\documentclass[oneside]{amsart}

\usepackage{amsthm}
\usepackage{amsmath}
\usepackage{amssymb}
\usepackage{enumerate}
\usepackage{graphicx}
\usepackage{booktabs}
\usepackage{color}
\usepackage[dvipsnames]{xcolor}
\usepackage{import}
\usepackage{tikz-cd}

\AtBeginDocument{%
   \def\MR#1{}
}

\usepackage{marginnote}
\long\def\@savemarbox#1#2{\global\setbox#1\vtop{\hsize\marginparwidth 
  \@parboxrestore\tiny\raggedright #2}}
\numberwithin{equation}{section}
\theoremstyle{plain}
\newtheorem{theorem}[equation]{Theorem}
\newtheorem{corollary}[equation]{Corollary}
\newtheorem{lemma}[equation]{Lemma}

\newtheorem{proposition}[equation]{Proposition}

\newtheorem*{namedtheorem}{\theoremname}
\newcommand{\theoremname}{testing}

\theoremstyle{definition}

\newtheorem{remark}[equation]{Remark}

\newtheorem{convention}[equation]{Convention}

\newcommand{\calT}{{\mathcal{T}}}

\newcommand{\cut}{\backslash\backslash}
\newcommand{\calH}{\mathcal{H}}
\newcommand{\calB}{\mathcal{B}}

\newcommand{\calF}{\mathcal{F}}

\makeatletter

\def\chaptermark#1{}

\def\chapter{%
  \if@openright\cleardoublepage\else\clearpage\fi
  \thispagestyle{plain}\global\@topnum\z@
  \@afterindenttrue \secdef\@chapter\@schapter}

\def\@chapter[#1]#2{\refstepcounter{chapter}%
  \ifnum\c@secnumdepth<\z@ \let\@secnumber\@empty
  \else \let\@secnumber\thechapter \fi
  \typeout{\chaptername\space\@secnumber}%
  \def\@toclevel{0}%
  \ifx\chaptername\appendixname \@tocwriteb\tocappendix{chapter}{#2}%
  \else \@tocwriteb\tocchapter{chapter}{#2}\fi
  \chaptermark{#1}%
  \addtocontents{lof}{\protect\addvspace{10\p@}}%
  \addtocontents{lot}{\protect\addvspace{10\p@}}%
  \@makechapterhead{#2}\@afterheading}
\def\@schapter#1{\typeout{#1}%
  \let\@secnumber\@empty
  \def\@toclevel{0}%
  \ifx\chaptername\appendixname \@tocwriteb\tocappendix{chapter}{#1}%
  \else \@tocwriteb\tocchapter{chapter}{#1}\fi
  \chaptermark{#1}%
  \addtocontents{lof}{\protect\addvspace{10\p@}}%
  \addtocontents{lot}{\protect\addvspace{10\p@}}%
  \@makeschapterhead{#1}\@afterheading}
\newcommand\chaptername{Chapter}

\def\@makechapterhead#1{\global\topskip 7.5pc\relax
  \begingroup
  \fontsize{\@xivpt}{18}\bfseries\centering
    \ifnum\c@secnumdepth>\m@ne
      \leavevmode \hskip-\leftskip
      \rlap{\vbox to\z@{\vss
          \centerline{\normalsize\mdseries
              \uppercase\@xp{\chaptername}\enspace\thechapter}
          \vskip 3pc}}\hskip\leftskip\fi
     #1\par \endgroup
  \skip@34\p@ \advance\skip@-\normalbaselineskip
  \vskip\skip@ }
\def\@makeschapterhead#1{\global\topskip 7.5pc\relax
  \begingroup
  \fontsize{\@xivpt}{18}\bfseries\centering
  #1\par \endgroup
  \skip@34\p@ \advance\skip@-\normalbaselineskip
  \vskip\skip@ }
\def\appendix{\par
  \c@chapter\z@ \c@section\z@
  \let\chaptername\appendixname
  \def\thechapter{\@Alph\c@chapter}}

\newcounter{chapter}

\newif\if@openright

\makeatother

\title[An upper bound on Reidemeister moves for each link type]{A polynomial upper bound on Reidemeister moves for each link type}
\author{Marc Lackenby}
\address{Mathematical Institute, University of Oxford, \newline Woodstock Road, Oxford OX2 6GG, United Kingdom}
\thanks{The author was partially supported by EPSRC grant EP/Y004256/1. For the purpose of open access, the author has applied a CC BY public copyright licence to any author accepted manuscript arising from this submission.}

\begin{document}

\begin{abstract} 
For each link type $K$ in the 3-sphere, we show that there is a polynomial $p_K$ such that any two diagrams of $K$ with $c_1$ and
$c_2$ crossings differ by at most $p_K(c_1) + p_K(c_2)$ Reidemeister moves. As a consequence, the problem of recognising
whether a given link diagram represents $K$ is in the complexity class NP and hence can be completed deterministically in exponential time. 
We calculate this polynomial $p_K$ explicitly for various classes  of links. 
\end{abstract}
\maketitle
\tableofcontents
\newpage

\section{Introduction}\label{Sec:Intro}

Haken, Hemion and Matveev showed that the equivalence problem for knots and links is soluble \cite{Haken, Hemion, Matveev}. But the complexity class
of this problem is not at all well understood. In particular, it is not known whether there is an
efficient algorithm to decide whether two links are the same. One of the most simple and attractive ways
of approaching this problem is to find an upper bound on the number of Reidemeister moves
required to pass between two diagrams of a link that is a computable function of the number
of crossings in each diagram. In \cite{CowardLackenby}, Coward and the author provided such a bound
and thereby gave a conceptually simple algorithm to solve the link equivalence problem. However, the bound
was huge: a tower of exponentials that is exponentially high. In this paper, we show that, if we fix the
link type, one can actually find an upper bound that is a polynomial function of the number of crossings.

\begin{theorem}
\label{Thm:Main}
For each link type $K$ in the 3-sphere, there is a polynomial $p_K$ with
following property. If $D_1$ and $D_2$ are two diagrams for $K$ with $c_1$ and $c_2$ crossings
respectively, then they differ by a sequence of at most $p_K(c_1) + p_K(c_2)$ Reidemeister moves.
\end{theorem}

This has the following algorithmic consequence. For a link type $K$, the \emph{$K$-recognition problem}
is the question whether a given link diagram represents $K$.

\begin{theorem}
\label{Thm:LinkDetectionNP}
For each link type $K$, the $K$-recognition problem lies in NP.
\end{theorem}

By definition, a problem is in the complexity class NP if, whenever an instance of the problem
has an affirmative answer, there is a certificate that provides this affirmation that can be
checked in polynomial time. In the case of the $K$-recognition problem, this certificate
takes a particularly simple form. It consists of a fixed diagram of $K$, together with a
sequence of Reidemeister moves taking the given input diagram to this fixed diagram.
We note that NP problems can be solved deterministically in exponential time, simply by checking
all possible certificates of polynomially bounded length. 

\begin{corollary}
\label{Cor:LinkDetectionExp}
For each link type $K$, the $K$-recognition problem can be solved deterministically in
exponential time, as a function of the number of crossings in a given link diagram.
\end{corollary}

Theorem \ref{Thm:Main} asserts the existence of the polynomial $p_K$. The following result
shows that it is algorithmically computable.

\begin{theorem}
\label{Thm:PolynomialComputable}
There is an algorithm that takes, as its input, a diagram of a link $K$ and outputs the polynomial $p_K$.
\end{theorem}

We implement this procedure when $K$ is the figure-eight knot, and obtain the following.

\begin{theorem}
\label{Thm:FigEight}
For $K$ the figure-eight knot, we may set $p_K(c) = (10^{108} c)^{15460896}$.
\end{theorem}

At the end of the paper, we will investigate torus links, and we will find the following explicit polynomial.

\begin{theorem}
\label{Thm:TorusKnotPoly}
When $K$ is a torus knot, we may set $p_K(c) = (10^{11}c)^{299666}$.
\end{theorem}

Again this result has the following algorithmic consequence, which was recently proved by Baldwin and Sivek \cite{BaldwinSivek} using other techniques. 

\begin{theorem}
\label{Thm:TorusKnotNP}
Deciding whether a diagram represents a torus knot is in NP.
\end{theorem}

\begin{proof}
We are given a diagram $D$ for a link $K$ with $c$ crossings. 
We need to show that when $K$ is a torus knot, then there is a certificate that establishes this. 
Suppose that $K$ is the $(p,q)$-torus knot, where $0 < |p| \leq q$.
Note that $K$ has a standard diagram $D'$ with $c'= q(|p|-1)$ crossings,
and this realises the crossing number of $K$, by a theorem of Murasugi \cite{Murasugi:Braid}.
Our certificate is a sequence of Reidemeister moves, with length bounded by a polynomial function of $c$,
taking $D$ to $D'$. By Theorem  \ref{Thm:TorusKnotNP}, 
the number of Reidemeister moves required to take $D$ to $D'$
is at most $(10^{11}c)^{299666} + (10^{11}c')^{299666} $.
Since $c' \leq c$, we deduce that the number of Reidemeister moves is at most $2(10^{11}c)^{299666}$.
\end{proof}

In general, the degree of $p_K$ is quite easy to compute. Without loss of generality, $p_K(c) = (ac)^d$, for constants $a$ and $d$. When $K$ is non-split, the degree $d$ is solely a function of the length of a certain type of
hierarchy for the exterior of $K$. But the coefficient $a$ is related to finer details of the hierarchy and also the link.

It remains a very interesting question whether there is a \emph{universal} polynomial $p$ that works for every link $K$. If there were such a polynomial, then this would imply that the link equivalence problem lies in NP. That would represent a very significant improvement on the known computational complexity of this important problem. (See \cite{Kuperberg} for the best known upper bound on its computational complexity.)

\subsection{Outline of the proof}
This paper can be viewed as a continuation of the paper \cite{Lackenby:PolyUnknot}. There, an explicit polynomial
upper bound on Reidemeister moves was provided for the unknot. A variety of techniques were used:
Dynnikov's work on arc presentations of the unknot \cite{Dynnikov}, normal surface theory, and a detailed analysis
of surfaces carried by Euclidean branched surfaces. An outline of the argument there is as follows. Given a diagram $D$
of the unknot, convert it to a rectangular diagram. This specifies an arc presentation for the unknot, in the
following sense. An open book decomposition for the 3-sphere is fixed, with binding circle an unknot and with pages that
are open discs. A link is an arc presentation if it intersects binding circle at finitely many points, and elsewhere
intersects each open page in either the empty set or a properly embedded arc.
(For the terminology related to rectangular diagrams and arc presentations, see \cite{Cromwell}, \cite{Dynnikov}, or Section \ref{Sec:ArcPresentations}).

The unknot bounds a spanning disc, which can be placed in a form that respects the arc presentation,
called admissible form. A good measure of the disc's complexity is the number of times that it intersects
the binding circle, which we term its \emph{binding weight}. Now, arising naturally from an arc presentation of a link, 
there is a polyhedral structure
of the link's exterior. By using well-known estimates on the complexity of normal surfaces in such a polyhedral structure, due to
Hass and Lagarias \cite{HassLagarias}, we obtain an exponential upper bound on the disc's binding weight. In \cite{Dynnikov},
Dynnikov described a collection of moves that can be used to reduce the binding weight
of the spanning disc without increasing the arc index of the presentation. Using this,
it is straightforward to find an exponential bound on the number of Reidemeister moves.
But one can improve this to a polynomial by using normal surface theory in a deeper way.
If the spanning disc has high binding weight, then one can find many parts of the surface
that are normally parallel. Using this observation, it was shown in \cite{Lackenby:PolyUnknot} that one can reduce
the binding weight of the spanning disc by a definite factor using a single move.
This then leads to the polynomial bound on Reidemeister moves.

Many of these ideas can be translated to the setting of general non-split links. The main thing that
is absent is a spanning disc. Instead, we use an entire hierarchy for the link exterior.
We show that, after polynomially many moves on the arc presentation, all the surfaces
in this hierarchy can be arranged to have polynomially bounded binding weight. The hierarchy
specifies a handle structure on the exterior of $K$. By reattaching a regular neighbourhood
of $K$ with a simple handle structure, we obtain a handle structure on the 3-sphere,
and it is easy to arrange that each component of $K$ runs through a 0-handle and a 1-handle. 
We pick a fixed diagram $D'$ for $K$. The aim is to convert $D$ to $D'$ using a controlled
number of Reidemeister moves. We pick a fixed copy of $D'$ drawn on a specified 0-handle.
This specifies a realisation $K'$ of the knot within this 0-handle. The next step in the argument is to pick a fixed ambient isotopy
within this handle structure taking $K$ to $K'$. This achieved using a series of moves, for example,
sliding part of the knot across a 2-handle. Each 2-handle is embedded within $S^3$
in a distorted way, but this distortion is controlled. So, when the knot is slid across this 2-handle,
one can bound the resulting number of Reidemeister moves in the diagrams obtained by projecting onto the
horizontal plane. In this way, we obtain a upper bound on the number of
Reidemeister moves required to pass between the given diagram $D$ and the fixed
diagram $D'$. This is a polynomial function $p_K(c(D))$, where $c(D)$ is the crossing number
of $D$. Therefore, if one is given two diagrams $D_1$ and $D_2$ of $K$ with crossing number
$c_1$ and $c_2$ respectively, the number of moves required to pass from one to the other via $D'$
is at most $p_K(c_1) + p_K(c_2)$.

\subsection{Structure of the paper} 
In Section \ref{Sec:NormalSurfaces}, we give an introduction to some aspects of normal surface theory,
and we prove some new results. Possibly the most significant of these is Theorem \ref{Thm:JSJFundamental}, which asserts
that JSJ surfaces in a compact orientable irreducible 3-manifold with incompressible boundary can always be realised as fundamental normal surfaces.
The main aim of this section is to show that many surfaces in a hierarchy are \emph{exponentially controlled}.
This means that for any triangulation of the manifold with $t$ tetrahedra, the surface can be arranged
to intersect the 1-skeleton at most $ab^t$ times, where $a$ and $b$ are constants depending only
on the 3-manifold and the surface, but not the triangulation. 
The results in Section \ref{Sec:NormalSurfaces} are then used in Section \ref{Sec:Hierarchies}, where the main 
aim is to show that the exterior of every non-split link, other than the unknot, admits an exponentially controlled hierarchy.
This means that, for every manifold appearing in the hierarchy with its inherited boundary pattern, the next surface 
in the hierarchy is exponentially controlled. We therefore fix such a hierarchy $H$ for our given link $K$.
Although Sections \ref{Sec:NormalSurfaces} and \ref{Sec:Hierarchies} are focused on triangulations,
it is also convenient to use handle structures. We therefore give a very brief introduction to handle structures
in Section \ref{Sec:HandleStructures}.

Section \ref{Sec:ArcPresentations} introduces our other main tools: arc presentations and rectangular diagrams.
It is primarily expositional, but contains a few new results. We mostly discuss admissible form for a surface,
and the resulting singular foliation that such a surface inherits. In Section \ref{Sec:AdmissibleHierarchies}, we 
consider a partial hierarchy for a link exterior where each surface is admissible form. We define a polyhedral
decomposition, a handle structure and a triangulation for the exterior of this partial hierarchy. In Section 
\ref{Sec:ExtendingHierarchy}, we consider the next surface $S$ in the hierarchy. This can be arranged to be normal, 
with an exponential upper bound on its weight, and a further modification makes it admissible.

In Section \ref{Sec:BranchedSurfaces}, we introduce branched surfaces. We show that the admissible
surface $S$ is naturally carried by a branched surface $B$. Each patch of $B$ comes from parallel tiles 
of $S$.

Section \ref{Sec:ModificationsAdmissible} mostly discusses the modifications to an admissible surface $S$
that were introduced by Dynnikov \cite{Dynnikov}, and which drew on work of Bennequin \cite{Bennequin}
and Birman-Menasco \cite{BirmanMenasco}. Here, one considers the singular foliation on $S$.
There are primarily three types of singularity: interior vertices (where the interior of the surface intersects the binding circle transversely),
interior saddles (where the interior of the surface is arranged like a saddle with respect to the circle-valued function on the
complement of the binding circle) and poles (where the surface has a local maximum or minimum with respect
to this height function). However, there are two further types of singularity on $\partial S$:
boundary vertices and boundary saddles. Coming out of each saddle, there are four leaves
of this foliation. Such leaves are called \emph{separatrices}. They typically end on a boundary
or interior vertex. The \emph{valence} of a vertex is the number of separatrices emanating from it.
The moves discussed by Dynnikov arise when there is an interior vertex of valence at most 3
or a boundary vertex of valence 1. He showed that when $K$ is the unknot in an arc presentation,
then any admissible spanning disc always has such a vertex. The moves then modify the link
and the spanning disc, and result in a reduction of the disc's binding weight. In Section \ref{Sec:ModificationsAdmissible},
we recall these moves, but we introduce one further modification. In the case where
there was a boundary vertex of valence 1, Dynnikov was able to slide the two arcs of
$K$ incident to this vertex across a subset of the disc, and thereby reduce the arc index of
$K$. Such a move is not appropriate in our context. The boundary of our surface $S$ runs
over previous surfaces in the hierarchy, and so one cannot simply move $\partial S$ across
a portion of $S$ and keep it properly embedded. Instead, we introduce an alternative move
called a wedge insertion.

Each of the above moves may only reduce the binding weight of $S$ by 1 or 2.
As in \cite{Lackenby:PolyUnknot}, a crucial part of our argument is to perform such moves
but with a much greater reduction in binding weight. This is possible when there is a collection
of vertices of $S$, each with low valence, and with parallel stars. Here, the \emph{star} of a vertex
is the union of the leaves of the singular foliation emanating from it. We define parallelism
in terms of the branched surface $B$. Therefore in Section \ref{Sec:Parallelism}, we 
revisit the moves from Section \ref{Sec:ModificationsAdmissible} in the context of many
low valence vertices with parallel stars. These modifications change the surface and also
the branched surface $B$. It is important that they do not make $B$ more complicated.
Therefore, we define a notion of complexity for $B$ and show that it does not increase.
In fact, we have to introduce a second branched surface $B'$, in order to control what happens
during wedge insertion. We package $B$ and $B'$, along with the admissible surface, in
a structure that we call an admissible envelope.

The above argument works well when there are `many' vertices with low valence and
having parallel stars. Here, `low' means `at most 3' for interior vertices. The definition of `low' valence for boundary vertices
is slightly more complicated.
Our task is to show that there are many such vertices, unless $S$ has polynomially bounded
binding weight. A fairly simple Euler characteristic argument, presented in Section \ref{Sec:EulerCharacteristic},
establishes that there are these vertices unless there are many vertices with valence 4 in the double $S_+$ of $S$.
If there are such 4-valent vertices in $S_+$ then their incident tiles patch together to form a subsurface of $S_+$
with a natural Euclidean metric. In Section \ref{Sec:Euclidean}, we consider this subsurface.
In \cite{Lackenby:PolyUnknot}, this subsurface was also considered. It was possible to show
that if the Euclidean subsurface was sufficiently big, then it contained a large rectangle or annulus
with that the property that opposite sides of the rectangle or annulus are parallel. It was
then argued that this rectangle or annulus could be patched together to form a normal torus
that was a normal summand for (some multiple of) the spanning disc. This contradicted an initial assumption that
the surface had no such summand. Here, we need to make a similar argument, but a somewhat
more technical one. In \cite{Lackenby:PolyUnknot}, it was possible to hypothesise that the spanning
disc had particularly nice properties as a normal surface, whereas in our situation all we have is that
$S$ is exponentially controlled. So we need to present a more technical variation of the argument in
\cite{Lackenby:PolyUnknot}. However, the conclusion is the same, that the Euclidean subsurface of $S_+$
cannot be too large, and hence there must in fact be many vertices of $S$ with parallel stars and 
low valence. 

We piece together these various parts of the argument in Section \ref{Sec:PolynomiallyBounded}.
We are given a diagram $D$ for $K$ with crossing number $c$. An isotopy of $D$ makes it into
a rectangular diagram with arc index $n$ at most $(81/20)c$. It then determines an arc presentation
for $K$. We have an exponentially controlled hierarchy $H$ for the exterior of $K$. In Theorem \ref{Thm:PolyBoundExchange},
we establish the existence of a polynomial $q$ so that $H$ may be arranged to have binding weight
at most $q(n)$, and this is achieved using at most $n^2 q(n)$ exchange moves and cyclic permutations of $K$.
This is achieved by arranging each of the surfaces of the hierarchy in turn to have polynomially bounded
binding weight.

This hierarchy $H$ induces a handle structure on the exterior of $K$. We then attach
a regular neighbourhood of $K$, to obtain a handle structure on the 3-sphere. We have arranged
that all the 2-handles are thin regular neighbourhoods of admissible surfaces, and that these have
polynomially bounded binding weight. In Sections \ref{Sec:IsotopeThroughHS}, \ref{Sec:IsotopyToReidemeister} and \ref{Sec:Explicit}, we isotope the
link $K$ through this handle structure until it lies within a 0-handle, and also so that its projection
to the horizontal plane in $\mathbb{R}^3$ is some fixed diagram $D'$ for $K$. By considering
the projection of $K$ to this plane, we need at most polynomially many Reidemeister moves to
convert the given diagram $D$ to $D'$. This then completes the proof of Theorem \ref{Thm:Main}.

In Section \ref{Sec:Fig8}, we run through the details of this construction for the figure-eight knot.
We find an explicit hierarchy $H$ that is exponentially-controlled, and from this it is quite straightforward
to compute the polynomial $q$ given above. We also bound the number of Reidemeister moves
needed when isotoping $K$ through the handle structure. This leads to the proof of Theorem
\ref{Thm:FigEight}, where an explicit polynomial upper bound for Reidemeister moves is given.
A similar analysis is given in Section \ref{Sec:TorusKnots} in the case of torus knots.

\subsection{Notation} If $M$ is a compact manifold and $X$ is a closed subset of $M$, the manifold $M \backslash \backslash X$
\emph{obtained by cutting $M$ along $X$} is the following compactification of $M - X$. We assign $M$ some Riemannian metric,
and then $M \cut X$ is the completion of $M - X$ with respect to its path metric. The inclusion
$M -X \rightarrow M$ induces a map $M \backslash \backslash X \rightarrow M$.

\section{Normal Surfaces}
\label{Sec:NormalSurfaces}

\subsection{Definition}
Let $\mathcal{T}$ be a triangulation of a compact 3-manifold $M$. Then a surface $S$ properly embedded in $M$ is \emph{normal} with respect to $\mathcal{T}$ if it intersects each tetrahedron in a collection of triangles and squares, as shown in Figure \ref{Fig:Normal}. Triangles and squares are known as \emph{elementary normal discs}.

\begin{figure}[h]
\includegraphics[width=3in]{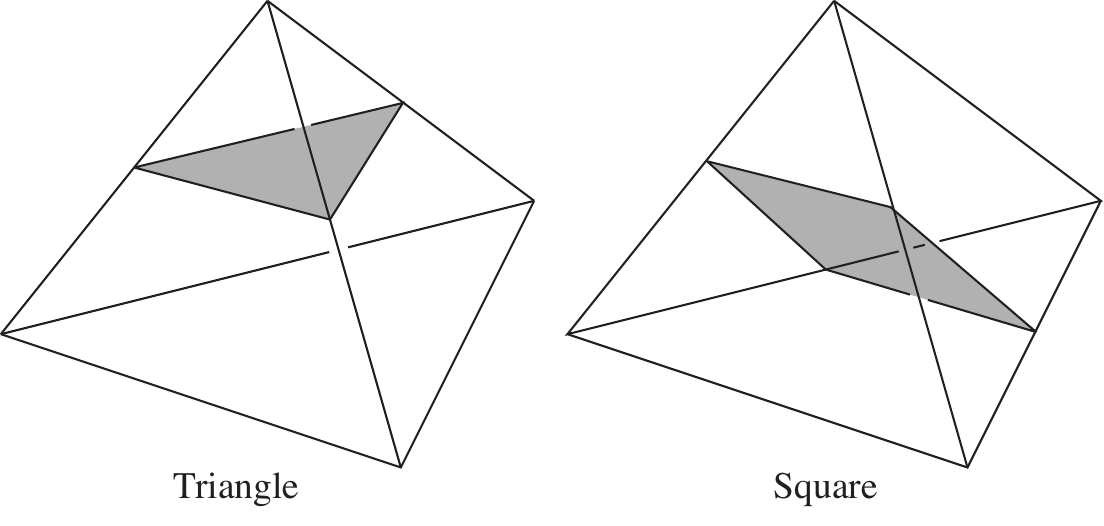}
\caption{Elementary normal discs in a tetrahedron}
\label{Fig:Normal}
\end{figure}

A normal surface $S$ can be encoded by a collection of non-negative integers as follows. One simply counts the
number of triangles and squares of each type in each tetrahedron. There are 7 types of triangle and square in
each tetrahedron. So, if $\mathcal{T}$ contains $t$ tetrahedra, then we obtain a list of $7t$ non-negative integers.
Each integer is a \emph{co-ordinate} of $S$ and the complete list is called the \emph{normal surface vector} for $S$ and is denoted $(S)$.

\subsection{The weight of a surface}

Let $M$ be a compact 3-manifold with a triangulation $\mathcal{T}$. Let $S$ be a properly embedded
surface that is disjoint from the vertices of $\mathcal{T}$ and that intersects the edges transversely.
The \emph{weight} $w(S)$ of $S$ is its number of points of intersection with the edges.

\subsection{Boundary patterns}

We will need to conduct normal surface theory in the presence of boundary patterns.
Therefore, in this subsection, we recall some of the terminology relating to boundary patterns.

A \emph{boundary pattern} $P$ for a compact orientable 3-manifold $M$
is a (possibly empty) collection of disjoint simple closed curves and trivalent graphs
embedded in $\partial M$.

A surface embedded in $M$ is \emph{clean} if it is disjoint from $P$.

We say $(M,P)$ is \emph{boundary-irreducible} if for any clean disc $D$ properly embedded
in $M$, there is a clean disc $D'$ in $\partial M$ such that $\partial D' = \partial D$.
This is equivalent to the requirement that $\partial M - P$ is incompressible.

Let $S$ be a properly embedded surface in $M$ transverse to $P$. A \emph{boundary-compression disc}
for $S$ (with respect to the pattern $P$) is a clean disc $D$ embedded in $M$ such that 
\begin{enumerate}
\item $D \cap S$ is an arc $\alpha$ in $\partial D$;
\item $\partial D \cut \alpha = D \cap \partial M$;
\item $\alpha$ does not separate off a component of $S \cut \alpha$ that is a clean disc.
\end{enumerate}
The surface $S$ is \emph{boundary-compressible} (with respect to the pattern) if it admits a boundary-compression disc;
otherwise it is \emph{boundary-incompressible}.

We say that a \emph{pattern-isotopy} of a properly embedded surface $S$ in $M$ 
is an ambient isotopy that preserves the boundary pattern throughout. When two surfaces differ by 
such an isotopy, we say that they are \emph{pattern-isotopic}.

Two surfaces properly embedded in $M$ are \emph{strongly equivalent} if there is a homeomorphism
of $M$, whose restriction to $\partial M$ is pattern-isotopic to the identity, taking one surface to the other.

\subsection{Boundary patterns and normal surfaces}

It is well known that incompressible boundary-incompressible surfaces may be placed in normal form, under rather mild hypotheses.
The following version of this result, which is Corollary 3.3.25 in \cite{Matveev}, achieves this in the presence of boundary patterns.

\begin{proposition}
\label{Prop:IsotopicToNormal}
\sloppy
Let $(M,P)$ be a compact orientable irreducible boundary-irreducible 3-manifold
with boundary pattern. Let $\mathcal{T}$ be a triangulation of $M$ in which
$P$ is simplicial. Let $S$ be an incompressible, boundary-incompressible surface properly embedded in $(M,P)$,
other than a 2-sphere or a boundary-parallel disc intersecting $P$ in zero or two points.
Then $S$ is pattern-isotopic to a surface $S'$ that is normal with respect to $\mathcal{T}$.
Moreover, we can ensure that the weight of $S'$ is at most the weight of $S$. 
\end{proposition}

\subsection{Annuli and tori in $M$}

We present some definitions from \cite{Matveev}.

Let $M$ be a compact orientable 3-manifold with a boundary pattern $P$.

An annulus $A$ in $\partial M$ is \emph{almost clean} if its intersection with $P$ is
a (possibly empty) collection of disjoint core curves in the interior of $A$.

An annulus $A$ in $\partial M$ is a \emph{necklace annulus} if $P$ is disjoint from $\partial A$ and separates the components
of $\partial A$ and there is an arc in $A$ joining the two components of $\partial A$ such that
$A \cap P$ is a single point.

A clean annulus properly embedded in $M$ is \emph{essential} if it is incompressible and boundary-incompressible (with respect to the pattern)
and is not parallel to an almost clean annulus in $\partial M$. Note that when $(M,P)$ is irreducible and boundary-irreducible, then a clean
annulus is essential if and only if it is incompressible and not parallel to an almost clean annulus in $\partial M$.

\begin{figure}[h]
\includegraphics[width=3in]{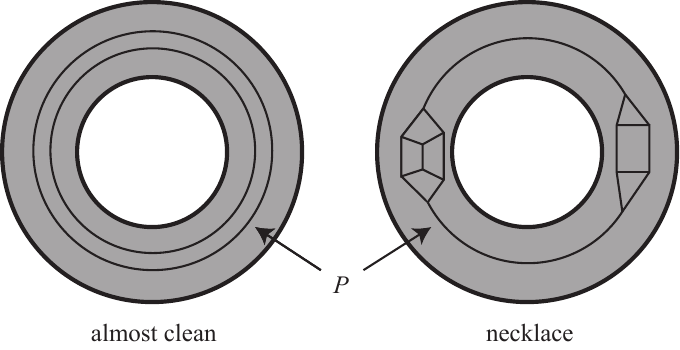}
\caption{Almost clean annulus and necklace annulus in $\partial M$}
\label{Fig:AlmostCleanAndNecklace}
\end{figure}

A torus $T$ properly embedded in $(M,P)$ is \emph{essential} if it is incompressible and not parallel to a torus in $\partial M$. (In this definition, it is not relevant whether this torus in $\partial M$ is clean or not.)

We say that $(M,P)$ is \emph{simple} if it is irreducible and boundary-irreducible and contains no properly embedded essential tori and no properly embedded essential clean annuli.

\subsection{Normal summation and fundamental surfaces}
Let $M$ be a compact orientable 3-manifold with a triangulation $\mathcal{T}$.

We say that a properly embedded normal surface $S$ is the \emph{sum} of two properly embedded
normal surfaces $S_1$ and $S_2$ if their
normal surface vectors satisfy
$(S) = (S_1) + (S_2)$. We write $S = S_1 + S_2$.

A normal surface $S$ is a \emph{fundamental surface} if it cannot be written as
the sum of two non-empty normal surfaces.

The following was proved by Hass and Lagarias in \cite{HassLagarias}.

\begin{proposition}
\label{Prop:BoundOnFundamental}
Let $M$ be a compact 3-manifold and let $\mathcal{T}$ be a triangulation
of $M$ with $t$ tetrahedra. Then, for any fundamental normal surface $S$, each coordinate of $S$
is at most $t2^{7t+2}$.
\end{proposition}

We note the following immediate consequence.

\begin{corollary}
\label{Cor:FundamentalWeight}
Let $M$, $\mathcal{T}$ and $t$ be as above. Then for any fundamental normal surface $S$,
$w(S) \leq t^2 2^{7t + 6}$.
\end{corollary}

\subsection{A topological interpretation of normal summation}

There is a topological way of viewing normal summation, which is as follows.

Suppose that $S$ is a properly embedded normal surface that is a sum of normal surfaces $S_1$ and $S_2$.
We may perform a small pattern-isotopy to $S_1$, keeping it normal throughout, so that
afterwards $S_1$ and $S_2$ intersect transversely in a collection of properly
embedded simple closed curves and arcs. Then $S$ is obtained from
$S_1 \cup S_2$ by removing a regular neighbourhood of $S_1 \cap S_2$,
and then attaching annuli, M\"obius bands and discs which join $S_1$ and $S_2$
together in this regular neighbourhood. Along each curve and arc of $S_1 \cap S_2$,
there are two possible ways to perform this smoothing, but only one correctly
gives the normal surface $S$. This is called a \emph{regular switch}. If the smoothing
is performed the other way, it is called an \emph{irregular switch}. In a regular neighbourhood of each component
of $S_1 \cap S_2$, there is an $I$-bundle over this component properly embedded in
$M \cut S$ called a \emph{trace surface}. It is an annulus, disc or M\"obius band.
A \emph{patch} is a component of $(S_1 \cup S_2) \cut (S_1 \cap S_2)$.

\begin{figure}[h]
\includegraphics[width=3.7in]{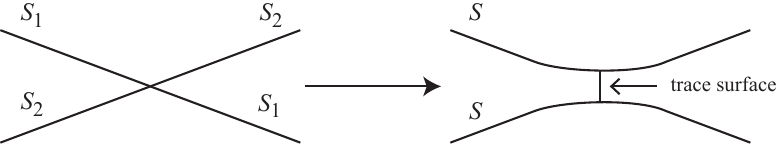}
\caption{Normal summation}
\label{Fig:NormalSummation}
\end{figure}

A normal sum $S = S_1 + S_2$ of two surfaces is \emph{in reduced form} if $S$ cannot be written as a sum $S'_1 + S'_2$ where each $S'_i$ is pattern-isotopic to $S_i$ and $S'_1 \cap S'_2$ has fewer components than $S_1 \cap S_2$.

There is a well-established theory of summation of normal surfaces, and we record some
of its main results for our future reference. These are essentially proved in 4.1.8, 4.1.36, 6.3.20 and 6.3.21 in \cite{Matveev}.

\begin{theorem}
\label{Thm:SummandsEssentialEtc}
Let $(M,P)$ be a compact orientable irreducible boundary-irreducible 3-manifold with boundary pattern.
Let $\mathcal{T}$ be a triangulation of $M$ in which $P$ is simplicial. 
Let $S$ be a properly embedded, incompressible, boundary-incompressible, normal surface that has least possible weight,
among all normal surfaces that are pattern-isotopic to it. Assume either that $S$ is connected or that no component of $S$ is a sphere, clean disc, projective plane, boundary-parallel torus, or clean inessential annulus.
Suppose that $S$ can be expressed as a normal sum $S_1 + S_2$,
where $S_1$ and $S_2$ are in general position and non-empty, and the summation is in reduced form. Then the following hold.
\begin{enumerate}
\item No patch is a disc disjoint from $\partial M$ or a disc that intersects $\partial M$ in a single clean arc.
\item Every patch is incompressible and boundary-incompressible (with respect to the pattern $P$).
\item The trace surfaces are incompressible in $M \cut S$.
\item $S_1$ and $S_2$ are incompressible and boundary-incompressible (with respect to the pattern $P$).
\item No component of $S_1$ or $S_2$ is a sphere, clean disc or projective plane. 
\item If $S$ is orientable, then no component of $S_1$ or $S_2$ is a clean annulus that is parallel to an almost clean annulus in $\partial M$.
\item If $S$ is orientable and has least weight up to strong equivalence, then no component of $S_1$ or $S_2$ can be parallel to a toral component $T$ of $\partial M$
such that $|S \cap T| \leq 1$.
\item If $S$ is orientable, then no component of $S_1$ or $S_2$ can be parallel to a toral component $T$ of $\partial M$
disjoint from $P$.
\end{enumerate}
\end{theorem}

Note that (3) is not explicitly stated in \cite{Matveev} but it is an immediate consequence of (1) and (2). This is because a compression disc for a
trace surface must be parallel to the boundary of the surface, and hence can be used to create a compression disc
for an incident patch. Then by (2) this patch is a disc disjoint from $\partial M$, contradicting (1).

Proposition 4.1.36 in \cite{Matveev}, which gives (4) and (5), has the assumption that $S$ is connected.
However, this assumption is only required to rule out the possibility that $S$ has a component that
is a sphere, clean disc or projective plane. We therefore make this alternative hypothesis in the above theorem.
Similarly, Propositions 6.3.20 and  6.3.21 in \cite{Matveev}, which give (6) and (7), also assume that $S$ is connected.
However, this assumption is made in order to apply Proposition 4.1.36 in \cite{Matveev} and to exclude
the possibility that a component of $S_1$ or $S_2$ is a boundary-parallel toral component of $S$ or a clean
inessential annulus component of $S$. However, we
assume that $S$ has no such component.

Note also that (8) is slightly different from Proposition 6.3.21 in \cite{Matveev}. There, it is assumed that $|T \cap S| \leq 1$, when
aiming to establish that no component of $S_1$ or $S_2$ can be parallel to a toral component $T$ of $\partial M$. 
It is also assumed that $S$ has least weight up to strong equivalence. However, we note that
when $T$ is disjoint from $P$, then the argument of Proposition 6.3.21 in \cite{Matveev} implies that
there is a pattern isotopy that reduces the weight of $S$. The argument reduces to the case
where $S_1$ and $S_2$ are arranged near $T$ as shown in Figure \ref{Fig:IsotopyBoundaryParallelTorus}, 
with $S'_2$ a component of $S_2$ that is parallel to $T$,
and the summation $S_1 + S_2$ is in the same direction along each curve of $S_1 \cap S'_2$. Hence, there is
an isotopy that takes $S$ to $S_1 + (S_2 - S_2')$, thereby reducing its weight. This isotopy may move the
curves $S \cap T$, but this is nevertheless a pattern isotopy because $T$ is disjoint from $P$.

\begin{figure}[h]
\includegraphics[width=4.5in]{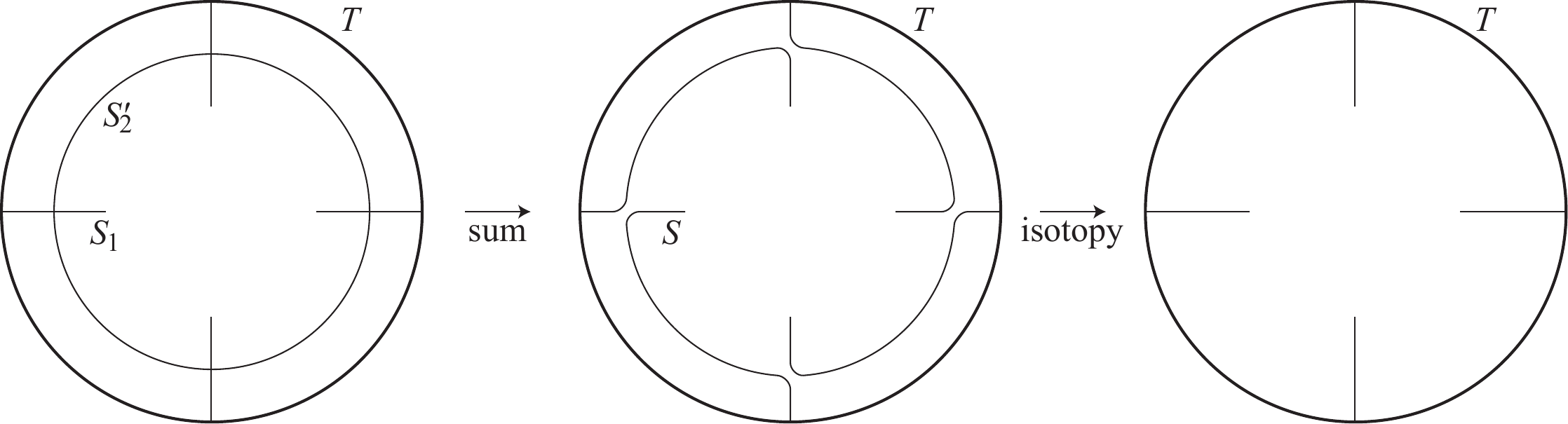}
\caption{Isotopy of $S$ when a summand has a boundary-parallel torus component $S_2'$}
\label{Fig:IsotopyBoundaryParallelTorus}
\end{figure}

\newpage
\subsection{Components of least weight surfaces}

\begin{theorem}
\label{Thm:ComponentsLeastWeight}
Let $(M,P)$ be a compact orientable irreducible boundary-\linebreak irreducible
3-manifold with boundary pattern. Let $\mathcal{T}$ be a triangulation of $M$ in which $P$ is simplicial.
Let $S$ be an incompressible, boundary-incompressible, orientable normal surface properly embedded
in $(M, P)$. Then $S$ has least weight in its pattern-isotopy class if and only if each component of $S$ does.
\end{theorem}

\begin{proof}
Suppose that $S$ is pattern-isotopic to a normal surface $S'$ with smaller weight.
Then some component of $S'$ has smaller weight than the corresponding component of $S$.
Hence, we deduce that this component of $S$ did not have least weight.

Suppose now that $S$ has least weight. Suppose that some component $S_1$ of $S$ does not
have least weight. Let $S_2$ be the remaining components of $S$. Let $S'_1$ and $S'_2$
be normal surfaces that have least weight among surfaces pattern-isotopic to
$S_1$ and $S_2$. Hence, $w(S'_1) < w(S_1)$
and $w(S'_2) \leq w(S_2)$. Now, by Proposition 3.7 in \cite{Jaco-Rubinstein:PL-equivariant-surgery},
we may form the normal sum
$S' = S'_1 + S'_2$ and this surface $S'$ is pattern-isotopic to $S$. 
But $w(S') = w(S'_1) + w(S'_2) < w(S_1) + w(S_2) = w(S)$,
which contradicts the assumption that $S$ had least weight.

Note that Proposition 3.7 in \cite{Jaco-Rubinstein:PL-equivariant-surgery} was not stated for 3-manifolds with boundary patterns. However,
the presence of the simplicial boundary pattern $P$ does not affect the argument there and so
Proposition 3.7 in \cite{Jaco-Rubinstein:PL-equivariant-surgery} still applies.
\end{proof}

\subsection{Exponentially controlled and weakly fundamental surfaces}
We will need to consider two types of normal surface that are less constrained than a fundamental surface.
We term these weakly fundamental surfaces and exponentially controlled surfaces. Such surfaces will play an important role in this paper.

Let $M$ be a compact orientable 3-manifold with a boundary pattern $P$. Let $\mathcal{T}$ be a triangulation of $M$
in which $P$ is simplicial. We say that a normal surface $S$ is \emph{weakly fundamental} if 
\begin{enumerate}
\item no component of $S$ can be expressed as the
sum of two non-empty normal surfaces $S_1$ and $S_2$, where $S_2$ is disjoint from $P$ and is a sphere, projective plane, disc, torus 
or annulus, and
\item no component of $S$ can be expressed as a sum $S_1 + S_2 + S_3$ of non-empty normal surfaces where $S_2$ and $S_3$ are M\"obius bands disjoint
from $P$.
\end{enumerate}

\begin{proposition}
\label{Prop:WeaklyFundamentalBound}
Let $M$ be a compact orientable 3-manifold with a boundary pattern $P$. Let $\mathcal{T}$ be a triangulation of $M$
in which $P$ is simplicial, and let $t$ be the number of tetrahedra in $\mathcal{T}$.
Then, for any weakly fundamental normal surface $S$ with no 2-sphere components, 
$$w(S) \leq t^2 2^{7t+6} (-\chi(S) + 2 |S \cap P| + 2|S|).$$
\end{proposition}

\begin{proof}
Let $\{ S_i \}$ be the set of fundamental normal surfaces. We write $S$ as a sum of its components,
and then we write each component as a sum of fundamental normal surfaces. We thereby obtain an expression
$$S = \sum_i n_i S_i$$
where each summand is identified with a summand of some component of $S$.
By disregarding the surfaces $S_i$ that do not appear in this sum, we may assume that $n_i > 0$ for all $i$.
Weight, Euler characteristic and the number of intersections with $P$ are all additive with respect to normal summation, and therefore
$$w(S) = \sum_i n_i w(S_i),$$
$$-\chi(S) = - \sum_i n_i \chi(S_i),$$
$$|S \cap P| =  \sum_i n_i |S_i \cap P|.$$
Hence,
$$-\chi(S) + 2|S \cap P| = \sum_i n_i (- \chi(S_i) + 2 |S_i \cap P|).$$
Since $S$ is weakly fundamental, $-\chi(S_i) + 2 |S_i \cap P| \geq 1$, for each $i$ with at most $2$ exceptions for each component of $S$.
The exceptions occur when $S_i$ is a torus, projective plane, disc or annulus disjoint from $P$ (in which case
at most one $S_i$ can be of this form and this is the only summand for this component of $S$) or a M\"obius band
(in which case at most two $S_i$ can be of this form for this component of $S$). We deduce that
$$\sum_i n_i \leq -\chi(S) + 2 |S \cap P| + 2|S|.$$
Corollary \ref{Cor:FundamentalWeight} gives a bound on the weight of each
fundamental normal surface. Hence,
$$w(S) \leq t^2 2^{7t+6}(-\chi(S) + 2 |S \cap P| + 2|S|),$$
as required. 
\end{proof}

Let $M$ be a compact orientable 3-manifold with a boundary pattern $P$. For constants $c$ and $k$, 
we say that a surface $S$ properly embedded
in $M$ is \emph{$(c,k)$-exponentially controlled} if, for any triangulation 
$\mathcal{T}$ of $M$ in which $P$ is simplicial and having $t$ tetrahedra, $S$ is strongly equivalent to a surface $S'$ satisfying $w(S') \leq ck^t$.
If a surface is $(c,k)$-exponentially controlled for some $c$ and $k$, we say that it is \emph{exponentially controlled}.

An immediate consequence of Proposition \ref{Prop:WeaklyFundamentalBound} is the following.

\begin{corollary}
\label{Cor:WeaklyFundExpControlled}
Let $(M,P)$ be a compact orientable 3-manifold with boundary pattern. Let $S$ be a compact surface properly embedded in $M$.
Suppose that for any triangulation $\mathcal{T}$ of $M$ in which $P$ is simplicial, $S$ is strongly equivalent to a weakly fundamental
normal surface. Then $S$ is $(c,k)$-exponentially controlled, where $c = -\chi(S) + 2 |S \cap P| + 2|S|$ and $k = 2^{15}$.
\end{corollary}

\subsection{Making surfaces weakly fundamental}
Many topologically relevant surfaces may be placed in this form, as the following result demonstrates.

\begin{theorem}
\label{Thm:MakeWeaklyFundamental}
Let $(M,P)$ be a compact orientable irreducible boundary-irreducible 3-manifold with boundary pattern. Let $\mathcal{T}$ be a triangulation of $M$
in which $P$ is simplicial. Suppose that $M$ contains no properly
embedded essential tori or essential clean annuli.  
Let $S$ be an incompressible, boundary-incompressible, orientable surface properly embedded in $(M,P)$.
Suppose that each incompressible toral component of $\partial M$ is disjoint from $P$.
Suppose also that no component of $S$ is a sphere, a clean disc, a clean annulus or a boundary-parallel disc intersecting $P$ twice.
Then $S$ is pattern-isotopic to a weakly fundamental normal surface.
\end{theorem}

\begin{proof}
By Proposition \ref{Prop:IsotopicToNormal}, $S$ is pattern-isotopic to a normal surface. 
We may assume that $S$ has smallest weight, up to pattern isotopy. By Theorem \ref{Thm:ComponentsLeastWeight},
each component of $S$ has least weight in its pattern-isotopy class.
We will show that each component $S'$ of $S$ is weakly fundamental, which will prove the theorem.

Suppose that $S'$ is not weakly fundamental. Assume first that it can be expressed as $S_1 + S_2$,
where $S_2$ is disjoint from $P$ and is a sphere, projective plane, disc, torus or annulus. We may
assume that this summation is in reduced form. However, by Theorem \ref{Thm:SummandsEssentialEtc},
$S_2$ cannot be a sphere, projective plane, clean disc, clean annulus or torus. For example, suppose that
$S_2$ is a torus. By Theorem \ref{Thm:SummandsEssentialEtc} (4), it is incompressible. By assumption, it must be
parallel to a toral boundary component that is disjoint from $P$.
But this is ruled out by Theorem \ref{Thm:SummandsEssentialEtc} (8).

Suppose that $S'$ can be expressed as can be expressed as $S_1 + S_2 + S_3$ where $S_2$
and $S_3$ are M\"obius bands disjoint from $P$. By Theorem \ref{Thm:SummandsEssentialEtc} (4),
$S_2$ and $S_3$ are incompressible. The orientable double cover of each of the M\"obius bands is a properly embedded
clean incompressible annulus, which must therefore be parallel to an almost clean annulus in $\partial M$. 
Therefore, $M$ is a solid torus and $P$ is a (possibly empty) collection of essential non-meridional curves.
We may give $M$ a Seifert fibration in which $P$ is a union of fibres.
Now, $[S_2]$ and $[S_3]$ are both non-trivial elements of $H_2(M, \partial M; \mathbb{Z}/2\mathbb{Z}) \cong \mathbb{Z}/2\mathbb{Z}$ 
and so their sum is homologically trivial. The sum $S_2 + S_3$ is incompressible and boundary-incompressible
by Theorem \ref{Thm:SummandsEssentialEtc} (4). It is also orientable, since it is homologically
trivial. It is therefore pattern-isotopic to a surface that is vertical or horizontal, plus possibly some boundary-parallel
components. If any component is vertical, then it is a
clean annulus, which is inessential, contradicting Theorem \ref{Thm:SummandsEssentialEtc} (6). 
If any component is boundary-parallel, it is a disc intersecting the pattern twice, since
any other connected boundary-parallel surface is either a clean disc or boundary-compressible,
contradicting Theorem \ref{Thm:SummandsEssentialEtc} (5) or (4). If any component is horizontal, it
is a meridian disc. Each component of $S_2 + S_3$ is therefore a disc. This is impossible since the Euler characteristic of $S_2 + S_3$ is zero.
\end{proof}

\subsection{JSJ surfaces}

Let $(M,P)$ be a compact orientable irreducible boundary-irreducible 3-manifold with boundary pattern. 

Let $F$ be an incompressible torus or an incompressible boundary-incompressible clean annulus properly embedded in $(M, P)$. Then $F$ is \emph{rough} if any incompressible torus and any clean incompressible boundary-incompressible  annulus in $M$ is pattern-isotopic to a surface disjoint from $F$. This is Definition 6.4.12 of \cite{Matveev}.

A surface $F$ properly embedded in $M$ is \emph{JSJ} if it is either a rough torus not parallel to a clean torus in $\partial M$, or a rough clean annulus not parallel to an almost clean annulus in $\partial M$. This is Definition 6.4.29 of \cite{Matveev}. 

It is possible to show that any collection of JSJ surfaces can be pattern-isotoped to be pairwise disjoint. Hence, the \emph{JSJ surfaces} are obtained by taking one representative of each JSJ surface, up to pattern-isotopy, and arranging them to be disjoint from each other. The resulting surface is in fact well-defined up to pattern-isotopy, when $(M,P)$ is irreducible and boundary-irreducible. (See Theorem 6.4.31 in \cite{Matveev}.)

Let $F$ be the JSJ surfaces for $(M,P)$. Give $M \cut F$ the boundary pattern consisting of the union of $\partial F$ and $(M \cut F) \cap P$. Then each component of $M \cut F$ with this boundary pattern is a \emph{JSJ piece}. (Matveev \cite{Matveev} calls them \emph{JSJ-chambers}.)

We now introduce two types of properly embedded clean essential annulus. A clean essential annulus $A$ properly embedded in $M$ is \emph{longitudinal} if any other clean essential annulus $A'$ can be pattern-isotoped so that $\partial A \cap \partial A' = \emptyset$. Otherwise $A$ is \emph{transverse}. This is Definition 6.4.13 in \cite{Matveev}. Note that any JSJ annulus is always longitudinal and never transverse.

We now recall Theorem 6.4.42, Lemma 6.4.32, Proposition 6.4.35 and Proposition 6.4.41 from \cite{Matveev}, which combine to give the following result.

\begin{theorem}
\label{Thm:JSJPieces}
Let $(M,P)$ be a compact orientable irreducible boundary-irreducible 3-manifold manifold with boundary pattern. Then its JSJ pieces fall into the following types:
\begin{enumerate}
\item simple 3-manifolds with pattern;
\item Seifert fibred manifolds with pattern consisting of fibres;
\item $I$-bundles, with pattern disjoint from the horizontal boundary and intersecting each vertical boundary component in a non-empty collection of parallel copies of its core curve.
\end{enumerate}
Moreover, if a JSJ piece contains an essential torus or a clean essential longitudinal annulus, it is of the second type. If a JSJ piece contains a clean essential transverse annulus, it is of the third type.
\end{theorem}

\begin{remark}\label{Rem:IsotopeOffMob}
Note that a JSJ annulus or torus can be pattern-isotoped off any properly embedded clean M\"obius
band or Klein bottle $F$, for the following reason. The orientable double cover $\tilde F$ is a properly embedded
clean annulus or torus. We claim that this is incompressible. Otherwise it would compress to two clean
discs or to a 2-sphere. In the case of two discs, the boundary-irreducibility of
$(M,P)$ implies that the boundary of each of these discs would bound a clean disc in $\partial M$.
One of these discs in $\partial M$ can be extended so that its boundary is $\partial F$.
But then the union of this disc and $F$ is a projective plane. The irreducibility of $M$
implies that it is a copy of $\mathbb{RP}^3$, which contradicts the fact that
$M$ has non-empty boundary. When $\tilde F$ compresses to a sphere, this bounds
ball in $M$ that does not contain the Klein bottle. Hence, again $M$ is closed, which is
a contradiction. This proves the claim.
Therefore any JSJ annulus or torus can be pattern-isotoped off $\tilde F$. It is therefore
disjoint from the regular neighbourhood $N(F)$ or lies within it. In the latter case, it can also
be pattern-isotoped off $F$. For we may give $N(F)$ the structure of an $I$-bundle or circle bundle
over a M\"obius band in which $F$ is a union of fibres. Then the JSJ torus or annulus
can also be made a union of fibres and hence is an annulus or torus parallel to $\tilde F$.
\end{remark}

\subsection{Connected JSJ surfaces are fundamental}

Many topologically relevant surfaces can be made fundamental. In particular, this is true of the JSJ annuli
and tori, as follows. 

\begin{theorem}
\label{Thm:JSJFundamental}
Let $(M,P)$ be a compact orientable irreducible boundary-irreducible 3-manifold with boundary pattern.
Let $\mathcal{T}$ be a triangulation of $M$ in which $P$ is simplicial. Then each connected JSJ surface 
that has least weight in its pattern-isotopy class is fundamental.
\end{theorem}

Our proof will use the following notion. We say that
two surfaces $S_1$ and $S_2$ properly embedded in $M$ are \emph{partially coincident} if $S_1 \cap S_2$ is a compact subsurface of $S_1$ and $S_2$.
We say that they have \emph{essential intersection} if no component of $S_1 \cap S_2$ or $(S_1 \cup S_2) \cut (S_1 \cap S_2)$ is a disc disjoint from $\partial M$ or a disc intersecting $\partial M$ in a single arc.

When at least one of $S_1$ and $S_2$ is a torus and they have essential intersection, then $S_1 \cap S_2$ consists of annuli or possibly a torus. When $S_1$ and $S_2$ are both clean annuli and they have essential intersection, then either every component of $S_1 \cap S_2$ is an annulus parallel to a core curve of $S_1$ and $S_2$ or every component of $S_1 \cap S_2$ is a \emph{square}, which is a disc consisting of an arc in $\partial S_i$, an essential arc properly embedded in $S_i$, another arc in $\partial S_i$ and another essential arc properly embedded in $S_i$.

A \emph{generalised product region} is a connected 3-manifold of the form $(F \times [-1,1]) / \! \sim$, where $\sim$ is an equivalence relation with each equivalence classes equal either to a singleton or to an interval of the form $x \times [-1,1]$ for points $x \in J$, where $J$ is a finite union of closed intervals and circles in $\partial F$. A \emph{horizontal boundary component} is the image of $F \times \{1\}$ or $F \times \{ -1 \}$. A \emph{vertical boundary component} is the closure of the image of a component of $(\partial F - J) \times [-1,1]$. The \emph{pinching locus} is the image of $J \times [-1, 1]$.

\begin{figure}[h]
\includegraphics[width=4in]{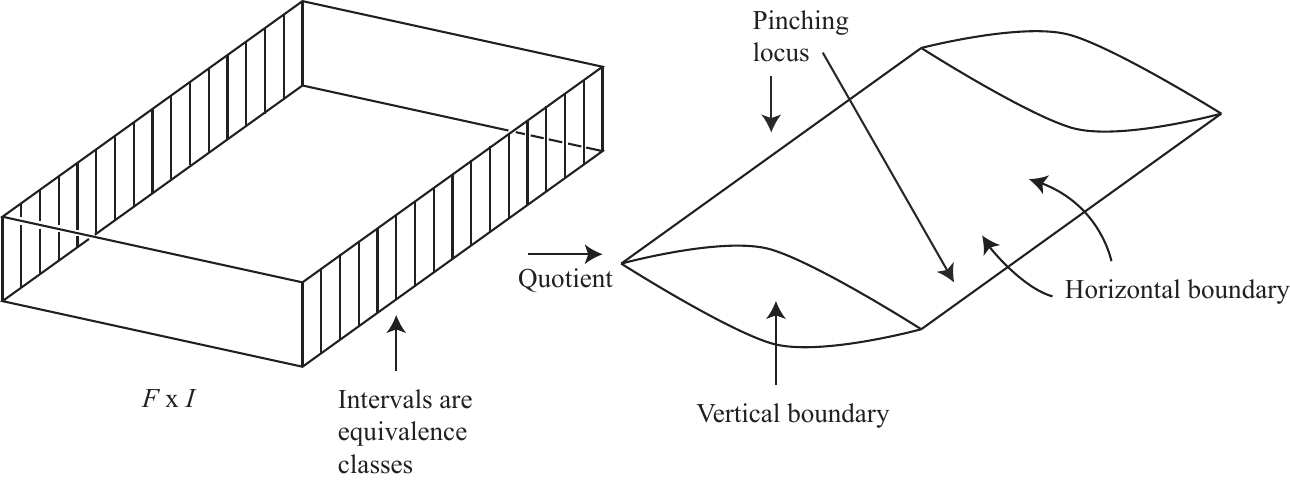}
\caption{A generalised product region is obtained from a product by collapsing certain arcs in the boundary}
\label{Fig:GeneralisedProductRegion}
\end{figure}

\begin{lemma}
\label{Lem:ExistenceGeneralisedProduct}
Let $S_1$ and $S_2$ be surfaces properly embedded in a compact orientable 3-manifold $M$ with boundary pattern $P$, where each $S_i$ is an incompressible boundary-incompressible
torus or clean annulus. Suppose that they are partially coincident but not equal, that they have essential intersection, and that there is a pattern-isotopy taking $S_1$ off $S_2$. Then there is a component $X$ of  $M \cut (S_1 \cup S_2)$ that is a generalised product region $X$ with the following properties:
\begin{enumerate}
\item each horizontal boundary component lies in $S_1$ or in $S_2$;
\item each vertical boundary component lies in $\partial M$ and is disjoint from $P$;
\item the pinching locus is a union of components of $S_1 \cap S_2 \cap X$.
\end{enumerate}
It is possible that other components of $S_1 \cap S_2$ lie in the horizontal boundary.
It is also possible that the inclusion of the product region into $M$ is not injective, since components of $S_1 \cap S_2$ in the boundary of the product region may be identified.
\end{lemma}

\begin{proof}
Note that each component $F$ of $S_1 \cap S_2$ is a square or annulus. We say that $F$ is \emph{removable} if the two components of $S_1 - S_2$ incident to $\partial F$ emanate from the same side of $S_2$ (or equivalently, the two components of $S_2 - S_1$ incident to $\partial F$ emanate from the same side of $S_1$). Otherwise $F$ is \emph{non-removable}. (See Figure \ref{Fig:PartiallyCoincidentIsotopy}.)

\begin{figure}[h]
\includegraphics[width=4in]{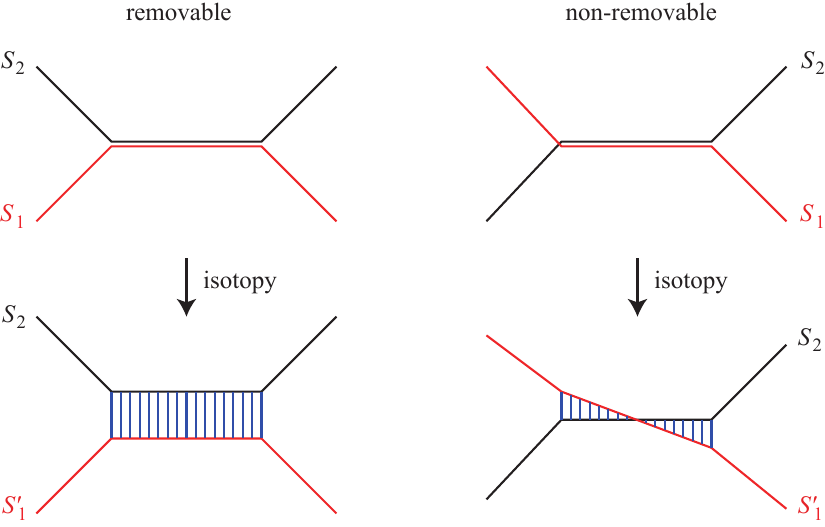}
\caption{Isotoping partially coincident surfaces $S_1$ and $S_2$ so that they intersect transversely}
\label{Fig:PartiallyCoincidentIsotopy}
\end{figure}

We now perform a small isotopy to $S_1$ so that the resulting surface $S'_1$ intersects $S_2$ transversely, where this intersection consists of a single arc or curve in each non-removable component of $S_1 \cap S_2$. Note that we may obtain the original arrangement $S_1$ and $S_2$ by collapsing some product regions as follows. For each removable component $F$ of $S_1 \cap S_2$, we have a copy of $F \times [-1,1]$ in $M$, with $(F \times [-1,1]) \cap (S'_1 \cup S_2) = F \times \{ -1,1 \}$. If we collapse the each interval $\{ \ast \} \times [-1,1]$ to a point, we obtain the original arrangement of $S_1$ and $S_2$.
Similarly, for each non-removable component $F$ of $S_1 \cap S_2$, we have a copy of $(F - (S'_1 \cap S_2))\times [-1,1]$ embedded in $M$ and when the copies of $[-1,1]$ are collapsed to points, we again revert to the original arrangement of $S_1$ and $S_2$.

We now apply a theorem of Waldhausen \cite{Waldhausen} (in the case where the surface are closed) or Johannson \cite{Johannson} (in the case where the surfaces have boundary). This states that there is a generalised product region arising as a component of $M \cut (S'_1 \cup S_2)$ where one horizontal boundary component lies in $S'_1$, the other horizontal boundary component lies in $S_2$, any vertical boundary lies in $\partial M$ and is disjoint from $P$, and the binding locus consists of components of $S'_1 \cap S_2$.

We now perform the collapses described above. The generalised product region collapses to a collection of generalised product regions for $S_1$ and $S_2$, with components of $S_1 \cap S_2$ attached. Any one of these new generalised product regions gives the desired component of $M \cut (S_1 \cup S_2)$. Note that the collapses may create self-identifications on the boundary of the generalised product region.
\end{proof}

\begin{lemma}
\label{Lem:TraceAnnulusEssential}
Let $(M,P)$ be a compact orientable irreducible boundary-irreducible 3-manifold with boundary pattern. Let $\calT$ be a triangulation of $M$ in which $P$ is simplicial. Let $F$ be a clean properly embedded surface, with least weight in its  pattern-isotopy class, and that is incompressible and boundary-incompressible. Suppose that $F = F_1 + F_2$ is in reduced form.  
\begin{enumerate}
\item
If $A$ is a trace annulus or disc, then $A$ is incompressible and boundary-incompressible in the manifold $M \cut F$ with boundary pattern consisting of $P$ and the copies of $\partial F$.
\item
Suppose that $A$ is a trace annulus. Suppose also that for any other summation $F = F_1' + F'_2$, then $|F'_1 \cap F'_2| \geq |F_1 \cap F_2|$. Then $A$ is essential in $M \cut F$.
\item
If $F$ is a JSJ annulus, and $A$ is a trace disc, then the JSJ piece of $M$ containing $A$ is an $I$-bundle chamber, in which $F$ is a vertical boundary-component.
\end{enumerate}
\end{lemma}

\begin{proof}
(1) Each trace annulus and disc is incompressible by Theorem \ref{Thm:SummandsEssentialEtc}. Suppose that $A$ is boundary-compressible and hence parallel to an annulus or disc in $F$. If we were to perform an irregular switch along $A$ but regular switches along all other trace discs, annuli and M\"obius bands, the result would be a copy of $F$ as well as a compressible annulus or torus. After performing an irregular switch, there is always an isotopy that strictly reduces the weight of the surface. Hence, we deduce that $F$ did not have least weight in its pattern-isotopy class, which is contrary to hypothesis.

(2) Note first that $F$ is connected, because otherwise there is a summation $F = F'_1 + F'_2$ with $F'_1 \cap F'_2 = \emptyset$, contradicting our minimality assumption. Suppose that $A$ is an inessential trace annulus and hence, by (1), parallel to an almost clean annulus in $M \cut F$. Thus, $A$ separates off a solid torus $V$ in $M \cut F$ in which the pattern is a non-empty collection of disjoint longitudes. There may be other trace annuli in $V$, but they are all incompressible and boundary-incompressible by (1). Hence, there is a trace annulus in $V$ that is outermost in $V$. By considering this trace annulus instead, we may assume that the interior of $V$ contains no trace annuli. Let $F'_1$ be $(F_1 \cut V) \cup (F_2 \cap V)$ and define $F'_2$ similarly. Then $F = F'_1 + F'_2$ and $|F'_1 \cap F'_2| < |F_1 \cap F_2|$, contrary to assumption.

(3) Since $A$ is a trace disc, $F$ is decomposed by $F \cap A$ into two squares. Each of these squares has a product structure as $I \times I$ where $(I \times I) \cap \partial M = I \times \partial I$. Similarly, $A$ has such a structure. So a regular neighbourhood of $F \cup A$ is an $I$-bundle over a surface $S$. Note that $A$ is not attached to both sides of $F$, since this would imply the existence of an essential clean annulus that could not be isotoped off $F$, contradicting the assumption that $F$ is a JSJ annulus. Hence, $S$ is a pair of pants or a once-punctured M\"obius band. Suppose first that a vertical boundary component of this $I$-bundle is compressible. Then it bounds a copy of $D^2 \times I$ with interior disjoint from the bundle, and we can extend the bundle over this 3-ball. In the case where the base surface $S$ was a once-punctured M\"obius band, this implies that the component of $M \cut F$ containing $A$ was an $I$-bundle chamber, as required. So suppose that the base surface was a pair of pants. There may have been other trace discs in $D^2 \times I$, but there is one that outermost, and we hence may assume that $A$ is outermost in $D^2 \times I$. But then the same patch is incident to both components of $A \cap F$, which is impossible, since one lies in $F_1$ and one lies in $F_2$. Thus, we have reduced to the case where the vertical boundary components of the $I$-bundle are incompressible. If any is inessential, then extend the $I$-bundle over the incident solid torus. We end with an $I$-bundle having essential vertical boundary. This must be part of an $I$-bundle piece in $M$. Furthermore, $F$ is a vertical boundary component of this piece, as claimed.
\end{proof}

\begin{proof}[Proof of Theorem \ref{Thm:JSJFundamental}]
Let $T$ be a connected JSJ surface. Pick a least weight normal representative for $T$, up to pattern-isotopy. If $T$ is not fundamental, then we may write $T = T_1 + T_2$ where $T_1$ is connected. We may assume that $|T_1 \cap T_2|$ is minimal among all such summations, and in particular that the sum is in reduced form.
By Theorem \ref{Thm:SummandsEssentialEtc}, no $T_i$ has positive Euler characteristic. Hence, $\chi(T_1) = \chi(T_2) = 0$. Furthermore, $T_1$ and $T_2$ are incompressible and boundary-incompressible. Hence, $T_1$ is a clean annulus, torus, Klein bottle or M\"obius band. Therefore, since $T$ is rough, there is a pattern-isotopy taking $T_1$ off $T$. In the case where $T_1$ is a M\"obius band or Klein bottle, this is the content of Remark \ref{Rem:IsotopeOffMob}.

The summation $T = T_1 + T_2$ places $T$ and $T_1$ into partially coincident position. Specifically, each patch in $T_1$ becomes a component of $T \cap T_1$. Each trace annulus, disc or M\"obius band becomes a component of $T_1 \cut T$. Each patch in $T_2$ becomes a component of $T \cut T_1$. Note that for each trace annulus or disc in $T_1 \cut T$, the incident components of $T \cut T_1$ emanate from opposite sides of the trace disc or annulus. This is a consequence of normal summation.

We claim that the trace annuli and discs are incident to just one side of $T$. Any trace disc or annulus is incompressible and boundary-incompressible in $M \cut T$ by Lemma \ref{Lem:TraceAnnulusEssential} (1). If there is a trace annulus $A$, then this is essential in $M \cut T$ by Lemma \ref{Lem:TraceAnnulusEssential} (2). So, by Theorem \ref{Thm:JSJPieces}, the JSJ piece of $M$ containing $A$ is either Seifert fibred or an $I$-bundle chamber, but in the latter case, the horizontal boundary is one or two copies of $T$. Thus, in both cases, the JSJ piece of $M$ containing $A$ is Seifert fibred, and $A$ is a union of fibres. If there is a trace disc on one side of $T$, then Lemma \ref{Lem:TraceAnnulusEssential} (3) implies that this lies in an $I$-bundle chamber of $M \cut T$, and this must be a solid torus. Again the JSJ piece of $M$ containing $A$ is Seifert fibred. Note that there cannot be both trace annuli and trace discs, since trace discs would cut the annulus $T$ into discs. Hence if there are trace discs or annuli on both sides of $T$, then we deduce that on both sides of $T$ there are Seifert fibre spaces. Moreover, as trace annuli are pairwise disjoint, we deduce that these Seifert fibre spaces have isotopic regular fibres. In all case, we deduce that $T$ is not JSJ, which is a contradiction, proving the claim.

Since there is an admissible isotopy taking $T$ off $T_1$, Lemma \ref{Lem:ExistenceGeneralisedProduct} gives a generalised product region $V$. Since its boundary is orientable, it can contain no trace M\"obius band components. Consider first the horizontal boundary component $F_1$ of $V$ lying in $T_1$. We claim that this contains exactly one removable component of $T \cap T_1$. If there were more than one such component, then between two of them would be a component of $T_1 \cut T$, which would be a trace annulus or disc, with the two incident components of $T \cut T_1$ emanating from the same side, which is impossible. If there was no removable component of $T \cap T_1$ in $F_1$, then $F_1$ would be an entire trace disc or annulus, but then again the two incident components of $T \cut T_1$, which lie in $\partial V$, would emanate from the same side of $F_1$.

Now consider the horizontal boundary component $F$ of $V$ lying in $T$. We claim that this contains no removable components of $T \cap T_1$. For if there were such a component, then each component of $T \cut T_1$ incident to $\partial F$ would have components of $T_1 \cut T$ emanating from both sides. Hence, there would be trace annuli or discs incident to both sides of $T$. But we have shown above that this does not arise.

Thus, we have shown that the product region $V$ must have exactly two components of $T_1 \cut T$ in its boundary, in other words exactly two trace annuli or two trace discs.

Note that the map $V \rightarrow M$ is an injection. This is because the only way that it can fail to be an injection is when the two horizontal boundary components of $V$ contain components of $T_1 \cap T$ that are identified. But we have shown that $F$ contains no such components. So the two horizontal boundary components of $V$ are disjoint apart from along their boundaries. Consider the surfaces $T'_1$ and $T'_2$ that are obtained from $T_1$ and $T_2$ by swapping the two horizontal boundary components of $V$. In other words, $T'_1 = (T_1 \cut F_1) \cup F$ and $T'_2$ is obtained from $T_2$ by removing $F \cap T_2$ and attaching the remainder of $\partial V$. Then $T'_1$ and $T'_2$ are normal, and $T = T'_1 + T'_2$, but $|T'_1 \cap T'_2| < |T_1 \cap T_2|$. This contradicts the assumption that the summation $T_1 + T_2$ is in reduced form.
\end{proof}

\begin{theorem}
\label{Thm:JSJUnionFundamental}
Let $(M,P)$ be a compact orientable irreducible boundary-irreducible 3-manifold with boundary pattern.
Let $\mathcal{T}$ be a triangulation of $M$ in which $P$ is simplicial. Then the JSJ surfaces 
may be realised as a disjoint union of fundamental surfaces.
\end{theorem}

\begin{proof} 
Pick a normal representative for the JSJ surfaces that has least weight in its pattern-isotopy class.
By Theorem \ref{Thm:ComponentsLeastWeight}, each component has least weight in its
pattern-isotopy class. Hence by Theorem \ref{Thm:JSJFundamental}, each component is fundamental.
\end{proof}

\subsection{Normal annuli in Seifert fibre spaces}
\label{Subsec:AnnuliSFS}

We will be considering the JSJ tori for the exterior of a non-split link in the 3-sphere.
It is possible that certain complementary regions of these tori may be Seifert fibre spaces.
These require techniques that are rather different from other types of 3-manifolds.
So, in this section, we recall and develop the necessary machinery.

We start by recalling what Seifert fibre spaces can arise.
By \cite[Proposition 3.2]{Budney}, such Seifert fibre spaces fall into one of the following types:
\begin{enumerate}
\item a solid torus, which has base space a disc with at most one cone point;
\item the product of a planar surface with the circle; this has no exceptional fibres;
it arises, for example, as the exterior of the link shown in Figure \ref{Fig:SeifertFibredLinks} (i);
\item the exterior of a torus link, with zero, one or both of the core curves
of the associated solid tori removed, as shown in Figure \ref{Fig:SeifertFibredLinks} (ii); 
its base space is punctured 2-sphere with at most two cone points; moreover,
if the space has two exceptional fibres, these have coprime orders.
\end{enumerate}

\begin{figure}[h]
\includegraphics[width=4in]{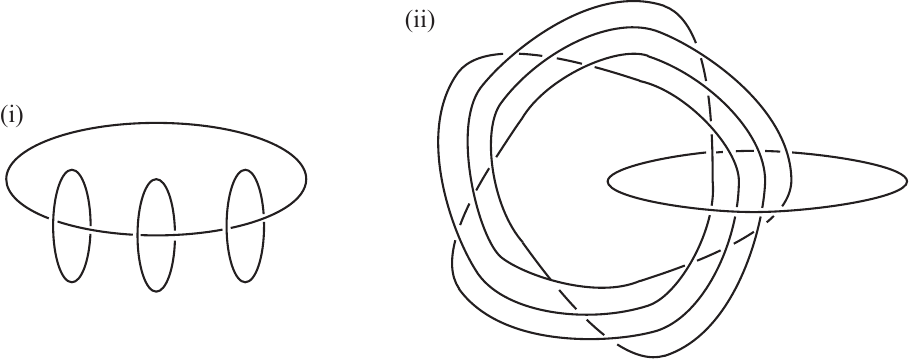}
\caption{Some links with Seifert fibred exteriors}
\label{Fig:SeifertFibredLinks}
\end{figure}

We will be concerned with annuli properly embedded in such a Seifert fibre space.

\begin{lemma}
\label{Lem:AnnuliInSFS}
Let $M$ be a Seifert fibre space with non-empty boundary, that embeds in the 3-sphere,
other than a solid torus or $S^1 \times S^1 \times I$.
\begin{enumerate}
\item If $\partial M$ is a single torus, then there is an essential annulus $A$ properly
embedded in $M$, such that the exterior of $A$ is two solid tori.
\item If $\partial M$ is more than one torus, let $T_1$ and $T_2$ be distinct
boundary components. Then there is an essential annulus $A$ properly embedded
in $M$, with one boundary component on $T_1$ and one boundary component on $T_2$.
\end{enumerate}
In both cases, $A$ is unique, up to a homeomorphism of $M$ that is
isotopic to the identity on $\partial M$.
\end{lemma}

\begin{proof}
We start by describing $A$, and then go on to explain why it is unique. 
When $\partial M$ is a single torus, the base space of the Seifert fibre space is a disc
with two cone points. We pick a properly embedded arc in the disc that separates
the two cone points. When $\partial M$ is more than one torus, we pick a properly
embedded arc in the base space joining the boundary components corresponding to
$T_1$ and $T_2$. In each case,  the inverse image in $M$ of this arc is the required annulus $A$. 

\begin{figure}[h]
\includegraphics[width=4in]{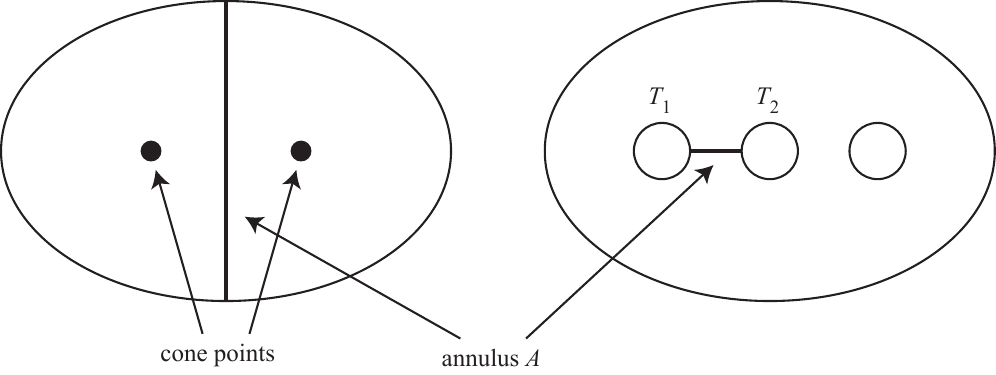}
\caption{Some possible base spaces of $M$ containing the projection of $A$}
\label{Fig:SeifertFibredBases}
\end{figure}

We now show that $A$ is unique, up to a homeomorphism of $M$ that is
isotopic to the identity on $\partial M$. The hypothesis that $M$ embeds in the 3-sphere, has non-empty boundary,
and is not a solid torus or $S^1 \times S^1 \times I$ implies that its Seifert fibring is unique up to isotopy \cite{Waldhausen:Klasse, Waldhausen:Gruppen}.
Any properly embedded orientable essential surface
in a Seifert fibre space is isotopic to a surface that is horizontal or vertical. If an annulus 
is horizontal, then the base orbifold must have zero Euler characteristic, but this is ruled
out by our hypothesis that $M$ embeds in the 3-sphere and is not $S^1 \times S^1 \times I$. So, we may assume that $A$
is vertical. It therefore projects to an arc in the base orbifold, which avoid the
cone points. It is easy to see that there is a unique such arc, up 
to orientation-preserving homeomorphism of the base space that preserves each boundary component and each cone point.
This implies that $A$ is unique, up to a homeomorphism of $M$ that is
isotopic to the identity on $\partial M$. 
\end{proof}

\begin{proposition}
\label{Prop:AnnuliWeaklyFundamental}
Let $M$ and $A$ be as in Lemma \ref{Lem:AnnuliInSFS}. Let $P$ be a boundary pattern, such that for
each component $T$ of $\partial M$, one of the following holds:
\begin{enumerate}
\item $T$ is disjoint from $P$;
\item $T \cap P$ is a disjoint union of essential simple closed curves;
\item $T \cut P$ is discs.
\end{enumerate}
Suppose that among the annuli in Lemma \ref{Lem:AnnuliInSFS}, $A$ intersects $P$ as few times as possible.
When $A$ is disjoint from $P$ and $\partial M$ is connected, we also require that both components of
$\partial A$ lie in the same component of $\partial M \cut P$. Pick a triangulation of $M$ in which
$P$ is simplicial. Then $A$ is strongly equivalent to a
weakly fundamental surface.
\end{proposition}

\begin{proof}
Since $A$ is essential, it can be isotoped to a normal surface. We pick a 
representative for $A$ that has minimal weight, up to strong equivalence. 
Suppose that $A$ is not weakly fundamental.

Suppose first that $A$ can be expressed as a sum
$S_1 + S_2$ of normal surfaces, where $S_2$ is a sphere, projective plane, torus, clean disc or clean annulus. We may assume
that the sum is in reduced form.
By Theorem \ref{Thm:SummandsEssentialEtc}, these surfaces are both incompressible,
and neither has a sphere, projective plane, clean disc or clean inessential annulus component. Hence, $\chi(S_1) = \chi(S_2) = 0$.

Consider first the case when $S_2$, say, is a torus. Then, $\partial A = \partial S_1$.
So either some component of $S_1$ is an annulus or two components of $S_1$ are
M\"obius bands. But the latter case cannot arise, because each M\"obius band would have to be vertical and so would contain
a singular fibre of order 2, whereas $M$ contains at most two singular fibres and these have coprime orders.
So, $S_1$ has an essential annulus component. It cannot be boundary-parallel by our assumption on $P$.
Since it has the same boundary as $A$, it
is strongly equivalent to $A$. But $S_1$ has smaller weight than $A$, which is a contradiction.
So we may assume that $A$ has no torus summand.

Suppose now that $S_2$ is a clean essential annulus and that $S_1$ has an annulus component. 
Consider first the case where $M$ has base space a disc and two exceptional fibres.
Then, by our minimality assumption about $|A \cap P|$, $A$ must be disjoint from $P$.
So, by assumption, both components of $\partial A$ lie in the
same component of $\partial M \cut P$. Hence, $\partial S_1$ and
$\partial S_2$ also lie in this component. Hence, $S_2$ is strongly
equivalent to $A$. But it has smaller weight than $A$, which is a contradiction.
So suppose that $M$ has more than one boundary component. Then by assumption, the components of $\partial A$
lie on distinct tori $T_1$ and $T_2$. Now $\partial S_1 \cup \partial S_2$
lies on the same boundary components as $\partial A$. Since $S_2$ is clean,
the torus or tori $T$ containing $\partial S_2$ have a clean fibre. By our minimality 
assumption on $|A \cap P|$, $A \cap T$ is disjoint from $P$. Hence, $A \cap T$ lies
in the same component of $T \cut P$ as $S_1 \cap T$ and $S_2 \cap T$.
If $S_2$ has one boundary component
on $T_1$ and one boundary component on $T_2$, then we may replace $A$ by this surface.
So suppose that both boundary components of $S_2$ lie on some boundary
component $T$ of $\partial M$, then $\partial S_1 \cap (\partial M - T) = \partial A \cap (\partial M - T)$,
which is single curve. This curve cannot be the boundary of a M\"obius band component of $S_1$,
since then every other component of $S_1$ is a torus or an annulus with both boundary curves on $T$,
which implies that $[A \cap T] = [S_1 \cap T] + [S_2 \cap T] =0 \in H_1(T; \mathbb{Z}_2)$,
but this is not the case.
So we may replace $A$ by the annular component of $S_1$ containing $\partial A \cap (\partial M - T)$
and reduce its weight, which is a contradiction.

Suppose now that $S_2$ is a clean essential annulus and that $S_1$ has no annulus or torus components.
Then $S_1$ is a collection of M\"obius bands. But $M$ cannot support two or more disjoint M\"obius bands.
So $S_1$ is a single M\"obius band.
If $M$ has a single boundary component, then, as above, $S_2$ is strongly equivalent to $A$ and has smaller
weight, which is a contradiction. So, suppose that $M$ has more than one boundary component,
and that one component of $\partial A$ lies on $T_1$ and the other lies on $T_2$.
Then $\partial S_1$ is a single regular fibre on some component ($T_1$, say) of
$\partial M$. The curves $\partial S_2$ are isotopic to regular fibres on $\partial M$.
If both curves of $\partial S_2$ lie on $T_2$, then $[A \cap T_2] = [S_2 \cap T_2] = 0 \in H_1(\partial M; \mathbb{ Z}_2)$,
which is not the case. If one curve of $\partial S_2$ lies on $T_1$ and the other lies on $T_2$,
then $[A \cap T_1] = [\partial S_1 \cap T_1] + [\partial S_2 \cap T_1] = 0 \in H_1(\partial M; \mathbb{ Z}_2)$, which 
again is not the case.
Finally if both curves of $\partial S_2$ lie on $T_1$, then $A$ is disjoint from $T_2$, which is not the case.
In all cases, we obtain a contradiction.

Suppose now that $A = S_1 + S_2 + S_3$ where $S_2$ and $S_3$ are clean M\"obius bands. Suppose
also that $M$ has a single boundary component.
Then $S_2$ (say) is a vertical surface that projects to an arc in the base space running from
the boundary of the base space to a singular point of order 2. So $2 S_2$ is a clean essential annulus.
By assumption, $A$ is therefore clean and both its boundary components lie in the
same component of $\partial M \cut P$. Hence, the annuli $2S_2$ and $2S_3$
also have boundary lying in this component, and so are strongly equivalent to $A$.
Now $w(S_2 + S_3)$ is the mean of $w(2S_2)$ and $w(2S_3)$.
So one of $w(2S_2)$ and $w(2S_3)$ is at most $w(S_2 + S_3)$, the former say. But
then $w(2S_2) \leq w(S_2 + S_3) < w(A)$, contradicting the assumption that $A$ has least weight.

Suppose now that $A = S_1 + S_2 + S_3$ where $S_2$ and $S_3$ are clean M\"obius bands, and that
$M$ has at least two boundary components. Now $[\partial S_2]$ represents the homology class
of a fibre of $T_1$ or $T_2$. The same is true of $[\partial S_3]$. Hence, $S_2 + S_3$
cannot be a single M\"obius band and some tori. It also cannot have more than one
M\"obius band component, since $M$ does not support two disjoint M\"obius bands.
So, $S_2 + S_3$ contains a clean annulus component. We may therefore write $A = S_1' + S'_2$
where $S'_2$ is a clean annulus, and we have dealt with that case above.

\end{proof}

\begin{proposition}
\label{Prop:MakeWeaklyFundamentalTorusTimesI}
Let $M$ be homeomorphic to $T^2 \times [0,1]$, with boundary pattern $P$ so that 
$(M,P)$ is boundary-irreducible. Suppose that $M$ contains no clean essential annulus.
Let $\mathcal{T}$ be a triangulation of $M$ in which $P$ is simplicial. 
Then any incompressible, boundary-incompressible annulus $A$ properly embedded in $M$, in general position with respect to $P$ and joining the
two components of $\partial M$, is strongly equivalent to a weakly fundamental normal surface.
\end{proposition}

\begin{proof}
By Proposition \ref{Prop:IsotopicToNormal}, $A$ is pattern-isotopic to a normal surface. 
We may assume that $A$ has smallest weight, up to strong equivalence.
By Theorem \ref{Thm:SummandsEssentialEtc} (5), no summand of $A$ is a sphere,
clean disc or projective plane. By Theorem \ref{Thm:SummandsEssentialEtc} (4) and (7), 
no summand is a torus. By Theorem \ref{Thm:SummandsEssentialEtc} (4) and (6), no summand
is a clean inessential annulus. We are assuming that $M$ contains no clean essential annulus.
The manifold $T^2 \times I$ does not contain a properly embedded M\"obius band. Hence,
$A$ is weakly fundamental.
\end{proof}

\begin{proposition}
\label{Prop:TorusTimesI}
Let $M$ be homeomorphic to $T^2 \times [0,1]$, with non-empty boundary pattern $P$ so that 
$(M,P)$ is boundary-irreducible. Let $\mathcal{T}$ be a triangulation of $M$ in which $P$ is simplicial.
Let $A$ be any essential annulus disjoint from the pattern, joining the two boundary components.
Then $A$ is strongly equivalent to a fundamental surface.
\end{proposition}

\begin{proof}
We may assume that $A$ has least weight, up to strong equivalence.
Suppose that $A$ is not fundamental, and hence is a normal sum
$S_1 + S_2$. We may assume that this summation is in reduced form. Hence, by Theorem \ref{Thm:SummandsEssentialEtc},
$S_1$ are incompressible and boundary-incompressible and have zero Euler characteristic.
Now $T^2 \times [0,1]$ supports no properly embedded M\"obius band. Therefore,
$S_1$ and $S_2$ are unions of clean annuli or tori. 
For at least one $i$,
$[S_i \cap \partial M]$ is non-trivial in $H_1(\partial M; \mathbb{Z}_2)$. Say that this is true 
of $S_1$. Then $S_1$ contains a clean annulus joining the two components of $\partial M$. The components
of $\partial M \cut P$ that it intersects must be the same components that $A$ intersects. These components
of $\partial M \cut P$ are annuli.
So, there is a pattern isotopy taking $\partial S_1$ to $\partial A$. There is then a homeomorphism
equal to the identity on $\partial M$ taking $S_1$ to $A$. This contradicts the assumption
that $A$ had least weight.
\end{proof}

\subsection{Boundary parallel surfaces}
\label{Subsec:BoundaryParallel}

\begin{lemma}
\label{Lem:BoundaryParallel}
Let $(M,P)$ be a compact orientable irreducible boundary-irreducible 3-manifold with boundary pattern. Let $S'$ be a union of edges of $P$, circles of $P$, and
components of
$\partial M \cut P$ that are not necessarily disjoint in $\partial M$, and let $S$ be an incompressible boundary-incompressible surface that is obtained from
$S'$ by taking the frontier of a regular neighbourhood and that contains no disc component intersecting $P$ zero or two times. 
Then $S$ is $(6,2)$-exponentially controlled.
\end{lemma}

\begin{proof}
Let $\mathcal{T}$ be any triangulation of $M$ in which $P$ is simplicial. Then a representative for $S$ intersects each edge of $\mathcal{T}$
at most twice. There are at most $6t$ edges, where $t$ is the number of tetrahedra of $\mathcal{T}$. Hence, the weight of $S$
is at most $12t$. It need not be normal, but it can be isotoped to a normal surface without increasing its weight, by Proposition \ref{Prop:IsotopicToNormal}.
Therefore, its weight is at most $12t \leq 6 \cdot 2^t$. Hence, it is $(6,2)$-exponentially controlled.
\end{proof}

\subsection{Sections of circle bundles}
\label{Sec:CircleBundle}

In this subsection, we consider the case where $M$ is a circle bundle over a surface with non-empty boundary. Our goal will be to show that, for any triangulation of $M$, there is a surface isotopic to a section of the bundle having controlled weight. However, we need somewhat more than this. The bundle $M$ will be decomposed into solid tori, each of which is a union of circle fibres. Moreover, the intersection of any two such solid tori will consist of a (possibly empty) union of annuli, which again is a union of fibres. We want to ensure that the surface intersects each solid torus in a collection of meridian discs. In order to ensure this, we will need to generalise some of the results earlier in this section.

The first thing that we do is to slightly relax the notion of an exponentially controlled surface. Given a compact orientable 3-manifold $M$ with boundary pattern $P$,
we say that a surface $S$ properly embedded in $M$ is \emph{$(c,k)$-nearly exponentially controlled} if, given any triangulation of $M$ with $t$ tetrahedra in which $P$ is simplicial, there is an orientation-preserving homeomorphism of $M$ that preserves each edge, vertex and circle of $P$, and that takes $S$ to a surface with weight at most $ck^t$.

The next thing that we need to do is to consider a more general notion of weight for a surface.
Suppose that $E_1$ and $E_2$ are sets of edges of the triangulation, where $E_1 \subset E_2$. 
The \emph{prioritised weight} of a properly embedded surface $S$ is the triple 
$$(|S \cap E_1|, |S \cap E_2|, w(S)).$$ 
Two such triples are compared using lexicographical ordering.

In our setting, each solid torus will be simplicial in the triangulation, and $E_2$ will be the edges in the boundary of these solid tori,
and $E_1$ will be the set of edges that lie either in three such solid tori or in two such solid tori and $\partial M$.

We note that there is the following version of Theorem \ref{Thm:SummandsEssentialEtc} that refers to prioritised weight.

\begin{theorem}
\label{Thm:SummandsPrioritised}
Let $M$, $P$ and $\mathcal{T}$ be as in Theorem \ref{Thm:SummandsEssentialEtc}. Let $E_1 \subset E_2$ be collections of edges of the triangulation.
Let $S$ be a properly embedded, connected, incompressible, boundary-incompressible, normal surface that has least possible prioritised weight,
among all normal surfaces that are pattern-isotopic to it. Suppose that $S$ can be expressed as a normal sum $S_1 + S_2$,
where $S_1$ and $S_2$ are in general position and non-empty, and the summation is in reduced form. Then the conclusions 
of Theorem  \ref{Thm:SummandsEssentialEtc} all hold.
\end{theorem}

The results from \cite{Matveev} used in the proof of Theorem \ref{Thm:SummandsEssentialEtc} do not refer to the prioritised weight, 
but rather the weight of $S$. However, we observe
that all the arguments used in these proofs in \cite{Matveev} work just as well for prioritised weight,
and hence Theorem \ref{Thm:SummandsPrioritised} follows.

\begin{lemma}
\label{Lem:NormalisePrioritisedWeight}
Let $M$ be a compact orientable 3-manifold that is a circle-bundle over a surface. 
Let $P$ be a boundary pattern for $M$. 
Let $C$ be a disjoint union of fibres in $M$, each of which lies in the interior of $M$ or in $P$,
and where each component of $\partial M$ contains at least one component of $C$.
Let $A$ be a collection of annuli in $M$, each of which is embedded, is a union of fibres and has boundary in $C$. 
Suppose that the interior of these annuli are disjoint,
and that for each component of $C$, there are three sheets of $A$ and $\partial M$ coming into it. 
Suppose that $M$ is triangulated, with $P$, $C$ and $A$ all simplicial.
 Let $E_1$ be the set of edges in $C$,
and let $E_2$ be the set of edges in $A \cup \partial M$. 
Let $S$ be a surface properly embedded in $M$, and suppose that there is a homeomorphism of $M$
that preserves $A$ and that takes $S$ to a section. 
Assume that $S$ has minimal number of intersections with $P$ among all such surfaces.
Then there is an isotopy preserving $A \cup P$ taking $S$ to a normal surface $S'$ with prioritised weight at most that of $S$.
\end{lemma}

\begin{proof}
Consider a surface $S'$ that is isotopic to $S$, via an isotopy preserving $A \cup P$, and that has least prioritised weight
among all such surfaces. Note that $S'$ intersects each component of
$C$ at least once, since it is isotopic to a section, and hence $S'$ intersects each component
of $C$ exactly once, because its prioritised weight has been minimised. Furthermore, for each annulus of $A$, $S'$ intersects that
annulus in a single essential arc.

Suppose that $S'$ is not normal. Then we may apply one of the
normalisation moves in the proof of Corollary 3.3.25  in \cite{Matveev}. We claim that each
such move gives an isotopy supported away from $A \cup \partial M$. For example, suppose that
there is a simple closed of curve of intersection between $S'$ and some face of the triangulation.
Since $S'$ intersects each annulus of $A$ in a single essential arc, we deduce that the face does
not lie in $A$. Similarly, it does not lie in $\partial M$. Consider such a simple closed curve that is innermost in the face. It bounds
a disc $D$ in the face. By the incompressibility of $S'$, $\partial D$ bounds a disc $D'$ in $S'$,
and by the irreducibility of $M$, $D \cup D'$ bounds a ball. We may isotope $D'$ across this ball
and then a bit beyond the face. This procedure removes intersections between $D'$ and the 1-skeleton
of the triangulation. However, it has no effect on $A \cap S'$, since if any component of $A \cap S'$
were removed, the resulting surface would be disjoint from some fibre, which is impossible.

Other moves in the proof of Corollary 3.3.25  in \cite{Matveev} remove two points of intersection between
some edge and $S'$. We claim that such an edge cannot lie in $A \cup \partial M$. Certainly, the edge cannot lie in $C$,
since each component of $C$ intersects $S'$ just once. Also, the edge cannot lie in $P$ by our assumption
that the number intersection points with $P$ is minimal. On the other hand, if the edge lies in some
annular component $A'$ of $(A \cup \partial M) \cut C$ and is disjoint from $P$, then the isotopy replaces the essential arc $A' \cap S'$ with a pattern-isotopic
arc with smaller weight plus possibly a simple closed curve. This contradicts the assumption that the prioritised weight of $S'$ is minimal.

In all cases, we deduce that the isotopy is supported away from $A \cup \partial M$. The resulting surface is normal.
\end{proof}

\begin{lemma}
\label{Lem:TwoEquivalences}
Let $M$, $P$, $C$, $A$, $E_1$ and $E_2$ be as in Lemma \ref{Lem:NormalisePrioritisedWeight}.
Suppose that $M \cut A$ is a union of solid tori.
Let $\mathcal{T}$ be a triangulation of $M$ in which $C$, $P$ and $A$ are simplicial. 
Let $S$ be a surface that, after a homeomorphism preserving $A$, is equal to a section.
Let $S'$ be a surface that is the image of $S$ under a homeomorphism of $M$ that preserves
each vertex, edge and circle of $P$, and which has no greater prioritised weight.
Then there is a homeomorphism of $M$ preserving each vertex, edge and circle of $P$ and each component of $A$ taking $S$ to $S'$.
\end{lemma}

\begin{proof}
The homeomorphism of $M$ taking $S$ to $S'$ preserves each of the fibres in $C \cap \partial M$
and so is isotopic to a fibre-preserving homeomorphism.
The surface $S$ intersects each component of $C$ exactly once. Since $S'$ differs from $S$ by a fibre-preserving homeomorphism and an isotopy,
it also intersects each component of $C$ an odd number of times. Because the prioritised weight
of $S'$ is at most that of $C$, we deduce that $S'$ intersects each component of $C$ exactly once.
Hence, $S'$ intersects each annulus in $A$ and $\partial M \cut P$ in exactly one essential arc plus
possibly some inessential simple closed curves. However, $S'$ cannot intersect any component of
$\partial M \cut P$ in an inessential simple closed curve, since $S'$ is isotopic to a section. 
Moreover, if $S'$ had an inessential simple closed curve of intersection with some component of $A$,
then we could perform an isotopy to $S'$ supported in the interior of $M$ to reduce its
prioritised weight. So, $S' \cap (A \cup \partial M)$ is homeomorphic to $S \cap (A \cup \partial M)$.
Furthermore, their regular neighbourhoods in $S'$ and $S$ respectively are homeomorphic.
The surfaces $S'$ and $S$ are obtained from these regular neighbourhoods by attaching $S' \cut 
N(A \cup \partial M)$ and $S \cut N(A \cup \partial M)$ respectively. The latter is a disjoint union of discs.
Hence, as $S'$ and $S$ have the same Euler characteristic,  $S' \cut 
N(A \cup \partial M)$ is also discs. Each such disc is a meridian disc for a solid torus component
of $M \cut A$. So, there is a homeomorphism of $M$ preserving each vertex, edge and circle of $P$ and each component of $A$ taking
$S'$ to a section. Therefore, there is also such a homeomorphism taking $S$ to $S'$.
\end{proof}

\begin{proposition}
\label{Prop:SectionCircleBundle}
Let $M$, $P$, $C$, $A$, $E_1$ and $E_2$ be as in Lemma \ref{Lem:NormalisePrioritisedWeight}.
Suppose that $M \cut A$ is a union of solid tori.
Let $\mathcal{T}$ be a triangulation of $M$ in which $C$, $P$ and $A$ are simplicial. 
Let $S$ be a properly embedded surface that is isotopic, via an isotopy preserving $A$, to a section.
Suppose that, among such surfaces, $S$ intersects $P$ as few times as possible.
Then there is a homeomorphism of $M$ preserving each vertex, edge and circle of $P$ and each component of $A$ 
taking $S$ to a weakly fundamental normal surface.
\end{proposition}

\begin{proof}
By Lemma \ref{Lem:NormalisePrioritisedWeight}, there is an isotopy preserving
$A \cup P$ taking $S$ to a normal surface without increasing its prioritised weight. 
We now fix $S$ to be such a surface that has least prioritised weight,
up to homeomorphism of $M$ preserving each vertex, edge and circle of $P$ and each component of $A$. Then by Lemma \ref{Lem:TwoEquivalences},
$S$ has least prioritised weight, up to a homeomorphism of $M$ that preserves each vertex, edge and circle of $P$
and each component of $A$. Hence,
the same is true of the surface $2S$, which is two parallel copies of $S$.

Suppose that $S$ is not weakly fundamental. Suppose first that $S = S_1 + S_2$ where 
$S_1$ is non-empty and $S_2$ is a sphere, projective plane, torus, clean disc
or clean annulus.
The surface $2S$ is a sum $S'_1 + S'_2$ where each $S'_i$ is the (possibly disconnected)
orientable double cover of $S_i$. We may assume that the summation is in reduced form.
By Theorem \ref{Thm:SummandsPrioritised}, $S'_1$ and $S'_2$ are both incompressible and boundary-incompressible,
have no component that a 2-sphere, a clean disc or a clean annulus parallel to an almost clean annulus in $\partial M$. 
Hence, $S'_1$ and $S'_2$ must be isotopic to surfaces that are horizontal or vertical.
Suppose that neither $S'_1$ nor $S'_2$ has
any horizontal components. Then each component is vertical.
Hence, the same is true of each component of $S_1$ and $S_2$. But then $S_1$ and $S_2$ 
would intersect each fibre an even number of times, whereas $S$ does not, which 
is a contradiction. So, one of $S'_1$ and $S'_2$ has a horizontal component, and so the same is
true of a component of $S_1$ or $S_2$. This is isotopic to a section, since it intersects each fibre in $C$ at most once.
Now, $S_2$ is disjoint from $P$ and so cannot be a section. Therefore, $S_1$ is isotopic to a section.
Moreover, its boundary is equal to $\partial S$ away from annular components of $\partial M \cut P$. So 
there is a homeomorphism of $M$ preserving  each vertex, edge and circle of $P$,
taking $S_1$ to $S$. It has smaller prioritised weight than $S$. So, by Lemma \ref{Lem:TwoEquivalences},
there is a homeomorphism of $M$ preserving  each vertex, edge and circle of $P$ and each component of $A$,
taking $S_1$ to $S$. But this contradicts our assumption that $S$ had smallest prioritised weight
among such surfaces.

Suppose now that $S = S_1 + S_2 +S_3$, where $S_1$ is non-empty and $S_2$ and $S_3$
are clean M\"obius bands. The only orientable circle bundles over a compact surface that have a section and that contain a properly embedded
M\"obius band are the solid torus and the orientable circle bundle over the M\"obius band. Suppose first $M$ is
a solid torus. In that case, both $S_2$ and $S_3$ represent the non-trivial
element of $H_2(M, \partial M; \mathbb{Z}_2)$. So $S_2 + S_3$ is homologically trivial and 
therefore has no M\"obius band component. (Note that the solid torus does not contain
two disjoint properly embedded M\"obius bands.) Now, $S_2 + S_3$ has zero Euler characteristic
and therefore some component is a clean disc or clean annulus. So, $S$ can be written as
the sum of a non-empty normal surface and a clean disc or a clean annulus. We have already
ruled that out above. Suppose now that $M$ is the orientable circle bundle over the M\"obius band.
Then $S_2$ and $S_3$ are isotopic to sections of the bundle. However, by assumption, each
component of $\partial M$ contains a component of $C$, which lies in $P$, and hence
$M$ does not admit a clean section. This contradiction completes the proof.
\end{proof}

As above, let $M$ be a circle bundle, and suppose that $C$ is a union of finitely many fibres with $C \cap \partial M \subset P$,
and that $A$ is a collection of annuli with boundary in $C$. Suppose that these divide $M$ into
solid tori, and that each component of $\partial M$ contains at least one component of $C$.
A \emph{broken section} for $M$ is obtained by picking distinct angles in $[0, 2\pi)$, one
for each solid torus, and isotoping the intersection between $S$ and each solid torus, so that for
any fibre in that solid torus, its intersection with $S$ has been moved through the relevant angle.

\begin{proposition}
\label{Prop:PerturbSection}
Let $M$, $P$, $C$, $A$, $E_1$ and $E_2$ be as in Proposition \ref{Prop:SectionCircleBundle}.
Let $S$ be a properly embedded surface that is isotopic, via an isotopy preserving $A$, to a section.
Suppose that, among such surfaces, $S$ intersects $P$ as few times as possible.
Then $M$ has a broken section that is $(3c,k)$-nearly exponentially controlled, where $c = -\chi(S) + 2 |S \cap P| + 2|S|$ and $k = 2^{15}$.
\end{proposition}

\begin{proof} 
By Proposition \ref{Prop:SectionCircleBundle},
there is a homeomorphism of $M$ preserving each vertex, edge and circle of $P$ and each component of $A$, 
taking $S$ to a weakly fundamental normal surface.
By Proposition \ref{Prop:WeaklyFundamentalBound},
$$w(S) \leq t^2 2^{7t+6} (-\chi(S) + 2 |S \cap P| + 2|S|).$$
We now perturb it to a broken section. By choosing the perturbation
angles to be sufficiently close to zero, the resulting surface is still normal and it has at most three times the weight.
(The weight on the edges of $C$ will increase by a factor of three, since the three adjacent meridian discs
have been perturbed to be disjoint.) Hence, it is $(3c,k)$-nearly exponentially
controlled.
\end{proof}

\section{Hierarchies}
\label{Sec:Hierarchies}

In this section, we recall some of the main concepts from the theory of hierarchies, as developed
by Haken \cite{Haken} and elucidated by Matveev \cite{Matveev}.

\subsection{Hierarchies and boundary patterns}

Let $M$ be a compact orientable 3-manifold with a boundary pattern $P$.
A \emph{partial hierarchy} for $M$ is a sequence of surfaces $S_1, \dots, S_{n}$ and
a sequence of 3-manifolds $M = M_1, \dots, M_{n+1}$, where each $S_i$ is a
properly embedded orientable incompressible surface in $M_i$, and each $M_{i+1}$ is obtained from $M_i$ by cutting along $S_i$. 
We write this as 
$$M = M_1 \buildrel S_1 \over \longrightarrow M_2 \buildrel S_2 \over \longrightarrow \dots
\buildrel S_{n} \over \longrightarrow M_{n+1}.$$

A partial hierarchy is a \emph{hierarchy} if $M_{n+1}$ is a collection of 3-balls. 

In the partial hierarchies that we consider, the boundaries of the surfaces will be 
in general position. Hence, $\partial S_i$ intersects $P \cup \partial S_1 \cup \dots \cup \partial S_{i-1}$
in finitely many points. Also, $\partial S_i$ is non-embedded at only finitely many points.
This latter phenomenon occurs when $S_i$ runs over both sides of an earlier surface. We will
ensure that this latter situation does not arise with the hierarchies that we consider.

The topological space $S_1 \cup \dots \cup S_i$ is homeomorphic to a 2-complex,
but this construction cannot be made in a canonical way, in general. We therefore introduce
the following structure.

We say that a space is a \emph{generalised 2-complex} if it is obtained as follows. Let
the \emph{1-skeleton} $K^1$ be the disjoint union of finitely many circles and a finite 1-complex.
Then attach a compact surface $F$ via an attaching map $f \colon \partial F \rightarrow K^1$.
We say that the image of each component of $F$ is a \emph{generalised 2-cell}.

Note that $S_1 \cup \dots \cup S_i$ has the structure of a generalised 2-complex,
as follows. The union $P \cup \partial S_1 \cup \dots \cup \partial  S_i$ is the 1-skeleton.
The 0-cells are the points where some $\partial S_i$ meets $P$, some other $\partial S_j$
or itself. 

Let $\Gamma_i$ be $P \cup \partial S_1 \cup \dots \cup \partial S_i$. The inverse image of
$\Gamma_i$ in $M_{i+1}$ is a union of simple closed curves and  graphs $P_{i+1}$
embedded within $\partial M_{i+1}$. This is the boundary pattern that $M_{i+1}$ \emph{inherits} from the partial hierarchy.

Note that, when visualising this boundary pattern, it is convenient to regard each surface $S_i$
as `transparent', in the following sense. When a later surface $S_j$ has boundary that runs over a part
of $\partial M_j$ associated with $S_i$, then boundary pattern is created not only at this
part of $\partial M_{j+1}$ but also in the part of $\partial M_{j+1}$ lying on the other side of $S_i$.
See Figure \ref{Fig:PartialHierarchyPattern}, for example.

\begin{figure}[h]
\includegraphics[width=4in]{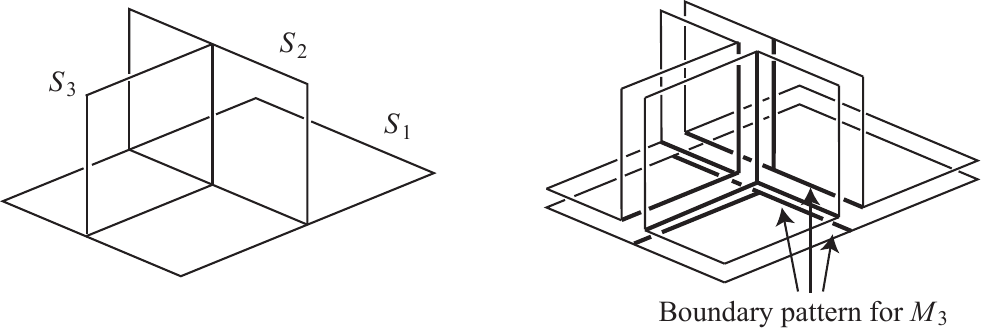}
\caption{A partial hierarchy and associated boundary pattern}
\label{Fig:PartialHierarchyPattern}
\end{figure}

The following simple fact is well known. See, for instance, the discussion following Definition 6.3.6 in \cite{Matveev}.

\begin{lemma}
\label{Lem:IrreduciblePreserved}
Let $(M,P)$ be a compact orientable irreducible boundary-irreducible 3-manifold with boundary pattern. 
Let
$$M = M_1 \buildrel S_1 \over \longrightarrow M_2 \buildrel S_2 \over \longrightarrow \dots
\buildrel S_{n} \over \longrightarrow M_{n+1}$$
be a partial hierarchy, where no component of any $S_i$ is a sphere or clean disc.
For each $i$, let $P_i$ be the boundary pattern that $M_i$ inherits.
Then $(M_i, P_i)$ is irreducible and boundary-irreducible.
\end{lemma}

In our proof of Theorem \ref{Thm:Main}, we will reduce to the case where $K$ is a non-split link other than the unknot.
In the hierarchies we consider, the initial manifold $M$ will be the exterior of $K$. By Lemma \ref{Lem:IrreduciblePreserved},
the manifolds with boundary pattern in such a hierarchy will all be irreducible and boundary-irreducible.

\subsection{Exponentially controlled hierarchies}
Let $M$ be a compact orientable 3-manifold with boundary pattern $P$.
Let
$$M = M_1 \buildrel S_1 \over \longrightarrow M_2 \buildrel S_2 \over \longrightarrow \dots
\buildrel S_{n} \over \longrightarrow M_{n+1}$$
be a hierarchy. We say that it is \emph{$(c,k)$-exponentially controlled} if, for each $i \leq n$
and each triangulation of $M$ with $t$ tetrahedra, in which $P, S_1, \dots, S_{i-1}$
are simplicial, there is a homeomorphism of $M$ that is isotopic
to the identity on $\partial M$ and that preserves $P$ and $S_j$, for each $j < i$, and that takes
$S_i$ to a surface with weight at most $ck^t$.

\begin{remark}
Clearly, if each surface in the hierarchy is exponentially controlled, then the entire hierarchy is
also exponentially controlled. However, there is one circumstance where we cannot ensure
that the surfaces are exponentially controlled. This is when the manifold $M_i$ is a union of solid
tori and 3-balls, and where the solid tori patch together to form a circle bundle embedded
in $M$. Any two such solid tori intersect in a collection of clean annuli. The next surface in
the hierarchy will be a broken section. It might not be the case that any two sections
of the circle bundle are strongly equivalent. However, applying the techniques
of Section \ref{Sec:CircleBundle}, and using specific properties of the partial hierarchy that we have
constructed, we are able to find a homeomorphism of $M$ that is isotopic
to the identity on $\partial M$ and preserving all the surfaces in the partial hierarchy,
taking one section to another.
\end{remark}

The following result will be crucial.

\begin{theorem}
\label{Thm:WeaklyFundamentalHierarchy}
Let $M$ be the exterior of a non-split link in the 3-sphere other than the unknot.
For each component of $\partial M$ that is incident to a Seifert fibred JSJ piece other than $T^2 \times [0,1]$,
suppose that $P \cap \partial M$ is a single curve with slope $0$ or $\pm 1$, other than
the fibre slope. Suppose also that $P$ is disjoint from the other components of $\partial M$.
When $M$ is a copy of $T^2 \times [0,1]$, then suppose that $\partial M \cap P$
consists of a longitude in each component of $\partial M$.
Then $(M,P)$ has an exponentially controlled hierarchy.
\end{theorem}

The outline of the proof is as follows. Suppose that we have formed a partial hierarchy that is
exponentially controlled, and let $M_i$ be the final manifold. Suppose that it is not a disjoint union of 3-balls.
We would like to find an exponentially controlled surface in $M_i$. Since $M_i$ is not a disjoint union of 3-balls,
it contains a properly embedded (non-separating) orientable incompressible boundary-incompressible surface $S$.
In the case that $M_i$ contains no properly embedded essential clean annuli
or essential tori, then Theorem \ref{Thm:MakeWeaklyFundamental} can probably be used to make $S$ weakly fundamental. 
On the other hand, if $M_i$ contains an essential clean annulus or essential torus, then it probably contains a JSJ torus or annulus.
In this case, Theorem \ref{Thm:JSJUnionFundamental} can be used to make the JSJ annuli and tori fundamental. 

There are, of course, several complications. One is that Theorem \ref{Thm:MakeWeaklyFundamental} 
also has a hypothesis about the way that the boundary pattern intersects the incompressible boundary tori of $M_i$, and this leads to
some technical issues. 

The second complication is that it is not clear that the above procedure terminates.
Fortunately, one can guarantee that the decomposition along $S$ reduces some notion of complexity, and hence
that the hierarchy eventually terminates, as long as one is careful about the choice of surface $S$.
A significant complication is that the JSJ annuli in $M_i$ may in fact be boundary parallel.
Cutting along such a surface produces a solid torus together with a copy of the original manifold,
but with different boundary pattern. In order not to get stuck in a repeating loop, considerable
care is required here.

The third main complication is that it is possible for a manifold to contain an essential clean annulus or essential torus,
and yet have no JSJ torus or annulus. This is clearly problematic, because neither 
Theorem \ref{Thm:MakeWeaklyFundamental} nor Theorem \ref{Thm:JSJUnionFundamental} can be used.
But according to Theorem \ref{Thm:JSJPieces}, this only occurs in two very restricted situations:
when a component of $M_i$ is Seifert fibred with pattern consisting of fibres, and when a component of $M_i$ 
is an $I$-bundle over a surface and the pattern $P_i$ is a collection of simple closed curves
in the vertical boundary of the bundle. 
This second case arises, for example, when the first surface in the hierarchy is a fibre in a fibration over the
circle. When Haken first developed the theory of hierarchies \cite{Haken}, this was a situation that he could
not handle. Fortunately, in the case of a non-split link in the 3-sphere other than the unknot, it is now known that one
can avoid this situation, as follows.

Let $S$ be a surface properly embedded in a compact orientable 3-manifold $M$.
Then $S$ is a \emph{fibre surface} if it is a fibre in a fibration of $M$ over the circle.
This is equivalent to $M \cut S$ having the structure of an $I$-bundle with horizontal boundary equal to the
copies of $S$. 

The surface $S$ is a \emph{semi-fibre} if it  is connected and separating and each component of $M \cut S$ is homeomorphic
to an $I$-bundle over some non-orientable surface, with $S$ forming the horizontal boundary of the bundle. 
Semi-fibres do not arise as properly embedded surfaces when
$M$ is the exterior of a link in the 3-sphere. This is because the zero-section of each $I$-bundle is a non-orientable surface,
and one may join these two surfaces by annuli in the boundary to form a closed non-orientable surface properly embedded in $M$.
However, the 3-sphere does not support any such surface.

The following is due to Culler and Shalen \cite{CullerShalen}.

\begin{theorem}
\label{Thm:NotFibreSemiFibre}
Let $M$ be the exterior of a hyperbolic link. Then $M$
contains a properly embedded orientable incompressible, boundary-incompressible surface that
is neither a fibre nor a semi-fibre, and with no sphere or disc components.
\end{theorem}

We now give a procedure for finding an exponentially controlled hierarchy for the exterior $M$ of a non-split link other than the unknot.
We give $M$ the following boundary pattern. For each component of $\partial M$ that is incident to a Seifert fibred JSJ piece
of $M$ other than $T^2 \times [0,1]$, we give it boundary pattern that equals a single curve with slope $0$ or $\pm 1$, chosen so that its slope does
not agree with the fibre slope. The remaining components of $\partial M$ are given empty boundary pattern.
In the case where $M$ is a copy of $T^2 \times [0,1]$ (which is exactly when the
link is the Hopf link), then we we pick boundary pattern that is a longitude on each boundary component.
\begin{enumerate}
\item If $M$ has any JSJ tori, decompose along two parallel copies of each torus. Let $X$ be the union of the copies of $T^2 \times I$ between these tori. We will not touch any component of $X$ until both of its boundary components have inherited non-empty boundary pattern.
\item For each Seifert fibred JSJ piece of $M_2 - X$, pick a properly embedded annulus 
as in Section \ref{Subsec:AnnuliSFS}. Insert one or two copies of it, depending on whether the annulus is separating or non-separating.
Repeat this process until each Seifert fibred component of $M_2 - X$ has been decomposed into solid tori.
\item In each hyperbolic complementary region, find a properly embedded orientable incompressible boundary-incompressible surface that
is not a union of fibre surfaces and that has no sphere or disc components. Cut along it or two parallel copies of it, depending on whether it is separating
or non-separating.
\item We now embark upon a sequence of the following decompositions. The process continues until one of two situations arises. If every component of the 3-manifold is a ball, then the procedure terminates with the required hierarchy. If every component of the 3-manifold is a 3-ball or a solid torus with a clean longitude, and there is at least one component of the latter type, then pass to step (5). The decompositions are as follows:
\begin{enumerate}[(i)]
\item If a component of the 3-manifold contains a clean essential annulus that is parallel to an annulus in the boundary, then cut along this annulus.
\item If a component of the 3-manifold does not contain an annulus as in (i), but has some JSJ annuli, then cut along them all.
\item Suppose that a component of the 3-manifold is a copy of $T^2 \times [0,1]$ with non-empty pattern consisting of essential parallel curves. If it is a component of $X$, suppose also that both components of $T^2 \times \{ 0, 1 \}$ have non-empty intersection with the pattern. Pick an essential clean annulus, joining the two boundary components, and cut along two copies of it.
\item Suppose that a component of the 3-manifold is a copy of $T^2 \times [0,1]$ and is simple (with respect to the boundary pattern).
 If it is a component of $X$, suppose also that both components of $T^2 \times \{ 0, 1 \}$ have non-empty intersection with the pattern. 
Pick an incompressible boundary-incompressible annulus joining the two boundary components and cut along two copies of it.
\item If a component of the 3-manifold is simple (with respect to the boundary pattern), and its boundary has a compression disc (that intersects the boundary pattern), but the manifold is not a solid torus with a clean longitude, then pick a disjoint collection of such discs that are boundary-incompressible (with respect to the boundary pattern), taking two parallel copies of a disc if it is non-separating, and cut along this collection.
\item If a component of the 3-manifold is simple and its boundary is \linebreak incompressible, but it is not a copy of $T^2 \times [0,1]$, then, in this \linebreak component, pick a properly embedded orientable incompressible surface with non-empty boundary, with no sphere or disc components, that is boundary-incompressible with respect to the inherited pattern and with respect to the empty pattern. Cut along one or two parallel copies of it, depending on whether it is separating or non-separating. 
\end{enumerate}
\item By reaching this stage, every component of the 3-manifold is a 3-ball or a solid torus with a clean longitude.
We now perform the following decompositions until no more such decompositions are possible.
\begin{enumerate}[(i)]
\item If a solid torus contains a necklace annulus in its boundary, then push this annulus into the manifold
so that it becomes properly embedded, and then cut along it. This creates a new solid torus. Cut along two copies of a meridian disc
for this solid torus that intersects the pattern as few times as possible.
\item If two solid tori have a component of intersection that is a disc, then cut along a parallel copy of this disc lying in one of
the solid tori.
\end{enumerate}
\item Upon reaching this stage, it is still the case that
every component of the 3-manifold is a 3-ball or a solid torus with a clean longitude.
Consider the union of these solid tori. This is a circle bundle over a surface. The interior of the circle
bundle intersects the union of the previous surfaces in the partial hierarchy in a union of annuli $A$. 
Let $S$ be a properly embedded surface that is isotopic, via an isotopy preserving $A$, to a section.
Suppose that, among such surfaces, $S$ intersects the boundary pattern $P_n$ as few times as possible.
Then $M$ has a broken section that is $(3c,k)$-nearly exponentially controlled, where $c = -2\chi(S) +  |S \cap P_n| + 2|S|$ and $k = 2^{15}$.
This therefore decomposes the solid tori to 3-balls.
\end{enumerate}

\begin{remark} 
\label{Rem:MultipleComponents}
In the above procedure, each surface $S_i$ lies in a single component of $M_i$. However, that restriction can be weakened,
as follows. We can instead require that no two components of $M_i$ that contain components of $S_i$ are incident across an
earlier surface of the hierarchy.
\end{remark}

\begin{theorem}
\label{Thm:ProducesWeaklyFundamentalHierarchy}
This procedure produces a $(\mu,2^{15})$-exponentially controlled hierarchy, where $\mu$ is the maximal value of
$-6\chi(S_i) + 3|S_i \cap P_i| + 6 |S_i|$ over all surfaces $S_i$ in the hierarchy, and where $P_i$ is the boundary
pattern of the $i$th manifold.
\end{theorem}

\begin{proof}
The fact that each of these surfaces is $(\mu,2^{15})$-exponentially controlled follows from the results in Section \ref{Sec:NormalSurfaces}, apart from the broken section in (6) which we will discuss below. Specifically, the fact that each JSJ torus is fundamental is the content of Theorem \ref{Thm:JSJUnionFundamental}. Propositions \ref{Prop:AnnuliWeaklyFundamental}, \ref{Prop:MakeWeaklyFundamentalTorusTimesI} and \ref{Prop:TorusTimesI} give that the surfaces in (2) are weakly fundamental. Theorem \ref{Thm:MakeWeaklyFundamental} implies that the surfaces in (3) are weakly fundamental. Similarly, the surfaces in (4) and (5) are exponentially controlled, either by Theorem \ref{Thm:MakeWeaklyFundamental}, Theorem \ref{Thm:JSJUnionFundamental}, Proposition \ref{Prop:MakeWeaklyFundamentalTorusTimesI}, Proposition \ref{Prop:TorusTimesI} or Lemma \ref{Lem:BoundaryParallel}. 
However, Theorem \ref{Thm:MakeWeaklyFundamental} has some hypotheses about toral boundary components of $\partial M_i$ that need to be checked.

We claim that from step (3) onwards $M_i$ has no essential tori. Suppose that on the contrary there is a properly embedded essential torus $T$ in $M_i$. Note that $T$ is disjoint from the surfaces $S_1, \dots, S_{i-1}$ because it lies in the interior of $M_i$. Since $T$ is essential in $M_i$, it is $\pi_1$-injective in $M_i$. As at each stage of the partial hierarchy, we have decomposed along a 2-sided incompressible surface, $M_i$ $\pi_1$-injects into $M$. Hence, $T$ is incompressible in $M$. It is disjoint from the JSJ surfaces in $M$, since these are the first decomposition in (1). It does not lie in a Seifert fibred piece, since these are decomposed into solid tori in (2). Hence $T$ is boundary parallel in $M$ or parallel to a JSJ torus of $M$. So, there is a copy of $T^2 \times [0,1]$ between this torus and a boundary component of $M_2$, where $T^2 \times \{ 1 \} = T$. Now each surface in the partial hierarchy is incompressible, and the only connected incompressible surfaces properly embedded in $T^2 \times [0,1]$ disjoint from $T^2 \times \{ 1 \}$ are parallel to a subsurface of $T^2 \times \{ 0 \}$. Thus, we deduce that in $M_i$, there is a copy of $T^2 \times [0,1]$ with $T^2 \times \{ 0 \}$ being a component of $\partial M_i$ and $T^2 \times \{ 1 \} = T$. In other words, $T$ is boundary parallel in $M_i$ and hence not essential, as claimed.

We also claim that from step (3) onwards, each incompressible toral boundary component of $M_i$ either is clean or lies in a component of $M_i$ that is a copy of $T^2 \times [0,1]$. Consider an incompressible toral boundary component $T$ of some $M_i$. As argued above, $T$ is boundary parallel in $M$ or parallel to a JSJ torus of $M$. Suppose that the component of $M_i$ containing $T$ is not a copy of $T^2 \times [0,1]$. Hence, $T$ lies in $M_2 - X$. If $T$ is a boundary component of $M_2$, then it must lie in a hyperbolic piece of $M_2$ since the Seifert fibred pieces are decomposed into solid tori in (2). But then $T$ is clean, because any component of $X$ on the other side of $T$ is not decomposed until the copy of $T$ in $X$ inherits non-empty pattern, by which stage $T$ is no longer a component of $\partial M_i$. If $T$ intersects a component of $\partial M_2$ but does not lie entirely in $\partial M_2$, then consider the first surface $S_j$ that intersects the region between $T$ and $\partial M_2$. This is parallel to a clean subsurface of $\partial M_j$ and hence is boundary-compressible, contrary to our construction. Finally, suppose that $T$ is disjoint from $\partial M_2$. Then it is parallel to a boundary component $T'$ of $M_2$. Between $T$ and $T'$ is a copy of $T^2 \times [0,1]$, and as argued above, this implies that the component of $M_i$ containing $T$ is a copy of $T^2 \times [0,1]$, as claimed.

When a component of $M_i$ that is a copy of $T^2 \times [0,1]$ is decomposed, we ensure in (4)(iii) and (4)(iv) that we cut along an annulus joining the two boundary components. In (4)(iii), we use Proposition \ref{Prop:TorusTimesI}. In (4)(iv), we use Proposition \ref{Prop:MakeWeaklyFundamentalTorusTimesI}.
Hence, when we apply Theorem \ref{Thm:MakeWeaklyFundamental}, its hypotheses hold. In particular, when we use it in steps (3) and (4), any incompressible toral boundary components in the relevant component of $M_i$ are clean.

We need to show that the process terminates, with the final manifold being a collection of balls. To do so, we use the machinery developed by Matveev in \cite{Matveev} that was based on work of Haken \cite{Haken}.

The procedure described here should be compared with the procedure given by Matveev in \cite{Matveev}. The initial part of Matveev's procedure, described in \cite[Section 6.5.1]{Matveev} is to perform the following moves whenever they are possible:
\begin{itemize}
\item[(E1)] If $(M_i, P_i)$ contains an essential torus, then add one or two parallel copies of such a torus depending on whether the torus is separating or non-separating.
\item[(E2)] If $(M_i, P_i)$ contains a clean essential longitudinal annulus that is not parallel to an annulus in $\partial M_i$, then add one or two parallel copies of such an annulus, depending on whether the annulus is separating or non-separating.
\item[(E3)] If $(M_i, P_i)$ contains a clean essential annulus that is parallel to an annulus in the boundary, then cut along this annulus.
\item[(E4)] If a component of $(M_i, P_i)$ is simple, and its boundary has a compression disc, but the manifold is not a solid torus with a clean longitude, then pick a compression disc that intersects $P_i$ as few times as possible, and cut along one or two parallel copies of this, depending on whether it is separating or non-separating.
\item[(E5)] If a component of $(M_i, P_i)$ is simple and its boundary is incompressible, then, in this component, pick a properly embedded connected orientable incompressible boundary-incompressible surface $F$ with non-empty boundary, other than a sphere or disc, and for which $-\chi(F) + |F \cap P_i|$ is as small as possible. Then cut along one or two parallel copies of it, depending on whether it is separating or non-separating.
\end{itemize}

Matveev shows in Corollary 6.5.7 in \cite{Matveev} that the above procedure terminates, although the final manifold is not necessarily a 3-ball. Matveev then uses further `extension moves' (E6)-(E10) and `auxiliary moves' (E3$'$), (E4$'$) and (E4$''$).

We do not follow Matveev's approach precisely for several reasons. First of all, we require the hierarchy to be exponentially controlled, whereas Matveev requires less of his hierarchies. Secondly, we are in a more particular situation where the initial manifold is the exterior of a link, whereas he has to work with more general Haken 3-manifolds.

We claim that by the time we have completed steps (1)-(4) above, none of Matveev's moves (E1)-(E5) can be applied. Hence, the manifold with boundary pattern that we obtain after completing moves (1)-(4) has all the properties that Matveev requires before he embarks on (E6)-(E10) and (E3$'$), (E4$'$) and (E4$''$). Let $(M_i, P_i)$ be the manifold with pattern obtained after steps (1)-(4).

Suppose now that some component of $M_i - X$ contains a clean essential longitudinal annulus, but that none of 4(i), 4(ii) and 4(iii) can be applied. After a pattern-isotopy, this annulus lies in a JSJ piece for $M_i$, which by Theorem \ref{Thm:JSJPieces} is Seifert fibred with pattern consisting of fibres. This is not a solid torus, because such a manifold does not support an essential longitudinal annulus. The Seifert fibre pieces of $M$ were decomposed in (2). Hence, the component of $M_i - X$ containing the annulus is a copy of $T^2 \times [0,1]$. It might be a regular neighbourhood of a JSJ torus for $M$ or it might be parallel to component of $\partial M$. Its pattern must be non-empty, since the annulus is longitudinal. Hence, 4(iii) can be applied to it, which is contrary to hypothesis.

So we have shown that neither (E1) nor (E2) can be applied to $M_i - X$. Note that (E3) is exactly 4(i). Note further that if (E4) can be applied, then so too can 4(v). Similarly, if (E5) can be applied to a component of $M_i - X$, then so too can 4(vi) or (iv). Since (4) has been completed, we deduce that none of (E3) - (E5) can be applied to $M_i - X$.

Hence, by Lemma 6.5.10 in \cite{Matveev}, each component of $(M_i - X, P_i - X)$ is either a 3-ball, a solid torus having a clean longitude or an $I$-bundle chamber. Matveev defines an \emph{$I$-bundle chamber} to be an $I$-bundle over some compact surface with negative Euler characteristic, with boundary pattern consisting a non-empty collection of simple closed curves in each vertical boundary component. He also requires that in each vertical boundary component of the $I$-bundle, its two boundary curves are part of the pattern.

We will show that in fact there are no $I$-bundle chambers. As explained by Matveev in the proof of Proposition 6.5.14 in \cite{Matveev}, each $I$-bundle chamber has horizontal boundary equal to the horizontal boundary of another $I$-bundle chamber. Hence, by \cite[Proposition 6.5.14]{Matveev}, the union of the $I$-bundle chambers, minus a small regular neighbourhood of the vertical boundary, consists of the following four types of components:
\begin{enumerate}
\item a copy of $F \times I$ where $F$ is a compact orientable surface, with $F \times \partial I \subset \partial N(K)$;
\item a twisted $I$-bundle $F \tilde \times I$, where $F$ is a compact non-orientable surface, with the horizontal boundary of the $I$-bundle in $\partial N(K)$;
\item a fibration over the circle with fibre a compact orientable surface;
\item a union of two twisted $I$-bundles glued along their horizontal boundary, called a \emph{quasi-Stallings manifold}.
\end{enumerate}
We say that a component of the union of the $I$-bundle chambers, minus a small regular neighbourhood of their vertical boundary, is a \emph{chamber compendium}. Its \emph{vertical boundary} is the union the vertical boundaries of its constituent $I$-bundles. It is a collection of annuli and tori properly embedded in $M$, the exterior of $K$.

We claim that the vertical boundary of each chamber compendium is incompressible in $M$. Suppose not, and hence that the vertical boundary is compressible. Hence, there is a vertical boundary component $T$ that admits a compression disc with interior disjoint from the vertical boundary. Any compressible torus in the irreducible 3-manifold $M$ either lies in a 3-ball or bounds a solid torus. Any compressible annulus lies in a 3-ball. The surface $T$ is incident to at least one horizontal boundary component $F$ of an $I$-bundle chamber. Each component of $\partial F$ represents a non-trivial element of the fundamental group of the chamber, and hence is a non-trivial element of $\pi_1(M)$. Hence, $T$ cannot lie in a 3-ball. It is therefore a torus bounding a solid torus $V$. Consider the first surface $S_i$ in the hierarchy to intersect $V$. Then $S_i \cap V$ is incompressible in $V$, for the following reason. The boundary of a compression disc for $S_i \cap V$ would have to bound a disc in $S_i$, since $S_i$ is $\pi_1$-injective. This disc does not lie wholly in $V$. But each surface in the hierarchy incident to $V$ but not lying wholly in $V$ must contain a horizontal boundary component $F$ of an $I$-bundle chamber incident to $V$. Hence, we deduce that a boundary component of $F$ is homotopically trivial in $M$, which is ruled out above. Since $S_i \cap V$ is incompressible and orientable in the solid torus $V$, it is a collection of incompressible annuli and discs. In fact, it cannot contain any discs, since this would again imply that a horizontal boundary component of a chamber did not $\pi_1$-inject into $M$. There is therefore some annulus $A$ of $S_i \cap V$ that is outermost in $V$. This is incident to horizontal boundary components $F$ and $F'$ of $I$-bundle chambers. Lying between $F$ and $F'$ is a union of $I$-bundle chambers, which is again an $I$-bundle. It is possible that $F = F'$ if they bound a twisted $I$-bundle. Enlarge this $I$-bundle so that it contains
the region in $V$ between $A$ and $\partial V$. This is again an $I$-bundle, but $A$ is now a vertical boundary component. Now the base of this $I$-bundle has negative Euler characteristic and so contains an essential properly embedded arc with both endpoints on the boundary component corresponding to $\partial A$. Over this arc is a disc $E$ consisting of $I$-fibres, with $E \cap S_i = \partial E$. By the incompressibility of $S_i$, $\partial E$ bounds a disc $E'$ in $S_i$. The intersection of $E'$ with $F$ is a planar surface component of $F \cut \partial E$. Hence, we deduce that some component of $\partial F$ bounds a disc in $S_i$, which we have shown to be impossible.

Thus, we have shown that the vertical boundary of each chamber compendium is incompressible. 

The vertical boundary of chamber compendium cannot contain any annulus component, since (1) or (2) would then apply, and hence the horizontal boundary of the bundle would lie in $\partial M$. It would then be a disc, annulus or torus, all of which are ruled out by the assumption that the base of the bundle has negative Euler characteristic.

So, the vertical boundary of the chamber compendiums consists of incompressible tori. Each torus is therefore boundary parallel or JSJ or lying in the interior of a Seifert fibre piece of the JSJ decomposition of $M$. The latter situation can be ruled out as follows. Suppose some vertical boundary component lies in a Seifert fibre piece but is not JSJ or boundary parallel. This Seifert fibre space has been decomposed into solid tori in step (2). Hence, we deduce that some $I$-bundle chamber in the compendium is a 3-ball or solid torus, but the $I$-bundle chambers cannot be 3-balls or solid tori.

Hence, each chamber compendium is either a component of the JSJ decomposition of $M$ or a copy of $T^2 \times I$ lying in a regular neighbourhood of a JSJ torus or a component of $\partial M$. The latter cannot happen because each chamber compendium is a surface bundle or Stallings manifold with fibre having negative Euler characteristic. Thus, we deduce that each chamber compendium is a component of the JSJ decomposition of $M$. The intersection between each surface of the hierarchy and the chamber compendium is horizontal in the chamber compendium. However, in step (3), we decomposed the hyperbolic JSJ pieces along a surface that was neither a fibre nor a semi-fibre, and we decomposed the Seifert fibred pieces into solid tori.

Thus, we have shown that there are no $I$-bundle chambers.

So, once Step (4) has been completed, each component of $M_i - X$ is a 3-ball or solid torus. We claim in fact that each component of $M_i$ is a 3-ball or solid torus. Consider any component of $X$. The two components of $M_i - X$ that are incident to it have been decomposed into solid tori and 3-balls. Hence, both its boundary components have inherited non-empty boundary pattern. Therefore the restriction in 4(iii) and (iv) on decomposing such a component of $X$ no longer applies. Hence, this component of $X$ has indeed been decomposed into solid tori in 4(iii) or 4(iv). Any further decompositions to these components only give solid tori and 3-balls. This proves the claim.

We now apply the moves in (5). These are exactly moves (E3$'$), (E4$'$) of Matveev \cite{Matveev}. Note that Matveev's move (E4$''$) is not possible in our situation, since it only applies when the initial manifold $M_1$ is closed. After these moves are performed, the union of the solid torus chambers are what Matveev terms \emph{faithful} (as in Definition 6.5.52 in \cite{Matveev}). He proves in Lemma 6.5.53 that when this holds, the union of the solid toral chambers is a circle bundle over a surface. Moreover, each solid torus chamber is a union of circle fibres, and each component of intersection between them is an annulus that is a union of circle fibres. 
Thus, in step (6), this circle bundle is decomposed along a broken section. In particular, each remaining solid torus is decomposed along a meridian disc. Thus this procedure does terminate with a hierarchy.

We need to show that this hierarchy is $(\mu,2^{15})$-exponentially controlled. This was explained at the beginning of the proof for all surfaces in the hierarchy except the broken section, which is only $(\mu,2^{15})$-nearly exponentially controlled, by Proposition \ref{Prop:PerturbSection}. (See Section \ref{Sec:CircleBundle} for the precise definition of being nearly exponentially controlled.) In our case, the surface that we are considering is a broken section, that is obtained from a section of a circle bundle $Y$ lying in $M$. Let $P$ be the boundary pattern for $Y$. Then $\partial Y \cut P$ is a collection of discs and annuli. The annular components lie in $\partial M$ and are vertical in the fibration. Any homeomorphism of $Y$ that preserves each vertex, edge and circle of $P$ is, up to pattern-isotopy, a composition of Dehn twists about clean vertical annuli and tori. A Dehn twist about a vertical torus restricts to the identity on $\partial M$ and hence determines a strong equivalence of $M$. Consider a Dehn twist about a vertical annulus $R$. We claim that this Dehn twist is isotopic to the identity on $M$. 

Now, we assigned boundary pattern to components of $\partial M$ that were incident to Seifert fibred JSJ pieces, and this pattern consisted of a single essential curve not having a fibre slope. Hence, when the Seifert fibred pieces are decomposed in step (2), these components of $\partial M$ incident to JSJ pieces inherit pattern that decomposes these tori into discs. We therefore deduce that the annular components of $\partial Y \cut P$ are disjoint from those components of $\partial M$. The annulus $R$ has clean boundary in $\partial M$, and hence misses the Seifert fibred pieces of $M$. It is therefore inessential in $M$. Hence, a Dehn twist about $R$ gives a homeomorphism of $M$ that is isotopic to the identity, as claimed.

So suppose that we are given a triangulation of $M$ with $t$ tetrahedra in which the pattern and all the surfaces before the broken section are simplicial. The broken section is $(\mu,2^{15})$-nearly exponentially controlled, and hence, as argued above, there is a homeomorphism of $M$, isotopic to the identity on $\partial M$, preserving the earlier surfaces in the hierarchy, taking the broken section to one that has weight at most $\mu \cdot 2^{15t}$, as required.
\end{proof}

\subsection{Adequately separating hierarchies}
Let
$$M = M_1 \buildrel S_1 \over \longrightarrow M_2 \buildrel S_2 \over \longrightarrow \dots
\buildrel S_{n} \over \longrightarrow M_{n+1}$$
be a hierarchy. This is said to be \emph{adequately separating} if the following two conditions hold:
\begin{enumerate}
\item for each $i < n$, there is no curve in $M_i$ that intersects $S_i$ transversely once, and
\item for each $i < n$, $S_i$ is contained within a single component of $M_i$.
\end{enumerate}
Furthermore, either these two conditions hold for $i = n$ or $S_n$ is broken section of a circle bundle.

The first condition implies that for each component $Y$ of $M_i \cut S_i$, the restriction of
$M_i \cut S_i \rightarrow M_i$ to $Y$ is an injection. Hence, with both these conditions,
we deduce that no surface $S_j$ ($j < n$) runs over two parts of $\partial M_j$ that are identified
in $M$. Note that this remains true for $S_n$ even when it is a broken section of a circle bundle.
This is technically useful, since $\partial S_j$ creates boundary pattern, and if two parts of
$\partial S_j$ were to be identified in $M$, then this would create possibly unexpected vertices
in the boundary pattern.

We note the following, particularly in reference to Remark \ref{Rem:MultipleComponents}.

\begin{remark} 
\label{Rem:AdequatelySeparating}
The hierarchy constructed in Theorem \ref{Thm:ProducesWeaklyFundamentalHierarchy} 
is adequately separating.
\end{remark}

\section{Handle structures}
\label{Sec:HandleStructures}

We will deal with handle structures on 3-manifolds. These will always have the following properties:

\begin{convention}
\label{Convention:HS}
Let $\mathcal{H}$ be a handle structure on a 3-manifold. Then, for $0 \leq i \leq 3$,
the union of the $i$-handles is denoted $\mathcal{H}^i$. The \emph{index} of an $i$-handle is $i$.
We will always require the handle structure $\mathcal{H}$ to satisfy the following conditions.
\begin{enumerate}
\item For each $i$-handle $D^i \times D^{3-i}$, its intersection with the handles of lower index is
$\partial D^i \times D^{3-i}$.
\item For each 2-handle $H_2 = D^2 \times D^1$ and 1-handle $H_1 = D^1 \times D^2$, 
the intersection $H_1 \cap H_2$ must respect the product structure of each, in the sense that $H_1 \cap H_2$ is
equal to $\alpha \times D^1$ in $H_2$, where $\alpha$ is a collection of arcs in $\partial D^2$, and is equal to
$D^1 \times \beta$ in $H_1$, for a collection of arcs $\beta$ in $\partial D^2$.
\item Each 2-handle runs over at least one 1-handle.
\end{enumerate}
\end{convention}

A triangulation of a 3-manifold $M$ determines a handle structure for $M$, where each $i$-simplex not lying
in $\partial M$ gives rise to a $(3-i)$-handle.

When analysing handle structures of 3-manifolds, the following surface is of particular relevance.
We let $\mathcal{F}$ be $\mathcal{H}^0 \cap (\mathcal{H}^1 \cup \mathcal{H}^2)$. This is a subsurface
of the spheres $\partial \mathcal{H}^0$. It has a handle
structure where the 0-handles, denoted $\mathcal{F}^0$, are $\mathcal{H}^0 \cap \mathcal{H}^1$,
and the 1-handles, denoted $\mathcal{F}^1$, are $\mathcal{H}^0 \cap \mathcal{H}^2$.
Note that $\mathcal{F}^1$ is indeed a collection of 1-handles, because each 2-handle
of $\mathcal{H}$ runs over at least 1-handle of $\mathcal{H}$, and so $\mathcal{F}^1$
has no annular components.

\subsection{Normal surfaces in handle structures}
\label{Sec:NormalHS}

Let $\mathcal{H}$ be a handle structure for a 3-manifold $M$. Then a surface properly embedded
in $M$ is said to be in \emph{standard form} with respect to $\mathcal{H}$ if
\begin{enumerate}
\item it is disjoint from the 3-handles;
\item it intersects each 2-handle $D^2 \times D^1$ in discs of the form $D^2 \times F$, 
for some finite set of points $F$ in the interior of $D^1$;
\item it intersects each 1-handle $D^1 \times D^2$ in discs of the form $D^1 \times \alpha$,
for a collection of properly embedded arcs $\alpha$ in $D^2$;
\item it intersects each 0-handle in a collection of properly embedded discs.
\end{enumerate}

A standard surface $S$ is \emph{normal} if, for each component $D$ of $S \cap \mathcal{H}^0$,
the following hold:
\begin{enumerate}
\item $D$ intersects each component of $\mathcal{F}^1$ in at most one arc;
\item $D$ intersects each component of $\partial \mathcal{H}^0 \cut \mathcal{F}$ in at
most one arc and no simple closed curves that bound discs in $\partial \mathcal{H}^0 \cut \mathcal{F}$;
\item $D$ does not intersect a component of $\partial \mathcal{F}^0 \cut \mathcal{F}^1$
more than once;
\item $D$ does not intersect components of 
$\partial \mathcal{H}^0 \cut \mathcal{F}$ and
$\mathcal{F}^1$ 
that are adjacent.
\end{enumerate}

These components $D$ of $S \cap \mathcal{H}^0$ are called \emph{elementary normal discs}.

\subsection{Curves and graphs in the boundary}

Let $\mathcal{H}$ be a handle structure for a 3-manifold $M$. We say that a union of disjoint simple closed curves $C$ in $\partial M$ is
\emph{standard} if 
\begin{enumerate}
\item $C$ is disjoint from the 2-handles of $\mathcal{H}$;
\item $C$ intersects each 1-handle $D^1 \cap D^2$ in $D^1 \times F$
for some finite subset $F$ of $\partial D^2$;
\item $C$ intersects each 0-handle in a collection of arcs.
\end{enumerate}

A graph $P$ in the boundary of $M$ is in \emph{transverse form} if 
\begin{enumerate}
\item $P$ is disjoint from the 0-handles of $\mathcal{H}$;
\item $P$ intersects each 1-handle $D^1 \cap D^2$ in a subset of
$\{ \ast \} \times \partial D^2$ for some point $\ast$ in the interior of $D^1$;
\item each vertex of $P$ lies in a 2-handle of $\mathcal{H}$;
\item for any 2-handle $D^2 \times D^1$ of $\mathcal{H}$, the intersection between $P$
and a component of $D^2 \times \partial D^1$ is at most one vertex of $P$
plus arcs that each run from the vertex to the boundary of $D^2 \times \partial D^1$.
\end{enumerate}

\subsection{The weight and extended weight of a standard surface}
\label{Sec:WeightHS}

The \emph{weight} of a standard surface $S$ in a handle structure $\mathcal{H}$ is the number of components of $S \cap \mathcal{H}^2$.
Note that, when the handle structure arises from a triangulation $\mathcal{T}$ of a closed 3-manifold, then there is a
bijection between the components of $S \cap \mathcal{H}^2$ and the points of intersection
between $S$ and the 1-skeleton of $\mathcal{T}$. Thus, in this situation, the two measures of
weight coincide.

We also introduce another measure of complexity. The \emph{extended weight} of a standard surface $S$ is defined
to be
$$|S \cap \mathcal{H}^2| + |S \cap \mathcal{H}^1|.$$
The rationale for this definition is that weight is a somewhat imperfect measure of complexity. For example,
there can be surfaces that avoid the 2-handles completely but that run through the 0-handles and
1-handles many times. These have zero weight but large extended weight. Furthermore, we have the following lemma.

\begin{lemma}
\label{Lem:NormaliseHS}
Let $\mathcal{H}$ be a handle structure for a compact 3-manifold $M$. Let $P$ be a boundary pattern that
is in transverse form. Let $S$ be a standard incompressible boundary-incompressible surface with no component that is a
sphere or disc intersecting $P$ at most twice. Then $S$ is pattern-isotopic to a normal surface $S'$ such that the
extended weight of $S'$ is at most the extended weight of $S$. Furthermore, for each 2-handle of $\mathcal{H}$,
$S'$ intersects the 2-handle at most as many times as $S$ did.
\end{lemma}

This is essentially proved in \cite[Theorem 3.4.7]{Matveev}, where normalisation moves are given. It is straightforward to
check that they do not
increase the extended weight.

\subsection{Distinguished subsurfaces of the boundary}
\label{Subsec:DistinguishedSubsurfaces}

When we analyse a handle structure of a 3-manifold $N$, its boundary will sometimes contain a distinguished
subsurface $X$. In addition, $\partial N$ will contain some arcs $\alpha$ properly embedded in $\partial N \cut X$.
These will satisfy the condition that $\partial X$ is a collection of standard curves and that each component of $\alpha$ lies in a 0-handle. We then
say that $\mathcal{H}$ is a handle structure for $(N,\alpha, X)$.

When dealing with such extra structure on the boundary, we will slightly vary the definition of a normal surface $S$.
We now say that a properly embedded surface $S$ is \emph{normal} if it is standard, its boundary is disjoint from $X$, and
for each component $D$ of $S \cap \mathcal{H}^0$,
the following hold:
\begin{enumerate}
\item $D$ intersects each component of $\mathcal{F}^1$ in at most one arc;
\item $D$ intersects each component of $\partial \mathcal{H}^0 \cut (\mathcal{F} \cup \alpha \cup X)$ in at
most one arc and no simple closed curves that bound discs in $\partial \mathcal{H}^0 \cut (\mathcal{F} \cup \alpha \cup X)$;
\item $D$ does not intersect a component of $\partial \mathcal{F}^0 \cut (\mathcal{F}^1 \cup \alpha \cup X)$
more than once;
\item $D$ does not intersect components of $\partial \mathcal{H}^0 \cut (\mathcal{F} \cup \alpha \cup X)$ and
$\mathcal{F}^1$ that are incident.
\end{enumerate}
We also say that $S$ is \emph{crudely normal} if it is standard, its boundary is disjoint from $X$, and
each component $D$ of $S \cap \mathcal{H}^0$ satisfies (1) 
above, as well as:
\begin{enumerate}
\item[($3'$)] no component of $D \cap \calF^0$ is an arc with endpoints on the same component of
$\partial \mathcal{F}^0 \cut (\mathcal{F}^1 \cup \alpha \cup X)$;
\item[($4'$)] no component of $D \cap \calF^0$ is an arc with endpoints on components of \linebreak $\partial \mathcal{F}^0 \cut (\mathcal{F}^1 \cup \alpha \cup X)$ and
$\partial \mathcal{F}^0 \cap \mathcal{F}^1$ that are incident.
\end{enumerate}

\subsection{Parallelism}

Let $\mathcal{H}$ be a handle structure for $(N,P, X)$, and let $S$ be a normal surface.
We say that two disc components $D_0$ and $D_1$ of $S \cap \mathcal{H}^i$ are
\emph{parallel} if there is an isotopy of $N$, preserving all the handles of $\mathcal{H}$
and preserving $X$ and $P$, which takes $D_0$ onto $D_1$, and such that the
restriction of this isotopy to $D_0$ is an embedding of $D_0 \times [0,1]$
into $N$. We then say that $D_0$ and $D_1$ are of the same \emph{type}.

More generally, we say that two subsurfaces $S_0$ and $S_1$ of $S$ are said to
be \emph{normally parallel} if there are subsurfaces $S'_0$ and $S'_1$ of $S$,
each of which is a union of elementary normal discs, and satisfying
$S'_0 \supseteq S_0$ and $S'_1 \supseteq S_1$, and an isotopy $H \colon N \times [0,1] \rightarrow N$
preserving all the handles of $\calH$, and preserving $X$ and $P$, such that the following hold:
\begin{enumerate}
\item The restriction of $H$ to $S'_0 \times [0,1]$ is an embedding.
\item For each elementary normal disc $D$ of $S'_0$ and each $t \in [0,1]$,
$H(D,t)$ is an elementary normal disc.
\item $H(S'_0 \times \{ i \}) = S_i'$ for $i =0$ and $1$.
\item $H(S_0 \times \{ i \}) = S_i$ for $i =0$ and $1$.
\end{enumerate}

\section{The basic theory of arc presentations}
\label{Sec:ArcPresentations}

In this section, we give an introduction to rectangular diagrams and arc presentations, as
developed mostly by Cromwell \cite{Cromwell} and Dynnikov \cite{Dynnikov}.

\subsection{Rectangular diagrams}

A \emph{rectangular diagram} is a diagram in the plane $\mathbb{ R} \times \mathbb{ R}$,
consisting of arcs, each of which is horizontal (of the form $[s_1, s_2] \times \{ t \}$ for
$s_1, s_2, t \in \mathbb{ R}$) or vertical (of the form $\{ s \} \times [t_1, t_2]$ for $s, t_1, t_2 \in \mathbb{ R})$.
Each crossing arises when a vertical arc and a horizontal arc cross; the vertical arc is
required to be the over-arc. Moreover, no two arcs may be collinear. (See Figure 
\ref{Fig:RectangularDiagram}.)

There are the same number of horizontal and vertical arcs. This is the \emph{arc index}
of the rectangular diagram.

\begin{figure}[h]
\includegraphics[width=1.5in]{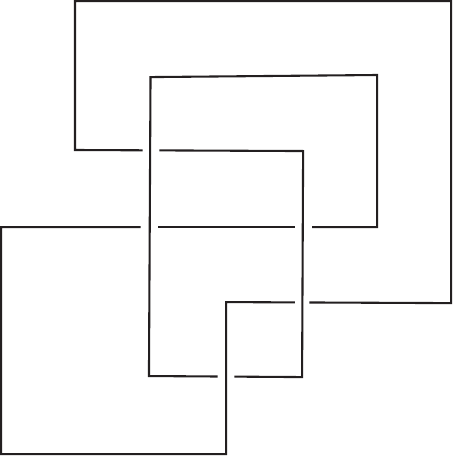}
\caption{A rectangular diagram}
\label{Fig:RectangularDiagram}
\end{figure}

It is an elementary matter to convert any link diagram into a rectangular diagram by an isotopy of the plane.
In \cite[Lemma 2.1]{Lackenby:PolyUnknot}, the following was proved, which gives a quantified version of this observation.

\begin{lemma}
\label{Lem:IsotopeToRectangular}
Let $D$ be a diagram of a link with $c$ crossings.
Then $D$ is isotopic to a rectangular diagram with arc index at most $(81/20) c$.
\end{lemma}

\subsection{Arc presentations}

Closely related is the concept of an arc presentation for a link.

One fixes a description of the 3-sphere as the join $S^1 \ast S^1$ of two circles.
Thus, a point in $S^3$ is specified by three coordinates $(\phi, \tau, \theta)$,
where $\phi, \theta \in \mathbb{ R} / (2 \pi \mathbb{ Z})$ and $0 \leq \tau \leq 1$.
The points $(\phi, 0, \theta)$ and $(\phi, 0, \theta')$ are identified for all $\theta$ and $\theta'$.
The resulting circle $\{ \tau = 0 \}$ is denoted $S^1_\phi$ and is called the \emph{binding circle}. 
Similarly,  $(\phi, 1, \theta)$ and $(\phi', 1, \theta)$ are identified for all $\phi$ and $\phi'$,
and the resulting circle $\{ \tau = 1\}$ is written $S^1_\theta$.
For $t \in \mathbb{ R}/(2 \pi \mathbb{ Z})$, the open disc $\{ \theta = t, \tau > 0 \}$ is termed a \emph{page}
and is denoted $\mathcal{D}_t$. 

A link $K$ is in an \emph{arc presentation} if
\begin{enumerate}
\item it intersects the binding circle in finitely many points, called \emph{vertices};
\item its intersection with each page is either empty or a single arc joining distinct vertices.
\end{enumerate}
The number of vertices is equal to the number of arcs, and is called the \emph{arc index}
of the arc presentation.

\subsection{The relationship between arc presentations and rectangular diagrams}

An arc presentation determines a rectangular diagram as follows. Let the arc index of the
arc presentation be $n$. For each arc in a page $\mathcal{D}_t$ joining vertices $s_1$ and $s_2$ in $[0, 2\pi)$,
insert the horizontal arc $[s_1, s_2] \times \{ t \}$ in the plane. For each vertex $s \in S^1_\phi$,
let $t_1, t_2$ be the $\theta$ values in $[0, 2 \pi)$ of the arcs incident to it, and insert a vertical
arc $\{ s \} \times [t_1, t_2]$ into the plane. These vertical and horizontal arcs in the
plane specify a rectangular diagram, once we declare that at each crossing, the
over-arc is the vertical one.

It is clear that this procedure may be reversed: given a rectangular diagram, one may
form the corresponding arc presentation.

It was shown by Cromwell \cite{Cromwell} that the arc presentation and the rectangular diagram
specify the same link.

\subsection{Modifications to a rectangular diagram}

In \cite{Cromwell}, Cromwell introduced a collection of modifications that one can make to a rectangular diagram:
\begin{enumerate}
\item[(1)] \emph{Cyclic permutation}: Here one cyclically permutes the horizontal (or vertical arcs).
\end{enumerate}

\begin{figure}[h]
\includegraphics[width=2.7in]{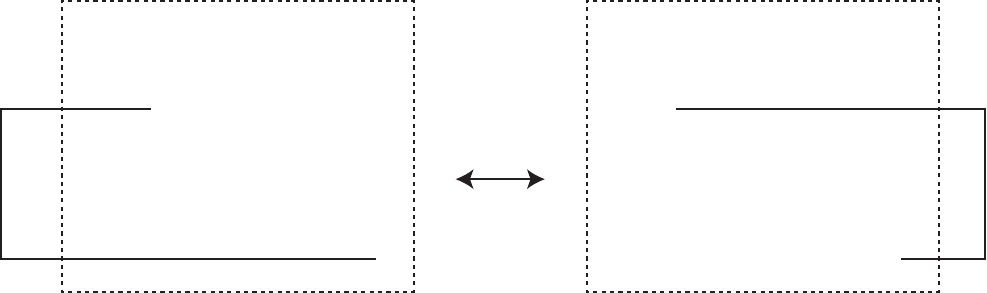}
\caption{Cyclic permutation of the vertical arcs}
\label{Fig:CyclicPerm}
\end{figure}

\begin{enumerate}
\item[(2)] \emph{Stabilisation/destabilisation}: These are local modifications to the diagram, as shown in Figure \ref{Fig:StabDestab},
which increase or decrease the arc index.
\end{enumerate}

\begin{figure}[h]
\includegraphics[width=3.2in]{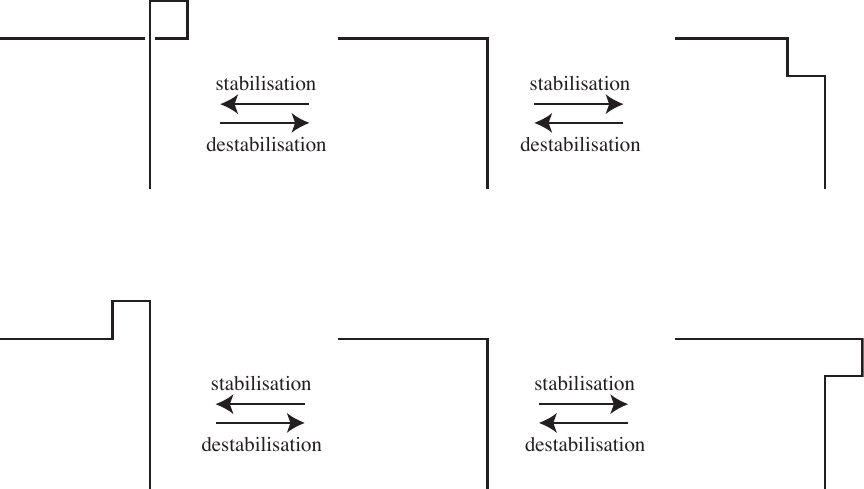}
\caption{Stabilisations and destabilisations}
\label{Fig:StabDestab}
\end{figure}

\begin{enumerate}
\item[(3)] \emph{Exchange move}: This moves two horizontal (or two vertical) arcs of the diagram past each other.
The arcs are required to have adjacent vertical heights, which means that no other arc has vertical height lying between
them. Also, the horizontal endpoints of these arcs must not be interleaved.
\end{enumerate}

\begin{figure}[h]
\includegraphics[width=3.2in]{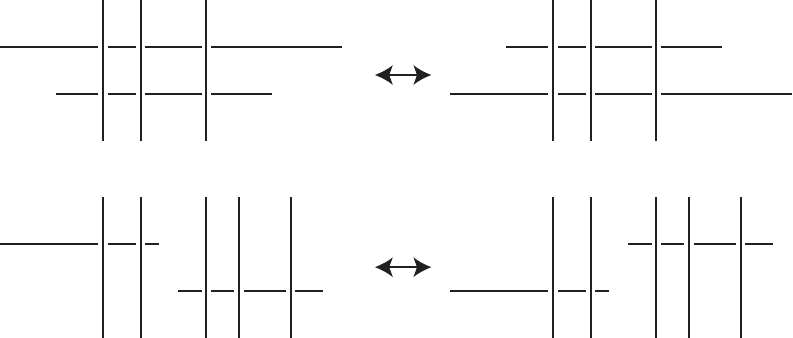}
\caption{Exchange moves}
\label{Fig:Exchange}
\end{figure}

Cromwell showed in \cite{Cromwell} that any two diagrams of a link differ by a sequence of these moves. In the case where
the link $K$ is unknot, Dynnikov \cite{Dynnikov} proved that the striking result that stabilisations (which are the only moves that increase the arc index)
are not required to reduce a rectangular diagram to the diagram with arc index 2.

\subsection{Upper bounds on Reidemeister moves}

It is clear that the above modifications to a rectangular diagram can be achieved
using Reidemeister moves, and it is unsurprising that one can obtain explicit
upper bounds. A stabilisation or a destabilisation can obviously be achieved using
at most one Reidemeister move. Exchange moves and cyclic permutations were
examined in \cite{Lackenby:PolyUnknot}, where the following bounds were established.

\begin{lemma}
\label{Lem:ExchangeRM}
Let $n$ be the arc index
of a rectangular diagram $D$. Suppose that an exchange move
is performed and let $D'$
be the resulting rectangular diagram. Then $D'$ and $D$
differ by a sequence of at most $n$ Reidemeister moves.
\end{lemma}

\begin{lemma}
Let $n$ be the arc index
of a rectangular diagram $D$. Suppose that a cyclic permutation
is performed on the vertical (or horizontal) arcs giving a rectangular diagram $D'$. Then $D$ and $D'$
differ by a sequence of at most $(n-1)^2$
Reidemeister moves.
\end{lemma}

\subsection{Generalised exchange moves}

These are defined as follows.

Let $0 < s_1 < s_2 < s_3 < 2 \pi$ be values of $\phi$ disjoint from the vertices of the link $K$. Let $0 \leq t_1 < t_2 < 2\pi$ be values of $\theta$ disjoint from the arcs of $K$. Suppose that, for each horizontal arc $[s, s'] \times \{ t \}$ of the rectangular diagram, the following hold:
\begin{enumerate}
\item If $t \in (t_1, t_2)$, then $\{ s, s' \}$ is not interleaved with $\{ s_2, s_3 \}$.
\item If $t \in S^1_\phi - (t_1, t_2)$, then $\{ s, s' \}$ is not interleaved with $\{ s_1, s_2 \}$. 
\end{enumerate}
Then one can change the rectangular diagram by shifting the $\phi$ value of all vertices between $s_1$ and $s_2$ so that they lie between $s_2$ and $s_3$ in the same order, and by shifting the $\phi$ value of all vertices between $s_2$ and $s_3$ so that they lie between $s_1$ and $s_2$ in the same order. This is a \emph{generalised exchange move}.

The effect on the rectangular diagram is shown in the case where $t_1 = 0$ in Figure \ref{Fig:GenExchange}.

\begin{figure}[h]
\includegraphics[width=2.9in]{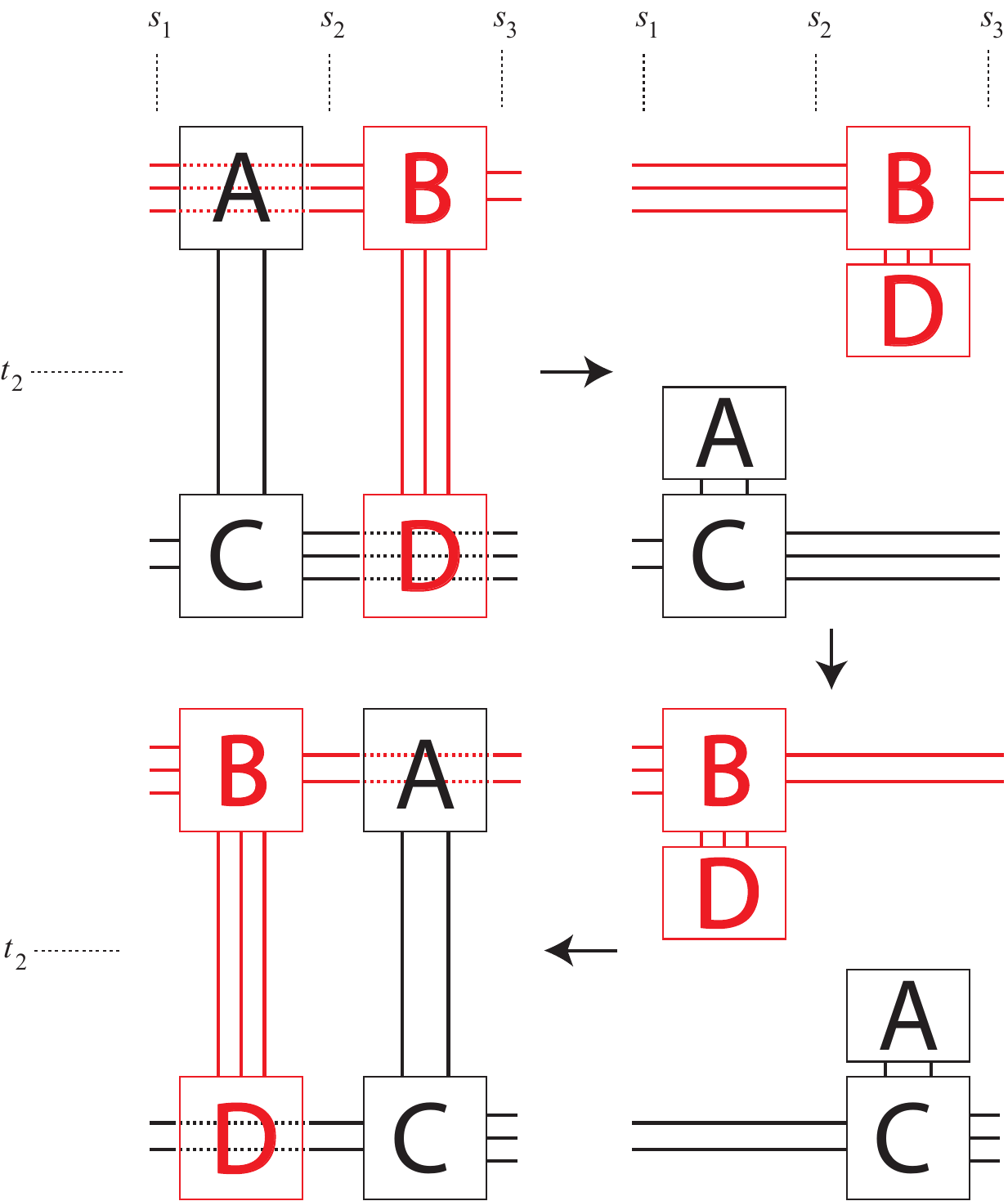}
\caption{A generalised exchange move}
\label{Fig:GenExchange}
\end{figure}

The following was Lemma 2.4 in \cite{Lackenby:PolyUnknot}.

\begin{lemma}
\label{Lem:GenExchange}
Let $n$ be the arc index of an arc presentation of $K$. A generalised exchange move on this arc presentation is a composition of at most $(3/2)n^3$ Reidemeister moves. It is also a composition of at most $n$ cyclic permutations and at most $(3/4)n^2$ exchange moves.
\end{lemma}

\subsection{Nearly admissible form}

We fix $S^3$ as the join $S^1_\phi \ast S^1_\theta$. Let $S$ be a surface embedded in $S^3$.
Then $S - S^1_\phi$ inherits a singular foliation $F$ with tangent space at each point given by the kernel of the
1-form $d\theta$.

We say that $S$ is in \emph{nearly admissible form} if it satisfies the following:
\begin{enumerate}
\item The boundary $\partial S$ is in an arc presentation.
\item The surface $S$ is smoothly embedded, except at $\partial S \cap S^1_\phi$.
\item $S - \partial S$ intersects $S^1_\phi$ transversely at finitely many points.
\item The foliation $F$ has only finitely many singularities, each of which is a point
of tangency with some page.
\item Each singularity is of Morse type, in other words, a local maximum, a local
minimum, an interior saddle or a boundary saddle. (See Figure \ref{Fig:Singularities}). Note that local minima and local maxima
do not occur on $\partial S$, since $\partial S$ is in an arc presentation.
\item The foliation $F$ is radial near each vertex of $S$.
\item Each page contains at most one arc of $\partial S$ or one singularity of ${F}|_{S - \partial S}$
but not both.
\item Each arc of $\partial S$ contains at most one singularity of $F$ in its interior (which is 
a boundary saddle as shown in Figure \ref{Fig:Singularities} (e)).
\end{enumerate}

Near a singular point of ${F}$ or vertex of $S$, the foliation has one of the following local pictures shown in Figure \ref{Fig:Singularities}.

\begin{figure}[h]
\includegraphics[width=4.5in]{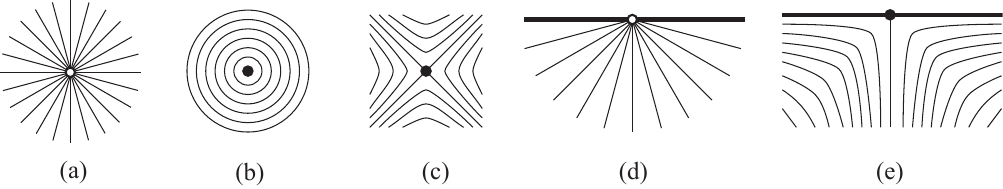}
\caption{Singularities of the foliation}
\label{Fig:Singularities}
\end{figure}

In (a), the interior of $S$ intersects the binding circle transversely. This is an \emph{interior vertex}. In (b), a local
maximum or minimum is shown, which is known as a \emph{pole}. In (c), there is an \emph{interior saddle}. A \emph{boundary vertex} is shown in (d), which is
where the boundary of $S$ intersects the binding circle. A \emph{boundary saddle} is given in (e). When a singular point
is an interior saddle or a boundary saddle, we call it a \emph{saddle}.

An elementary general position argument gives that any embedded surface, with boundary that is in an 
arc presentation, may be isotoped into nearly admissible form, keeping its boundary fixed.

There are various arrangements for the surface $S$ that are closely related to nearly admissible form.
These were introduced by Dynnikov \cite{Dynnikov} and the author \cite{Lackenby:PolyUnknot}. We briefly recall these below.

\subsection{Admissible form and winding vertices}

Let $v$ be a boundary vertex of $S$. Pick a properly embedded arc $\alpha$ in $S$ that lies within small regular neighbourhood
of $v$ and with endpoints on each side of $v$. Then the \emph{winding angle} of $v$ is $|\int_\alpha d \theta |$. We say
that $v$ is a \emph{winding vertex} if its winding angle is more than $2 \pi$.

\emph{Admissible form} was defined by Dynnikov \cite{Dynnikov}. The surface is required to nearly admissible,
but also to satisfy some extra conditions which constrain the behaviour near $\partial S$. 
Specifically, there can be at most one winding vertex. There can also be at most one
boundary saddle. Moreover, if there is both a winding vertex and a boundary saddle, then these are required
to lie in the same arc of $\partial S$. 

Dynnikov shows in \cite{Dynnikov} that if $S$ is an embedded surface and $\partial S$ is connected and
in an arc presentation, then there is an ambient isotopy keeping $\partial S$ fixed taking $S$ into admissible form.

\subsection{Alternative admissible form}

A surface $S$ is in \emph{alternative admissible form} if it is nearly admissible and it contains no winding
vertices.

This was defined by the author in \cite{Lackenby:PolyUnknot}. It has some advantages over admissible form
and some disadvantages. The main advantage is the absence of winding vertices,
which can be problematic. The main disadvantage is that it need not be the case that
a surface can be isotoped into alternative admissible form keeping its boundary fixed.
This is because the constraints that there are no winding vertices and that there is at most one boundary
saddle per arc of $\partial S$ impose restrictions on the framing of $\partial S$.

\subsection{Generalised admissible form}

It will be convenient to consider a further variation on the above theme, where we permit the surface to have
singularities that are more general than Morse singularities.

We consider $\mathbb{R}^3$ with height function $z$. We say that a surface embedded 
in $\mathbb{R}^3$ has a \emph{generalised interior saddle} at the origin
if when using cylindrical polar coordinates near the origin, it has the following form:
$$\{ (r, \alpha, z): z = r^2 \sin (k\alpha) \}$$
for some integer $k \geq 2$. The $k = 2$ is the case of a Morse saddle.

We also consider surfaces embedded in the half space $\{ 0 \leq \alpha \leq \pi \}$ (which in
Cartesian coordinates is just $\{ y \geq 0 \}$). We say that the surface has a \emph{generalised boundary saddle}
at the origin if it has the form
$$\{ (r, \alpha, z): z = r^2 \sin (k\alpha) \} \cap \{ 0 \leq \alpha \leq \pi \}$$
for some integer $k \geq 2$.
Thus, at a generalised boundary saddle, the boundary of the surface is level
with respect to the height function $z$. A \emph{generalised saddle} is a generalised
interior saddle or a generalised boundary saddle.

We say that a surface is in \emph{generalised admissible form} if it satisfies (1)-(7) in the definition
of nearly admissible form, except that generalised interior saddles and generalised boundary saddles
are permitted, and it has no winding vertices.

The singular foliation near a generalised interior saddle and a generalised boundary saddle is shown in Figure \ref{Fig:GenSaddle}.

\begin{figure}[h]
\includegraphics[width=2.5in]{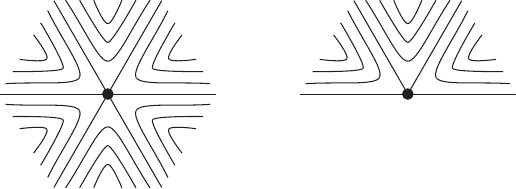}
\caption{Generalised interior saddle and generalised boundary saddle (in the case $k = 3$)}
\label{Fig:GenSaddle}
\end{figure}

\subsection{Nearly embedded surfaces}

In this paper, it will be necessary to consider surfaces that are not quite embedded.
We say that a surface $S$ in a 3-manifold $M$ is \emph{nearly embedded} if
$S \cap \partial M = \partial S$ and $S - \partial S$ is embedded.

Thus, the boundary of a nearly embedded surface is required to lie in $\partial M$ but it
may not be embedded there. However, if one were to remove a very thin regular neighbourhood
of $\partial M$, the intersection between $S$ and the boundary of the resulting 3-manifold
would be a collection of embedded curves. Hence, although $\partial S$ is not necessarily
embedded in $\partial M$, it is possible to perform a small homotopy to $\partial S$ in $\partial M$ to
make it embedded.

We will consider nearly embedded surfaces that are in alternative admissible form.
In this situation, we will require that $\partial S$ is a union of fibres, and so
it can be viewed as being in an arc presentation. However, because $\partial S$
is not necessarily embedded, its image will not necessarily be a union of disjoint
simple closed curves.

\subsection{The binding weight of surfaces}

For an embedded surface $S$ in nearly admissible form, its \emph{binding weight} $w_\beta(S)$ is $|S \cap S^1_\phi|$.
However, when $S$ is nearly embedded, we use a different definition. A nearly embedded surface
is given by a map $f \colon S \rightarrow S^3$. The \emph{binding weight} $w_\beta(S)$ of $S$ is $|f^{-1}(S^1_\phi)|$.
Thus, when $\partial S$ runs over a vertex several times, this vertex is counted with multiplicity.

\subsection{Separatrices and tiles}

All the above configurations for a surface share various common features, which we now describe.

If one removes the vertices and generalised saddles, the result
is a foliation of a subsurface of $S$.

A \emph{separatrix} is a component of a leaf of this foliation that is incident to a generalised saddle. Since every 
generalised saddle has a different value of $\theta$, the separatrices do not join distinct generalised saddles.
They therefore run from a generalised saddle to a vertex or from a generalised saddle back to itself.

If one removes the singularities, the separatrices and $\partial S$ from $S$, each component of
the resulting surface is called a \emph{tile}.

Each tile has a foliation, and since it is non-singular and defined by a global function $\theta$,
this foliation induces a product structure on the tile. Therefore, each tile is an open annulus
or open disc. The separatrices in the boundary of a disc tile run from a generalised saddle to a vertex. Hence,
disc tiles come in three types (see Figure \ref{Fig:DiscTiles}):
\begin{enumerate}
\item \emph{square} tiles, which have two vertices in their boundary (which may be identified)
and two generalised saddles;
\item \emph{half} tiles, which have two vertices in their boundary (which may again be identified) and one generalised saddle;
the arc in the boundary of the half tile joining the vertices lies in $\partial S$;
\item \emph{bigon} tiles, which have two vertices in their boundary and no generalised saddles. In this case,
the bigon tile is a component of $S$.
\end{enumerate}

\begin{figure}[h]
\includegraphics[width=3.8in]{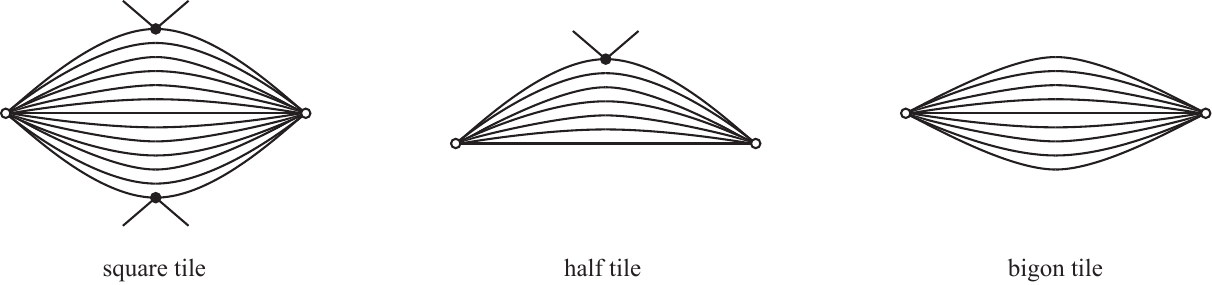}
\caption{Disc tiles}
\label{Fig:DiscTiles}
\end{figure}

It is frequently the case that annular tiles can be removed without increasing
the binding weight of the surface, or can be otherwise avoided. Thus, we will frequently make the hypothesis that
a given nearly admissible or generalised admissible surface contains no annular tiles.

\subsection{Relating the number of saddles and vertices}

\begin{lemma}
\label{Lem:EulerInequality}
Let $S$ be a generalised admissible surface with no annular tiles.
Let $v^i(S)$ denote the number of interior vertices, let $v^b(S)$ be the number of boundary vertices,
let $x^i(S)$ be the number of generalised interior saddles, and let $x^b(S)$ denote the number of generalised boundary saddles.
Then $$2v^i(S) + v^b(S) - 2x^i(S) - x^b(S) \geq 2 \chi(S).$$
\end{lemma}

\begin{proof}
We consider first the case where $S$ is closed. The singular foliation can be used 
to create a cell structure on $S$, as follows. The 0-cells are the vertices. Within each tile,
we pick a non-singular leaf, and declare it to be a 1-cell. The complement of this 1-skeleton
is a collection of discs, each of which forms an open regular neighbourhood of a generalised saddle.
Here, we are using the assumption that there are no annular tiles. Let $V$, $E$ and $F$
be the number of 0-cells, 1-cells and 2-cells of this cell structure. Then $4F \leq 2E$,
because each 2-cell is at least 4-valent. Hence,
$$\chi(S) = V - E + F \leq V - F = v^i(S) - x^i(S),$$
proving the lemma in this case.

Note that when a component of $S$ contains a bigon tile, then that tile
is all of the component. The required formula is readily verified in this situation.
So we may assume that $S$ contains no bigon tiles.

When $S$ has non-empty boundary, we form its double $\hat S$. The singular foliation
on $S$ induces one on $\hat S$. Since $\partial S$ is in an arc presentation and $S$ has no bigon tiles, no annular
tiles on $\hat S$ are formed. Thus, we may create a cell structure on $\hat S$, exactly
as in the closed case. Let $V$, $E$ and $F$ be its number of 0-cells, 1-cells and 2-cells.
Again each 2-cell is at least 4-valent, because it arises either from an interior generalised saddle of $S$,
or from two generalised boundary saddles fused together. Hence,
$$2 \chi(S) = \chi(\hat S) = V - E + F \leq V - F = 2v^i(S) + v^b(S) - 2x^i(S) - x^b(S).$$
\end{proof}

\section{Admissible partial hierarchies}
\label{Sec:AdmissibleHierarchies}

\subsection{Definition}

We are given a link $K$ in an arc presentation. This arc presentation corresponds to a rectangular diagram for $K$.

We give the exterior of $K$ a boundary pattern
as described in Theorem \ref{Thm:WeaklyFundamentalHierarchy}. On each component of $\partial N(K)$,
this pattern is either empty or single curve with slope $0$ or $\pm 1$.
We stabilise the arc presentation enough times so that for each component of $\partial N(K)$
with non-empty boundary pattern, the slope of this boundary pattern is equal to the writhe of the
relevant component of $K$ in its rectangular diagram.

Let $S_1, \dots, S_n$ be a partial hierarchy for the exterior $K$. 
We set $S_0$ to be $\partial N(K)$. Thus, although $S_0$ is not properly
embedded in the ambient manifold $S^3 \cut N(K)$, we can
still think of it as, in some sense, the first surface of the hierarchy.
Because of the above stabilisations, the initial boundary pattern can be realised
as a union of fibres in the singular foliation on $\partial N(K)$.

We say that $S_0, \dots, S_n$ is in \emph{admissible form} if each $S_i$
is an embedded nearly admissible surface with no annular tiles and no winding vertices,
and $\partial S_i$ avoids all the saddles in the earlier surfaces.
Note that we require each $S_i$ to be embedded, rather than just nearly embedded.
The \emph{binding weight} of this partial hierarchy is $\sum_{i=0}^n w_\beta(S_i)$.

Let $M = M_{n+1}$ be the final manifold in the partial hierarchy arising from $S_1, \dots, S_n$.
It inherits a boundary pattern $P = P_{n+1}$ which is in an arc presentation.
Each component of $\partial M_{n+1} \cut P_{n+1}$
lies in some $S_i$ and so is in nearly admissible form.

\subsection{A polyhedral decomposition for an admissible partial hierarchy}
\label{Subsec:TriangulatingExterior}

Let $H$ be the partial hierarchy in admissible form for the exterior of the link $K$. 
We now explain how to construct a polyhedral decomposition $\Delta$
of the exterior of $K$ that contains $H$ as a subcomplex. Note that the boundary of $M$ is comprised of parts of $H$ and so inherits 
the singular foliation of $H$.

The polyhedral decomposition will be an expression of the exterior of $K$ as a union of polyhedra with some faces
identified in pairs. The polyhedra that we will use will be \emph{expanded tetrahedra}, \emph{expanded turnovers},
and \emph{rugby balls}, which are defined as follows.

A \emph{bigon} is a cell complex, consisting of two 0-cells, two 1-cells, and a 2-cell.
The 1-cells are each attached to both 0-cells, so that the 1-skeleton is a circle.
The boundary of the 2-cell is attached homeomorphically onto the 1-skeleton.

An \emph{expanded tetrahedron} is a 3-ball that is obtained from a solid tetrahedron by replacing
some of its edges by bigons. 

A \emph{turnover} is a cell complex, consisting of three
0-cells, three 1-cells, two 2-cells and a 3-cell. The 1-cells are attached to the 0-cells
so that the 1-skeleton is a circle. The boundary of each 2-cell is attached homeomorphically
onto this circle. The 2-skeleton is then a 2-sphere, to which the boundary of the 3-cell
is attached homeomorphically. An \emph{expanded turnover} is a 3-ball that is obtained from a turnover
by replacing some 1-cells by bigons.

A \emph{rugby ball of order $n$} is a cell complex consisting of two 0-cells, $n$ 1-cells,
$n$ 2-cells and a 3-cell. Each 1-cell joins the two 0-cells. Each 2-cell is a bigon. The 2-skeleton
forms a 2-sphere to which the boundary of the 3-cell is attached homeomorphically.

\begin{figure}[h]
\includegraphics[width=3.8in]{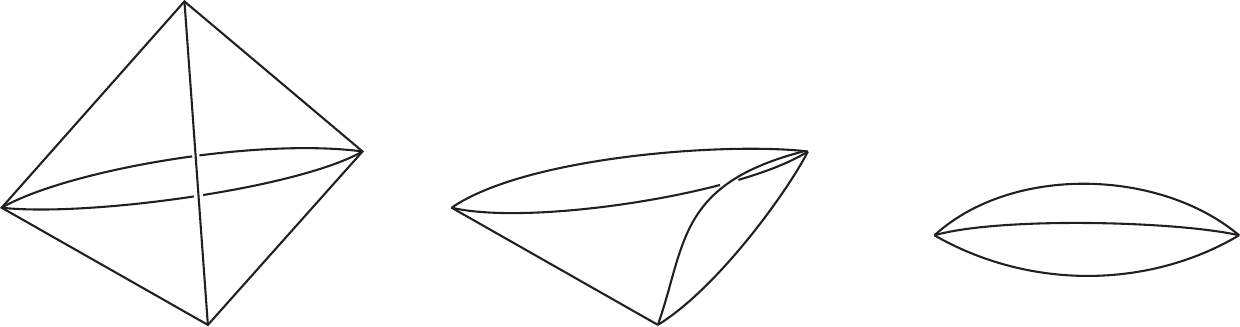}
\caption{The polyhedra used in the polyhedral decomposition. Left: expanded tetrahedron. Middle: expanded turnover. Right: rubgy ball of order 3.}
\label{Fig:PolyhedralTypes}
\end{figure}

The binding circle intersects $H \cup N(K)$ in a finite collection of points and arc. Cutting the binding circle
at these points and arcs gives a union of closed arcs.
We declare that each of these is a 1-cell of $\Delta$.

The complement of the binding circle is foliated by pages $\mathcal{D}_t$, where
$t \in \mathbb{R}/(2 \pi \mathbb{Z})$. We say that a page $\mathcal{D}_t$ contains an \emph{event}
if it has non-empty intersection with the boundary of some surface of the partial hierarchy,
or some Morse singularity of some surface in the partial hierarchy. There are only finitely many
events, and in each component of the complement of these events, the
intersections $\mathcal{D}_t \cap H$ are isotopic as $t$ varies. Between
the event pages, we will also pick a finite collection of pages and these, together
with the event pages, will be called \emph{specified pages}.

Consider an event page that contains an interior saddle of a surface in $H$. Then just below this point,
two arcs of the surface come together, forming the interior saddle, and then just above the interior saddle,
they separate into two different arcs. We pick a page just below and a page just above the interior saddle,
and we declare that these are specified pages. Each of these pages intersects the exterior of $H$ in
a collection of discs. We triangulate each of these discs by coning from a point on the boundary,
unless a disc is a bigon, in which case we leave it untouched. We ensure that the
arcs involved in the interior saddle move are edges of the polyhedral decomposition. These pages are 
now decomposed into triangles and isolated bigons.
Despite the possible existence of bigons, we call this a `triangulation'.

Now suppose that the event page contains an arc $\alpha$ in the boundary of one of the surfaces $S_j$
in $H$. This runs over a previous surface $S_i$ in the hierarchy. By assumption, this arc in $S_i$
misses its interior saddles and boundary saddles. Hence, near $\alpha$, one part of $S_i \cut \alpha$ lies
just above $\alpha$ and the other part lies just below $\alpha$. There are two possible arrangements
for $S_j$: either $\alpha$ contains a boundary saddle of $S_j$ or it does not. 

Let us first consider the case where $\alpha$ does not contain a boundary saddle of $S_j$. Then
near $\alpha$, $S_j$ lies entirely above $\alpha$ or entirely below it.
We have three sheets of $H$ emanating from this arc. Say that one sheet
lies below the arc and two sheets lie above. We declare that a page just above
and a page just below the event page are specified pages. Between the higher page 
and the event page, we insert a rugby ball of order $3$.
The event page intersects the exterior of $H$ in discs. Pick a triangulation of
these discs by coning from points on the boundary. This specifies triangulations in the two nearby specified pages.
Between these triangulations, we insert expanded turnovers and expanded bigons.
We will consider later the case where $\alpha$ contains a boundary saddle.

When two event pages are adjacent, the discs in the specified page just below the higher event page
and in the specified page just above the lower event page have been triangulated.
These triangulations need not agree. We can change one triangulation into the other
using 2-2 Pachner moves. Pick such a sequence of 2-2 Pachner moves that is
as short as possible. We insert new specified pages, the number being
one less than the number of Pachner moves. We triangulate these as indicated
by the Pachner moves. Each of these moves can be realised by the addition of 
an expanded tetrahedron. Two of the triangular faces of the expanded tetrahedron
lie in a specified page, as do the other two triangular faces. Between these,
the expanded tetrahedron has four bigons. For each of the triangles 
not involved in the 2-2 Pachner move, we insert an expanded turnover.

We now need to explain how to triangulate the space between the pages
just below an interior saddle and just above. The triangulations on these pages
mostly agree. Where they do agree, we insert expanded turnovers. The remainder
lies near the separatrices emanating from the interior saddle. We declare that the interior saddle
point is a vertex of $\Delta$, and that the four separatrices emanating from it
are edges of $\Delta$. The interior saddle is contained in some surface of the hierarchy $H$,
and near this interior saddle, there is a square in the surface bounded by edges of $\Delta$
lying in the adjacent pages. This square is divided into 4 triangles by the interior saddle
and its incident separatrices. We declare that each of these triangles is a face of $\Delta$.
The 4 separatrices lie in a page, and this has also been decomposed by other edges of 
$\Delta$. In the complement of these edges, there are 4 triangles that are incident to
the separatrices. We also declare that these 4 triangles are faces of $\Delta$. 
We attach polyhedra of $\Delta$ onto these 4 faces. Four of these polyhedra are expanded
tetrahedra, and four are expanded turnovers. See Figure \ref{Fig:PolyhedraNearSaddle}. 

\begin{figure}[h]
\includegraphics[width=3.5in]{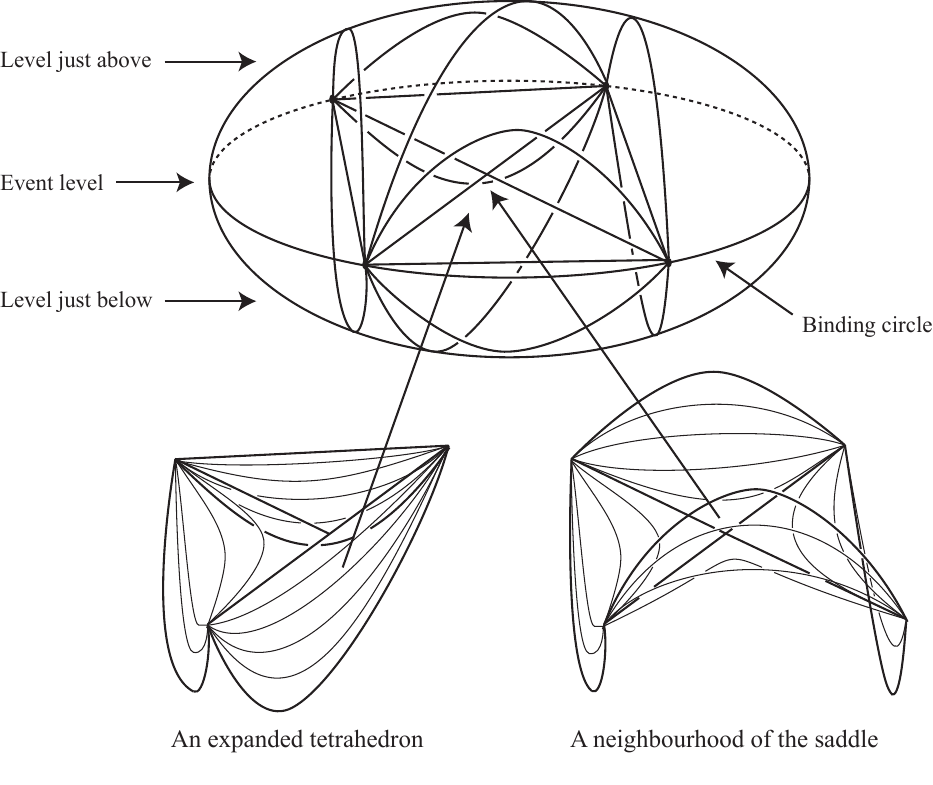}
\caption{An expanded tetrahedron lying near an interior saddle}
\label{Fig:PolyhedraNearSaddle}
\end{figure}

When there is a boundary saddle in some $S_j$, $\Delta$ is defined similarly. The boundary saddle is a vertex
of $\Delta$. The three separatrices in $S_j$ emanating from it are edges of $\Delta$. 
On each side of $S_j$, we insert an expanded tetrahedron and an expanded turnover.

Note that by construction, the boundary pattern $P$ for the exterior of $H$ is a subcomplex of the 1-skeleton of the
polyhedral decomposition.
This is because the boundary pattern is a union of fibres, each of which is part of the
boundary of a surface in the hierarchy. This therefore lies in an event page and becomes
an edge of $\Delta$.

The hierarchy $H$ is a subcomplex of $\Delta$, and hence the exterior of $H$ inherits a polyhedral
decomposition $\Delta_M$.

\subsection{A triangulation for an admissible partial hierarchy}
\label{Subsec:TriExterior}

In the previous subsection, we defined a polyhedral decomposition $\Delta$ for the exterior of
the link, with the admissible partial hierarchy $H$ as a subcomplex. In this section, we show that this can be collapsed to
a triangulation $\mathcal{T}$. We collapse each rugby ball to a single edge, in such a way that
each edge of the rubgy ball maps homeomorphically onto the new edge. We collapse each
expanded turnover to a single triangular face. Again, each edge of the expanded turnover is mapped
homeomorphically to an edge of the triangle. We collapse each expanded tetrahedron
to a tetrahedron. This gives a triangulation $\mathcal{T}$ for the link exterior.

Note that, in principle, this sequence of collapses might create a topological space that is not a 3-manifold.
For example, suppose that a sequence of expanded turnovers were glued together in a circular fashion,
so that they are indexed by the integers modulo $n$, and so that the expander turnover with index $i$
is glued the one with index $i-1$ and the one with index $i+1$ along their triangular faces. Then
when these expanded turnovers are collapsed, they are collapsed to a single triangle, and hence points
in this triangle are not 3-manifold points. However, it is easy to check that in our situation,
we do obtain a triangulation for the link exterior.

By performing the same collapses to $\Delta_M$, we obtain a triangulation $\calT_M$ for $M$.

\subsection{Bounding the number of tetrahedra}

\begin{lemma}
\label{Lem:NumberEventLevels}
Let $H = \{ S_0, S_1, \dots, S_n \}$ be an admissible partial hierarchy for the exterior of $K$.
Let $w$ be its binding weight. Then the total number of event pages is at most
$2 ( w -  \sum_{i=1}^n \chi(S_i))$. 
\end{lemma}

\begin{proof}
An event page occurs when some surface $S_i$ has an interior saddle or when its
boundary has an arc in a page. The number of interior saddles of the surface $S_i$ is $x^i(S_i)$,
which is at most $2x^i(S_i) + x^b(S_i)$. By Lemma \ref{Lem:EulerInequality}, this is at most
$2v^i(S_i) + v^b(S_i) - 2 \chi(S_i)$. So the sum of the number of interior saddles
and boundary vertices is at most $2 (v^i(S_i) + v^b(S_i) - \chi(S_i))$. Summing this
over $i$ gives the required upper bound on the number of event pages.
\end{proof}

\begin{lemma}
\label{Lem:NumberSpecifiedLevels}
The total number of specified pages is at most
$4w(w -  \sum_{i=1}^n \chi(S_i))$.
\end{lemma}

\begin{proof}
Consider a specified page just below or just above an event page. This intersects the exterior of $H$
in discs. Some of these are bigons which are not triangulated, and some of these are triangles. The remainder
are polygons with at least four edges in their boundary. Each edge of each polygon is either an arc in the binding
circle or an arc in the interior of the page. There are at most $w$ of the former arcs. There are at most
$w/2$ arcs in the interior of the page, and each contributes at most two edges to the boundaries
of the polygons. So the polygons have a total of at most $2w$ sides.
Each of these discs has been triangulated by coning from a vertex in its boundary.
Now, the number of 2-2 Pachner moves relating two such triangulations
of a polygon with $n$ sides is at most $n-2$. Hence, the number of 2-2 Pachner moves
needed to relate the two triangulated specified pages is at most $2w-2$. So, the number
of specified pages strictly between two adjacent event pages is at most $2w-1$.
Adding the number of event pages gives the required bound.
\end{proof}

\begin{lemma}
\label{Lem:NumberTetrahedra}
The number of tetrahedra in the triangulation $\mathcal{T}$ of the link exterior
is at most $16w (w - \sum_{i=1}^n \chi(S_i))$. This is also an upper bound for the number
of tetrahedra in $\calT_M$.
\end{lemma}

\begin{proof}
The tetrahedra in the triangulations $\mathcal{T}$ and $\calT_M$ come from expanded tetrahedra of
the polyhedral decompositions $\Delta$ and $\Delta_M$.
We observe that between two specified pages, neither of which are event pages,
there is at most one expanded tetrahedron. Between any event page and an adjacent
specified page, there are at most four expanded tetrahedra. The result now follows from
Lemma \ref{Lem:NumberSpecifiedLevels}.
\end{proof}

\begin{lemma}
\label{Lem:NumberExpandedTurnover}
The number of expanded turnovers in the polyhedral decomposition $\Delta$ is at most $8w^2 (w - \sum_{i=1}^n \chi(S_i))$.
\end{lemma}

\begin{proof}
Between adjacent specified pages, there are at most $2w$ expanded turnovers. Now apply
Lemma \ref{Lem:NumberSpecifiedLevels}.
\end{proof}

\begin{lemma} 
\label{Lem:NumberEdgePreimages}
Each edge of $\calT$ has pre-image in $\Delta$ consisting of at most $8w (w - \sum_{i=1}^n \chi(S_i))$
edges. Similarly, each edge in $\calT_M$ has pre-image in $\Delta_M$ consisting of at most $8w (w - \sum_{i=1}^n \chi(S_i))$
edges.
\end{lemma}

\begin{proof}
The edges of $\calT$ are obtained by collapsing bigons of $\Delta$. Within each page, each edge is incident to at most
one bigon. In adjacent specified pages,
edges of $\Delta$ may be joined by bigons. So the number of edges of $\Delta$ collapsed to a single edge of $\calT$
is at most twice the number of specified pages. Now apply Lemma \ref{Lem:NumberSpecifiedLevels}.
The same argument also applies to the edges in $\calT_M$.
\end{proof}

\begin{lemma}
\label{Lem:EdgesIncidentToVertex}
For each point of intersection between $\partial M$ and the binding circle, the number of edges of $\Delta_M$ that are incident to that vertex is at most $4w^2(w - \sum_{i=1}^n \chi(S_i))$.
This is also an upper bound for the number of edges of $\calT_M$ that are incident to a vertex in $\partial M$.
\end{lemma}
 
\begin{proof}
We see from the construction of $\Delta$ in Section \ref{Subsec:TriangulatingExterior} that each edge lies in a specified page. It runs from a vertex of $H$ either to another vertex or to a saddle. Therefore, in any page, the number of edges running from a given vertex is at most $w$, the binding weight of the partial hierarchy. So, the total number of edges of $\Delta_M$ emanating from a given vertex of $H$ is at most $w$ times the number of specified pages, hence at most $4w^2(w - \sum_{i=1}^n \chi(S_i))$ by Lemma \ref{Lem:NumberSpecifiedLevels}. When $\Delta_M$ is collapsed to form $\calT_M$, several edges may be combined to a single edge, but the number of edges emanating from a given vertex in $\partial M$ does not go up.
\end{proof}

\subsection{A handle structure for the exterior of the partial hierarchy}
\label{Subsec:HierarchyExteriorHS}

Let $M$ be the exterior of the partial hierarchy $H$.
We can attach a partial collar onto $\partial M$, as follows. For each edge of $\Delta_M$ in $\partial M$,
attach a bigon along one of the bigon's edges. For each triangle of $\Delta_M$ in $\partial M$,
attach on an expanded turnover with three bigons in its boundary. Identify each of 
these bigons with the bigon incident to the relevant edge in $\partial M$.

Let $\hat \Delta_M$ be this new polyhedral decomposition of $M$, and let $\mathcal{K}$
be the dual handle structure. Recall that this has an $i$-handle for each $(3-i)$-cell of $\hat \Delta_M$ that
does not lie wholly in $\partial M$.

For example, dual to the expanded tetrahedron shown in Figure \ref{Fig:PolyhedraNearSaddle}, there is 
0-handle of $\mathcal{K}$. Its intersection with the 1-handles and 2-handles is shown in Figure \ref{Fig:IsotopeBoundary}.

Suppose that $S$ is a surface properly embedded in $M$ that is in general position with respect
to $\mathcal{T}_M$. The inverse image of $S$ is a surface in $\Delta_M$, which we will call $S'$.
This also gives a surface $\hat S$ that is standard in the handle structure $\mathcal{K}$.

\begin{lemma}
\label{Lem:ExtendedWeightEstimate}
Let $S$ be a surface properly embedded in $M$ that is in general position with respect to $\calT_M$, 
and let $\hat S$ be the resulting surface in $\mathcal{K}$.
Then the extended weight of $\hat S$ is at most $40w(w - \sum_i \chi(S_i))w(S)$, where $w$ is
the binding weight of $H$ and $w(S)$ is the weight of $S$ in $\calT_M$.
\end{lemma}

\begin{proof}
The polyhedral decomposition $\Delta_M$ collapses to $\calT_M$. By Lemma \ref{Lem:NumberEdgePreimages},
each edge of $\calT_M$ has pre-image in $\Delta_M$ consisting of at most $8w (w - \sum_{i=1}^n \chi(S_i))$
edges. So, the inverse image of $S$ in $\Delta_M$ has weight at most $8w (w - \sum_{i=1}^n \chi(S_i)) w(S)$.
There is a 1-1 correspondence between the edges of $\Delta_M$ not lying in $\partial M$ and
the 2-handles of $\mathcal{K}$. Hence, $8w (w - \sum_{i=1}^n \chi(S_i)) w(S)$ is an upper bound for
the weight of $\hat S$ in $\mathcal{K}$. 

Each arc of intersection between $S$ and a face of $\calT_M$ has both its endpoints on edges of $\calT_M$.
The same is true of the inverse image of $S$ in $\Delta_M$. Hence, each component of intersection between 
$\hat S$ and a 1-handle of $\mathcal{K}$ is incident to
a 2-handle of $\mathcal{K}$. Each 2-handle of $\mathcal{K}$ runs over at most four 1-handles of 
$\mathcal{K}$. Hence, the number of intersections between $\hat S$ and the 1-handles of $\mathcal{K}$
is at most 4 times the weight of $\hat S$. Therefore, the extended weight of $\hat S$ in $\mathcal{K}$ is at most
$40w (w - \sum_{i=1}^n \chi(S_i)) w(S)$, as required.
\end{proof}

\subsection{Making a surface transverse to $P$}
\label{Subsec:TransverseToPattern}

We are considering a partial hierarchy $H$ in admissible form. We also suppose that 
it is \emph{well-spaced}, which means that it satisfies all the following conditions:
\begin{enumerate}
\item the closure of the union of tiles in $H$ adjacent to $P$ is a regular neighbourhood $N(P)$ of $P$;
\item for each vertex $s$ of $H$ that lies on $P$ and for each annular component $A'$ of $N(P) \cut P$ containing $s$, there is a tile in $A'$ incident to $s$ such that the other vertex of the tile lies on $\partial A' - P$;
\item the above vertices include all the vertices lying in $A'$.
\end{enumerate}
Section \ref{Subsec:WellSpaced} describes how to make a hierarchy well-spaced, and Figures \ref{Fig:WellSpacedFoliation} and \ref{Fig:WellSpacedBoundarySaddle} show examples
of the singular foliations in a well-spaced hierarchy.

We have defined a polyhedral decomposition $\Delta$ and a triangulation $\mathcal{T}$ for the link exterior.
These restrict to a polyhedral decomposition $\Delta_M$ and a triangulation $\mathcal{T}_M$
for the exterior $M$ of $H$. Suppose that $S$ is a surface properly embedded in $M$ that is normal with respect
to $\mathcal{T}_M$. In this subsection, we will explain how to pattern-isotope $\partial S$ so that afterwards it intersects the annuli $N(P) \cut P$ in a collection of essential arcs,
each of which intersects $P$ transversely at a single point away from the vertices of $P$.

Let $A'$ be a component of $N(P) \cut P$, which is an annulus. By assumption, for each vertex $s$ of $H$ that lies on $A' \cap P$, there is a tile in $A'$ incident to $s$ such that the other vertex of the tile lies on $\partial A' -  P$. This tile is a union of faces of $\calT$. Each such face has corners that are vertices of $H$ or a saddle, and there is at most one corner of the latter type. Hence, there is some edge of $\calT$ running across the tile from $s$ to the other vertex. Pick one such edge for each vertex of $H$ on $P \cap A'$. We may choose a product structure $S^1 \times [0,1]$ on $A'$ so that each of these edges is of the form $\{ \ast \} \times [0,1]$ for some point $\{ \ast \}$ on $S^1$. Pick a very thin regular neighbourhood $A''$ of $A' \cap P$ in $A'$ that also intersects these edges vertically. There is an isotopy of $\partial M$ that expands $A''$ to all of $A'$, and that respects the product structures on $A''$ and $A'$. Now consider the curves $\partial S$. We may assume that $\partial S$ intersects $A''$ in a collection of vertical arcs. Hence, when we expand $A''$ to fill $A'$, afterwards the intersection $\partial S \cap A'$ is a union of vertical arcs in $A'$, as required. At certain times during the isotopy, the image of $\partial S$ passes through a vertex on $\partial A' - P$. The number of times this happens is equal to the number of intersections between the initial position of $\partial S$ and the edges running from the vertices on $\partial A' - P$ to $P$. Hence, the number of times this happens is at most $w(\partial S)$. We can realise the isotopy of $\partial S$ as an isotopy of $S$ supported in a collar neighbourhood of $\partial M$. At each moment in time that $\partial S$ passes through a vertex on $\partial M$, we obtain new intersection points between $S$ and the edges of $\calT_M$. By Lemma \ref{Lem:EdgesIncidentToVertex}, the total number of these intersection points is at most 
$4w^2(w - \sum_{i=1}^n \chi(S_i)) w(\partial S) $.

\subsection{Homotoping $S$ into alternative admissible form}
\label{Subsec:HomotopeAdmissible}

As in the previous subsection, suppose that the partial hierarchy $H$ is in admissible form and is well-spaced.
Again suppose that suppose that $S$ is a surface properly embedded in $M$ that is normal with respect
to $\mathcal{T}_M$. We also suppose that $S$ is incompressible and boundary-incompressible with no component that is a
sphere or disc intersecting $P$ at most twice.

We pattern-isotoped $S$ so $\partial S$ intersects each component of $N(P) \cut P$ in a collection of essential arcs,
each of which intersects $P$ transversely at a single point away from the vertices of $P$.
This isotopy need not have kept $S$ normal in $\calT_M$.

The boundary of $S$ need not be in an arc presentation, in other words, a concatenation of properly embedded 
arcs in pages. We will homotope $\partial S$ to achieve this. The homotopy runs from time $0$ to time $1$, say.
For all times in the interval $[0,1)$, the curves remain embedded. But at time $1$, this need not be the case.
This homotopy therefore specifies a modification to $S$ near the boundary, taking $S$ to a nearly embedded surface.
In fact, $S$ will embedded away from the binding circle.

Recall that associated with the surface $S$ in $\mathcal{T}_M$, there are surfaces $S'$ and $\hat S$ in
$\Delta_M$ and $\mathcal{K}$. The homotopy that we will perform to $S$ will also induce a homotopy to $S'$ and $\hat S$.
We start by describing the homotopy to $S$.

Since the hierarchy $H$ is in admissible form, $\partial M$ inherits the structure of a union of tiles.
The boundary pattern $P$ is a union of fibres, each running between vertices. 
We now homotope $\partial S$, via a homotopy supported away from $P$, as follows.
In each triangle of $\mathcal{T}_M$ in $\partial M$, we homotope $\partial S$ near the
edges incident to a saddle. We homotope its points of intersection along that edge away
from the saddle until they run through the vertex at the other end of the edge.
(See Figure \ref{Fig:HomotopeBoundary}.)

\begin{figure}[h]
\includegraphics[width=3.5in]{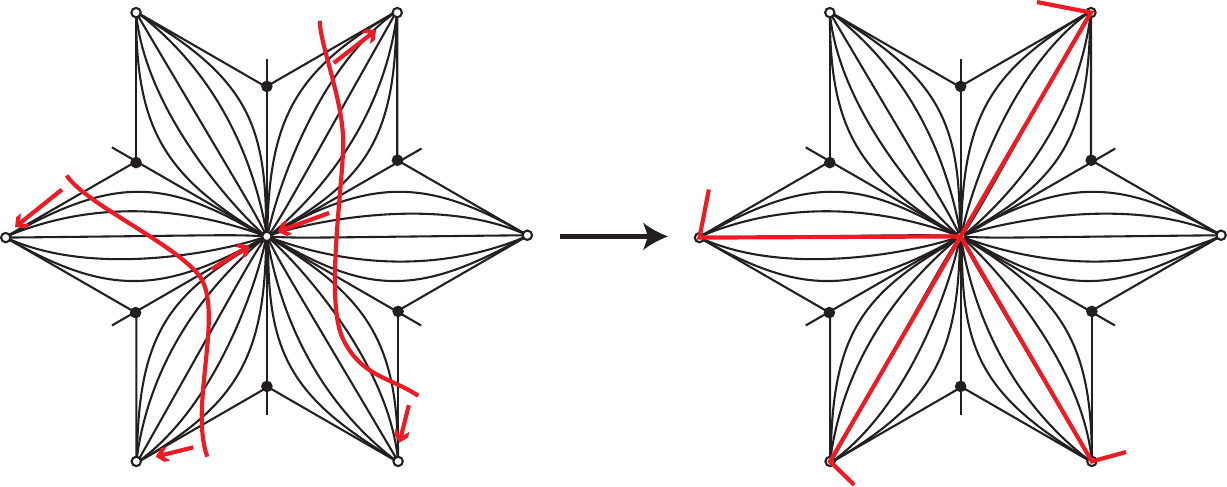}
\caption{Homotoping the boundary curves of a surface}
\label{Fig:HomotopeBoundary}
\end{figure}

The reason that we do not require $\partial S$ to be embedded is as follows. The surface $S$ is constructed
using normal surface theory, and all we have is an exponential upper bound on its weight. 
Hence, $\partial S$ may run over the tiles in $S_0 \cup \dots \cup S_n$ many times.
So it might not be possible to isotope $\partial S$ to become a union of fibres
and preserve the property that it is embedded, without somehow introducing many more vertices
to $S_0 \cup \dots \cup S_n$. This runs counter to our main goal, which
is to control the binding weight of the hierarchy.

This homotopy creates the surfaces $S'$ in $\hat \Delta_M$ and $\hat S$ in $\mathcal{K}$.
The latter can also be viewed as a properly embedded surface in the handle structure $\mathcal{K}$.
Moreover, the homotopy that we have applied to $\hat S$ is actually an isotopy.
This is because in $\mathcal{K}$, we have removed a regular neighbourhood of the vertices on the binding circle,
and away from the binding circle, the surface really is embedded. An example of this isotopy is shown in Figure \ref{Fig:IsotopeBoundary}.
Note that after this isotopy, the surface is disjoint from the 2-handles that are dual to the edges
incident to saddles.

\begin{figure}[h]
\includegraphics[width=4.5in]{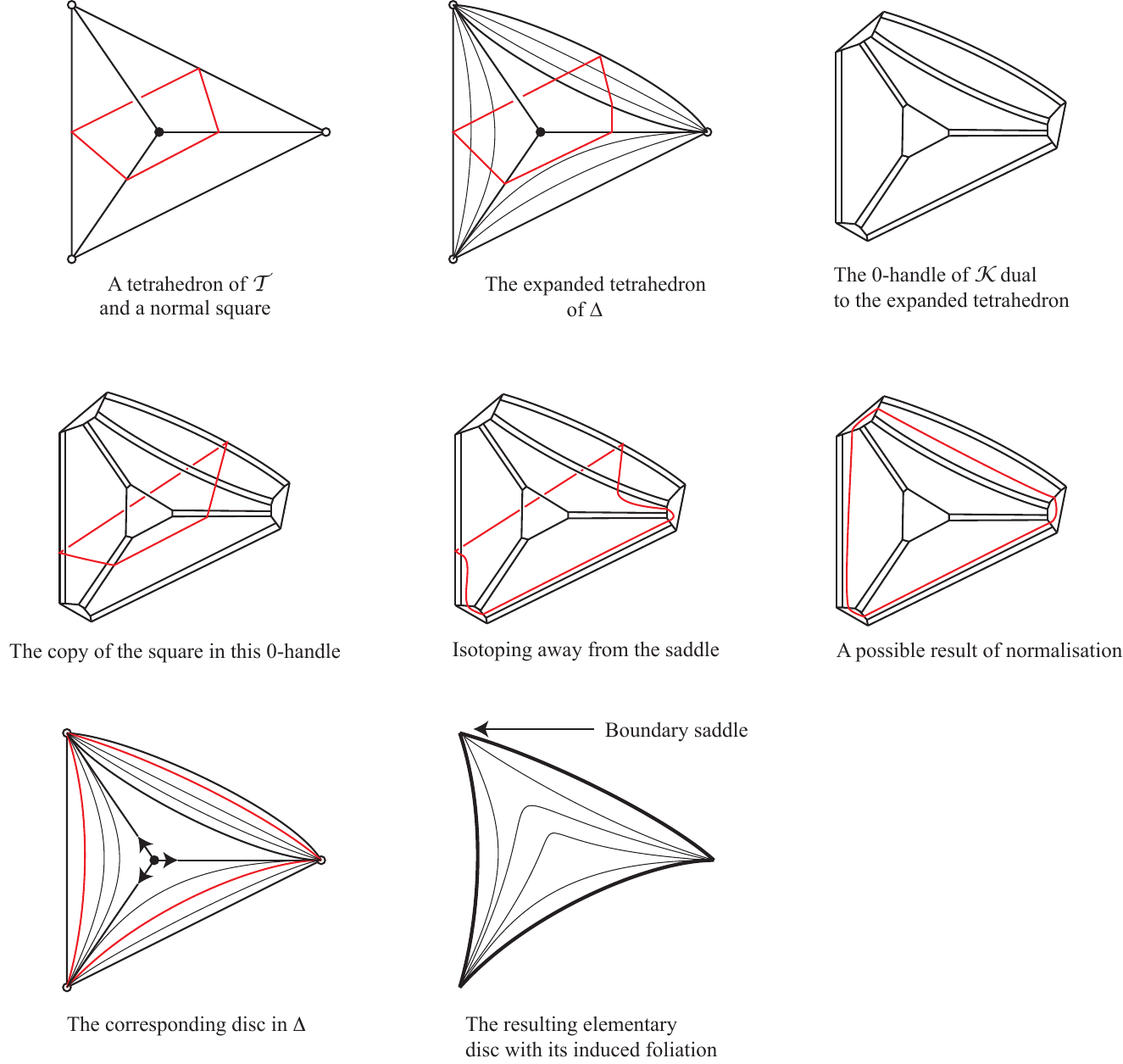}
\caption{Isotoping the boundary curves of a surface}
\label{Fig:IsotopeBoundary}
\end{figure}

This surface $\hat S$ can easily be made standard with respect to $\mathcal{K}$, but it is not necessarily normal. However,
a further pattern-isotopy makes $\hat S$ normal with respect to $\mathcal{K}$, by Lemma \ref{Lem:NormaliseHS}.
This does not increase the extended weight of the surface. 

Note also that this isotopy does not increase the number of components of intersection with any
2-handle. Hence, the resulting surface is also disjoint from the 2-handles that are dual to all the edges
incident to saddles. This imposes the following restrictions on how $\partial \hat S$ runs over $\partial M$.
Consider any face $F$ of a polyhedron of $\Delta_M$ that lies in $\partial M$. Dual to this polyhedron
is a 0-handle of $\mathcal{K}$ and dual to the face is a 1-handle incident to the 0-handle, since we
enlarged $\Delta_M$ to $\hat \Delta_M$.
In the case where $F$ is triangular, it has three vertices, one of which is a saddle and the other
two being on the binding circle. The 1-handle of $\mathcal{K}$ dual to $F$ has three 2-handles incident
to it. Two of these are dual to edges of $\Delta_M$ incident to the saddle, and so $\hat S$ misses these.
So, up to handle-preserving
isotopy, there is only one possible configuration for each component of intersection between $S$
and the 1-handle. This corresponds to an arc in $F$ joining the two vertices on the binding circle.
In the case where $F$ is a bigon, both of its vertices are on the binding circle. Again, there
is only one possible configuration for each component of intersection between the dual 1-handle
and $\hat S$.

We now summarise the result of the procedure so far.

\begin{lemma}
Let $S$ be a compact orientable incompressible boundary-incompressible surface properly embedded in $M$ with no component that is a
sphere or disc intersecting $P$ at most twice. Suppose that $S$ is normal with respect to $\calT_M$.
Then there is a pattern-isotopy taking $S$ to a surface $\hat S$ that is normal with respect to $\mathcal{K}$
and with extended weight at most 
$200w^3 (w - \sum_{i=}^n \chi(S_i))^2 w(S)$.
This surface $\hat S$ misses the 2-handles of $\mathcal{K}$ that are dual to the edges 
incident to saddles. 
This also induces a surface (also called $\hat S$) in $\calT_M$ that is nearly embedded, that has boundary in admissible form,
and that has binding weight at most 
$200w^3 (w - \sum_{i=1}^n \chi(S_i))^2 w(S)$.
It misses the saddles in $\partial M$ as well as the interior of any edge that is incident
to a saddle. 
\end{lemma}

\begin{proof}
We start with the surface $S$ that is normal with respect to $\calT_M$. As described in Section \ref{Subsec:TransverseToPattern},
we then pattern-isotope $S$ so afterwards it
intersects each component of $N(P) \cut P$ in a collection of essential arcs. Its weight in $\calT_M$ is then
at most 
$$w(S) + 4w^2(w - \sum_{i=1}^n \chi(S_i)) w(\partial S)  \leq 5w^2(w - \sum_{i=1}^n \chi(S_i)) w(S).$$
By Lemma \ref{Lem:ExtendedWeightEstimate}, the corresponding surface $\hat S$ in $\mathcal{K}$ has extended weight
at most $200w^3 (w - \sum_{i=}^n \chi(S_i))^2 w(S)$. By Lemma \ref{Lem:NormaliseHS}, we can then isotope $\hat S$ to a surface that
is normal with respect to $\mathcal{K}$ without increasing its extended weight and without introducing any
new intersections with 2-handles. The bound on the extended weight in $\mathcal{K}$ then gives a bound
on the binding weight of the surface in $\calT_M$.
\end{proof}

Thus, we have now specified, for each polyhedron of $\Delta_M$, the intersection between its
boundary and $\hat S$. In particular, the boundary of $\hat S$ is in an arc presentation.
We now arrange the intersection between $\hat S$ and the interior
of the polyhedra, using the following lemmas.

We start by examining the elementary discs of $\hat S$ in $\mathcal{K}$. The following straightforward
result classifies elementary normal discs in a rugby ball.

\begin{lemma}
\label{Lem:DiscsRugby}
Let $D$ be an elementary normal disc in a rugby ball with an open regular neighbourhood of its vertices
removed. Then $D$ satisfies one of the following:
\begin{enumerate}
\item The boundary of $D$ is a concatenation of two arcs, each of which lies in a bigon and runs
between the vertices of the rugby ball.
\item The boundary of $D$ runs over each bigon in exactly one arc, which runs between the
two edges of the bigon.
\end{enumerate}
\end{lemma}

\begin{proof}
We refer to the definition of normality in Section \ref{Sec:NormalHS}.
By (1) in the definition, $\partial D$ runs over the interior of each edge of the rugby ball at most once. By (2), it runs over each vertex of the rugby ball at most once. By (4), it cannot run over both a vertex and the interior of an incident edge. By (3), it cannot be comprised of an arc in a face with endpoints in the same vertex. Hence, the only possibilities are as described in the lemma.
\end{proof}

Similarly, we have the following classification of normal discs in an expanded turnover.

\begin{lemma}
\label{Lem:DiscsTurnover}
Let $D$ be an elementary disc in an expanded turnover with an open regular neighbourhood of its vertices
removed. Then $D$ satisfies one of the following:
\begin{enumerate}
\item The boundary of $D$ is a concatenation of two arcs, one in a bigon and one in triangular face,
each one running between distinct vertices.
\item The boundary of $D$ runs over each bigon in at most one arc, and it runs over each triangular
face in exactly one arc.
\item The boundary of $D$ is a concatenation of three arcs, each of which runs
between distinct vertices.
\end{enumerate}
\end{lemma}

Examples of these possible discs are shown in Figure \ref{Fig:ElemDiscsTurnovers}. In (1) and (2) of each
of these lemmas, the disc $D$ can be arranged so that its singular foliation is a product foliation with
no singularities. But in (3) of Lemma \ref{Lem:DiscsTurnover},
it has a singular foliation with one boundary saddle (which may be placed at a vertex of the expanded turnover).

\begin{figure}[h]
\includegraphics[width=3.5in]{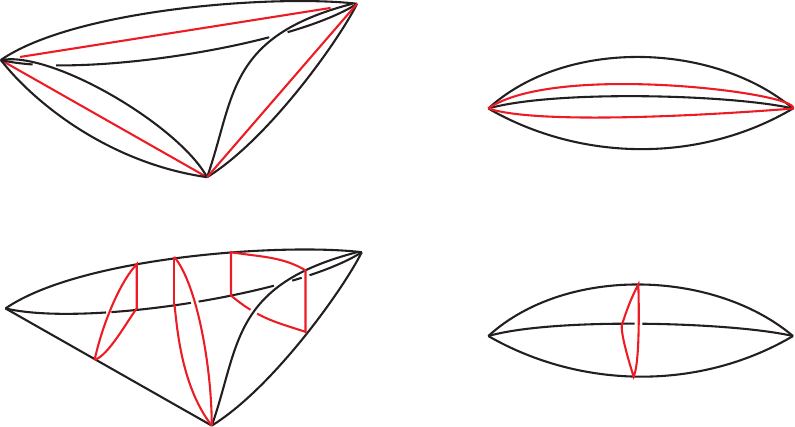}
\caption{Some normal discs in expanded turnovers and in a rugby ball}
\label{Fig:ElemDiscsTurnovers}
\end{figure}

\begin{lemma}
\label{Lem:DiscsSaddleTetrahedron}
Let $B$ be an expanded tetrahedron of $\Delta_M$, with an open neighbourhood of its vertices removed,
that is incident to a saddle of the partial hierarchy $H$. Let $D$ be an elementary normal disc in $B$ that misses
the 2-handles dual to edges incident to the saddle. Then $\partial D$ is a concatenation of two or three
arcs, each of which runs between distinct vertices of the expanded turnover. Furthermore,
$D$ can be arranged so that its singular foliation either is a product or has a single boundary saddle, at a vertex
of the expanded turnover.
\end{lemma}

\begin{proof}
The expanded tetrahedron has four triangular faces. Three of these are incident to the saddle. As shown above,
when $\partial D$ intersects one of these three faces, it does so in an arc joining the two vertices in its boundary
that are not the saddle. Hence, if $D$ intersects the fourth triangular face, it must also do so in arc
running between two distinct vertices of $B$. For otherwise, it runs over the interior of an edge of that triangular face, but
this then forces $\partial D$ to intersect the interior of one of the edges of the other three triangular faces, and we have
ruled this out. So, $\partial D$ is a concatenation of arcs, each of which runs between distinct vertices of $B$.
By normality, it cannot intersect a vertex of $B$ more than once. It also misses the vertex that is the saddle.
Hence, it is a concatenation of two or three arcs. In each case, we can arrange its singular foliation to be of the
specified form.
\end{proof}

See Figure \ref{Fig:IsotopeBoundary} for an example of such a normal disc. A similar case-by-case analysis gives
the following result for an expanded tetrahedron of $\Delta_M$ that is not incident to a saddle of the
partial hierarchy $H$. Recall from Section \ref{Subsec:TriangulatingExterior} that each 
triangular face of the expanded tetrahedron lies in some page, with two triangular faces lying one page, and two lying
in another.

\begin{lemma}
\label{Lem:DiscsExpandedTet}
Let $B$ be an expanded tetrahedron of $\Delta_M$, with an open neighbourhood of its vertices removed,
that is not incident to a saddle of the partial hierarchy $H$. Let $D$ be a normal disc in $B$. Then $D$ may be isotoped,
keeping its boundary fixed, so that it satisfies one 
of the following possibilities:
\begin{enumerate}
\item It inherits a product foliation. Its boundary consists of an arc in the union of the top two triangular faces and an
arc in the union of the bottom two triangular faces, possibly joined by arcs in the bigons.
\item It inherits a product foliation. Its boundary consists of an arc in a triangular face and an arc in bigon face,
both joining distinct vertices.
\item It has a single boundary saddle. Its boundary consists of three, four or five arcs. If it has four arcs, then two of these
lie in the top two triangles or the bottom two triangles and are incident. If it has five arcs, then 
two of these lie in the top two triangles and are incident, and two of the arcs lie in the
bottom two triangles and are incident. The remaining arcs run between distinct
vertices and lie in a bigon or triangular face.
\item Its boundary runs through the four bigons, with each arc joining one vertex of the bigon to the other. It inherits
a singular foliation with two boundary saddles, on distinct vertices.
\item Its boundary consists of six arcs. Two of these
lie in the top two triangles and are incident, and two of the arcs lie in the
bottom two triangles and are incident. Each of the remaining two arcs runs between distinct
vertices and lies in a bigon. This disc has two boundary saddles.
\item Its boundary runs through each of the four triangular faces and across each of the bigons. It is disjoint from
the vertices and comes from a square in the collapsed triangulation. It inherits a singular foliation, with a single
interior saddle and with separatrices running to each of the four bigons.
\end{enumerate}
\end{lemma}

See Figure \ref{Fig:DiscsInExpandedTet} for examples of these discs.

\begin{figure}[h]
\includegraphics[width=3.5in]{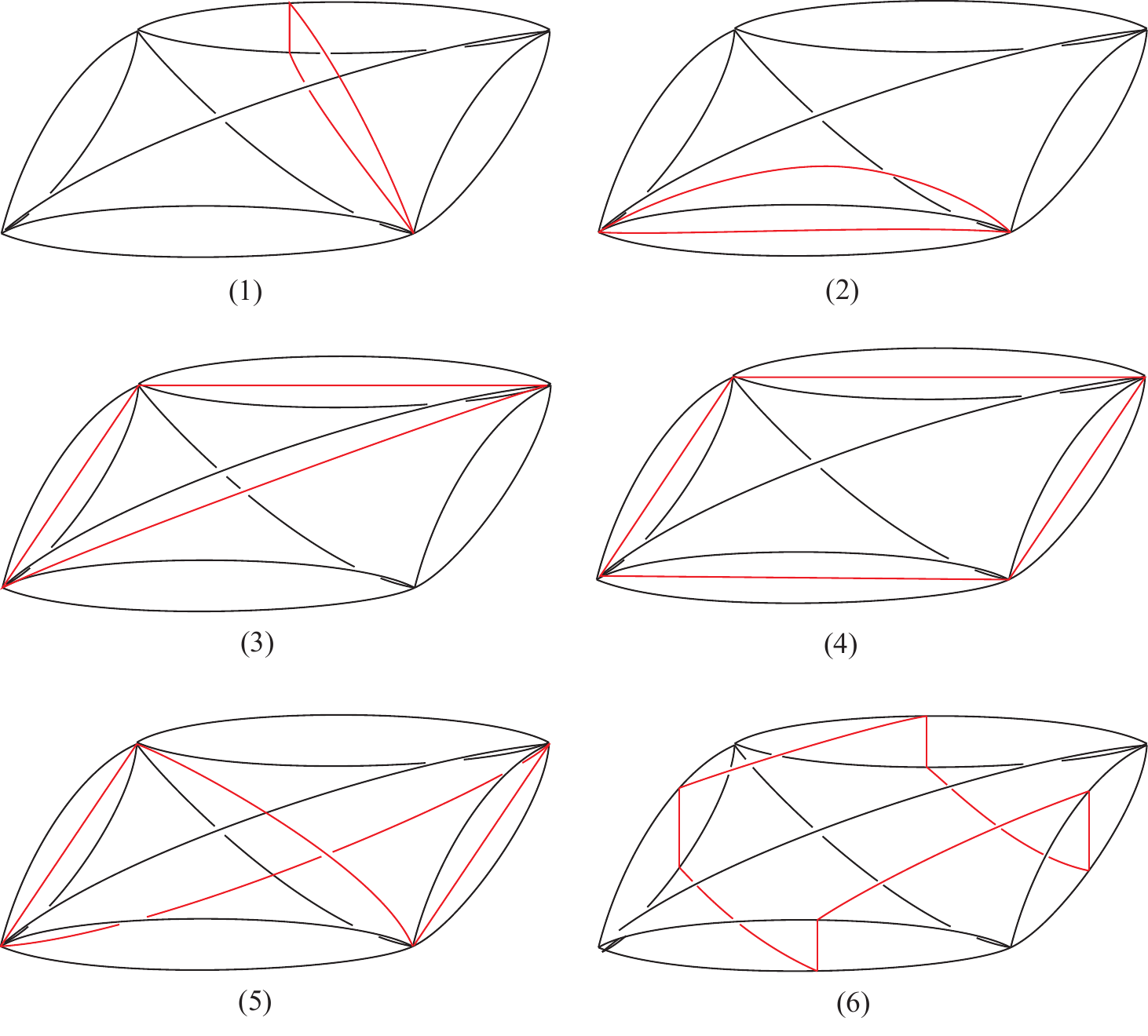}
\caption{The five possibilities for elementary normal discs in an expanded tetrahedron, as described in Lemma \ref{Lem:DiscsExpandedTet}}
\label{Fig:DiscsInExpandedTet}
\end{figure}

\begin{lemma}
\label{Lem:AtMostTilesInElemDisc}
Each elementary normal disc of $\hat S$ intersects at most $4$ separatrices and most $4$ tiles.
\end{lemma}

\begin{proof}
Each separatrix starts at an interior saddle or boundary saddle. Consider the case of a boundary saddle first. This occurs in an elementary disc
as in (3), (4) or (5) of Lemma \ref{Lem:DiscsExpandedTet} or as in (3) of Lemma \ref{Lem:DiscsTurnover} or as in Lemma \ref{Lem:DiscsSaddleTetrahedron}. 
Therefore, the separatrices starting at a boundary saddle lie wholly within a single elementary normal disc.

Now consider the case of an interior saddle. This lies in a disc as in (6) of Lemma \ref{Lem:DiscsExpandedTet}, which we will call a \emph{square}.
 Just above and below the event page containing the interior saddle there
are two specified pages. These two pages have triangulations
with no vertices of the triangulation in the interior of these pages. Note that $S$ intersects of these discs in normal arcs.
As the height of the pages increase, some of these arcs come together in interior saddles lying in squares parallel to the given square. 
The separatrices emanating from an interior saddle lie in a square. They then run to bigons in the
boundary of the expanded tetrahedron. After that, they run through discs as in (2) of Lemma \ref{Lem:DiscsTurnover} and
Lemma \ref{Lem:DiscsRugby} until they reach a vertex on the binding circle. These latter discs
do not contain any boundary saddles. So they each intersect at most one separatrix, and hence intersect
at most 2 tiles. We also deduce that the only separatrices lying in squares
are the ones emanating from the interior saddle. So they intersect 4 separatrices and 4 tiles.
\end{proof}

\section{Extending the partial hierarchy}
\label{Sec:ExtendingHierarchy}

\subsection{Overview}

Theorem \ref{Thm:WeaklyFundamentalHierarchy} provides an exponentially controlled hierarchy for the exterior of the link $K$.
Let $S_1, \dots, S_n$ be the first $n$ surfaces in this hierarchy.
Let $S_0$ be $\partial N(K)$. Suppose that this partial hierarchy is in admissible form,
and let $w$ be its binding weight.

We now explain, in overview, how to extend this partial hierarchy.

\begin{enumerate}
\item Modify the hierarchy so that it is well-spaced, as defined in Section \ref{Subsec:HomotopeAdmissible}.
\item Build a triangulation $\mathcal{T}$ and a polyhedral decomposition $\Delta$ for the exterior of $K$,
with $S_0 \cup \dots \cup S_n$ as a subcomplex,
as described in Section \ref{Subsec:TriangulatingExterior}. Let $\mathcal{T}_M$ be the restriction of $\mathcal{T}$
to $M$, the exterior of $S_0 \cup \dots \cup S_n$.
\item Let $S = S_{n+1}$ be the next surface in the hierarchy. Since $S$ is exponentially controlled, we may
assume that it has exponentially bounded weight in $\mathcal{T}_M$.
\item Homotope $S$ to a nearly embedded alternative admissible surface $\hat S$ that is embedded away
from the binding circle, that is normal with respect to 
the handle structure $\mathcal{K}$, 
and with binding weight at most $200 w^3(w - \sum_i \chi(S_i))^2 w(S)$, 
as described in Section \ref{Subsec:HomotopeAdmissible}.
\item This bound on the binding weight is too large for our purposes, and
so we modify $\hat S$ to reduce its binding weight significantly. This may require exchange moves
and also a modification to the earlier surfaces in the hierarchy.
\item Modify $S_0, \dots, S_n, \hat S$ further so that $\hat S$ is embedded, rather than just nearly embedded.
\item The surface $\hat S$ is in generalised admissible form, but we require it to be alternative admissible. A small perturbation to the surface achieves this.
\end{enumerate}

Step (1) is described in Section \ref{Subsec:WellSpaced}. 
Step (2) is explained in Sections \ref{Subsec:TriangulatingExterior}, \ref{Subsec:TriExterior} and \ref{Subsec:HierarchyExteriorHS}.
Step (4) is described in Section \ref{Subsec:HomotopeAdmissible}. 
Step (5) will require considerable work
and will be explained in Sections \ref{Sec:ModificationsAdmissible}, \ref{Sec:Parallelism},  \ref{Sec:Euclidean}, \ref{Sec:Thin}
and \ref{Sec:EulerCharacteristic}. Step (6) is described in Section \ref{Subsec:MakeEmbedded}. Step (7) is given in Section \ref{Subsec:GenAdmissToAltAdmiss}.

\subsection{Making the hierarchy well-spaced}
\label{Subsec:WellSpaced}

We suppose that we have an admissible partial hierarchy $S_0, \dots, S_n$. We now show how to make it
well-spaced. The boundary pattern $P$ forms an embedded graph in an arc presentation.
Let $F$ be a component of $(S_0 \cup \dots \cup S_n) \cut P$. 
We will show how to modify $F$ so that it is still nearly admissible, with the same boundary,
and so that the closure in $F$ of the union of the tiles incident to $\partial F$ is a regular
neighbourhood of $\partial F$. Performing this operation for each component of 
$(S_0 \cup \dots \cup S_n) \cut P$ will make the hierarchy well-spaced.

Pick a transverse orientation on $F$. We cut the hierarchy along $\partial F$ and isotope $F$
a little in this transverse direction. Thus, $\partial F$ no longer becomes attached to $P$.
We then insert between $\partial F$ and $P$ an annulus $A$ as shown in Figure \ref{Fig:WellSpaced}.
Specifically, between each fibre of $\partial F$ and the adjacent fibre of $P$, we
insert another fibre in the same page. The union of these new fibres will form a core curve
$C$ for the annulus $A$, which will divide $A$ into two smaller annuli. One of these annuli
as attached to $\partial F$, extends beyond $F$ and then rapidly returns to $C$. The other
annulus emanates from $P$ initially in the same way that the original copy of $F$ did, but
it then rapidly returns to $C$. (See Figure \ref{Fig:WellSpaced}.)

\begin{figure}[h]
\includegraphics[width=4in]{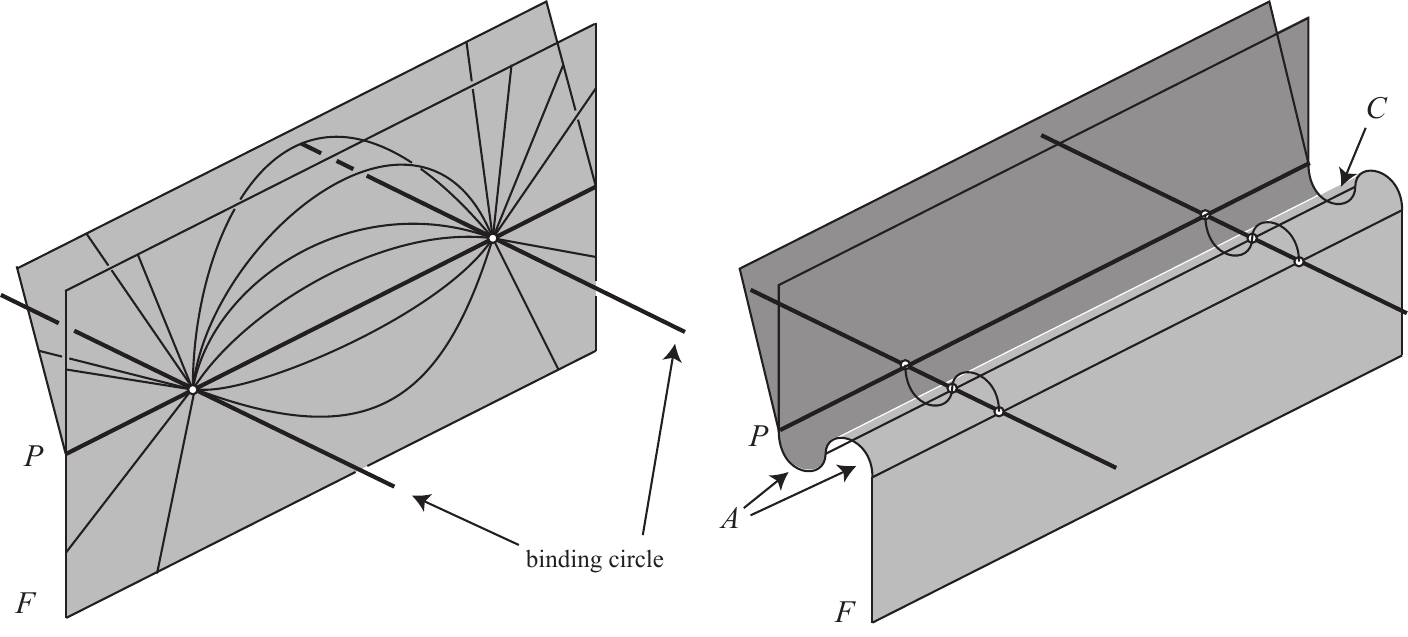}
\caption{Making $F$ well-spaced}
\label{Fig:WellSpaced}
\end{figure}

A small isotopy places $A$ into nearly admissible form keeping its boundary fixed. In the case
where $F$ has no boundary saddles, the induced
foliation on $F$ is then as shown in Figure \ref{Fig:WellSpacedFoliation}.

\begin{figure}[h]
\includegraphics[width=4.5in]{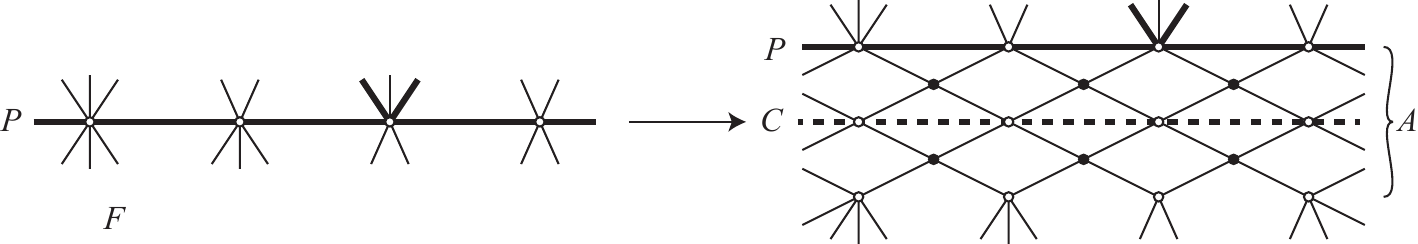}
\caption{The effect on the foliation}
\label{Fig:WellSpacedFoliation}
\end{figure}

When $F$ has a boundary saddle incident to an arc in $\partial F$, then one possibility for
the new singular foliation near that arc is shown in Figure \ref{Fig:WellSpacedBoundarySaddle}.

\begin{figure}[h]
\includegraphics[width=4.5in]{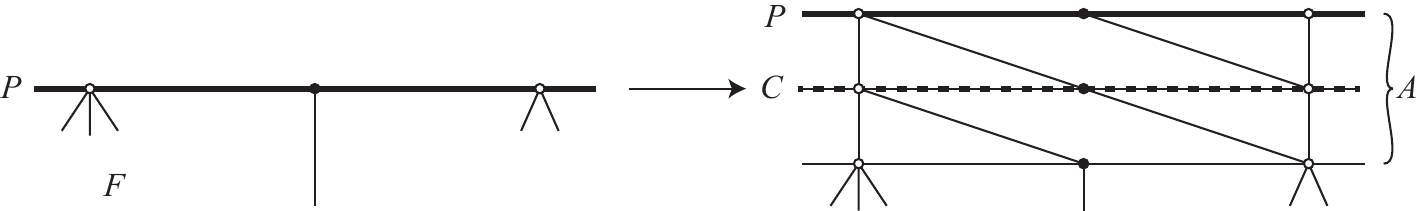}
\caption{The effect on the foliation near a boundary saddle}
\label{Fig:WellSpacedBoundarySaddle}
\end{figure}

Note that each vertex of $\partial F$ has been replaced by 3 vertices. Since
in each surface $S_i$ of the partial hierarchy, the graph $(\partial S_0 \cup \dots \cup \partial S_n) \cap S_i$
has vertices with degree at most 3, we deduce that this gives rise to at most 6 new vertices.
Thus, the binding weight of the hierarchy increases by at most a factor of $7$.

\subsection{Making the surface embedded}
\label{Subsec:MakeEmbedded}

In this subsection, we suppose that we have a hierarchy $S_0, \dots, S_n$ in admissible form.
In particular, each surface is embedded. We also have the next surface $\hat S$ in alternative admissible
form, but it might just be nearly embedded. However, it is embedded away from the binding circle.
We will show how to modify the hierarchy so that
all the surfaces end up embedded.

The exterior of $S_0 \cup \dots \cup S_n$ is a 3-manifold $M$. Its boundary has a singular foliation.
The boundary of $S$ maps to a graph $\Gamma$ in $\partial M$ that is a union of fibres. 
We will modify the surfaces $S_0, \dots, S_n$ so that $\Gamma$ becomes a union of disjoint
simple closed curves. We can then make $\partial S$ run along these curves so that it becomes
embedded.

The first stage is to cut $S_0 \cup \dots \cup S_n$ along $\Gamma$. Then isotope the resulting components
of $(S_0 \cup \dots \cup S_n) \cut \Gamma$ away from $\Gamma$ but keeping them admissible.
The second stage is to realise $\partial S$ as a collection of simple closed curves in their arc presentation.
The third stage is to thicken $\partial S$ to a union $A$ of admissible annuli in which $\partial S$
forms the core curves. The final stage is to isotope $\partial A$ so that it runs along $(S_0 \cup \dots \cup S_n) \cut \Gamma$.
In this way, we replace $(S_0 \cup \dots \cup S_n)$ by a new 2-complex $((S_0 \cup \dots \cup S_n) \cut \Gamma) \cup A$. The
binding weight of this 2-complex has gone up by at most $3w_\beta(\partial S)$, since each new vertex is
a vertex of $A$, and $A$ has three vertices for each vertex of $\partial S$. See Figure \ref{Fig:MakingSEmbedded}.

\begin{figure}[h]
\includegraphics[width=3.5in]{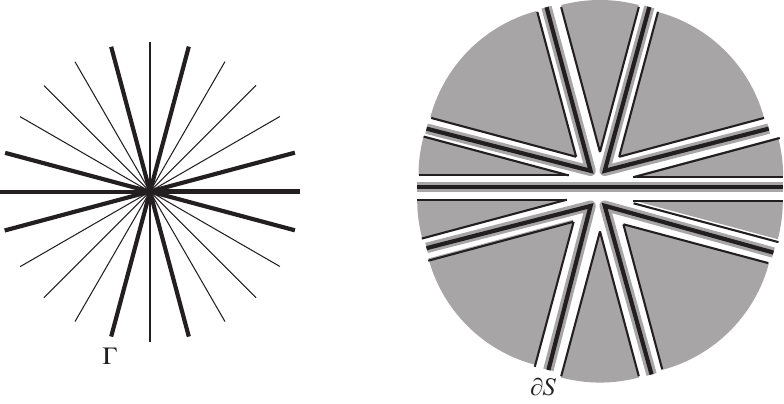}
\caption{Making $S$ embedded. Left: a neighbourhood of a vertex in $S$. The image of $\partial S$ is a graph $\Gamma$. Right: Cutting $S$ along $\Gamma$, and thickening $\partial S$ to annuli.}
\label{Fig:MakingSEmbedded}
\end{figure}

Since the binding weight of $S_0 \cup \dots \cup S_n$ increases by at least the binding weight of $\partial S$,
we only perform this procedure when the latter quantity is relatively small.

\subsection{From generalised admissible form to alternative admissible form}
\label{Subsec:GenAdmissToAltAdmiss}

Given a generalised admissible surface, we may perturb it to an alternative admissible surface, as follows.
Any generalised interior saddle with $n$ singular leaves emanating from it may be perturbed to $(n-2)/2$ saddles.
An example where $n = 6$ is shown in Figure \ref{Fig:PerturbGenSaddle}. Similarly, any generalised boundary saddle
with $n$ singular leaves may be perturbed to $(n-2)/2$ saddles (if $n$ is even) or $(n-3)/2$ saddles plus a boundary saddle
(if $n$ is odd). This does not affect the binding
weight of the surface.

\begin{figure}[h]
\includegraphics[width=3in]{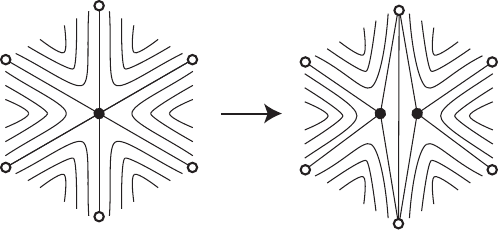}
\caption{A small isotopy replaces a generalised interior saddle with several interior saddles}
\label{Fig:PerturbGenSaddle}
\end{figure}

\section{Branched surfaces}
\label{Sec:BranchedSurfaces}

\subsection{Definitions}
\label{Sec:BSDefns}

A \emph{generalised branched surface} is a compact generalised 2-complex $B$ smoothly embedded in a 3-manifold $M$,
with the following extra structure:
\begin{enumerate}
\item At each point $x$ in $B$, there is a specified tangent plane in $T_x(M)$ that is
denoted by $T_x(B)$. All the 1-cells and generalised 2-cells containing $x$ are required to have tangent
spaces that lie in $T_x(B)$.
\item As a consequence, for each point $x$ in the interior of a 1-cell, the tangent plane $T_x(B)$
is divided into two half-planes by the tangent space of the 1-cell. These are the two \emph{sides} of the 1-cell.
At each such point $x$, we require that there are generalised 2-cells on both sides of $x$ or
just one side. The closure of the union of points of the former type is
the \emph{branching locus} of $B$. The closure of the union of points of the latter type is
$\partial B$, the \emph{boundary} of $B$.
\end{enumerate}
The generalised 2-cells of $B$ are known as \emph{patches}. We permit a 1-cell in the branch locus to have 
a single patch on each side.
A generalised branched surface is known as a \emph{branched surface} if each 1-cell either has patches on
both sides or it is incident to a single patch. The union of points of the latter type is the \emph{boundary}
of $B$ and is denoted $\partial B$.

We let $N(B)$ be a thickening of $B$ in $M$. This thickening is almost a regular neighbourhood, except
that $\partial B$ is required to lie in $\partial N(B)$. (This would hold true for a regular neighbourhood
if $\partial B \subset \partial M$, but we will not be making this requirement.) This thickening has
a decomposition as a union of \emph{fibres}, each of which is a closed interval. There is a map
$\pi \colon N(B) \rightarrow B$ that collapses each fibre to a point. Away from a small regular neighbourhood
of the 1-skeleton of $B$, $\pi$ is the projection map for an $I$-bundle. We choose this decomposition of $N(B)$
into fibres so that the following conditions hold:
\begin{enumerate}
\item For each $x \in B$, the fibre through $x$ is required to have tangent space that does not
lie in $T_x(B)$.
\item Each fibre intersects $\partial N(B)$ in its endpoints, together possibly with a collection of
closed intervals.
\item When the fibre does not intersect $\partial B$, these closed intervals in (2) are
required to lie in the interior of the fibre.
\end{enumerate}

The boundary of $N(B)$ decomposes into two subsurfaces:
\begin{enumerate}
\item the horizontal boundary $\partial_h B$, which is the closure of the union of the endpoints of the fibres;
\item the vertical boundary $\partial_vB$, which is $\partial N(B) \cut \partial_h B$.
\end{enumerate}
\noindent Each component of
$\partial_v B  \cut \pi^{-1}(\partial B)$ is a \emph{cusp}. Cusps arise near the branching locus of $B$. (See Figure \ref{Fig:BranchNbhd}.)

\begin{figure}[h]
\includegraphics[width=2.5in]{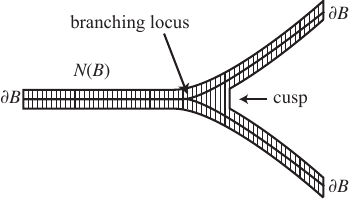}
\caption{The fibred neighbourhood of a branched surface $B$}
\label{Fig:BranchNbhd}
\end{figure}

Note that the branched surfaces that we consider are far from generic. In particular, more
than three 2-cells may be incident to a 1-cell of $B$. This means the cusps of $\partial N(B)$
can be quite complicated. However, the following result was established in Lemma 8.1 of \cite{Lackenby:PolyUnknot}.
When $B \cap \partial M = \partial B$ and $\pi^{-1}(\partial B) = N(B) \cap \partial M$ and
at each $x \in \partial B$, the tangent plane $T_x(B)$ does not equal $T_x(\partial M)$, then
each cusp is either an annulus or a disc $D$ such that $D \cap \partial M$ is two arcs in
$\partial D$.

\subsection{Natural handle structure}
\label{Subsec:NaturalHS}

Let $B$ be a branched surface with a fibred neighbourhood $N(B)$. Suppose that $B$ has
only disc patches. Then $N(B)$ has a handle structure $\mathcal{H}$ where each $i$-handle is a thickened
$i$-cell of $B$. We term this the \emph{natural handle structure} of $N(B)$.

As explained in Section \ref{Sec:HandleStructures}, the surface $\mathcal{F} = \partial \mathcal{H}^0 \cap (\mathcal{H}^1 \cup \mathcal{H}^2)$ plays an
important role when considering handle structures. In the case of the natural handle structure,
we can arrange that this surface is a union of fibres. So one could view this as a fibred neighbourhood
of a generalised train track (once this has been suitably defined). In fact, the
branched surface structure imposes further conditions on $\mathcal{F}$, as follows.

Let $v$ be a 0-cell of $B$ and let $H_0$ be the associated 0-handle of $\mathcal{H}$.
We require that $H_0$ is a copy of $D^2 \times [-1,1]$, where the fibres are of the form
$\{ \ast \} \times [-1,1]$. The vertex $v$ has a
well-defined tangent space $T_vB$, which we may view as intersecting
$\partial H_0$ in $\partial D^2 \times \{ 0 \}$. Because $B$ is smoothly embedded,
its intersection with $\partial H_0$ lies within $\partial D^2 \times [-1,1]$, provided $H_0$ is chosen
appropriately. Moreover, the tangent planes of $B$ near $v$ all are close
to the tangent plane through the equator. 

Similarly, we require that each 1-handle of $\mathcal{H}$ is of the form $[-1,1] \times [-1,1] \times I$,
where the fibres are of the form $\{ \ast \} \times I$ for a point $\ast \in [-1,1] \times [-1,1]$.

\subsection{Surfaces weakly carried by branched surfaces}

A compact surface $S$ is \emph{weakly carried} by a generalised branched surface $B$ if $S$ is properly embedded in $N(B)$
transversely to the fibres of $N(B)$, and also $\partial S = S \cap \partial N(B)$ lies within a small
regular neighbourhood of the 1-cells of $B$. Note that we permit $\partial S$ to intersect
the cusps and the horizontal boundary of $N(B)$.

\begin{figure}[h]
\includegraphics[width=2.5in]{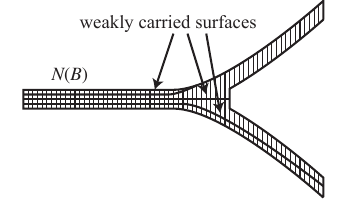}
\caption{Weakly carried surfaces}
\label{Fig:WeaklyCarried}
\end{figure}

When a surface $S$ is weakly carried by $B$, each patch of $B$ inherits a non-negative
integer, which is the \emph{weight} of $S$ in that patch. This is equal to the number
of intersections between $S$ and any fibre away from a regular neighbourhood of the branching locus. 

The surfaces that we will consider will be standard in the natural handle structure on $N(B)$.
Such surfaces satisfy one of the conditions of
normality, as follows.

\begin{lemma}
\label{Lem:WeaklyCarriedNormal}
Let $B$ be a generalised branched surface in which each patch is a disc.
Let $S$ be a surface weakly carried by $B$ that is standard in the natural handle
structure  $\mathcal{H}$ on $N(B)$.
Then in each 0-handle $H_0$ of $\mathcal{H}$, each disc $D$ of $S \cap H_0$ satisfies (1)
in the definition of normality. That is, $\partial D$ runs over each 1-handle of $\mathcal{F}$ at most once.
\end{lemma}

\begin{proof} As discussed in Section \ref{Subsec:NaturalHS}, the surface $\mathcal{F} \cap \partial H_0$ lies
within $\partial D^2 \times [-1, 1] \subset \partial H_0$, where $\partial D^2 \times \{ 0 \}$ is
the equator of $\partial H_0$ corresponding to the tangent plane of the branched surface.
Each fibre of $N(B) \cap \partial H_0$ lies within $\{ \ast \} \times [-1, 1]$ for
some point $\ast \in \partial D^2$. Each 1-handle of $\mathcal{F}$ has a product foliation by fibres,
each parallel to the co-core of the 1-handle. The boundary of the disc $D$
lies within $\partial D^2 \times [-1, 1]$ and intersects $\{ \ast \} \times [-1, 1]$
at most once for each $\ast \in \partial D^2$. It therefore runs over each 1-handle of $\mathcal{F}$ at most once.
\end{proof}

\subsection{Surfaces carried by branched surfaces}

We say that $S$ is \emph{carried} by a branched surface $B$ if it is weakly carried and, in addition, $S \cap \pi^{-1}(\partial B) = \partial S$.
Thus, in Figure \ref{Fig:WeaklyCarried}, only the lower surface satisfies this condition.

A surface $S$ carried by $B$ is determined, up to fibre preserving ambient isotopy, by its weights in the patches. The collection
of these weights is known as the \emph{vector} for $S$, and is denoted $[S]$.

When surfaces $S$, $S_1$ and $S_2$ are carried by $B$, we say that
$S$ is the \emph{sum} of $S_1$ and $S_2$ if $[S] = [S_1] + [S_2]$.
Then $S_1$ and $S_2$ are \emph{summands}. There is an intrinsic characterisation
of one carried surface $S_1$ being a summand of another carried surface $S$: this occurs
if and only if the weight of $S$ in each patch is at least that of $S_1$.

\subsection{A branched surface carrying an alternative admissible surface}
\label{Subsec:BranchedSurfaceCarryingAdmissible}

Suppose that we have a hierarchy $H$ in admissible form. We give the exterior of $H$ the handle structure $\mathcal{K}$ described in Section \ref{Subsec:HierarchyExteriorHS}. Let $S$ be a nearly embedded surface
in the exterior of $H$ in alternative admissible form with no annular tiles. We assume
that $S$ is embedded away from the binding circle. Hence, in $\mathcal{K}$,
the surface is actually embedded. We also make the assumption that it is normal with respect to $\mathcal{K}$.
We now construct a branched surface $B$ that carries $S$. 

Since $S$ is nearly admissible, it is a union of tiles. We say that two tiles of $S$ have the same \emph{type}
if they are normally parallel with respect to the handle structure $\mathcal{K}$. 
For each type of tile of $S$, take one representative of this tile. These tiles will
form the patches of $B$. When two tiles share a separatrix,
the corresponding patches are glued along a 1-cell of $B$. 
This may lead to a sequence of identifications.
For example, if a patch $P_1$ shares a separatrix with $P_2$, and $P_2$ shares the same separatrix with a patch $P_3$,
then $P_1$ and $P_3$ are, of course, glued along this separatrix.

The result is the branched surface $B$. Note that each patch of $B$ is a disc
and so $B$ is a 2-complex, rather than a generalised 2-complex. Let $N(B)$ be a fibred regular neighbourhood of $B$.
We give $N(B)$ its natural handle structure $\mathcal{H}$.
For each $i$-cell of $B$, its regular neighbourhood forms an $i$-handle of $\mathcal{H}$.

We will measure parallelism using the 3-manifold $N(B)$ and its handle structure $\mathcal{H}$. 
More specifically, we say that two vertices of $S$ have \emph{$\mathcal{H}$-parallel} stars 
if their stars are normally parallel with respect to $\mathcal{H}$.

Note that $S$ is indeed carried by $B$, for the following reason. Form a regular neighbourhood
$N(S)$ of $S$, so that $N(S)$ intersects each handle of $\mathcal{K}$ in a union of
elementary normal discs. When two tiles of $S$ are $\mathcal{K}$-parallel,
then attach the space between them to $N(S)$. The result is the 3-manifold $N(B)$.
By definition of parallelism, the region between two normally parallel tiles
has the structure as $D \times [0,1]$ for a disc $D$, where $D \times \{ 0,1 \}$
is the two tiles. Thus, $D \times [0,1]$ has the structure of a union of intervals.
These combine to form the fibres of $N(B)$. If we collapse each of these fibres to
a point, we obtain the branched surface $B$. By construction, $S$ is a subset
of $N(B)$ transverse to the fibres and so it is weakly carried by $B$. In fact, $S \cap \pi^{-1}(\partial B) = \partial S$,
for the following reason. Note that $\partial S = S \cap \partial N(S)$ lies in the interior of $\partial_v N(S)$. The fibred
neighbourhood $N(B)$ is obtained from $N(S)$ by adding to it. But none of these additions affects
the interior of $\partial_v N(S)$, and furthermore $\pi^{-1}(\partial B)$ contains $\partial_v N(S)$.
Hence, $S \cap \pi^{-1}(\partial B) = \partial S$, and so $S$ is carried by $B$.

\section{Modifications to an admissible surface}
\label{Sec:ModificationsAdmissible}

When an admissible surface has a vertex of small valence, there is a modification that can
be made to the surface and the arc presentation, which reduces the complexity of the surface.
These modifications were key to Dynnikov's analysis of the unknot and its spanning disc.
Similar moves had also been utilised in the context of braids by Birman and Menasco \cite{BirmanMenasco} 
and by Bennequin \cite{Bennequin}.

\subsection{Two-valent interior vertex}
\label{Subsec:TwoValInterior}

Let $S$ be a generalised admissible surface.
The star of a two-valent interior vertex $s$ in $S$ is shown in Figure \ref{Fig:TwoValVertex}. 

\begin{figure}[h]
\includegraphics[width=2in]{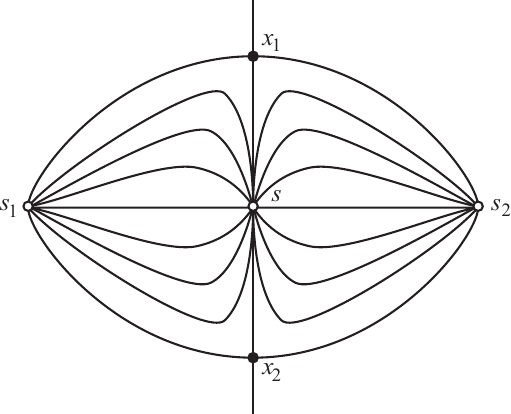}
\caption{A 2-valent interior vertex and its star}
\label{Fig:TwoValVertex}
\end{figure}

The arrangement of this star in 3-space is shown in Figure \ref{Fig:TwoValVertex2}. 

\begin{figure}[h]
\includegraphics[width=3in]{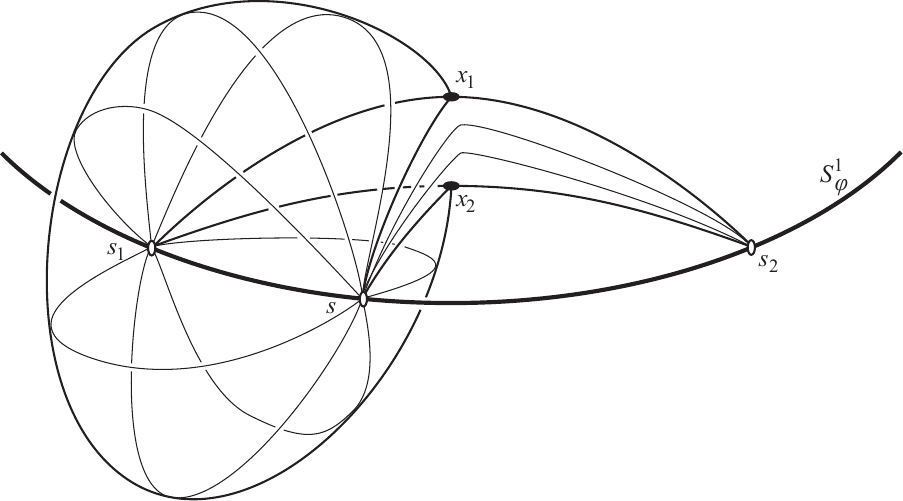}
\caption{The arrangement of the surface with respect to the arc presentation}
\label{Fig:TwoValVertex2}
\end{figure}

There are two separatrices emanating from $s$. Let $t_1$ and $t_2$ be their $\theta$ values,
where $0 \leq t_1 < t_2 < 2\pi$. For each $t \in (t_1,t_2)$, there is a leaf of the foliation in $\mathcal{D}_t$
running from $s$ to a vertex $s_1$ in the star of $s$. For every $t \in S^1_\theta - (t_1, t_2)$,
there is a leaf of the foliation in $\mathcal{D}_t$ running from $s$ to the other vertex $s_2$
in the star of $s$. We may suppose, after relabelling if necessary, 
that the vertices $s_1, s, s_2$ occur in that order around $S^1_\phi$.

Let $i \colon S^1_\phi - \{ s_1, s, s_2 \} \rightarrow S^1_\phi -  \{ s_1, s, s_2 \}$ be the following involution.
It sends the interval $(s_1, s)$ to the interval $(s,s_2)$ linearly, preserving this ordering.
Similarly, it sends $(s, s_2)$ to the interval $(s_1,s)$ also linearly, preserving this ordering.
It fixes the interval between $s_2$ and $s_1$.

We may now apply the following procedure, which will modify $S$, and which may also move the link and indeed
any other surfaces in $S^3$, including any surfaces in the partial hierarchy that have already been constructed.
\begin{enumerate}
\item Remove the star of $s$ from $S$.
\item For any leaf of the foliation or arc of the link with an endpoint in $\phi \in S^1_\phi \backslash \{s_1, s_2 \}$, replace it with a leaf
ending at $i(\phi)$.
\item For the leaves of the foliation that ended at $s_1$ or $s_2$, but which were not part of
the star of $s$, replace them with a leaf ending at $s$.
\end{enumerate}
This procedure can be achieved via an isotopy of the link and the surfaces. The effect
of this on the rectangular diagram is to perform a generalised exchange move (see \cite{Dynnikov}). The effect on $S$
is to remove the star of $s$ from $S$, and to identify the leaves running from $s_1$ in the boundary of this star
with the leaves running from $s_2$. Therefore this decreases the binding weight of $S$ by 2. It has no effect on
the singular foliation in the remainder of $S$, and no effect on the singular foliation of the other surfaces.

\subsection{Three-valent interior vertex}
\label{Subsec:ThreeValInterior}

Let $s$ be a 3-valent interior vertex of $S$. Let $s_2, s_3, s_4$ be the three vertices in its star.
After relabelling, we may suppose that these vertices occur in the order $s, s_2, s_3, s_4$
around $S^1_\phi$. Then the singular foliation near $s$ is as shown in Figure \ref{Fig:ThreeValVertex}.
(In the figure, only saddles are shown for simplicity, but they may be generalised saddles.)
Each generalised saddle has an \emph{orientation}. Here, one picks a transverse orientation on $S$ near $s$.
If the transverse orientation on $S$ at a generalised saddle points in the direction of increasing $\theta$,
we say that it is \emph{positive}. Otherwise it is \emph{negative}. There are three generalised saddles
in the boundary of the star of $s$. Therefore, there are two, $x_1$ and $x_2$ say, that have the
same orientation. By relabelling if necessary, we may assume that
$x_1$ is the generalised saddle involving $s$, $s_2$ and $s_3$,
and that $x_2$ is the generalised saddle involving $s$, $s_3$ and $s_4$.
Let $t_1 < t_2$ be the $\theta$ values of these saddles. 
Note that $t_1 \not= t_2$ because the pages $\mathcal{D}_{t_1}$ and $\mathcal{D}_{t_2}$ are distinct at $s$.
So $x_1$ and $x_2$ are definitely distinct generalised saddles.

\begin{figure}[h]
\includegraphics[width=4in]{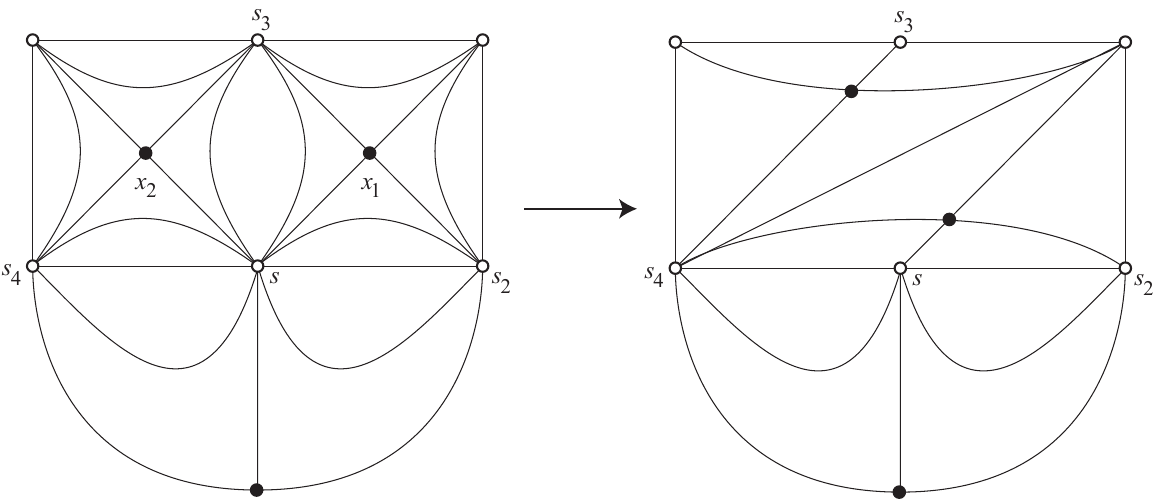}
\caption{A 3-valent interior vertex}
\label{Fig:ThreeValVertex}
\end{figure}

In previous work on this problem \cite{Dynnikov}, the following procedure was used.
The first stage of the modification was to remove all events from the interval $(t_1, t_2)$.
This was achieved by moving all saddles in the the interval $(t_1, t_2)$ with
separatrices ending in $(s, s_3)$ into the future, and all saddles with
separatrices ending in $(s_3, s)$ into the past. Similarly, all boundary arcs in the interval $(t_1, t_2)$
with endpoints in $(s, s_3)$ move into the future, and those with endpoints 
in $(s_3, s)$ move into the past. Then one increases
the $\theta$-value of $x_1$ until it equals that of $x_2$, and then is slightly greater
than it. The new singular foliation is shown in Figure \ref{Fig:ThreeValVertex}.
The effect of this process on the singular foliation is to turn $s$ into a 2-valent
interior vertex. One can then perform the procedure described in Section 9.1
to remove this vertex.

In this paper, we will do something that is a little different. Again, we remove
all events from the interval $(t_1,t_2)$, by moving some into the future and
some into the past. And again, we will increase the $\theta$-value of $x_1$,
but we will stop when $x_1$ and $x_2$ have the same $\theta$-values.
These can then be combined to form a single generalised saddle.
The effect on the singular foliation of $S$ is shown in Figure \ref{Fig:ThreeValVertex2}. 
Once again, the vertex $s$ has become 2-valent, and it can be removed using
the method in Section \ref{Subsec:TwoValInterior}.

\begin{figure}[h]
\includegraphics[width=4in]{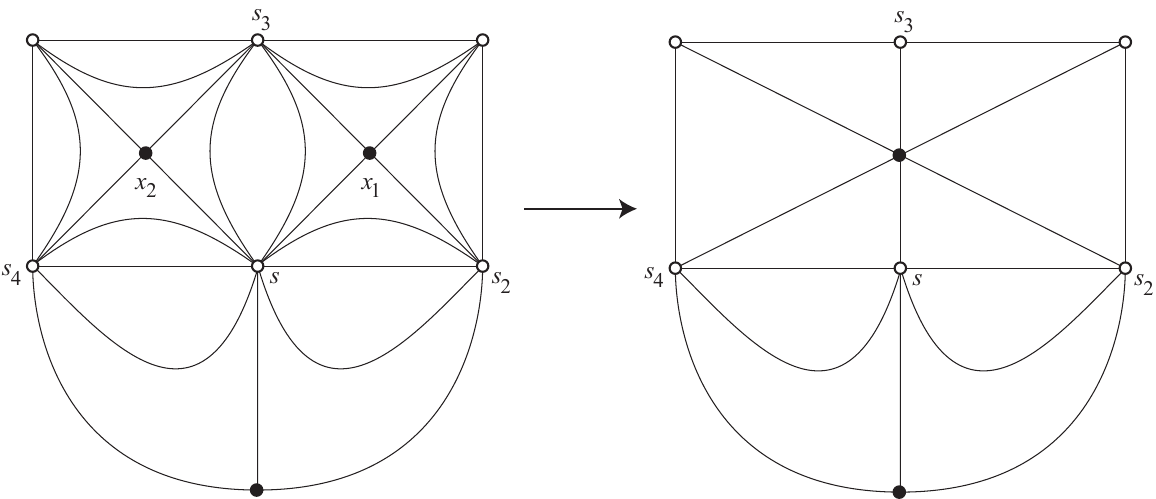}
\caption{The modification used in this paper}
\label{Fig:ThreeValVertex2}
\end{figure}

\subsection{Low-valence boundary vertices}

For a vertex $s$ on the boundary of a generalised admissible surface $S$, there are two
types of separatrix that are incident to it: those lying in $\partial S$, and those
lying in the interior of $S$. We say that the \emph{boundary valence} $d_b(s)$
is equal to the number of separatrices in $\partial S$ that are incident to $s$.
The \emph{interior valence} $d_i(s)$ is the number of interior separatrices incident to $s$.
We say that a vertex $s$ is a \emph{low-valence boundary vertex} if it lies on $\partial S$
and $2d_i(s) + d_b(s) < 4$. Thus, these vertices fall into one of five possible types:
\begin{enumerate}
\item $d_b(s) = 0$ and $d_i(s) = 0$.
\item $d_b(s) = 0$ and $d_i(s) = 1$.
\item $d_b(s) = 1$ and $d_i(s) = 0$.
\item $d_b(s) = 1$ and $d_i(s) = 1$.
\item $d_b(s) = 2$ and $d_i(s) = 0$.
\end{enumerate}

Case 1 cannot occur, because a vertex cannot have zero valence unless the component of $S$
containing it is a single bigon tile. The remaining cases are shown in Figure \ref{Fig:LowValBoundaryVertices}.

\begin{figure}[h]
\includegraphics[width=4in]{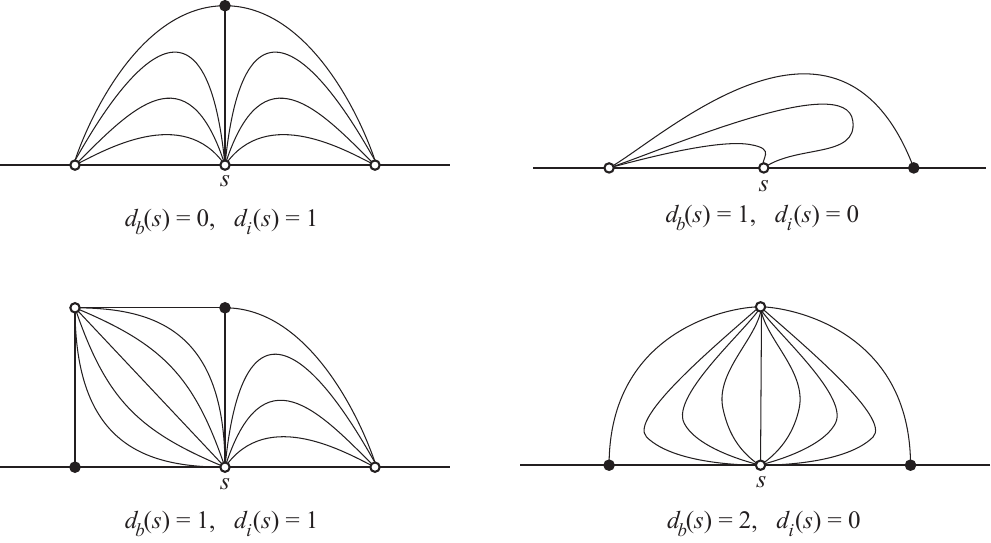}
\caption{Low-valence boundary vertices}
\label{Fig:LowValBoundaryVertices}
\end{figure}

In previous papers \cite{Dynnikov, Lackenby:PolyUnknot}, such an arrangement usually specifies an ambient isotopy of $\partial S$.
For example, when considering arc presentations of the unknot, which bounds a spanning disc $S$,
Case (2) above leads to an isotopy of the unknot, which results in a sequence of cyclic permutations
and exchange moves, followed by a destabilisation. We want to do something similar here,
but simply moving $\partial S$ would appear to be problematic. The boundary of $S$ lies in
earlier surfaces of the hierarchy, and this needs to remain the case. We therefore introduce
a completely new type of modification, called a \emph{wedge insertion}.

\subsection{Wedge insertion}

Let $M$ be the exterior of an admissible partial hierarchy. Let $P$ be its boundary pattern.
Let $D$ be a nearly admissible disc embedded in $M$ disjoint from $P$,
with $D \cap \partial M$ a single arc in $\partial D$, and where the remainder of $\partial D$
is an arc in some page. Then the associated \emph{wedge} $W$ is a regular neighbourhood of $D$ in $M$.
The 3-manifold $M' = M \cut W$ is obtained \emph{inserting the wedge $W$ into $M$}.
(See Figure \ref{Fig:WedgeInsertion}.)

It is clear that, after a small isotopy, $\partial M' \cut P$ may be placed in nearly admissible form. Moreover,
the singular foliation on $\partial M'$ can be determined from that of $M$ and $D$. We will only
be concerned with the binding weight of $\partial M'$, which one can arrange to be at most
the binding weight of $\partial M$, plus twice the binding weight of $D$. When $D$
is the closure of the star of a low-valence boundary vertex, its binding weight is at most 3 and so 
the binding weight of $\partial M$ increases by at most 6.

\begin{figure}[h]
\includegraphics[width=4in]{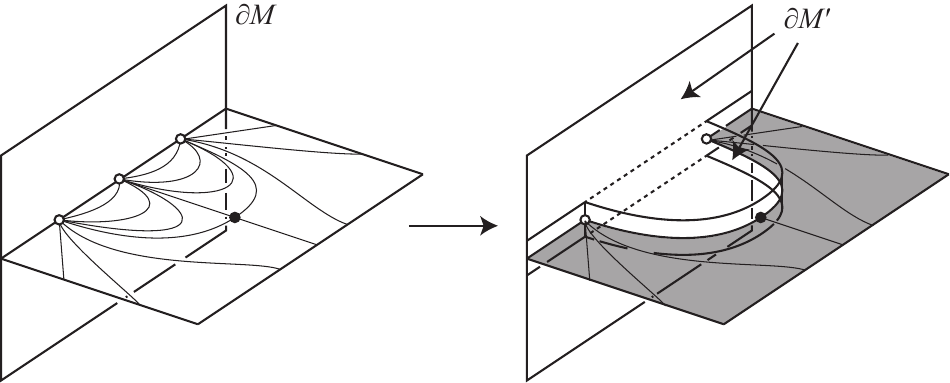}
\caption{Inserting a wedge into $M$}
\label{Fig:WedgeInsertion}
\end{figure}

We will apply this when there is a low-valence boundary vertex. We take $D$ to be the star
of this vertex. The effect of this is to remove this vertex, and thereby reduce the binding weight of $S$. However, it also makes 
$\partial M$ more complicated, and therefore this move needs to be applied carefully and sparingly.

In our situation, $M$ is the exterior of an admissible partial hierarchy for $K$. We are assuming
that $D$ is disjoint from the boundary pattern $P$ of $M$, and so it intersects
a single surface $S'$ in this partial hierarchy. We can view the insertion of the wedge $W$
as being achieved by an isotopy of $S'$. This isotopy slides $S'$ across $W$,
replacing $S' \cap W$ with $\partial W \cut S'$.

\subsection{Parallelism}

The above moves do all simplify the surface $S$, in the sense that they reduce its binding weight.
However, they do so at a cost. In the case of the 2-valent and 3-valent interior vertex, they require
a generalised exchange move, and this is achieved using a sequence of Reidemeister moves.
In the case of a low-valence boundary vertex, the link does not need to be moved,
but this modification does make the earlier surfaces in the partial hierarchy more complicated. We therefore need
to ensure that we use as few of these moves as possible. As in \cite{Lackenby:PolyUnknot}, the key is to perform a
variation of these moves, each of which has the same cost in terms of Reidemeister moves
or the effect on the earlier parts of the hierarchy, but which lead to a much more substantial
decrease in the binding weight of $S$. We will describe these moves in the next section.

\section{Parallelism and envelopes}
\label{Sec:Parallelism}

In the previous section, we described several modifications that one can make to a generalised admissible surface.
As we have already mentioned, it will be important that when we perform these moves, we take advantage
of parallelism. Initially, we will say that two parts of $S$ are parallel if they are normally parallel
in the handle structure $\mathcal{H}$. Recall from Section \ref{Subsec:BranchedSurfaceCarryingAdmissible}
that $\mathcal{H}$ was a handle structure for the fibred regular
neighbourhood $N$ of a branched surface $B$ that carries $S$. As we modify $S$, we will
also need to modify $B$, $N$ and $\mathcal{H}$. In this section, we will describe these modifications.

\subsection{Distinguished subsurfaces}

We will keep track of disjoint distinguished subsurfaces $X$ and $Y$ in $\partial N$. We initially set $X$ to be $\partial_h N$ and set $Y = \emptyset$.
As in Section \ref{Subsec:DistinguishedSubsurfaces}, $\partial X \cup \partial Y$ is a union of standard curves in $\mathcal{H}$.

\begin{figure}[h]
\includegraphics[width=3in]{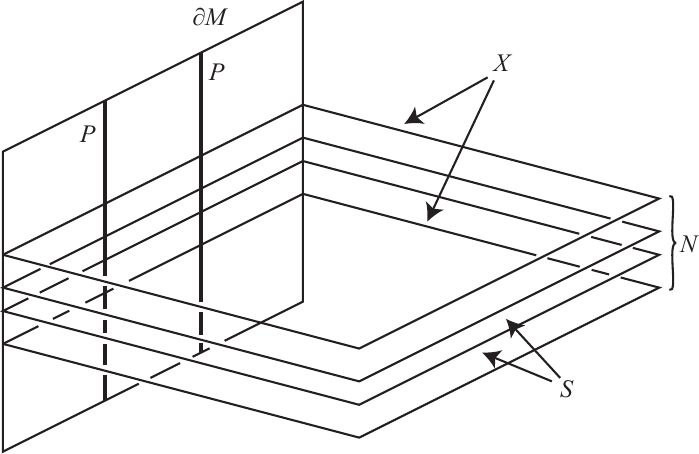}
\caption{The arrangement of $N$, $S$ and $X$}
\label{Fig:DistinguisedSubsurface}
\end{figure}

\subsection{Admissible envelope}
\label{Sec:AdmissibleEnvelope}

We formalise the properties of $S$, $B$, $N$ and $\mathcal{H}$ in the following definition. It
will also be useful to keep track of the two disjoint distinguished subsurfaces $X$ and $Y$, and the boundary pattern $P$.

We say that $(N, X, Y, P \cap N, \mathcal{H}, S)$ is an \emph{admissible envelope} if
\begin{enumerate}
\item $N$ is a fibred regular neighbourhood of a generalised branched surface with only disc patches.
\item $\mathcal{H}$ is the natural handle structure for $N$.
\item $X$ and $Y$ are disjoint subsurfaces of $\partial N$ with $\partial X \cup \partial Y$ standard in $\mathcal{H}$.
\item Each component of $\partial_h N$ lies in $X$ or $Y$ or is disjoint from both.
\item $P \cap N$ is a collection of disjoint arcs properly embedded in $\partial_v N$ with endpoints in $X$ and is otherwise disjoint from $X \cup Y$.
Each such arc lies in a 0-handle of $\calH$, and at most one arc lies in each 0-handle.
\item $S$ is a surface weakly carried by $N$ that is standard with respect to $\mathcal{H}$.
\item $S$ is disjoint from $X$.
\item $S$ is in generalised admissible form with no annular tiles.
\item Each 2-handle of $\mathcal{H}$ is a regular neighbourhood of a tile of $S$, and all the components
of intersection between $S$ and this 2-handle inherit corresponding singular foliations. In particular,
each 2-handle of $\mathcal{H}$ contains at least one tile of $S$.
\item For each 1-handle $H_1$ of $\mathcal{H}$, each component of $H_1 \cap S$ is a regular neighbourhood
either of a separatrix or of an arc in $\partial S$ joining two vertices and not containing a generalised boundary saddle.
\item For each 0-handle $H_0$ of $\mathcal{H}$, $H_0 \cap S$ consists either of a collection of regular neighbourhoods
of vertices of $S$, or of a collection of regular neighbourhoods of generalised saddles of $S$ (but not a mixture of the two). We
refer to 0-handles of these types as \emph{ vertex} or \emph{ saddle} 0-handles.
\end{enumerate}

\subsection{The role of the distinguished subsurface $X$}

The reason for introducing the distinguished subsurface $X$ is as follows. Consider a vertex $v$ of $S$ on the boundary of $M$, for example as in Figure \ref{Fig:BoundaryVxHandle}.
This lies in a 0-handle $H_0$ of $\mathcal{H}$. As explained in Section \ref{Subsec:NaturalHS}, this 0-handle is identified with $D^2 \times [-1,1]$, where the fibres of $N$ are of the form $\{ \ast \} \times [-1,1]$. Hence, $X \cap (D^2 \times [-1,1]) = D^2 \times \{ -1, 1 \}$. Attached to this 0-handle are 1-handles and 2-handles, and these also have an $I$-bundle structure, although the $I$-fibres of these handles intersect $D^2 \times [-1,1]$ in closed intervals in $\{ \ast \} \times [-1,1]$ rather than necessarily the whole of $\{ \ast \} \times [-1,1]$. Since $H_0$ is incident to $\partial M$, there is an arc $\alpha$ in $\partial D^2$ such that $H_0 \cap \partial M = \alpha \times [-1,1]$. In particular, no handles are attached to $(\alpha - \partial \alpha) \times [-1,1]$. 

As discussed in Section \ref{Sec:HandleStructures}, the surface $\mathcal{F} = \mathcal{H}^0 \cap (\mathcal{H}^1 \cup \mathcal{H}^2)$ plays an important role. This has a handle structure with $\mathcal{F}^0 = \mathcal{H}^0 \cap \mathcal{H}^1$ being the 0-handles and $\mathcal{F}^1 = \mathcal{H}^0 \cap \mathcal{H}^2$ equalling the 1-handles. Consider the 0-handles of $\mathcal{F}^0$ incident to some component of $\partial \alpha \times [-1,1]$. Arranged along this component, there may be several 0-handles of $\mathcal{F}^0$. Consider one 0-handle of $\calF$ that is outermost in $\partial \alpha \times [-1,1]$. The intersection between $S$ and this 0-handle of $\calF$ consists of arcs. Note that there is at least one such arc, since every handle of $\mathcal{H}$ has non-empty intersection with $S$. If we had not specified the distinguished subsurface $X$, these arcs would have violated condition (4) in the definition of normality in Section \ref{Sec:NormalHS}. However, in the presence of $X$, they do satisfy condition (4) in the definition of normality in Section \ref{Subsec:DistinguishedSubsurfaces}. Since it is convenient for $S$ to be normal in $\mathcal{H}$, this explains the reason for using $X$.

Note that $P$ is disjoint from $\calF$, for the following reason. Since $P \cap N$ lies in $\partial N$, the interior
of $P \cap \partial \calH^0$ is disjoint from $\calF$. Moreover, as the endpoints of each arc of $P \cap N$ lie in $X$,
these endpoints are disjoint from $\calF$.

\begin{figure}[h]
\includegraphics[width=4in]{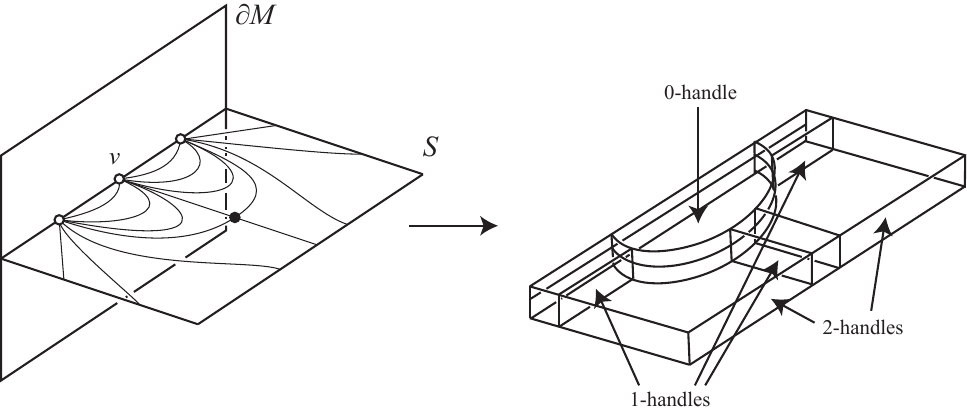}
\caption{A vertex $v$ in $\partial S$ and the 0-handle of $\calH$ that contains it}
\label{Fig:BoundaryVxHandle}
\end{figure}

\subsection{Another admissible envelope}

As explained above, we are keeping track of a 3-manifold $N$ with a handle structure $\mathcal{H}$.
In fact, we will keep track of two 3-manifolds $N$ and $N'$, with handle structures
$\mathcal{H}$ and $\mathcal{H}'$ respectively. The 3-manifold $N'$ will be a subset of $N$,
and each handle of $\mathcal{H}'$ will lie within a handle of $\mathcal{H}$.

Initially, $N' = N$ and $\mathcal{H}' = \mathcal{H}$. But when we insert the first wedge $W$, then
these two manifolds and handle structures become distinct. With the insertion of $W$,
$N$ and $\mathcal{H}$ remain unchanged, but we remove $W$ from $N'$. When $W$ is
a regular neighbourhood of the star of a vertex of $\partial S$, then there is a natural way of defining $\mathcal{H}'$.
We will also consider the surface $S' = S \cap N'$. Thus, when a wedge is inserted,
we view $S'$ as being shrunk but $S$ as remaining unchanged.

We say that two vertices of $S'$ have  \emph{$\mathcal{H}'$-parallel} stars 
if their stars are normally parallel with respect to $\mathcal{H}'$.

\begin{figure}[h]
\includegraphics[width=4in]{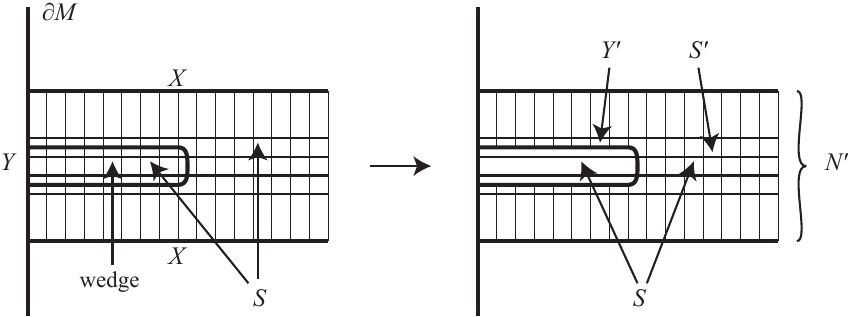}
\caption{The effect on $N'$ when a wedge is inserted}
\label{Fig:WedgeInsertionEnvelope}
\end{figure}

\subsection{Nested admissible envelopes}

We formalise the above arrangements with the following definition.
We say that two admissible envelopes $(N, X, Y, P \cap N, \mathcal{H}, S)$ and $(N', X', Y', P \cap N', \mathcal{H}', S')$ are \emph{nested} if
the following hold:
\begin{enumerate}
\item $N'$ is a subset of $N$, and is obtained from $N$ by removing a sequence of wedges, where
each wedge is transverse to the fibres of $N$ and is disjoint from $\partial_h N$.
\item Each fibre of $N'$ is a subset of a fibre of $N$.
\item $X' = X = \partial_h N$ and $P' \cap N'= P \cap N$.
\item $Y = \emptyset$.
\item $Y'$ is a surface properly embedded in $N$ and that is parallel to a subsurface of $\partial_v N$; the product region between $Y'$ and $\partial_v N$
is where the wedges are inserted.
\item $X' \cup Y'$ contains $\partial_h N'$.
\item $S \cap N' = S'$.
\item $S'$ is obtained from $S$ by wedge insertions and boundary-compressions along clean discs that intersect $\partial N'$ in arcs lying in $Y'$;
\item $S'$ is a union of tiles of $S$.
\item Each handle of $\mathcal{H}'$ lies in a handle of $\mathcal{H}$ with the same index and their product structures agree.
\end{enumerate}

We explained above why we keep track of the surface $X = X'$. The reason for the designated surface $Y'$ is that  it encodes
the location where the wedges are inserted when forming $N'$ from $N$.

The following result helps to clarify the relationship between $S$ and $S'$.

\begin{lemma}
Let $S$ be a compact incompressible boundary-incompressible surface properly embedded in $(M,P)$. Let $(N, X, Y, P \cap N, \mathcal{H}, S)$ and $(N', X', Y', P \cap N', \mathcal{H}', S')$
be nested admissible envelopes. Let $M'$ be the 3-manifold $M \cut (N \cut N')$, which is obtained from $M$ by removing wedges.
Then $S'$ is a copy of $S$ in $M$, plus possibly some clean properly embedded discs, in the sense that there is an isotopy
taking $M$ to $M'$ and taking $S$ plus possibly some clean discs to $S'$. In particular, $S'$ is
incompressible and boundary-incompressible in $M'$, and $S'$ is $\pi_1$-injective in $S$.
\end{lemma}

\begin{proof} By assumption, $S'$ is obtained from $S$ by wedge insertions and boundary-compressions along clean discs that intersect $\partial N'$ in arcs lying in $Y'$.
At each wedge insertion, $M'$ and $S'$ are obtained from what they were previously by an isotopy. When a boundary compression
is performed along a clean disc that intersects $\partial N'$ in an arc lying in $Y'$, then this disc is a boundary compression disc for $S'$
in $M'$. However, inductively, $S'$ is a copy of $S$ plus possibly some clean discs. Since $S$ is boundary-incompressible,
the arc of intersection between $S'$ and a clean compression disc is boundary parallel in $S'$. Hence, after the
boundary compression, $S'$ remains a copy of $S$ plus some clean discs.
\end{proof}

\subsection{The complexity of the handle structures}
\label{Subsec:ComplexityHS}

It will be important that, as we make the modifications to $\mathcal{H}$ and $\mathcal{H}'$, they remain of bounded `complexity'. We now make this
notion precise. It is defined using the vertex 0-handles.

We consider the intersection between the vertex 0-handles of $\mathcal{H}$ with the union of the 1-handles and 2-handles of $\mathcal{H}$.
This is a surface $\mathcal{F}_v$. We define $\mathcal{F}'_v$ similarly, but using the handle structure $\mathcal{H}'$.

For a component $F'$ of $\mathcal{F}'_v$, we define the following complexity: 
$$c(F') = \max\{ -\chi(F') + 2 |\partial X \cap F'|, 0 \}.$$
Here, $X'$ is the distinguished subsurface of $\partial N'$.
For a component $F$ of $\mathcal{F}_v$, we define the following: 
$$c_+(F) = -\chi(F) + 2 |\partial X \cap F| + 1.$$
We define $c(\mathcal{F}')$ to be the sum of $c(F')$ over each component $F'$ of $\mathcal{F}'$.
We define $c_+(\mathcal{F})$ to be the sum of $c_+(F)$ over each component $F$ of $\mathcal{F}$.
We will measure the complexity of $\mathcal{H}$ by $c_+(\mathcal{F})$, and will measure the complexity of
$\mathcal{H}'$ by $c(\mathcal{F}')$. We will ensure that, during any modification to $\mathcal{H}$
and $\mathcal{H}'$, these complexities do not increase.

We will give some justification in Section \ref{Sec:JustificationComplexity} for this definition of $c(\mathcal{F}')$.
The following also provides some explanation.

\begin{lemma}
\label{Lem:NumberArcs}
Let $F$ be a component of $\mathcal{F}$ or $\mathcal{F}'$. Then the maximal
number of disjoint non-parallel properly embedded essential arcs in $F - X$ is
at most $\max\{ 3c(F),1 \}$. This is less than $3c_+(F)$ unless $F$ is a disc disjoint from $X$. 
Here, an arc in $F - X$ is said to be essential if it cannot be homotoped,
keeping its endpoints fixed, to an arc in $\partial F - X$. Also, we say 
that two arcs $\alpha_1$ and $\alpha_2$ in $F$ are parallel if they separate off
a disc in $F$ containing exactly one copy of $\alpha_1$ and exactly one copy of
$\alpha_2$ and that is disjoint from $X$.
\end{lemma}

\begin{proof} Note first that we may assume that $F$ is not a disc
disjoint from $X$, because such a disc cannot contain any essential arcs.
Let $\alpha$ be a maximal set of arcs as in the lemma.
For a component $F_0$ of $F \backslash \backslash \alpha$, define its index $I(F_0)$ to be
$$-\chi(F_0) + 2|F_0 \cap X| + |F_0 \cap \partial \alpha|/4.$$
Then the sum of $I(F_0)$ over all components $F_0$ of $F \backslash \backslash \alpha$
is equal to $c(F)$. Unless $F$ is an annulus disjoint from $X$, each component $F_0$ has positive index, since
otherwise a component of $\alpha$ is inessential or two components
are parallel. In the case where $F$  is an annulus disjoint from $X$,
$\alpha$ consists of a single arc that cuts the annulus into a single region with zero index.
We also note that if $I(F_0) > 0$, then in fact $I(F_0) \geq 1/2$.
Hence, the number of components of $F \backslash \backslash \alpha$ is at most $2c(F)$
unless $F$ is an annulus disjoint from $X$.
Again using the additivity of index, we deduce that
$$|\alpha| = \sum_{F_0} |F_0 \cap \partial \alpha|/4 = c(F) + \sum_{F_0} \chi(F_0) - \sum_{F_0} 2 |F_0 \cap X| \leq c(F) + |F \backslash \backslash \alpha| \leq 3c(F).$$
Again, this formula does not hold when $F$ is an annulus disjoint from $X$, but in this case
$|\alpha| = 1$ and $c(F) = 0$. Thus, we obtain that $|\alpha| \leq \max\{ 3c(F),1 \}$ in general.
Note that this is clearly less than $3c_+(F)$ unless $F$ is a disc disjoint from $X$.
\end{proof}

\subsection{The complexity of the natural handle structure}
In this section, we compute an upper bound for $c(\mathcal{H})$ and $c_+(\mathcal{H})$ where
$\mathcal{H}$ is the natural handle structure arising from the initial alternative admissible surface $S$.

\begin{lemma}
\label{Lem:NumberStarTypesInitially}
Let $H = \{ S_1, \dots, S_k \}$ be a partial hierarchy with exterior $M$ and
boundary pattern $P$, and let $w$ be its binding weight. Let $\mathcal{K}$ be the
handle structure of $(M,P)$ described in Section \ref{Subsec:HierarchyExteriorHS}.
Let $S$ be a nearly embedded alternative admissible surface that is in normal form with respect to $\mathcal{K}$.
Then the number of tile types of $S$ is at most
$832w^2 (w - \sum_i \chi(S_i))$.
\end{lemma}

\begin{proof}
The possible normal disc types of $S$ are described in Lemmas \ref{Lem:DiscsRugby}, \ref{Lem:DiscsTurnover}, \ref{Lem:DiscsSaddleTetrahedron} and
\ref{Lem:DiscsExpandedTet}. Each tile must be incident to a saddle, and this must lie
in an expanded tetrahedron or an expanded turnover. We see that from Lemmas \ref{Lem:DiscsTurnover}
and  \ref{Lem:DiscsSaddleTetrahedron} that there are at most three such elementary disc type in each expanded
turnover and there are at most five such
normal elementary disc types in each expanded tetrahedron. The outermost tiles of any given type must intersect
some outermost disc in an expanded tetrahedron or expanded turnover. By Lemmas \ref{Lem:NumberTetrahedra} 
and \ref{Lem:NumberExpandedTurnover}, there
are at most $208w^2 (w - \sum_{i=1}^n \chi(S_i))$ such normal discs. By Lemma \ref{Lem:AtMostTilesInElemDisc},
each such disc can intersect at most 4 tiles. So, the number of tile types is at most
$832w^2 (w - \sum_{i=1}^n \chi(S_i))$.
\end{proof}

\begin{lemma}
\label{Lem:ComplexityNaturalHS}
The complexity of the initial natural handle structure $\mathcal{H}$ satisfies 
$$c(\mathcal{H}) \leq c_+(\mathcal{H}) \leq 11648w^2 (w - \sum_i \chi(S_i)).$$
\end{lemma}

\begin{proof}
The number of tile types of $S$ is at most $832w^2 (w - \sum_i \chi(S_i))$ by Lemma \ref{Lem:NumberStarTypesInitially}.
Each tile type gives rise to at most two 1-handles of $\calF_v$ (one at each vertex of the tile).
So, $|\mathcal{F}_v^1| \leq 1664w^2 (w - \sum_i \chi(S_i))$. Also, since each 0-handle of $\calH$ contains a component of $\calF^1$,
this also forms an upper bound on the number of 0-handles of $\calH$.

Each component of $X \cap \mathcal{F}_v$ intersects some 1-handle of $\calF_v$
and any given 1-handle intersects at most two such components of $X \cap \mathcal{F}_v$. 
Each component of $X \cap \mathcal{F}_v$ gives rise to at most two points of $\partial X \cap \mathcal{F}_v$.
So, $|\partial X \cap \calF_v| \leq 6656w^2 (w - \sum_i \chi(S_i))$.
So, 
$$c_+(\mathcal{H}) = c_+(\mathcal{F}_v) \leq |\mathcal{F}_v^1| + 2|X \cap \calF| \leq 11648 w^2 (w - \sum_i \chi(S_i)).$$
The inequality $c(\mathcal{H}) \leq c_+(\mathcal{H})$ is immediate.
\end{proof}

\subsection{Removing handles}

\begin{lemma}
\label{Lem:ComplexityRemovingHandles}
Let $\mathcal{H}_-'$ be a handle structure obtained from $\mathcal{H}'$
by removing some handles disjoint from $P$ and $X$. For $\mathcal{H}_-'$ to be a handle structure, we require that whenever a
handle of $\mathcal{H}'$ is removed, then so are all the incident handles with higher index. Then
$c(\mathcal{H}_-') \leq c(\mathcal{H}')$.
\end{lemma}

\begin{proof}
Let $F$ be a component of $\mathcal{F}'$.
When a 1-handle of $F$ disjoint from $X$ is removed, this decreases $-\chi(F)$ by $1$.
If this 1-handle is non-separating in $F$, the resulting component therefore has smaller complexity than $F$.
Suppose that the 1-handle is separating in $F$. Then the quantities $ -\chi(F) + 1$ and
$2 |\partial X \cap F|$ are shared out between the two
new components. So, the total complexity goes down by one, unless a component is
created that is a disc disjoint from $X$. But in this case, the other component
is a copy of $F$, and so complexity is unchanged. 

Suppose now that a 0-handle of $\mathcal{F}'$ is removed. The incident 1-handles have already been removed.
Therefore, this procedure is simply the removal of a component of $\mathcal{F}'$, which clearly does not increase
complexity.
\end{proof}

\subsection{Shrinking the handle structure to achieve crude normality}

We now give another modification to $N'$. 
The end result will be that $S'$ is crudely normal in $(\mathcal{H}', X')$, as defined in Section \ref{Subsec:DistinguishedSubsurfaces}.
The moves that we make are versions of the
usual normalisation moves and hence are guaranteed to terminate. More precisely, 
each move decreases the extended weight of $S'$, as defined in Section \ref{Sec:WeightHS}.

We saw in Lemma \ref{Lem:WeaklyCarriedNormal} that since $S'$ is weakly carried by $N'$, it satisfies (1) in the definition of
crude normality.

Suppose that, contrary to $(4')$ in the definition of crude normality, there is an arc $\alpha$ of intersection between $S'$ and a 0-handle of $\mathcal{F}'$
such one endpoint of $\alpha$ lies in a 1-handle of $\mathcal{F}'$ and the other endpoint lies
in $\partial N'$. Suppose that one of the arcs $\beta$ in the boundary of the 0-handle of $\mathcal{F}'$
joining $\partial \alpha$ is disjoint from $X$ and $P$ and from any other 1-handles of $\mathcal{F}'$.
Let us assume that the interior of $\beta$ is disjoint from $S'$.
Then $\alpha \cup \beta$ bounds a disc in the 0-handle of $\mathcal{F}'$, and the product of
this disc with the interval forms a ball $B_1$ in the incident 1-handle of $\mathcal{H}'$.
One of the endpoints of $\alpha$ lies in a 2-handle of $\mathcal{H}'$. This is divided
into balls by $S'$. Let $B_2$ be the ball that intersects the interior of $\beta$.
(See Figure \ref{Fig:ShrinkingHS}.)

\begin{figure}[h]
\includegraphics[width=3.5in]{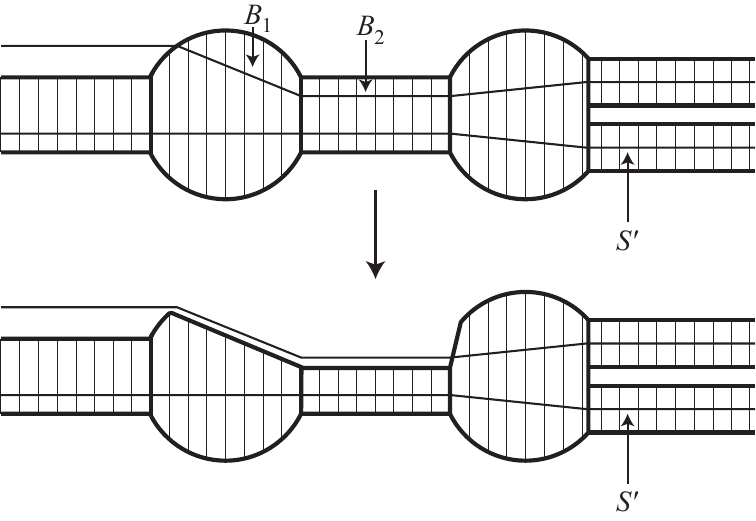}
\caption{Shrinking the handle structure}
\label{Fig:ShrinkingHS}
\end{figure}

The modification we make is to remove a regular neighbourhood of $B_1 \cup B_2$ from $N'$.
The intersection between $B_1 \cup B_2$ and $S'$ is a disc that intersects $\partial S'$ in a single arc.
The modification removes this disc from $S'$. This leaves $S'$ unchanged up to isotopy in $S$.

Alternatively, suppose that there is an arc $\alpha$ of intersection between $S'$ and a 0-handle of $\mathcal{F}'$
that separates off a disc from $\mathcal{F}'$ that is disjoint from the 1-handles of $\mathcal{F}'$
and from $X$, contrary to ($3'$) in the definition of crude normality. 
Let us assume that the interior of this disc is disjoint from $S'$.
The product of this disc with the interval is a ball $B_1$ in
the incident handle of $\mathcal{H}'$. The modification we can then perform is to remove
a regular neighbourhood of this ball from $N'$. This has the effect of 
boundary-compressing $S$ along a clean boundary-compression disc $D$. 
The intersection between $D$ and $\partial_h N'$ misses $X'$ and hence lies in $Y'$.
The intersection between $D$ and $\partial_v N'$ lies in those parts of $\partial_v N'$
intersecting $\partial S'$, which again lie in $Y'$. Hence, $D \cap \partial N'$ lies in $Y'$,
as required in the definition of nested admissible envelopes.

These modifications may create handles of $\mathcal{H}'$ that do not intersect $S'$. In this case, we remove
these handles. By Lemma \ref{Lem:ComplexityRemovingHandles}, this does not increase $c(\mathcal{H}')$.

Note that these are solely modifications to $N'$, $\mathcal{H}'$ and $S'$ and not $N$, $\mathcal{H}$ and $S$.

\begin{lemma}
\label{Lem:HandlesBothSides}
Suppose that $S'$ is crudely normal in $\mathcal{H}'$. Then, for each 0-handle of $\mathcal{F}'$
lying in a 0-handle $D^2 \times [-1,1]$ of $\mathcal{H}'$, one of the following must hold:
\begin{enumerate}
\item it intersects both $X \cap (D^2 \times \{1\})$ and $X \cap (D^2 \times \{-1\})$, or
\item it has 1-handles of $\mathcal{F}'$ on both sides of it.
\end{enumerate}
\end{lemma}

\begin{proof} 
Suppose that, on the contrary, there is a 0-handle $D$ of $\mathcal{F}'$ that is disjoint from 
$X \cap (D^2 \times \{1\})$ or $X \cap (D^2 \times \{-1\})$ and that has
1-handles emanating from at most one side of it. We are identifying $D$
with $[-1,1] \times I$ with fibres of the form $\ast \times I$. By assumption, the 1-handles of $\mathcal{F}'$
are disjoint from some component of $\{ -1, 1 \} \times I$, say $\{ -1 \} \times I$. They are also disjoint from 
$[-1,1] \times \partial I$. The intersection between
$S'$ and this 0-handle of $\mathcal{F}'$ is a non-empty collection of arcs. Each arc of $S' \cap D$ intersects
at most one 1-handle of $\mathcal{F}'$ since it is transverse to the fibres. Consider an outermost such arc.
It divides $D$ into two discs, one of which, $D'$ say, has interior disjoint from $S'$. We may choose the arc so that
$D'$ also misses $X$. The boundary of $D'$
cannot intersect an entire component of $D \cap (\mathcal{F}')^1$, since $S'$ runs over every 1-handle
of $\mathcal{F}'$. Thus, we deduce that this arc violates ($3'$) or ($4'$) in the definition of crude
normality.
\end{proof}

\subsection{Reduced handle structure}

We say that $\mathcal{H}'$ is \emph{reduced} if both the following hold:
\begin{enumerate}
\item no 0-handle of $\mathcal{F}'$
is disjoint from $\partial X$ and intersects the 1-handles of $\mathcal{F}'$ in at most one arc, and
\item for each 0-handle $H_0'$ of $\mathcal{H}'$,
$\partial H_0' \cap (\mathcal{F}' \cup X \cup P)$ is connected.
\end{enumerate}

We now show how to modify the handle structure $\mathcal{H}'$ to make it reduced without increasing its complexity $c(\mathcal{F}')$.

Suppose that there is a 0-handle of $\mathcal{F}'$ that
is disjoint from $X$ and intersects the 1-handles of $\mathcal{F}'$ in at most one arc.
Let $H'_0$ be the 0-handle of $\mathcal{H}'$ containing this 0-handle.
Then any component of $S' \cap H'_0$ violates (1), ($3'$) or ($4'$) in the definition of crude normality.
Hence, $S'$ must be disjoint from this 0-handle. But then the incident 1-handle of $\mathcal{H}'$,
and possibly an incident 2-handle, would have been removed.

Thus we may assume that $\mathcal{F}'$ satisfies (1) in the definition of reduced.
Suppose that for some 0-handle $H'_0$ of $\mathcal{H}'$, $\partial H_0' \cap (\mathcal{F}' \cup X \cup P)$
is disconnected. Then there is a disc $D$ properly embedded in $H_0'$ that separates the components of
$\partial H_0' \cap (\mathcal{F}' \cup X \cup P)$. We may assume that $D$
intersects each disc of $S' \cap H_0'$ in a (possibly empty) collection of arcs. The modification we
perform is to cut $H'_0$ along $D$. This has no effect on the complexity $c(\mathcal{F}')$.
But it does increase the number of 0-handles of $\mathcal{H}'$. Since each 0-handle of
$\mathcal{H}'$ contains a component of $\mathcal{F}'$, this process is guaranteed to terminate.
The effect on $S'$ is to boundary-compress it along clean boundary-compression discs.
Hence, $S'$ remains $\pi_1$-injective in $S$,
and $S \cut S'$ still lies in a collar on $\partial S$.

We need to justify why this modification maintains $N'$ as a thickening of a generalised branched surface weakly carrying $S'$.
Recall that $H'_0$ is identified with $D^2 \times [-1,1]$. The discs $D^2 \times \{ -1,1 \}$ lie in $\partial_hN'$, and so are disjoint from $\calF'$. 
Each component of $D^2 \times \{ -1,1 \}$ either lies in $X$ or is disjoint from $X$.
We may also assume that $\mathcal{F}' \cap (D^2 \times \{-1,1\}) \subset X$.
Since $\partial D$ is disjoint from $X$, we may isotope $\partial D$ so that it lies in $\partial D^2 \times [-1,1]$.
We claim that $\partial D$ is an essential curve in $\partial D^2 \times [-1,1]$. 
If not, some component of $\calF'$ lies within in a disc in
$\partial D^2 \times [-1,1]$. But this implies that some 0-handle of $\calF'$ has 1-handles of $\calF'$ emanating
from only one side of it and is disjoint from $X $, contradicting Lemma \ref{Lem:HandlesBothSides}.
Hence, $\partial D$ is essential in $\partial D^2 \times [-1,1]$. We may then isotope it so that it is everywhere
transverse to the $I$-fibres in $\partial D^2 \times [-1,1]$. For the only obstruction to doing this would be a
0-handle of $\calF'$ with 1-handles emanating from only one side of it, and where this component of $\calF'$
prevented $\partial D$ being pulled tight. So once $\partial D$ is transverse to the $I$-fibres, we can
make all of $D$ transverse to the $I$-fibres. Hence, cutting the 0-handle along $D$ creates two 0-handles
that are still fibred with intervals and hence copies of $D^2 \times [-1,1]$. Therefore, cutting $N'$ and $S'$ along $D'$
still results in a 3-manifold that is a thickening of a generalised branched surface weakly carrying the new $S'$.

\begin{lemma}
\label{Lem:NoDiscs}
Suppose that $\mathcal{H}'$ is reduced. Then $\mathcal{F}'$ contains
no components that are discs disjoint from $X$.
\end{lemma}

\begin{proof}
Let $F'$ be such a component. The handle structure on $F'$ is a thickened graph.
This graph is a tree, by the hypothesis that $F'$ is a disc. Hence, it is either an isolated
vertex or contains a leaf. This contradicts the hypothesis that $\mathcal{H}'$ is reduced.
\end{proof}

Note that the disc components $F'$ of $\mathcal{F}'$ disjoint from $X$ were a special case in
the definition of complexity $c(\mathcal{F}')$. The first part of the formula for $c(F')$ gives
$-1$, but this is upgraded to $0$ by the fact that one takes a maximum.
It is therefore convenient to be able to assume that such components
do not arise.

We now give some consequences of crude normality.

\begin{lemma}
\label{Lem:EssentialIntF}
Suppose that $\calH'$ is reduced and that $S'$ is crudely normal. Then no component of
$S' \cap \calF'$ is a simple closed curve that bounds a disc in $\calF'$ or an arc parallel
to an in arc in $\partial \calF' - X$.
\end{lemma}

\begin{proof}
Consider first the case where a component of $S' \cap \calF'$ is a simple closed curve that bounds a disc in $\calF'$.
An innermost such curve bounds a disc component of $\calF' \cut S'$. This inherits a handle structure that is a
thickened tree. This is either a single 0-handle or it has a leaf. When it is a single 0-handle, then $S'$ fails to be
standard, since its intersection with the corresponding 1-handle of $\calH'$ has an annular component.
When the disc component of $\calF' \cut S'$ has a leaf, this implies that the relevant component of $S' \cap (\calH')^0$ runs over a 1-handle of $\calF'$
more than once, contradicting crude normality.

Now suppose that a component of $S' \cap \calF'$ is an arc parallel to an arc in $\partial \calF' - X$. An
outermost such arc separates off a disc $D$ in $\calF'$. Again this inherits a handle structure that is a thickened tree.
If this is a single 0-handle, then this contradicts ($3'$) in the definition of crude normality.
So suppose that this handle structure has a leaf. This is a 0-handle of $D$ that
is incident to a single 1-handle of $D$. This 0-handle intersects $\partial D$ in an arc $\alpha$.
If $\alpha$ lies wholly in $\partial \calF'$, then this implies that $\calF'$ has a leaf, contradicting
the assumption that $\calH'$ is reduced. If both the endpoints of $\alpha$ lie in $S'$, then this
contradicts (1) in the definition of crude normality. If one endpoint of $\alpha$ lies in $S'$
and the other lies in $\partial \calF'$, then this contradicts ($4'$) in the definition of crude
normality. The remaining possibility is that $\alpha$ contains $\partial D \cap S'$ in its
interior. But then $D$ has at least one other leaf, and the intersection between that
leaf and $\partial D$ lies wholly in $\partial \calF'$, and again we have a contradiction.
\end{proof}

\begin{lemma}
\label{Lem:ParallelInF}
Suppose that $\calH'$ is reduced and that $S'$ is crudely normal. Suppose that
two components of $S' \cap \calF'$ are topologically parallel in $\calF' - X$. Then they are
normally parallel.
\end{lemma}

\begin{proof}
Suppose that two components of $S' \cap \calF'$ are topologically parallel in $\mathcal{F}' - \partial X$, 
and that these are arcs $\alpha_1$ and $\alpha_2$. Then these
two arcs separate off a disc $D$ from $\mathcal{F}'$. This disc inherits a handle
structure. One can define the index of a 0-handle of $D$ to be half the number of
intersection components with the 1-handles of $D$ plus a quarter of the number of
intersection points with $\partial \alpha_1 \cup \partial \alpha_2$ minus 1.
Then the sum of the indices of the 0-handles of $D$ is equal to 0.
Each 0-handle of $D$ has non-negative index, since otherwise crude
normality is violated. So, every 0-handle of $D$ has zero index. Thus, $D$ consists
of a line of 0-handles and 1-handles. Suppose that there is at
least one 1-handle, for otherwise $\alpha_1$ and $\alpha_2$ are normally parallel.
Then the end 0-handles
each intersect $\partial \alpha_1 \cup \partial \alpha_2$ in two points.
There are then two possibilities: either $\alpha_1 \cap D$ and $\alpha_2 \cap D$
each lie in a single 0-handle of $D$, or $\alpha_1$ and $\alpha_2$ are normally parallel.
We must rule out the former possibility. Consider the 1-handle of $D$ incident to
the 0-handle containing $\alpha_1$. This must have non-empty intersection with $S'$.
Hence, at least one arc of $S'$ runs from this 1-handle of $D$ into the 0-handle.
But then crude normality is violated.

A similar but simpler argument applies when $\alpha_1$ and $\alpha_2$ are simple
closed curves of $S' \cap \calF'$ that are topologically parallel. They cobound an
annulus $A$ in $\calF'$. Again each 0-handle of $A$ has index zero. So $A$ is formed
of a circle of 0-handles and 1-handles. Therefore, $\alpha_1$ and $\alpha_2$ are
normally parallel.
\end{proof}

\subsection{Low-valence boundary vertices (revisited)}

Suppose that $\mathcal{H}'$ is reduced and that $S'$ is crudely normal. 

One can take advantage of parallelism in the case of wedge removal. Suppose that
$s^1, \dots, s^n$ are low-valence boundary vertices of $S'$. Let $D_i$ be the star of $s^i$. Suppose
that these stars are parallel in $\mathcal{H}'$. Then there is an embedding of $D \times [1,n]$ in
$N'$, where $D_i = D \times \{ i \}$, and where the remainder of $S'$ is disjoint from $D \times [1,n]$.
Then we make take a regular neighbourhood of $D \times [1,n]$ to be the wedge $W$ that is
removed from $N'$. 

Thus, with a single wedge removal, the binding weight of $S'$ is reduced by $n$. 

We can see from Figure \ref{Fig:LowValBoundaryVertices} that each of these stars contains at most one
interior vertex of $S'$ in its boundary. The effect on $N'$ and $\mathcal{H}'$ is as follows:
\begin{enumerate}
\item The component of $\mathcal{F}'$ containing the vertices $s^1, \dots, s^n$ is sliced along its intersection with $W$,
which is a regular neighbourhood of a properly embedded arc $\gamma$ in $\mathcal{F}'$.
This increases $\chi(\mathcal{F}')$ by 1. We will show below that this does not create any disc components of $\mathcal{F}'$
that are disjoint from $X$. Hence, $c(\mathcal{F}')$ decreases by 1.
\item If any vertices of the star other than $s$ lie in $\partial S'$, then the effect on $\mathcal{F}'$
at these points is cut along an arc, with precisely one boundary point in $\partial \mathcal{F}'$.
Thus, the topology of $\mathcal{F}'$ remains unchanged, and hence so does its complexity.
\item If there is an interior vertex in the star, then the component of $\mathcal{F}'$
containing this vertex is cut open along an arc. This arc lies within the interior of $\mathcal{F}'$,
and so the effect on $c(\mathcal{F}')$ is to increase it by 1.
\end{enumerate}
So, the overall effect is leave $c(\mathcal{H}')$ unchanged. But we need to show
that the procedure in (1), where $\mathcal{F}'$ is sliced along an arc $\gamma$, 
does not create a disc component disjoint from $X$. The arc $\gamma$ is a component
of $S' \cap \calF'$ and so is not parallel to a component of $\partial \calF' - X$, by Lemma
\ref{Lem:EssentialIntF}. Moreover, it is not an essential arc in an annular component of
$\calF'$ because its endpoints are joined by an arc of $S' \cap (\partial (\calH')^0 \cut (\calF' \cup X \cup P))$.
This establishes the claim.

Note also that, when we insert a wedge, we do not change $N$ and $\mathcal{H}$
and so their complexity remains unchanged.

\subsection{Cut vertices}
\label{Sec:CutVertices}

The operation of inserting a wedge shrinks $S'$ by a removing discs incident to $\partial S'$. In doing so,
$\partial S'$ may, at some stage, bump into itself. At that point, a \emph{cut vertex} may be created. This is
a vertex in $\partial S'$ at which $S'$ is not a surface with boundary. Instead, a neighbourhood of that
point is a cone on at least two closed intervals.

We can cut $S'$ at these points, giving a surface $S'_{\mathrm{cut}}$. When we do this, we do not forget the remnants
of the cut vertices in $S'_{\mathrm{cut}}$. The reason is that if a cut-vertex of $S'_{\mathrm{cut}}$ is a low-valence vertex, then
we cannot insert a wedge along it.

The \emph{cut-adjusted Euler characteristic} of $S'_{\mathrm{cut}}$, denoted $\chi_c(S'_{\mathrm{cut}})$, is defined to be $\chi(S'_{\mathrm{cut}})$ minus
half its number of cut vertices. Since $S'_{\mathrm{cut}}$ is obtained from $S'$ by cutting along some arcs, and because
each such arc creates two cut vertices, we deduce that $\chi_c(S'_{\mathrm{cut}}) = \chi_c(S')$.

\subsection{Justification for the measure of complexity}
\label{Sec:JustificationComplexity}

We can now justify the measure of complexity $c(\mathcal{F}')$
that we chose for $\mathcal{F}'$. In some respects, it is an awkward
quantity, because certain components of $\mathcal{F}'$ have zero complexity.
Hence, if one just has a bound on $c(\mathcal{F}')$, then one cannot control
the number of these components. These components come in two types:
discs and annuli disjoint from $X$. The discs
disjoint from $X$ can be dealt with using Lemma \ref{Lem:NoDiscs}, and we therefore
may assume that $\mathcal{H}'$ is reduced. However, annuli disjoint from
$X$ are harder to control. The following example explains why.

Suppose that, initially, we have two components of $\mathcal{F}'$, an annulus disjoint
from $X \cup P$ and a pair of pants disjoint from $X \cup P$. (See Figure \ref{Fig:WedgeAnnuli}.)
Suppose that a wedge is inserted, thereby slicing the pair of pants along an arc, into
two annuli. Suppose also that the star contains an interior vertex and that
the relevant component of $\mathcal{F}'$ surrounding this interior vertex is
the annulus disjoint from $X \cup P$. Then a slit is inserted into this annulus,
forming a pair of pants. Thus, the wedge insertion has created a pair of
pants and two annuli from a pair of pants and a single annulus.
The conclusion that one must draw is that, whatever notion of complexity
we have for $\mathcal{F}'$, annuli disjoint from $X \cup P$
must have zero complexity if wedge insertion is not going
to increase complexity.

\begin{figure}[h]
\includegraphics[width=5in]{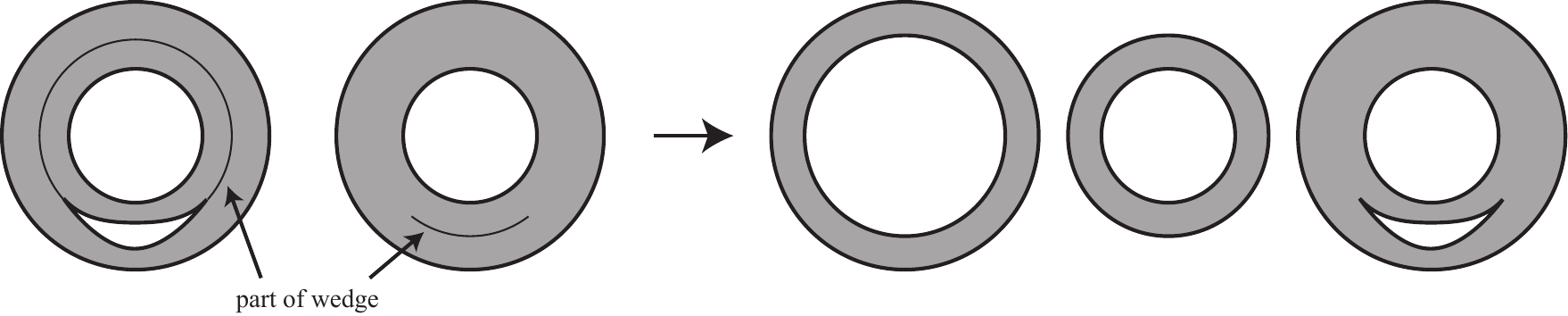}
\caption{An example of the effect of wedge insertion on $\mathcal{F}'$}
\label{Fig:WedgeAnnuli}
\end{figure}

\subsection{Approximate parallelism}

Consider two normal simple closed curves $\alpha$ and $\alpha'$ in $\mathcal{F}'$.
We say that they are \emph{approximately parallel} if
\begin{enumerate}
\item they have the same $\mathcal{H}$-type; and
\item when $\mathcal{F}'$ is decomposed along them forming a surface $F''$, the resulting pieces
of $F''$ lying in the annular region between them are just annuli in which every 0-handle of $F''$
has two arcs of intersection with the 1-handles of $F''$.
\end{enumerate}

We say that two disjoint properly embedded discs in some 0-handle
of $\mathcal{H}'$ are \emph{approximately parallel} if their boundary curves lie in $\mathcal{F}'$
and are approximately parallel. Similarly, we say that 
interior vertices of $S'$ are \emph{approximately parallel} if the discs of intersection
between $S'$ and the 0-handles of $\mathcal{H}'$ that contain these vertices
are approximately parallel.

When two curves, discs or vertices are approximately parallel, we also say that they
have the same \emph{approximate type}.

\subsection{Interior 2-valent vertices (revisited)}
\label{Subsec:InteriorTwoValRevisited}

We now examine the `parallel' version of the modification in Section 8.1. 
Suppose that $\mathcal{H}'$ is reduced.

Consider again a 2-valent interior vertex $s$ of $S'$. We will now describe a modification to
$S$ and $S'$ which removes this vertex and all other vertices of $S$ with approximately parallel stars.

Let $s^1, \dots, s^n$ be these vertices, and for each $i$,
let $s^i_1$ and $s^i_2$ be the two vertices in its star. Then these vertices occur along
$S^1_\phi$ in the order
$$\dots, s^1_1, \dots, s^n_1, s^n, \dots, s^1, s^1_2, \dots s^n_2, \dots$$
Let $i \colon S^1_\phi \rightarrow S^1_\phi$ map the interval $(s_1^n, s^n)$ linearly
onto the interval $(s^1, s^1_2)$, map the interval $(s^1, s^1_2)$ linearly
onto the interval $(s_1^n, s^n)$, and fix everything else.
\begin{enumerate}
\item Remove the stars of $s^1, \dots, s^n$ from $S$.
\item For any leaf of the foliation or arc of the link with an endpoint in $\phi \in (s_1^n, s^n) \cup (s^1, s^1_2)$, replace it with a leaf
ending at $i(\phi)$.
\item For the leaves of the foliation that ended at $s_1^i$ or $s_2^i$, but which were not part of
the star of $s^i$, replace them with a leaf ending at $s^i$.
\end{enumerate}
Thus, a single generalised exchange move is performed, but the binding weight of $S$ is reduced by $2n$.
Furthermore, if $m$ is the number of vertices among $s^1, \dots, s^n$ that
lie in $S'$, then the binding weight of $S'$ is reduced by $2m$.

Adjacent to $s$ are two tiles, $T_1$ and $T_2$, say,
each of which lies in a 2-handle of $\mathcal{H}$. Any other vertex with
star that is $\mathcal{H}$-parallel to that of $s$ contains parallel copies of $T_1$
and $T_2$. Consider a sequence of these stars, moving away from
$s$ in some transverse direction. At some stage, we reach the last
star parallel to $s$. Beyond this may be a parallel copy of $T_1$
or a parallel copy of $T_2$, but not both. For the sake of being definite,
suppose that there is a parallel copy of $T_1$. Then the first stage of the
procedure is to slice the 2-handle of $\mathcal{H}$ containing $T_1$ into two handles, thereby
separating this parallel copy of $T_1$ that is not part of a star parallel
to that of $s$ from the stars that are parallel to that of $s$. We also
slice the two 1-handles that are incident to this 2-handle and that
contain the separatrices incident to $s$. This procedure is shown in
Figure \ref{Fig:SlidingHandleTwoVal}.

\begin{figure}[h]
\includegraphics[width=4in]{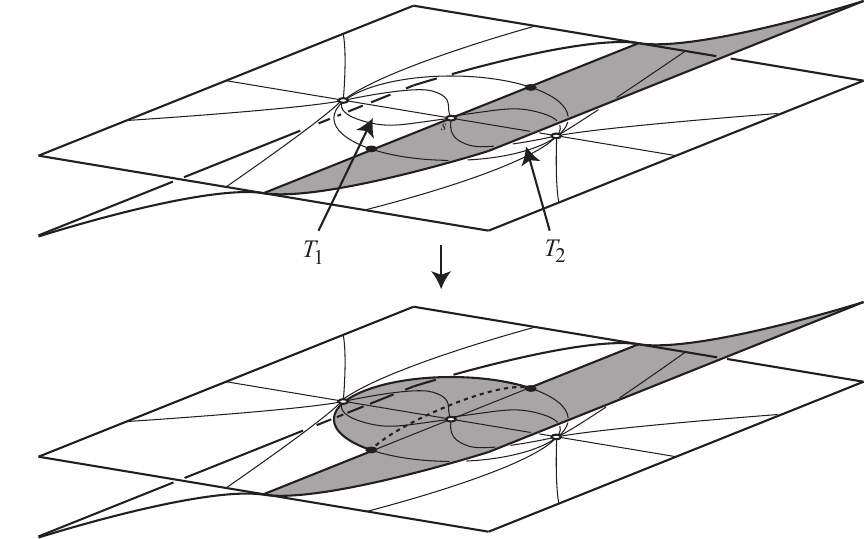}
\caption{The 2-handle of $\calH$ containing $T_1$ is sliced into two 2-handles}
\label{Fig:SlidingHandleTwoVal}
\end{figure}

We perform a similar procedure, but going away from $s$ in the other
transverse direction. Again for the sake of being definite, suppose
that in that direction, there is a parallel copy of $T_2$ but not $T_1$.
The effect of these two procedures on the
0-handle containing $s$ is to slice it into three 0-handles. The way
that this affects $\mathcal{F}$ in these 0-handles is shown in Figure \ref{Fig:EffectFTwoVal}.

\begin{figure}[h]
\includegraphics[width=4in]{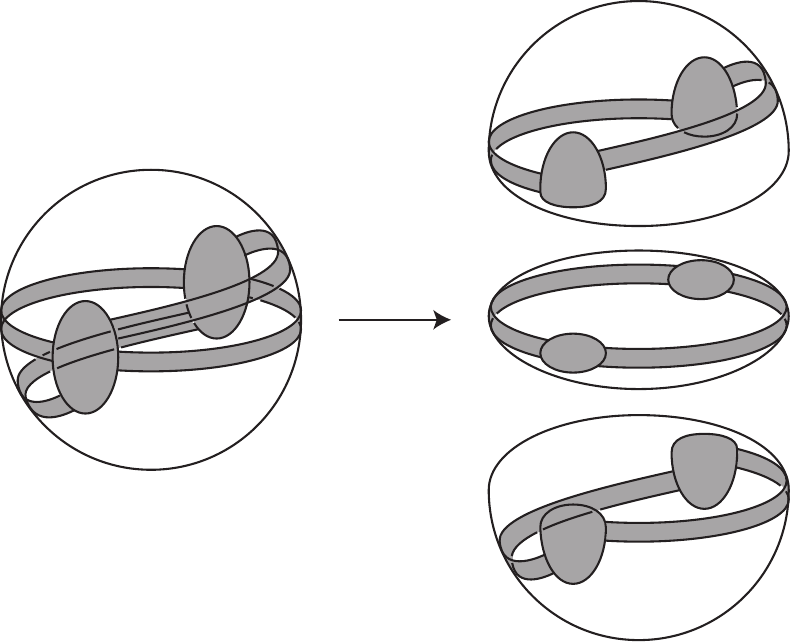}
\caption{The effect on $\mathcal{F}$ near $s$}
\label{Fig:EffectFTwoVal}
\end{figure}

Note that the original component of $\mathcal{F}_v$ near $s$ is sliced along two arcs,
thereby forming three components. This decreases $-\chi(\mathcal{F}_v)$ by 2,
but it increases the number of components by 2, and so the contribution to
$c_+(\mathcal{F}_v)$ from $-\chi(\calF_v)$ is unchanged.

However, this is not the only effect on $\mathcal{F}_v$. Let $s_1$ and $s_2$ be
the vertices in the star of $s$. These lie in vertex 0-handles which are also
affected by this modification. Each 0-handle intersects the 2-handle containing
$T_1$ or $T_2$ in a 1-handle of $\mathcal{F}_v$. This 1-handle of $\mathcal{F}_v$
is replaced by two 1-handles. When this performed in the two vertex 0-handles
of $\mathcal{H}$, this increases $-\chi(\mathcal{F}_v)$ by
2, but it leaves the number of components of $\mathcal{F}_v$ unchanged. So,
the contribution to $c_+(\mathcal{F}_v)$ from $-\chi(\calF_v)$ 
has increased by 2.

We must also analyse the contribution to $c_+(\calF_v)$ arising from $\calF_v \cap \partial X$.
When the 2-handle of $\calH$ containing $T_1$ is sliced in two,
this enlarges the horizontal boundary of $N$. If any components of $\partial_h N$
containing $X$ are enlarged, then we enlarge $X$ also. This moves $\partial X$
but it does not increase the number of vertex 0-handles that it runs over.
So the contribution to $c_+(\mathcal{F}_v)$ from $\calF_v \cap \partial X$
remains unchanged.

We now perform the procedure described above, which removes
the star of $s$ and all parallel stars. This new surface is carried by
a new handle structure. One removes the vertex handle containing $s$,
and all the incident handles. Thus, the annular component of $\mathcal{F}$
in this 0-handle is removed, thereby decreasing the complexity $c_+(\calF)$ by 1.

The vertex 0-handles containing $v_1$ and $v_2$ are amalgamated to
form a single 0-handle. The way that this affects $\mathcal{F}$ is as follows.
The two components of $\mathcal{F}$ in these 0-handles are incident to
a 2-handle that was removed because it contained $T_1$ or $T_2$.
These two components of $\mathcal{F}$ are then combined to form a
single component. The effect here is leave $-\chi(\mathcal{F})$ unchanged
but to reduce the number of components by 1. The number of points
of $\partial X \cap \calF_v$ remains unchanged. So the complexity goes down by 1.

We see that the net effect of this process is to leave the complexity
$c_+(\mathcal{F})$ unchanged.

We also observe that the complexity of $\mathcal{F}'$ is unchanged, as follows.
The modifications to $\partial X \cap \calF'_v$ are exactly as described above.
Hence, the only
thing that concerns us is the effect on Euler characteristic and whether
the modifications create or remove any components of $\mathcal{F}'_v$ that
are discs disjoint from $X$.

The first stage is to increase $-\chi(\mathcal{F}'_v)$ by 2, but it then decreases $-\chi(\mathcal{F}'_v)$ by $2$.
At no stage is a disc component of $\mathcal{F}'_v$ disjoint from $X$ created. At the first stage,
this is because each component of $\calF'_v$ is sliced along at most one arc that is parallel to a component
of $S' \cap \calF'_v$ with endpoints on the same component of $\partial \calF'_v$. This does not
create a disc component of $\mathcal{F}'_v$ disjoint from $X$ by Lemma \ref{Lem:EssentialIntF}.
In the second stage, a non-separating 1-handle of $\calF'_v$ is removed from two components of $\calF'_v$,
and then these are combined into a single component of $\calF'_v$, which again cannot be a disc
disjoint from $X$.

In the above discussion, we assumed that, in one transverse direction from the tiles parallel
to $s$, there is a copy of $T_1$ but not $T_2$, and in the other transverse direction, there is a
copy of $T_2$ but not $T_1$. However, this is not the only possible arrangement. For example,
it might be the case that in both transverse directions, there is a copy of $T_1$ but not $T_2$.
Alternatively, in one or both transverse directions, there might be neither a copy of $T_1$ nor
a copy of $T_2$. In the former situation, the 0-handle is still sliced into three 0-handles. The
above argument still applies to give that the complexities of $\mathcal{F}_v$ and $\mathcal{F}'_v$ are
unchanged. In the latter situation, where in one or both transverse directions, there is neither
a copy of $T_1$ nor a copy of $T_2$, then the 0-handle is sliced into two 0-handles or is not sliced at all.
Again it is easy to deduce that the complexities of $\mathcal{F}_v$ and $\mathcal{F}'_v$ are
unchanged.

\subsection{Interior 3-valent vertices (revisited)}

Consider a 3-valent interior vertex $s$. It is contained within an elementary normal disc $D$ in a 0-handle
of $\mathcal{H}$. Its star contains three tiles, and hence $D$ runs through three 1-handles
and three 0-handles of $\mathcal{F}$. Coming out of these three 0-handles are three 1-handles
of $\mathcal{H}$, which end at other vertices $s_2, s_3, s_4$. As in Section \ref{Subsec:ThreeValInterior},
suppose these vertices are arranged around $S^1_\phi$ in the order
$s, s_2, s_3, s_4$. Let $x_1$ and $x_2$ be the generalised saddles that
are defined in Section \ref{Subsec:ThreeValInterior} and that have the same orientation, and let $t_1$ and $t_2$ be their
$\theta$-values.

We wish to perform the procedure described in
Section \ref{Subsec:ThreeValInterior}, not just to the star of $s$ but to all the vertices with stars that
are parallel (with respect to $\mathcal{H}$). The procedure in Section \ref{Subsec:ThreeValInterior}
can be viewed as taking place in three steps:
\begin{enumerate}
\item Moving events occurring in the interval $(t_1, t_2)$ into the
future or past.
\item Collapsing the tile with boundary $s, x_1, s_3, x_2$ so that
$x_1$ and $x_2$ are amalgamated to form a single generalised saddle.
\item The process given in Section \ref{Subsec:TwoValInterior} for dealing with 2-valent
interior vertices. 
\end{enumerate}

Step (1) can be achieved by an isotopy of $\mathcal{H}$ and $\mathcal{H}'$, and so this does not affect
the complexities of $\mathcal{H}$  and $\mathcal{H}'$. We have already seen in Section \ref{Subsec:InteriorTwoValRevisited} that the
complexities of $\mathcal{H}$ and $\mathcal{H}'$ are unchanged when dealing with parallel copies of a
2-valent interior vertex. So, it suffices to show that the complexities of $\mathcal{H}$ and $\mathcal{H}'$
are unchanged by Step (2).

In Step (2), we collapse the 2-handle containing the tile. The 0-handles containing
the generalised saddles $x_1$ and $x_2$ are amalgamated into a single 0-handle. These are not
vertex 0-handles, and so do not contribute to the complexity of $\mathcal{H}$.
The 1-handles of $\mathcal{H}$ running from $s$ to $x_1$ and from $s$ to $x_2$
are amalgamated into a single 1-handle. The 1-handles running out of $s_3$
are similarly amalgamated. The effect on the two relevant components of 
$\mathcal{F}_v$ is to collapse, in each component, a 1-handle of $\mathcal{F}_v$. This leaves
$\mathcal{F}_v$ unchanged up to homeomorphism preserving $X \cap \mathcal{F}_v$, and so does not change its
complexity $c_+(\calF_v)$. It also leaves $\mathcal{F}'_v$ unchanged up to homeomorphism preserving $X \cap \mathcal{F}'_v$, and so does not change its
complexity $c(\calF'_v)$.

\subsection{The number of star types}

\begin{proposition}
\label{Prop:StarTypeNonCut}
Suppose that $\mathcal{H}'$ is reduced and that
$S'$ is crudely normal.
Then the number of star types of non-cut vertices in $\partial S'$ is at most 
$3c(\mathcal{H}')$.
\end{proposition}

\begin{proof}
Each vertex $s$ of $S'$ lies within a vertex 0-handle $H_0$ of $\mathcal{H}'$. The intersection between
$S'$ and $H_0$ is a collection of elementary normal discs. Let $E$ be the elementary normal disc
containing $s$. This disc $E$
intersects $\mathcal{F}'$ in a single arc or simple closed curve, because $s$ is a non-cut vertex.
For the purposes of this proof, we call such a disc \emph{relevant}. We also call
the intersection of the disc with $\mathcal{F}'$ a \emph{relevant} arc.

We now show that a relevant disc $E$ in $H_0$ is determined (up to an isotopy of $H_0$
preserving $X$, $P$ and the handles) by its relevant arc $E \cap \calF'$. 
First observe that, because
$\calH'$ is reduced, $H_0 \cap (\calF' \cup X \cup P)$ is connected, and hence
its complement is a collection of discs. The arc $E \cut \calF'$ intersects one or two
such discs, depending on whether it intersects $P \cap H_0$. (Note that there
is at most one arc of $P \cap H_0$.) Hence, the arc $\partial E \cut \calF'$ is determined
by its endpoints, as required.

It therefore suffices to bound the number of such relevant arcs.
By Lemma \ref{Lem:EssentialIntF}, $S' \cap \mathcal{F}'$ is essential in $\calF'$, in the 
sense that no simple closed curve bounds a disc in $\calF'$ or no arc is parallel to
an arc in $\partial \calF' - X$. By Lemma \ref{Lem:ParallelInF}, whenever two
arcs are topologically parallel in $\calF' - X$, they are normally parallel. 

So, when we consider a collection of disjoint relevant discs in $H_0$, no two of
which are $\mathcal{H}'$-parallel, we deduce that their intersection with
$\mathcal{F}'$ forms a collection of properly embedded arcs, no
two of which are topologically parallel. Each arc is essential in the
complement of $X$.
So, the maximal number of relevant non-parallel discs in $H_0$ is at most the number of disjoint non-parallel
essential arcs in $(\mathcal{F}' \cap H_0) - X$. This is at most $3 \max \{ c(\mathcal{F}' \cap H_0), 1\}$, by Lemma \ref{Lem:NumberArcs}.
However, in fact, the number of such arcs is bounded above by $3 c(\mathcal{F}' \cap H_0)$, for the following reason.
When the maximum is 1, then $\calF' \cap H_0$ is a disc or annulus disjoint from $X$.
We saw in Lemma \ref{Lem:NoDiscs} that when $\calH'$ is reduced, $\calF'$ has no disc components
disjoint from $X$. Any annular component of $\calF'$ disjoint from $X$ is composed
of a string of 0-handles and 1-handles joined together in a circular fashion. Hence,
any normal arc in such an annulus must join distinct components of $\partial H_0 \cut \calF'$.
Therefore such an arc is not relevant, as it comes from a cut vertex.
\end{proof}

\begin{proposition}
\label{Prop:NumberWeakTileTypes}
Suppose that $\mathcal{H}'$ is reduced and that
$S'$ is crudely normal.
Then the number of approximate types of interior vertices in $S'$ is at most 
$2c_+(\mathcal{H}) + c(\mathcal{H}')$.
\end{proposition}

\begin{proof}
Consider a maximal collection of interior vertices of $S'$, no two of
which have the same approximate type. Each gives rise to a simple closed curve in $\mathcal{F}_v$.
None of these curves bounds a disc in $\mathcal{F}_v$, by Lemma \ref{Lem:EssentialIntF}.
The maximal number of disjoint non-parallel simple closed curves in $\mathcal{F}_v$, none of which
bounds a disc, is at most $2c_+(\mathcal{F}_v)$. To prove this, consider a component $F$ of $\mathcal{F}_v$,
and maximal collection of disjoint non-parallel
curves in $F$, none of which bounds a disc in $F$. Since $F$ is planar, this
is $-2 \chi(F) + 1 \leq 2 c_+(F)$. 

However, it may be the case that two closed curves in $\mathcal{F}$ are topologically
parallel but not of the same approximate type. Certainly, if they are topologically parallel in
$\mathcal{F}$, then they have the same $\mathcal{H}$-type. However, condition (2) in the
definition of approximate type may not hold. We may cut $\mathcal{F}'$ along these curves
and one of the resulting pieces of surface between them may not be an annulus.
Note that we do not in fact obtain any components of this surface that are discs disjoint from $\partial X$, since these
discs would then have been part of $\mathcal{F}'_v$, contradicting the assumption
that $\mathcal{H}'$ is reduced. Thus, when $\mathcal{F}'_v$ is cut along this collection of vertices,
the result is a collection $\mathcal{F}''$ of surfaces, none of which is a disc disjoint
from $\partial X$. Since we are cutting along simple closed curves, $\chi(\mathcal{F}'')
= \chi(\mathcal{F}'_v)$. Therefore, the number of components of $\mathcal{F}''$
that are not annuli disjoint from $\partial X$ is at most $c(\mathcal{F}'_v)$.
These components can increase the number of approximate types of interior
vertices of $S'$ by at most $c(\mathcal{F}'_v)$.
\end{proof}

\subsection{The approximate star types of tiles}
Let $T$ be a tile of $S'$ with vertices $v_1$ and $v_2$ in its boundary, and let $T'$ be another tile of $S'$ with vertices $v'_1$ and $v'_2$ in its boundary.
We say that $T$ and $T'$ have the same \emph{approximate star type} if at least one of the following holds (after possibly relabelling $v'_1$ and $v'_2$):
\begin{enumerate}
\item the stars of $v_1$ and $v'_1$ are parallel in $\mathcal{H}'$, as are the stars of $v_2$ and $v'_2$;
\item all their vertices lie in the interior of $S'$, and the stars of $v_1$ and $v'_1$ are approximately parallel, as are the stars of $v_2$ and $v'_2$.
\end{enumerate}

\section{The Euclidean subsurface}
\label{Sec:Euclidean}

We continue to consider the link $K$ in an arc presentation, and an admissible partial hierarchy $H$ for
the exterior of $K$. We denote the exterior of $H$ by $M$.

We consider a generalised admissible surface $S$ in $M$ with no annular tiles. 
We also assume that no component of $S$ is a bigon tile. We suppose that
$S$ is part of nested admissible envelopes $(N, X, Y, P \cap N, \mathcal{H}, S)$ and $(N', X, Y', P \cap N', \mathcal{H}', S')$.
In particular, $S$ contains the $\pi_1$-injective subsurface $S'$.

Let $\text{approx-type}(S')$ be the number of approximate star types of tiles of $S'$. 
We will suppose that the
following quantity is minimised, up to pattern-preserving isotopy:
$$(w_\beta(S'), \text{approx-type}(S')).$$
As usual, such pairs are compared using lexicographical ordering. More specifically, suppose
that there is no pattern-isotopy of $M$ taking $S$ and $S'$ to $S_2$ and $S_2'$, and nested admissible envelopes 
$(N_2, X_2, Y_2, P \cap N_2, \mathcal{H}_2, S_2)$ and $(N'_2,  X_2, Y'_2, P \cap N'_2, \mathcal{H}'_2, S'_2)$ satisfying
$$(w_\beta(S'_2), \text{approx-type}(S'_2)) < (w_\beta(S'), \text{approx-type}(S')),$$
$$c_+(\mathcal{H}_2) \leq c_+(\mathcal{H}), \qquad c(\mathcal{H}'_2) \leq c(\mathcal{H}').$$

Now define the \emph{Euclidean subsurface} $E$ of $S'$ to be the result of removing
the interior vertices with valence other than $4$, the vertices $v$ on $\partial S'$ with $2d_i(v) + d_b(v) \not= 4$, the intersection $P \cap S'$, the generalised interior saddles 
that are not interior saddles, and the generalised boundary saddles that are not boundary saddles. So, $E$ includes:
\begin{enumerate}
\item The interior of each tile of $S'$.
\item The interior of each separatrix of $S'$.
\item Each interior vertex of $S'$ with valence 4.
\item Each boundary vertex $v$ with $2d_i(v) + d_b(v) = 4$ that is disjoint from the boundary pattern $P$.
\item Each interior saddle and boundary saddle of $S'$.
\end{enumerate}

A similar Euclidean surface was considered in \cite{Lackenby:PolyUnknot}. 

Each square tile in $E$ has four arc separatrices in its boundary.
We may therefore realise this tile as the interior of a Euclidean square with side length 1, where 
the separatrices form the sides. 
Each half tile of $E$ is realised as a right-angled isosceles triangle with side lengths $1$, $1$ and $\sqrt 2$,
where the hypotenuse lies in $\partial S'$. 
Thus, $E$ has a locally Euclidean Riemannian metric. The remainder $S' \setminus E$
is a collection of isolated points.

Define the \emph{degree} of a point in $S' \setminus E$ to be the number of separatrices to which it
is incident. The \emph{degree} of $S' \setminus E$, denoted $\mathrm{deg}(S' \setminus E)$, is the sum of the degrees of its points.
A key part of the argument is to find a bound on the area of $E$ in terms of 
$\mathrm{deg}(S' \setminus E)$. A simple geometric argument, analogous to Lemma 7.2 of \cite{Lackenby:PolyUnknot}, gives the following.

\begin{lemma}
\label{Lem:AreaEuclidean}
Suppose that each point of $E$ is at a distance at most $R$ from $S' \setminus E$.
Then the area of $E$ is at most $(\pi/4)R^2 \mathrm{deg}(S' \setminus E)$.
\end{lemma}

\begin{proof}
For each point $y$ in $E$, there is a shortest path from $y$ to $S' \setminus E$,
which is a Euclidean geodesic with length at most $R$. Thus, $E$ is covered by 
the union of the geodesics emanating from $S' \setminus E$ with length at most $R$.
These geodesics sweep out a region of area at most $(\pi/4)R^2 \mathrm{deg}(S' \setminus E)$.
Thus, this forms an upper bound for the area of $E$.
\end{proof}

\begin{remark}
Lemma 7.2 in \cite{Lackenby:PolyUnknot} asserts that the area of $E$ is bounded above by
$\pi(R+1)^2 \ell(\partial E)$, where $\ell(\partial E)$ is the `combinatorial length'
of $\partial E$. The definition of $\ell(\partial E)$ in \cite{Lackenby:PolyUnknot}
was however incorrect, and should instead have been the number of separatrices in $\partial E$
plus the sum of the degrees of the isolated points in $\partial E$.
With that corrected definition, the proof of Lemma 7.2 in \cite{Lackenby:PolyUnknot} does not
need to be adjusted, and its application in the proof of Theorem 7.1 of \cite{Lackenby:PolyUnknot} is valid.
\end{remark}

Thus, it remains to control the distance of each point of $E$ from $S' \setminus E$. This is achieved by the following
theorem, which is the main result of this section.

\begin{theorem}
\label{Thm:DistanceToBoundaryEuclidean}
Let $(N, X, Y, P \cap N, \mathcal{H}, S)$ and 
$(N', X, Y', P \cap N', \mathcal{H}', S')$ be
nested admissible envelopes. Suppose that there is no homeomorphism $h \colon N \rightarrow N$
such that all the following hold:
\begin{enumerate}
\item $h$ is equal to the identity on $P \cap N$ and $X$ (and so extends to $M$);
\item $h$ is a Dehn twist along a torus or clean annulus properly embedded in $M$;
\item $(N, X, Y, P \cap N, \mathcal{H}, h(S))$ and $(h(N'), X, h(Y'), P \cap N', h(\mathcal{H}'), h(S'))$ are nested admissible envelopes;
\item $w_\beta(h(S')) < w_\beta(S')$.
\end{enumerate}
Suppose also
that there is no pattern-isotopy taking $S$ and $S'$ to $S_2$ and $S_2'$, and
nested admissible envelopes
$(N_2, X_2, Y_2, P \cap N_2, \mathcal{H}_2, S_2)$ and $(N'_2, X_2, Y'_2, P \cap N'_2, \mathcal{H}'_2, S'_2)$ satisfying
$$(w_\beta(S'_2), \text{approx-type}(S'_2)) < (w_\beta(S'), \text{approx-type}(S')),$$
$$c_+(\mathcal{H}_2) \leq c_+(\mathcal{H}), \qquad c(\mathcal{H}'_2) \leq c(\mathcal{H}').$$
Let $e$ be the number of approximate star types of tiles of $E \setminus N_2(S' \setminus E)$.
Then, for each point of the Euclidean subsurface $E$ of $S'$, one of the following holds:
\begin{enumerate}
\item it has distance at most $4784e$ from $S' \setminus E$, or
\item it lies in a toral or clean annular component of $S'$ that lies in $E$ and has binding weight at most $23552 e^2$.
\end{enumerate}
\end{theorem}

The proof of this theorem relies heavily on work of the author in Section 9 of \cite{Lackenby:PolyUnknot}.
We will need to spend some time recalling some of this material, but for more details,
the reader should refer to \cite{Lackenby:PolyUnknot}.

Note that we may assume that there is no homeomorphism $h$ as in Theorem \ref{Thm:DistanceToBoundaryEuclidean} for the following reason.

\begin{lemma}
Let $(M,P)$ be a compact orientable irreducible boundary-irreducible 3-manifold
with boundary pattern. Let $S$ be a properly embedded surface. If $(M,P)$ contains an essential clean annulus or essential
torus, then suppose that $S$ is the JSJ tori and annuli.
Let $h \colon M \rightarrow M$ be a Dehn twist along a properly embedded
torus or clean annulus. Then $h(S)$ is strongly equivalent to $S$.
\end{lemma}

\begin{proof}
Let $T$ be the torus or annulus along which the Dehn twist is performed.
If $T$ is a torus, then $h(S)$ is strongly equivalent to $S$.
If $T$ is an inessential annulus, then $h$ is pattern-isotopic
to the identity and so $h(S)$ is pattern-isotopic to $S$. On the other hand, if $T$ is an
essential clean annulus, then $S$ is the JSJ tori and annuli,
by hypothesis. Hence, $S$ can be pattern-isotoped off $T$. 
Therefore, again $h(S)$ is pattern-isotopic to $S$.
\end{proof}

So, suppose that there is a homeomorphism $h$ as in Theorem \ref{Thm:DistanceToBoundaryEuclidean}. 
Then by the above lemma,
we may replace $S$ by $h(S)$. We may also replace the nested admissible envelopes
by $(N, X, Y, P \cap N, \mathcal{H}, h(S))$ and $(h(N'), X, h(Y'), P \cap N', h(\mathcal{H}'), h(S'))$.
These handle structures have the same complexity as the original ones, because $c(h(\mathcal{H}'))= c(\mathcal{H}')$.
This process decreases the binding weight of $S'$ by (4) of Theorem \ref{Thm:DistanceToBoundaryEuclidean}. So by choosing
$S'$ to have smallest binding weight, we can assume that there is no homeomorphism $h$.

\subsection{A `Euclidean' branched surface}

Recall that $S'$ is properly embedded in the manifold $N'$. Let $M'$ be obtained from $N'$ by attaching $M \cut N$.
Thus, $M'$ is obtained from $M$ by inserting wedges, but we suppose that when a wedge is inserted in a way that
creates a cut vertex, then we do not quite insert the full wedge, in order to maintain $M'$ as a 3-manifold.

The first step is to construct a branched surface $B_E$ lying in $M$. 
The principal type of patch of $B_E$ arises from the tiles of the Euclidean surface $E$ of $S'$.
More precisely, each approximate star type of tile in $\mathcal{H}$ of $E \setminus N_2(S' \setminus E)$ 
gives rise a patch of $B_E$. But we also consider a
slight enlargement of $E \setminus N_2(S' \setminus E)$, consisting of a regular neighbourhood of  $E \setminus N_2(S' \setminus E)$ in $E$,
and parallel pieces of this also form tiles of $B_E$. This enlargement of $E \setminus N_2(S' \setminus E)$ is carried
by the branched surface $B_E$. 

Let $N(B_E)$ be a thickening of $B_E$, and let $\pi \colon N(B_E) \rightarrow B_E$ be the map that
collapses fibres. Lemma 8.1 in \cite{Lackenby:PolyUnknot} gives that each cusp in $\partial N(B_E)$ is either an annulus or a disc.
Moreover, the cusps that are discs intersect $\pi^{-1}(\partial B_E)$ in two arcs. (Lemma 8.1 in 
\cite{Lackenby:PolyUnknot} has a couple of hypotheses which are easily checked for the 3-manifold that is a slight enlargement of $N(B_E)$.)

In Lemma 9.1 of  \cite{Lackenby:PolyUnknot}, it was shown that there are at most $24n^2$ types of Euclidean tile (where $n$ was
the arc index of the arc presentation for $K$). This estimate does not apply here. But the following gives an alternative
upper bound.

\begin{lemma}
The number of approximate star types of tiles of $E - N_2(S' \setminus E)$ is at most 
$$e = 16c_+(\mathcal{H}) + 32c(\mathcal{H}').$$
\end{lemma}

\begin{proof}
Pick a maximal collection of tiles in $E - N_2(S' \setminus E)$, all having the same approximate star type. Consider an outermost tile $T$ of a given approximate star type. One of its vertices must be outermost in its approximate star type.
By Propositions \ref{Prop:StarTypeNonCut} and \ref{Prop:NumberWeakTileTypes}, there are at most $4c_+(\mathcal{H}) + 8c(\mathcal{H}')$ such vertices. Since the tiles we
are considering lying in $E - N_2(S' \setminus E)$, their vertices lie in $E$ and so have valence at most 4. Hence, each such vertex is incident to at most 4 tiles.
The required bound then follows.
\end{proof}

Hence, our branched surface $B_E$ has at most this
many Euclidean patches.

\begin{lemma}
\label{Lem:CuspTypes}
The cusps of $B_E$ are annuli and discs in $\partial N(B_E)$. The discs lie in $N'$.
\end{lemma}

\begin{proof}
An application of Lemma 8.1 in \cite{Lackenby:PolyUnknot} gives that each cusp is either an annulus or a disc. In the latter case, the disc is a copy of $I \times I$, where each $\{ \ast \} \times I$ lies in a fibre of $N(B_E)$. It was also shown that $\partial I \times I$ lies in $\partial_h N(B_E)$. Now $N(B_E)$ is obtained from $N'$ by adding the spaces between tiles of $S'$ of the same star approximate type. Tiles of $S'$ that are incident to $\partial S'$ have the same approximate type if and only if they are parallel in $N'$. Thus we deduce that $\partial I \times I$ lies in $N'$. Hence, $\partial I \times I$ is disjoint from the wedges making up $N' \cut N$. These wedges intersect each cusp of $N(B_E)$ in regular neighbourhoods of arcs and curves that are transverse to the fibres and disjoint from the horizontal boundary. Hence if any wedge were to intersect a disc cusp of $N(B_E)$, it would do so in a regular neighbourhood of arcs, each lying entirely in the interior of the cusp. Consider the endpoint of one of these arcs. The tiles of $S'$ above and below it are not of the same approximate star type, contradicting the definition of $N(B_E)$. Hence, we deduce that the disc components of the cusps of $N(B_E)$ lie in $N'$, as claimed.
\end{proof}

\subsection{The monodromy homomorphism}
\label{Subsec:Monodromy}

For each component $B_0$ of $B_E$, there is a homomorphism $\pi_1(B_0) \rightarrow O(2)$ where $O(2)$
is the group of orthogonal transformations of $\mathbb{R}^2$. To define this, one
must consider a loop $\ell$ in $B_0$ based at some basepoint. As $\ell$ runs from one
patch to another, one may parallel translate a Euclidean 2-plane along with it.
By the time that the loop has returned to the basepoint, this 2-plane has been
subjected to an orthogonal transformation, which is the image of $[\ell]$ under this
homomorphism. Full details can be found in Section 9.3 in \cite{Lackenby:PolyUnknot}.

\subsection{A finite cover of the branched surface}

It is clear that the image of the monodromy homomorphism lies in the subgroup
of $O(2)$ that preserves the unit square centred at the origin. This is a dihedral group of order $8$. Hence, the kernel of
$\pi_1(B_0) \rightarrow O(2)$ is a finite index subgroup of $\pi_1(B_0)$. This corresponds
to a finite cover $\tilde B_0$ of $B_0$. Let $\tilde B_E$ the union of these covers, as
$B_0$ varies over each component of $B_E$. This is a branched surface. Indeed, we may take a thickening $N(B_E)$ of $B_E$,
and if we take the corresponding cover of this, the result is a thickening $N(\tilde B_E)$
of $\tilde B_E$. Thus $\tilde B_E$ is a branched surface in the 3-manifold $N(\tilde B_E)$.

The principal type of patches of $\tilde B$ are lifts of Euclidean patches of $B_E$.
We term these Euclidean patches of $\tilde B_E$. Since the covering map $\tilde B_E \rightarrow B_E$
has degree at most $8$ on each component of $\tilde B_E$, we deduce that $\tilde B_E$ has at most $8e$ Euclidean patches.

Let $\tilde S'$ be the inverse image of $S'$ in $N(\tilde B_E)$.
The inverse image of $E$ in $N(\tilde B)$ is a surface $\tilde E$ made up of Euclidean tiles.

One may define a monodromy homomorphism $\pi_1(\tilde B_0) \rightarrow O(2)$ for
each component $\tilde B_0$ of $\tilde B_E$. But by the construction of $\tilde B_E$, this monodromy
homomorphism is trivial. This implies that $\tilde B_E$ is transversely orientable,
since if it were not, then for some loop $\ell$ in $\tilde B_E$, the image of $\ell$
under the monodromy homomorphism would have determinant $-1$. We now 
fix a transverse orientation on $\tilde B_E$, which induces a transverse orientation on 
$\tilde E$.

\subsection{Strategy of the proof}

The main result of Section 9 in \cite{Lackenby:PolyUnknot} 
was Proposition 9.8, which was as follows. If there is a point in $\tilde E$ with distance
more than $8000n^2$ from $\partial \tilde E$ (where again $n$ is the arc index of the link),
then $\tilde E$ has a torus summand, when viewed as a surface carried by $\tilde B_E$.

This then implied the result that each point in $E$ has distance at most $8000n^2$ from $\partial E$.
For if not, then some point of $\tilde E$ would have distance more than $8000n^2$ from $\partial \tilde E$
and hence $\tilde E$ would have a torus summand. The projection of $\tilde E$ to $N(B_E)$ forms a non-zero
multiple of $E$, and the torus projects to a torus carried by $B_E$. So, we deduce that some non-zero multiple of $E$
has a torus summand. This then implied that some non-zero multiple of $S$ has a torus summand,
as a normal surface. (Here, $S$ was a spanning disc for the unknot.) But, in \cite{Lackenby:PolyUnknot}, $S$ was arranged
so that no multiple of $S$ has a torus summand, because $S$ is, in \cite{Lackenby:PolyUnknot}, a `boundary-vertex surface'.

Unfortunately, the same strategy does not work in our setting, because we are not assuming that
$S'$ is a `boundary-vertex' surface. In particular, there is no obvious reason to prevent some multiple of $S'$
having a torus summand. Instead, we will show the following.

\begin{theorem}
\label{Thm:ExistsDehnTwist}
Suppose that there is some point in $E$ that has distance more than $1196e$
from $S' \setminus E$. Then there is a torus or clean annulus $A$ properly embedded in $M$ that is normal in $\mathcal{H}$ and a homeomorphism
$h \colon M \rightarrow M$ that is a power of a Dehn twist along $A$ that satisfies (1) - (4) of Theorem \ref{Thm:DistanceToBoundaryEuclidean}.
\end{theorem}

\subsection{Doubling}
Let $E_+$ be the abstract double of $\mathrm{cl}(E)$ along $\partial S' \cap E$. Let $B_{E+}$ be the double of $B_E$ along $\partial B_E$.
Let $N(B_{E+})$ be the double of $N(B_E)$ along $\pi^{-1}(\partial B_E)$. Note that the cusps of $N(B_{E+})$ are all annuli.

The covers $\tilde B_E \rightarrow B_E$ and $\tilde E \rightarrow E$ induce covers $\tilde B_{E+} \rightarrow B_{E+}$ and $\tilde E_+ \rightarrow E_+$. Similarly, we have a cover
$N(\tilde B_{E+}) \rightarrow N(B_{E+})$.

\subsection{Grids and annuli}

A \emph{grid} is a subsurface of $\tilde E$ or $E$ that is a union of tiles glued together
to form a square of odd side length. Since the side length is odd, the centre of this
square is at the centre of some tile. For a positive integer $r$ and a point $x$ in $\tilde E$ or $E$,
we let $D(x,r)$ denote a grid centred at a tile containing $x$ with side length $2r+1$.

The proof of Proposition 9.8 in \cite{Lackenby:PolyUnknot} divided into two cases: either there is a `short' closed geodesic `far' from
$\partial \tilde E$, or there is not. In \cite{Lackenby:PolyUnknot}, `short' and `far' were measured in terms of the arc index $n$. Here,
we instead measure in terms of $e$. We similarly divide into two cases, but 
we work with $E$ instead $\tilde E$. For this, we require the following lemma.

\begin{lemma}
\label{Lem:GeodesicOrGrid}
Suppose that $E - N_{4784e}(S' \setminus E)$ is non-empty. Let $E_+$ be the abstract double of $\mathrm{cl}(E)$ along $\partial S' \cap E$.
Let $W$ be the inverse image of $S' \setminus E$ in $E_+$.
Then one of the following holds:
\begin{enumerate}
\item there is a closed geodesic $\alpha$ in $E_+ - N_{1472e}(W)$ with length at most $736e$ 
such that $N_{1472e}(\alpha)$ is either invariant under the involution of $E_+$ or disjoint from its image
under the involution;
\item there is a grid in $E_+ - N_{1472e}(W)$ with side length $512e+1$
that is either invariant under the involution or disjoint from its image
under the involution.
\end{enumerate}
\end{lemma}

\begin{proof}
Since $E - N_{4784e}(S' \setminus E)$ is non-empty, so is $E_+ - N_{4784e}(W)$.
Let $x'$ be some point in $E_+ - N_{4784e}(W)$. If $x'$ lies within distance $1840e$ of the copy of $\partial S' \cap E$ in $E_+$, then
there is a point $x$ on $\partial S' \cap E_+$ that lies in $E_+ - N_{2944e}(W)$. If on the other hand, $x'$
is more than $1840e$ from $\partial S' \cap E_+$, then we set $x = x'$. So, $x$ either lies on $\partial S' \cap E_+$
or is more than $1840e$ from $\partial S' \cap E_+$.

Let $r$ be the injectivity
radius at $x$ of $E_+$. If $r > 368e$, then there is an embedded grid with side
length $512e+1$ centred at $x$, since $512e+1<368\sqrt{2} e$. When $x$ lies on $\partial S' \cap E_+$, we may choose the grid
so that it is invariant under the involution of $E_+$. When $x$ is more than $1840e$
from $\partial S' \cap E_+$ then the grid is disjoint from its image under the involution.
Note that the grid lies in $E_+ - N_{2576e}(W)$, and hence (2) is satisfied.

So suppose that $r$ is at most $368e$. Let $\gamma_1$ and
$\gamma_2$ be distinct geodesics starting at $x$ with length $r$ and having the same endpoint $y$ in $E_+$.
Since $r$ is exactly the injectivity radius at $x$, these two geodesics patch together exactly
at $y$ to form a smooth geodesic there. However, it need not be the case that
they patch together to form a smooth closed geodesic at $x$. Let $\gamma$
be the concatenation $\gamma_1.\overline{\gamma_2}$, which is a geodesic based at $x$.
The monodromy of $\gamma$ lies in $O(2)$. Indeed it lies in $SO(2)$ because $E_+$ is orientable.

\medskip
\emph{Case 1.} $\mu(\gamma)$ is the identity.

Then $\gamma$ is a closed geodesic that is smooth at $x$. It has length at most $736e$ and lies in $E_+ - N_{1472e}(W)$.
Hence, when $x$ is disjoint from $\partial S' \cap E_+$ and hence $N_{1472e}(\gamma)$ is disjoint from $\partial S' \cap E_+$,
$\gamma$ is the required geodesic.

However, when $x$ lies on $\partial S' \cap E_+$, we need to verify that $\gamma$ is invariant under the involution of $E_+$.
In this case let $D$ be the open disc of radius $r$ about $x$. This is a Euclidean disc that is invariant under the
isometric involution of $E_+$. Hence, the involution acts by reflection in a straight line through $x$. We deduce that
$D \cap (\partial S' \cap E_+)$ is this straight line. If $\gamma'(0)$ points in the direction of this line or is orthogonal to it,
then $\gamma$ is invariant under the involution as required. So suppose, for a contradiction, that $\gamma'(0)$
makes an angle strictly between $0$ and $\pi/2$ with $\partial S' \cap E_+$. Within $D$, the only intersection between $\gamma$
and $\partial S' \cap E_+$ is at $x$. However, $\gamma_1$ lies in one half of $E_+$ and $\gamma_2$ lies in the other half,
where these two halves are swapped by the involution. We deduce that $\partial S' \cap E_+$ must also lie at the endpoint
of $\gamma_1$. It is therefore tangent to $\partial D$ at that point. However, the Euclidean subsurface is made of
squares and $\partial S' \cap E_+$ follows the sides of these squares. We deduce that the only possible angles between
$\partial S' \cap E_+$ and $\gamma$ at the endpoint of $\gamma_1$ are multiples of $\pi/2$. This is a contradiction.

\medskip
\emph{Case 2.} $\mu(\gamma)$ is a rotation of order $2$.

Then $\gamma_1$ and $\gamma_2$ are equal, contradicting the assumption that they are distinct.

\medskip
\emph{Case 3.} $\mu(\gamma)$ is a rotation of order $4$.

Then $\gamma^4$ can be developed onto a closed curve in the Euclidean plane.
We may shorten this curve in the plane, as shown in Figure \ref{Fig:EuclideanOrder4}. This is achievable
within $E_+ - W$ unless the curve hits $W$. But in this case, $x$ would lie
within distance at most $1472e$ from $W$, which is contrary to assumption.
But if we can shrink $\gamma^4$ to length $0$, it is homotopically trivial in $E_+$.
Hence so is $\gamma$, since the fundamental group of $E_+$ is torsion free.
This is a contradiction, since $\mu(\gamma)$ is non-trivial. 
\end{proof}

\begin{figure}[h]
\includegraphics[width=4in]{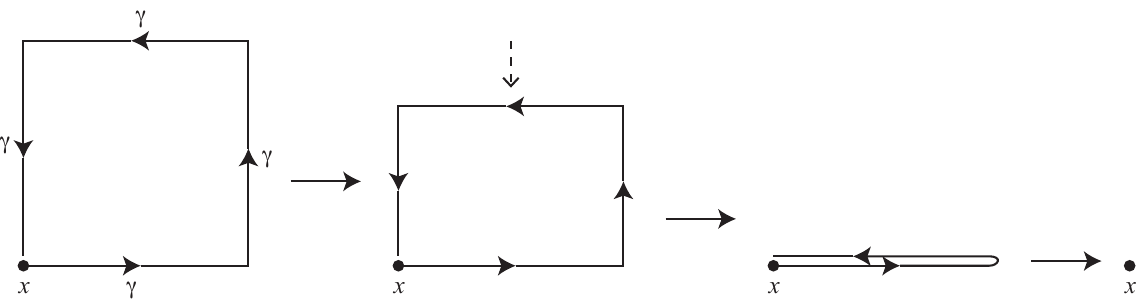}
\caption{Shrinking $\gamma^4$}
\label{Fig:EuclideanOrder4}
\end{figure}

The proof of Theorem \ref{Thm:DistanceToBoundaryEuclidean}
divides into the two cases in the above lemma. We leave the first case for a while, and now consider the case where 
there is a grid in $E_+ - N_{1472e}(W)$ with side length $512e+1$ and that is either invariant under the 
involution of $E_+$ or disjoint from its image. This lifts to a grid $D(\tilde x, 256e)$ in
$\tilde E_+ - N_{1472e}(\tilde W)$ also with side length $512e+1$, where $\tilde E_+$ is the double of
$\mathrm{cl}(\tilde E)$ along the inverse image of $\partial S' \cap E$, and $\tilde W$ is the inverse image of $W$ in $\tilde E_+$.

\subsection{First-return maps}

For a subsurface $F$  of $\tilde E$ or $\tilde E_+$, we may define a \emph{ first-return map}
$r_F \colon {\rm dom}(F) \rightarrow F$ as follows. Its domain of definition is
a subsurface ${\rm dom}(F)$ of $F$.

For a point $y$ of $F$, one considers the fibre $\alpha_y$ of $N(\tilde B_E)$ or $N(\tilde B_{E+})$ through $y$.
This intersects $F$ in a number a points. We consider the closest such point along $\alpha_y$
in the specified transverse direction on $\tilde B_E$ or $\tilde B_{E+}$. If there is no point
of $\alpha_y \cap F$ in this transverse direction, then $r_F(y)$ is not defined. But otherwise,
we take $r_F(y)$ to be this point.

We will consider surfaces $F$ that are a union of tiles. In this situation, the interior of each
tile of $F$ is either in ${\rm dom}(r_F)$ or is disjoint from it. Each tile of the latter sort is
outermost in the patch of $N(\tilde B_E)$ or $N(\tilde B_{E+})$ that contains it. Since there are at most $8e$
Euclidean patches of $\tilde B_E$, we deduce that all but at most $8e$ Euclidean tiles
of $F \subset \tilde E$ lie in ${\rm dom}(r_F)$. Similarly, all but at most $16e$ Euclidean tiles
of $F \subset \tilde E_+$ lie in ${\rm dom}(r_F)$.

\subsection{The first-return map for large grids}

As described above, we are currently considering the case where there is a grid $D(\tilde x, 256e)$ in
$\tilde E_+ - N_{1472e}(\tilde W)$ with side length $512e+1$. We are assuming that this
grid projects homeomorphically to a grid in $E_+$ and that it is invariant under the involution
of $E_+$.

\begin{proposition}
\label{Prop:GridFirstReturn}
Every point $\tilde y \in D(\tilde x, 174e)$ lies in $\mathrm{dom}(r_{D(\tilde y,81e)})$.
\end{proposition}

This is very similar to Proposition 9.14 in \cite{Lackenby:PolyUnknot}, and their proofs follow the same lines.
A key step is the following, which is an analogue of Proposition 9.15 in \cite{Lackenby:PolyUnknot}.

\begin{proposition}
\label{Prop:GridFirstReturnTildeE}
Let $D_1, \dots, D_m$ be a collection of disjoint grids in $\tilde E_+$, each with side length at least 
$64e+1$. Suppose that the restriction of the covering map $\tilde E_+ \rightarrow E_+$ to each $D_i$ is a homeomorphism
onto its image. Assume also that each $D_i$ is either invariant under the involution of $\tilde E_+$ or disjoint from
its image under the involution.
Let $d_i$ be the supremal distance of the points in $D_i$ from $\partial D_i$ at which
$r_{\tilde E_+}$ is not defined. Then $\sum_i d_i \leq 16e$.
\end{proposition}

Note that Proposition \ref{Prop:GridFirstReturnTildeE} has a hypothesis that was not present
in Proposition 9.15 in \cite{Lackenby:PolyUnknot}. This is that the restriction of the covering map 
$\tilde E_+ \rightarrow E_+$ to each $D_i$ is a homeomorphism onto its image. This hypothesis is required
because here we not assuming that $S$ or $S'$ is a `boundary-vertex' surface.

\begin{proof}
Consider the cusps of $\tilde B_{E+}$, and let $\partial_- C$
denote the 1-manifold in $\tilde E_+$ lying directly above these cusps. By `above',
we mean that the transverse orientation on $\tilde E_+$ at $\partial_- C$ points towards
the cusps. Note that $\partial_- C$
divides $\tilde E_+$ into two subsurfaces: $\mathrm{dom}(r_{\tilde E_+})$ and $E_+ \cut \mathrm{dom}(r_{\tilde E_+})$.
Transversely orient $\partial_- C$ in $\tilde E_+$ so
that it points towards $\mathrm{dom}(r_{\tilde E_+})$.
The 1-manifold $D_i \cap \partial_-C$ consists of
properly embedded arcs and simple closed curves.

We claim that each simple closed curve of $D_i \cap \partial_-C$
points into the disc in $D_i$ that it bounds.
This is an analogue of Claim 2 in the proof of Proposition 9.15 in \cite{Lackenby:PolyUnknot}.
Suppose that, on the contrary, there is a simple closed curve $\beta_-$ in $D_i \cap \partial_-C$ that points
out of the disc that it bounds in the grid $D_i$. This is an embedded curve in the
grid $D_i$ and hence it projects to an embedded curve $\overline{\beta}_-$ in 
$E_+$. Therefore, the cusp of $\tilde B_{E+}$ directly below $\beta_-$ projects
homeomorphically to a cusp of $B_+$. Let $\beta_+$ be the curve in $\tilde E_+$
directly below $\beta_-$. Then $\beta_+$ projects homeomorphically to a simple
closed curve $\overline{\beta}_+$ in $E_+$. This restricts to a simple closed curve
or properly embedded arc $\gamma_+$ in $E$. Similarly, $\overline{\beta}_-$
restricts to a simple closed curve
or properly embedded arc $\gamma_-$ in $E$. The disc bounded by $\overline{\beta}_-$
restricts to a disc $D_-$ in $S'$ separated off by $\gamma_-$

Suppose first that $\gamma_+$ and $\gamma_-$ are simple closed curves. Thus the union of the portions
of the fibres between them is an annulus $A$ containing an annular cusp of $N(B_E)$.
The union of $A \cup D_-$ is a disc bounded by $\gamma_+$ lying in $M$.
Since $S$ is incompressible, we deduce that $\gamma_+$ bounds a disc $D_+$ in $S$.
The boundary of this disc lies in $S'$, and hence all of $D_+$ lies in $S'$.
Then $D_- \cup A \cup D_+$ is a sphere that bounds a ball in $M$.

Now suppose that $\gamma_+$ and $\gamma_-$ are arcs. Thus the union of the portions
of the fibres between them is a disc $D$ containing a disc cusp of $N(B_E)$.
By Lemma \ref{Lem:CuspTypes}, $D$ lies in $N'$. Hence, after a small isotopy, $D_+ \cup D$
is a disc in $M'$, with boundary consisting of an arc in $S'$ and an arc in $\partial M'$,
and with interior disjoint from $S'$. By the boundary-incompressibility of $S'$,
we deduce that $\gamma_+$ separates off a clean disc $D_+$ in $S'$.
Hence, $D_- \cup A \cup D_+$ is a disc properly embedded in $M'$ disjoint
from the pattern. So by the boundary-irreducibility of $M'$, we deduce that $D_- \cup A \cup D_+$
separates off a 3-ball disjoint from the pattern. The interior of this ball is disjoint from $S'$
and $S$.

We can assume that $D_-$ and $D_+$ have the same binding weight, as otherwise we
could isotope $S'$ and $S$, replacing one disc by the other to reduce the binding weight of $S'$.
Note that in the case where $\gamma_+$ and $\gamma_-$ are simple closed curves,
there may be components of intersection between $S \setminus S'$ and the ball bounded by 
$D_- \cup A \cup D_+$. But we may replace these by discs parallel to $D_-$ or $D_+$.
This might introduce new 2-sphere components of $S \setminus S'$, in which case these
are discarded.

Now replace as many parallel copies of $D_+$ as possible by parallel copies of $D_-$. Eventually, this must stop when some approximate tile type
of $S'$ has been removed. Hence, we have found a surface $S'_2$ that is pattern-isotopic to $S'$
having the same binding weight, but where $\text{approx-type}(S'_2) < \text{approx-type}(S')$,
contrary to an assumption of Theorem \ref{Thm:DistanceToBoundaryEuclidean}. This proves the claim.

We claim that any point of $D_i \cap \partial_-C$ that is furthest from $\partial D_i$ lies on an
arc component of $D_i \cap \partial_-C$. This is an analogue of Claim 1 in the 
proof of in the proof of Proposition 9.15 in \cite{Lackenby:PolyUnknot}.
Cut $D_i$ along the arc components of $D_i \cap \partial_-C$
and let $D_i'$ be the  disc containing the centre of $D_i$. Since the arc components have length at most $32e$, 
$D'_i$ contains a grid with side length $32e+1$. So, $D_i'$ contains at least $(32e+1)^2$ tiles. 
We will rule out the possibility that there are any simple closed curves of $\partial_-C$
in $D'_i$. Let $\delta$ be the union of those simple closed components of $D_i' \cap \partial_-C$ that are
 outermost, in other words, that do not lie within another simple closed component of $D_i' \cap \partial_-C$. 
 The total length of $\delta$ is at most $32e$, and so the total number of tiles that it can bound is at most 
 $(8e)^2$. But by the above claim, $\mathrm{dom}(r_{\tilde E_+}) \cap D'_i$ lies within $\gamma$.
 So at least $(32e+1)^2 - (8e)^2$ tiles do not lie in $\mathrm{dom}(r_{\tilde E_+})$. However,
 $r_{\tilde E_+}$ is defined on all but at most $16e$ tiles of $\tilde E_+$. This is a contradiction, proving the claim.

We are now in a position to complete the proof of Proposition \ref{Prop:GridFirstReturnTildeE}.
Consider a point on $D_i \cap \partial_-C$ that is furthest from $\partial D_i$.
By the above claim, this lies on an arc component of $D_i \cap \partial_-C$, and hence
its distance from $\partial D_i$ is at most $16e$. So, we deduce that the grid with side length
$32e+1$ with the same centre as $D_i$
has interior disjoint from $\partial_-C$ and so either lies in $\mathrm{dom}(r_{\tilde E_+})$ or
disjoint from $\mathrm{dom}(r_{\tilde E_+})$. But this grid contains at least
$(32e+1)^2$ tiles whereas $r_{\tilde E_+}$ is defined on all but at most $16e$ tiles. 
So we deduce that a point in $D_i \cut \mathrm{dom}(r_{\tilde E_+})$ furthest
from $\partial D_i$ has distance from $\partial D_i$ that is at most half the length of
$\partial_- C \cap D_i$. So, $d_i$ is at most half the length of
$\partial_- C \cap D_i$. Thus, $\sum_i d_i$ is at at most half the length of
$\partial_- C$, and hence at most $16e$. 
\end{proof}

\begin{proof}[Proof of Proposition \ref{Prop:GridFirstReturn}]
The proof is a direct analogue of that of Proposition 9.14 in \cite{Lackenby:PolyUnknot}.
Let $\tilde y$ be a point in $D(\tilde x, 81e)$. We will show that there is a non-negative integer $m$ and
a map
$D(\tilde y, 32e) \times [0,m'+1] \rightarrow N(\tilde B_{E+})$ such that
\begin{enumerate}
\item the restriction to $D(\tilde y, 32e) \times [0,m'+1)$ is an embedding;
\item the image of $D(\tilde y, 32e) \times \{ 0 \}$ is $D(\tilde y, 32e)$;
\item the inverse image of $\tilde E_+$ is $D(\tilde y, 32e) \times ([0,m'+1] \cap \mathbb{Z})$;
\item for each point $z$ in $D(\tilde x, 32e)$, the image of $\{ z \} \times [0,m'+1]$ is a subset of
a fibre of $N(\tilde B_{E+})$;
\item the image of $D(\tilde y, 32e) \times \{ m'+1 \}$ has non-empty intersection with $D(\tilde y, 48e)$.
\end{enumerate}

We define two increasing sequences $m_i$ and $k_i$ of non-negative integers and maps
$D(\tilde y, 48e - k_i) \times [0,m_i] \rightarrow N(\tilde B_{E+})$ with properties analogous to
(1)-(4) above. The procedure starts with $k_0 = 0$
and $m_0 = 0$. Then the following procedure is performed:
\begin{enumerate}
\item If $D(\tilde y, 48e - k_i) \times  \{ m_i \}$ intersects $D(\tilde y, 48e)$ and $m_i > 0$, then the procedure terminates;
otherwise we pass to the next step.
\item If $r_{\tilde E_+}$ is defined on all of $D(\tilde x, 48e - k_i) \times  \{ m_i \}$, let $r_{\tilde E_+}(D(\tilde x, 48e - k_i) \times  \{ m_i \})$
be $D(\tilde x, 48e - k_i) \times  \{ m_i + 1\}$, increase $m_i$ by $1$, and between these two grids there is a product
which is added to the product region. Then return to step (1).
\item Suppose that $r_{\tilde E_+}$ is not defined for all of $D(\tilde x, 48e - k_i) \times  \{ m_i \}$. This means that there is at least one cusp
of $N(\tilde B_{E+})$ below $D(\tilde x, 48e - k_i) \times  \{ m_i \}$. Let $d_i$ be the maximal distance of such cusps from the boundary of 
$D(\tilde x, 48e - k_i) \times  \{ m_i \}$. Applying Proposition \ref{Prop:GridFirstReturnTildeE} to the
images of $D(\tilde x, 48e - k_0) \times  \{ m_0\}, \dots,  D(\tilde x, 48e - k_i) \times  \{ m_i\}$
gives that $\sum_{j=1}^i d_j$ is at most $16e$. Set $k_{i+1} = \sum_{j=1}^i d_j$, and define $m_{i+1} = m_i + 1$.
Increase $i$ by 1 and return to step (1).
\end{enumerate}

Note that our application of Proposition \ref{Prop:GridFirstReturnTildeE} in step (3) is permitted. This is because restriction of the
covering map $\tilde E_+ \rightarrow E_+$ to each of the grids $D(\tilde x, 48e - k_0) \times  \{ m_0\}, \dots,  D(\tilde x, 48e - k_i) \times  \{ m_i\}$ 
is an embedding. For suppose that this was not the case, and that two points in one of these grids had the same
image in $E_+$. Then, translating these two points along the fibres towards $D(\tilde x, 48e)$, we would deduce that two distinct
points in $D(\tilde x, 48e)$ had the same image in $E_+$. However, the restriction of $\tilde E_+ \rightarrow E_+$ to $D(\tilde x, 48e)$
is a homeomorphism onto $D(x, 48e)$, and in particular is injective. Thus, we have shown that there is
the required map $D(\tilde y, 32e) \times [0,m'+1] \rightarrow N(\tilde B_{E+})$.

Let $\tilde y$ be a point in $D(\tilde x, 174e)$. Around $\tilde y$, there is the grid $D(\tilde y, 32e) \subset D(\tilde x, 256e)$. Let 
$D(\tilde y, 32 e) \times [0,m'+1] \rightarrow N(\tilde B_{E+})$ be the above map. By
(5), there is a point $z$ in $D(\tilde y, 32e)$ such that $\{ z \} \times \{ m'+1 \}$
has image equal to a point in $D(\tilde y, 48e)$. Now $\{ \tilde y \} \times \{ m'+1 \}$ lies in the grid $D(\{ z \} \times \{ m'+1 \}, 32e)$ and
hence lies in $D(\tilde y, 1+80e) \subset D(\tilde y, 81e)$. There may be other points of $D(\tilde y, 81e)$ on the fibre between $\{ \tilde y \} \times \{0 \}$ and 
$\{ \tilde y \} \times \{ m'+1 \}$. But the first return map for $D(\tilde y, 81e)$ is certainly defined at $\tilde y$.
\end{proof}

\subsection{Translation invariance of the first-return map}

As in the previous subsection, let $D(\tilde x, 256e)$ be a grid in $\tilde E_+ - N_{1472e}(\tilde W)$. We are assuming that
$D(\tilde x, 256e)$ is either invariant under the involution of $\tilde E_+$ or disjoint from its image under the involution.
We are also assuming that $D(\tilde x, 256e)$ projects homeomorphically to a grid in $E_+$.
By Proposition \ref{Prop:GridFirstReturn},
for each $\tilde y \in D(\tilde x, 174e)$, the points $\tilde y$ and $r_{D(\tilde y,81e)}(\tilde y)$ both lie in the
disc $D(\tilde x, 256e)$ and so there is a well-defined vector $v_{\tilde y}$ in $T_{\tilde x} \tilde E_+$ taking $\tilde y$ to $r_{D(\tilde y,81e)}(\tilde y)$.

\begin{proposition}
\label{Prop:CovariantConstant}
The vector field $\{ v_{\tilde y} : \tilde y \in D(\tilde x, 48e) \}$ is covariant constant.
\end{proposition}

This was Proposition 9.16 in \cite{Lackenby:PolyUnknot}, with essentially the same proof, which we briefly sketch.
First consider moving $\tilde y$ within a tile. Then $D({\tilde y, 81e})$ remains the same grid as $\tilde y$ varies.
The map $r_{D(\tilde y,81e)}$ sends $\tilde y$ to a point in $N(\tilde B_{E+})$ directly below it.
Because $\tilde B_{E+}$ has trivial monodromy, the vector in $D({\tilde y, 80e})$ from $\tilde y$
to its image remains unchanged.

Now consider what happens when $\tilde y$ moves to a point $\tilde y'$ in an adjacent tile. Then $D({\tilde y, 81e})$ also moves,
and a new tile of $D({\tilde y', 81e})$ appears below $\tilde y'$. This is again because $\tilde B_{E+}$ has trivial monodromy.
The image of $\tilde y'$ under $r_{D(\tilde y',81e)}$ lies in this tile, and again the vector in $D({\tilde y', 81e})$ from $\tilde y'$
to its image remains unchanged.

This argument then gives the following result.

\begin{proposition}
\label{Prop:ProductRegion}
There is a positive integer $m$ with the following property. For any point $\tilde y \in D(\tilde x, 174e)$,
there is a map
$D(\tilde y, 81e) \times [0,m+1] \rightarrow N(\tilde B_{E+})$ such that
\begin{enumerate}
\item the restriction to $D(\tilde y, 81e) \times [0,m+1)$ is an embedding;
\item the image of $D(\tilde y, 81e) \times \{ 0 \}$ is $D(\tilde y, 81e)$;
\item the inverse image of $\tilde E_+$ is $D(\tilde y, 81e) \times ([0,m+1] \cap \mathbb{Z})$;
\item for each point $z$ in $D(\tilde x, 81e)$, the image of $\{ z \} \times [0,m+1]$ is a subset of
a fibre of $N(\tilde B_{E+})$;
\item the image of $D(\tilde y, 81e) \times \{ m+1 \}$ has non-empty intersection with $D(\tilde y, 81e)$.
\end{enumerate}
\end{proposition}

\subsection{The monodromy of paths}

In Section \ref{Subsec:Monodromy}, we defined the monodromy of a closed loop in $B_{E+}$. We now use this to define
the monodromy of certain paths in $E_+$.

Let $p$ be a path in $E_+$ such that $\pi p (0) = \pi p (1)$ where $\pi \colon N(B_{E+}) \rightarrow B_{E+}$
is the projection map that collapses each fibre to a point. Then there is a unique embedded
path $\epsilon$ in a fibre joining $p(1)$ to $p(0)$. We define the \emph{monodromy} $\mu(p)$ of $p$
to be the monodromy of $p . \epsilon$. Note that if $p$ is a closed loop, then the two
notions of monodromy coincide and so no confusion can arise. Note also that the monodromy
of a path $p$ is unchanged if a homotopy is applied to $p$ in $E_+$ that fixes its endpoints.

\subsection{Paths with close endpoints}

Let $p$ be a path in $D(x,174e)$ such that $p(0)$ and $p(1)$ lies in the same fibre of $N(B_{E+})$ and
with the property that $p$ lies in $D(p(0), 81e)$. Let $\epsilon$
be the part of the fibre between these two points. Suppose that $\epsilon$ lies in the image of 
$D(p(0), 81e) \times [0,m+1)$ or $D(p(0), 81e) \times (0,m+1]$ in $N(B_{E+})$ for the map described in Proposition
\ref{Prop:ProductRegion} composed with the covering map $N(\tilde B_{E+}) \rightarrow N(B_{E+})$. Then we say that $p$ has \emph{close endpoints}.
Our aim is to prove the following.

\begin{proposition}
\label{Prop:NoCloseEndpoints}
No path $p$ in $D(x, 81e)$ has close endpoints.
\end{proposition}

It suffices to consider the case where $p$ is a geodesic in $D(x,174e)$.

\begin{lemma}
\label{Lem:CloseEndpointsPath}
Let $p$ be a geodesic in $D(x,174e)$ with close endpoints. Then the following are homotopically trivial loops in $N(B_{E+})$:
$$\begin{array}{ll}
p . \epsilon. (-\mu(p)p). -\epsilon & \text{if } {\rm det}(\mu(p)) = 1; \cr
p . \epsilon. (-\mu(p)p). \epsilon & \text{if } {\rm det}(\mu(p)) = -1. \cr
\end{array}
$$
Here, $-\mu(p)p$ denotes the the geodesic that starts at $p(0)$ with initial vector $- \mu(p) p'(0)$ and
with length equal to that of $p$. Also, $-\epsilon$ denotes the embedded arc in the same
fibre as $\epsilon$ that ends at $\epsilon(1)$, that intersects $E_+$ the same
number of times as $\epsilon$ and that intersects $\epsilon$ only at $\epsilon(1)$.
\end{lemma}

\begin{proof}
Let $\tilde p$ be a lift of $p$ to $D(\tilde x, 174e)$. Let $\tilde \epsilon$ be the lift of $\epsilon$ starting at the endpoint of 
$\tilde p$. Then $\tilde p. \tilde \epsilon$ lies in the image of $D(p(0), 81e) \times [0, m+1)$  or $D(p(0), 81e) \times (0,m+1]$. The lift of 
$(-\mu(p)p). -\epsilon$ or $(-\mu(p)p). \epsilon$ (as appropriate) starting at the endpoint of $\tilde \epsilon$
runs along a path parallel to $\tilde p$ but in reverse and then closes up, ending at $\tilde p(0)$. Hence,
the composition of these 4 paths is a homotopically trivial loop in $D(p(0), 81e) \times [0, m+1)$  or $D(p(0), 81e) \times (0,m+1]$.
Thus, the image of these 4 paths is a homotopically trivial loop in $N(B_{E+})$.
\end{proof}

Therefore, when $p$ is a path in $D(x,174e)$ with close endpoints, the following paths are homotopic relative to their endpoints:
$$
\begin{array}{ll}
p . \epsilon \simeq \overline{-\epsilon}. \overline{(-\mu(p)p)}  & \text{if } {\rm det}(\mu(p)) = 1; \cr
p . \epsilon \simeq  \overline{\epsilon}. \overline{(-\mu(p)p)} & \text{if } {\rm det}(\mu(p)) = -1. \cr
\end{array}
$$
Here, the overline denotes the reverse of a path.

\subsection{Paths in grids with close endpoints have positive determinant}

\begin{lemma}
\label{Lem:PositiveDeterminant}
Any  path
in $D(x,81e)$ with close endpoints has monodromy with positive determinant.
\end{lemma}

\begin{proof}
Suppose that the monodromy of $p$ has negative determinant.
Therefore, $p. \epsilon$ is a homotopically non-trivial loop in $N(B_{E+})$. Therefore,
$p.\epsilon.p.\epsilon$ is also, since $N(B_{E+})$ has torsion free fundamental group
(being a compact 3-dimensional submanifold of the 3-sphere).
But
$$p.\epsilon.p.\epsilon \simeq p.\epsilon. \overline{\epsilon}. \overline{(-\mu(p)p)}
\simeq p. \overline{(-\mu(p)p)}.$$
However $p. \overline{(-\mu(p)p)}$ remains within $D(x,174)$ which is a grid and which
therefore does not contain any homotopically non-trivial loops.
\end{proof}

\subsection{Paths in grids with close endpoints have trivial monodromy}

\begin{lemma}
\label{Lem:TrivialMonodromy}
Any  path
in $D(x,81e)$ with close endpoints has trivial monodromy.
\end{lemma}

\begin{proof}
Let $p$ be a path in $D(x,81e)$ with close endpoints and non-trivial monodromy. 
By Lemma \ref{Lem:PositiveDeterminant}, this monodromy has positive determinant
and is therefore a rotation of order 2 or 4.
By assumption, the endpoints of $p$ lie in the same fibre
of $N(B_{E+})$. Indeed, the part of the fibre between $p(0)$ and $p(1)$ lies in the image of $D(\tilde x,81e) \times [0,m+1)$
or $D(\tilde x, 81e) \times (0,m+1]$.
We may homotope $p$ relative to its endpoints to a geodesic. Let $p_2$ be the geodesic starting at $p(1)$, with length
and direction given by $\mu(p) (p'(0))$ translated to $p(1)$. Then $p_2$ fellow travels with $p$, in the sense that
they have the same images under the projection map $N(B_{E+}) \rightarrow B_{E+}$.
Define $p_3$ and $p_4$ similarly.
Then $p.p_2.p_3.p_4$ has distinct endpoints in the same fibre of $N(B_{E+})$. However, $p.p_2.p_3.p_4$ is a homotopically
trivial loop in $D(x,174e)$, and in particular, it ends where it starts. This is a contradiction.
\end{proof}

\begin{proof}[Proof of Proposition \ref{Prop:NoCloseEndpoints}]
Suppose that on the contrary there is a path $p$ in $D(x,81e)$ with close endpoints. Let $\tilde p$ be a lift to $D(\tilde x, 81e)$.
Let $\epsilon$ be the path in $N(B)$
between $p(0)$ and $p(1)$. By Lemma \ref{Lem:TrivialMonodromy}, $p.\epsilon$ has trivial monodromy,
and so lifts to a loop in $N(\tilde B)$ starting at $\tilde p(0)$. So either $\tilde p(1)$ is a point in $D(\tilde p(0), 81e)$ that lies strictly between $\tilde p(0)$
and $r_{D(\tilde p(0),81e)}(\tilde p(0))$, or $\tilde p(0)$ is a point in $D(\tilde p(1), 81e)$ that lies strictly between $\tilde p(1)$
and $r_{D(\tilde p(1),81e)}(\tilde p(1))$. This contradicts the definition of the first-return map.
\end{proof}

\subsection{An embedded annulus in $N(\tilde B_+)$}
\label{Subsec:EmbeddedAnnulus}

We can now form an embedded annulus $\tilde C$ in $N(\tilde B_+)$. This is constructed as follows.

Let $\tilde \alpha$ be the geodesic in $D(\tilde x, 81e)$ from $\tilde x$ to $\tilde x' = r_{D(\tilde x, 81e)}(\tilde x)$. Let $\tilde \beta$ be a geodesic starting
at $\tilde x$ with length $32e$ that is parallel to a side of $D(\tilde x, 81e)$ and that has angle between $\pi/2$ and $3\pi/4$ from $\tilde \alpha$.
Let $\tilde \beta'$ be the result of translating $\tilde \beta$ using the vector $v_{\tilde x}$, so that it runs through $\tilde x'$.
The region $\tilde Q$ between $\tilde \beta$ and $\tilde \beta'$ is 
a Euclidean parallelogram. By the way that we have chosen $\tilde \beta$, $\tilde Q$ lies within $D(\tilde x, 81e)$. All
of interior angles lie between $\pi/4$ and $3\pi/4$. Note also that when $\tilde \alpha$ is invariant under the involution
of $\tilde E_+$, so too is $\tilde Q$.

The two sides $\tilde \beta$ and $\tilde \beta'$ of $\tilde Q$ have the same image in
$\tilde B_{E+}$. Hence, we can insert a rectangle $\tilde R$ between these two sides of $\tilde Q$, that is vertical in $N(\tilde B_{E+})$ in the
sense that it is a union of arcs that lie in fibres of $N(\tilde B_{E+})$. The union of $\tilde Q$ and $\tilde R$ is an annulus $\tilde C$.

When $D(\tilde x, 256e)$ is invariant under the involution of $N(\tilde B_{E+})$, then $\tilde x$ and $\tilde x'$ are fixed points. Hence in this case,
$\tilde \alpha$ is fixed pointwise by the involution and so $\tilde C$ is invariant under the involution.

\begin{lemma}
$\tilde C$ is embedded in $N(\tilde B_+)$.
\end{lemma}

\begin{proof}
The annulus $\tilde C$ consists of the parallelogram $\tilde Q$ and a vertical rectangle $\tilde R$.
Both $\tilde Q$ and $\tilde R$ are embedded. So the only way that $\tilde C$ could fail to be embedded
is if there is a point $z_1$ in $\tilde Q$ and a point $z_2$ in the interior of $\tilde R$ that coincide
in $N(\tilde B_{E+})$. Let $z_3$ be a point in $\tilde R$ that lies in the same fibre as $z_2$ and that also lies in $\tilde Q$.
Then $r_{\tilde Q}(z_3)$ is a point lying between $z_3$ and $z_2$ (possibly equalling $z_2$ but not $z_3)$.
Hence, $r_{D(z_3, 81e)}(z_3)$ also lies between $z_3$ and $z_2$. So the number of points
of $\tilde E_+$ strictly between $z_3$ and  $r_{D(z_3, 81e)}(z_3)$ is less than $m$, the constant from Proposition \ref{Prop:ProductRegion},
which is a contradiction.
\end{proof}

\begin{lemma}
\label{Lem:AnnulusDisjointFromTranslates}
$\tilde C$ is disjoint from its covering translates under the action of the covering group for $N(\tilde B_{E+}) \rightarrow N(B_{E+})$.
\end{lemma}

\begin{proof}
Suppose that on the contrary, $\tilde C$ intersects some image $g(\tilde C)$ under some non-trivial covering transformation $g$.
This covering translate is made up of a covering translate $g(\tilde Q)$ of $\tilde Q$ and 
a covering translate $g(\tilde R)$ of $\tilde R$. Let $z$ in $\tilde C$ be a point of intersection between
$\tilde C$ and $g(\tilde C)$. There are several possible cases to consider.

Suppose first that $z$ lies in $\tilde Q$ and $g(\tilde Q)$. Then $g^{-1}(z)$ lies in $\tilde Q$,
and there is a path $\tilde p$ in $\tilde Q$ from $g^{-1}(z)$ to $z$. This projects
a closed loop in $N(B_{E+})$ with non-trivial monodromy. However,
this path lies in the image of $\tilde Q$, which lies in $D(x,81e)$, and 
every closed loop in this grid has trivial monodromy.
This is a contradiction.

Now suppose that $z$ lies in $\tilde Q$ and $g(\mathrm{int}(\tilde R))$. Now $g(\tilde R)$ consists of a union of arcs,
each of which is a subset of a fibre. Let $z_2$ be an endpoint of the arc containing $z$, with $z_2$ lying in $g(\tilde Q)$. 
So, $g^{-1}(z_2)$ lies in $\tilde Q$. Let $\tilde p$ be a geodesic
in $\tilde Q$ joining $z$ to $g^{-1}(z_2)$. Let $p$ be the image of $\tilde p$ under the
projection map to $N(B)$. Then $p$ is a path in $D(x,81e)$ with close endpoints, contradicting
Proposition \ref{Prop:NoCloseEndpoints}. An analogous argument also rules out the case
where $z$ lies in $g(\tilde Q)$ and $\mathrm{int}(\tilde R)$.

The final case is where $z$ lies in $\mathrm{int}(\tilde R)$ and $g(\mathrm{int}(\tilde R))$.
However, if we move $z$ along the fibre it lies in, then we can find a point $z'$
of intersection between $\tilde C$ and $g(\tilde C)$ that lies in $\tilde Q$ or $g(\tilde Q)$ (or both).
This has been ruled out in the previous cases.
\end{proof}

Thus, $\tilde C$ projects to an embedded annulus $C$ in $N(B_{E+})$. This is either disjoint from its
image under the involution of $N(B_{E+})$ or invariant under this involution.

\subsection{Return to the annular case}

In Lemma \ref{Lem:GeodesicOrGrid}, we showed that either there is a closed geodesic 
in $E_+ - N_{1472e}(W)$ with length at most $736e$ or there is
a grid $D(x, 256e)$ in $E_+ - N_{1472e}(W)$ with side length $512e+1$. We have so far focused on the second case.
However, we now consider the first case. 

Let $\alpha$ be a simple closed geodesic in 
$E_+ - N_{1472e}(W)$ with length at most $736e$, such
that $N_{1472e}(\alpha)$ is either invariant under the involution or disjoint from its image
under the involution.

The following is a version of Lemma 9.9 in \cite{Lackenby:PolyUnknot},
which essentially the same proof, which is omitted.

\begin{lemma} 
Suppose that $\alpha$ is a simple closed
geodesic in $E_+ - N_{1472}(W)$ with length at most $736e$. Suppose also
that the component of $E_+$ containing $\alpha$ is not a torus 
with binding weight at most $23552e^2$. Then for 
all $r \leq 16e$, $N_r(\alpha)$ is a Euclidean annulus
with core curve $\alpha$.
\end{lemma}

Note that the monodromy along $\alpha$ is in $SO(2)$ because $E_+$ is orientable.
Note also that the monodromy preserves the unit tangent vector to $\alpha$ since $\alpha$
is a closed geodesic. So the monodromy along $\alpha$ is in fact trivial. Hence, the inverse
image of $\alpha$ in $\tilde E_+$ is a disjoint union of copies of $\alpha$. Let $\tilde \alpha$
be one such component.

Let $\tilde x$ be a point on $\tilde \alpha$, and let $\tilde \beta$ be a geodesic through $\tilde x$
that is orthogonal to $\tilde \alpha$ and that runs for distance $16e$ in both directions from $\tilde x$.
Let $\tilde Q$ be the rectangle that results from cutting $N_{16e}(\tilde \alpha)$ along
$\tilde \beta$. Let $Q$ be its image in $E_+$. We refer to the edges of $\tilde Q$ that
are identified in $\tilde E_+$ by $\tilde \beta$ and $\tilde \beta'$. Note that $\tilde Q$
is either invariant under the involution of $N(\tilde B_{E+})$ or disjoint from its image under
the involution.

We now consider the two cases of Lemma \ref{Lem:GeodesicOrGrid} simultaneously.

\subsection{Another vertical rectangle}
The following was proved in Section 9.8 of \cite{Lackenby:PolyUnknot}.

\begin{lemma}
There is a translate $\tilde \gamma$ of $\tilde \alpha$ that starts on $\tilde \beta$ and ends on $\tilde \beta'$
and a map $\tilde \gamma \times [0,K+1] \rightarrow N(\tilde B_{E+})$, for some integer $K \geq 0$, with the following properties:
\begin{enumerate}
\item the restriction to $(\tilde \gamma - \partial \tilde \gamma) \times [0,K+1)$ is an embedding;
\item if $x_1$ and $x_2$ are the start and endpoints of $\tilde \gamma$, then the restriction to each $x_i \times [0,K+1)$
is an embedding;
\item the images of $(x_1, y_1)$ and $(x_2, y_2)$ coincide if and only if $y_1 - y_2$ is the integer $m$ from 
Proposition \ref{Prop:ProductRegion};
\item the image of $\tilde \gamma \times \{ 0 \}$ is $\tilde \gamma$;
\item the inverse image of $\tilde E_+$ is $\tilde \gamma \times ([0,K+1] \cap \mathbb{Z})$;
\item for each point $z$ in $\tilde \gamma$, the image of $\{ z \} \times [0,K+1]$ is a subset of
a fibre of $N(\tilde B_{E+})$;
\item the image of  $\tilde \gamma \times [0,K+1)$ is disjoint from $\tilde Q$;
\item there is a point $y$ in the interior of $\tilde \gamma$ such that the image of $y \times \{ K+1 \}$
is equal to a point $y'$ in $\tilde Q$.
\end{enumerate}
Moreover when $\tilde \alpha$ is invariant under the involution of $N(\tilde B_{E+})$, then so
are $\tilde \gamma$, $y$ and $y'$.
\end{lemma}

Let $\tilde V$ be the image of $\tilde \gamma \times [0,K+1]$ in $N(\tilde B_{E+})$.
By construction, its interior is disjoint from $\tilde C$.

\begin{lemma}
\label{Lem:CCupVEmbedded}
$\tilde C \cup \tilde V$ is disjoint from its covering translates.
\end{lemma}

\begin{proof}
Suppose that, on the contrary, the image of $\tilde C \cup \tilde V$ under some non-trivial covering
transformation $g$ has non-empty intersection with $\tilde C \cup \tilde V$. Let $z$
be a point in the intersection.

We know by Lemma \ref{Lem:AnnulusDisjointFromTranslates} that $g(\tilde C) \cap \tilde C$
is empty. So, there are three remaining cases: $z$ can lie in $g(\tilde C) \cap \tilde V$,
$g(\tilde V) \cap \tilde C$ or $g(\tilde V) \cap \tilde V$. In the third case, we can slide $z$
vertically in $\tilde V$ until it lies in $g(\tilde C)$ or $\tilde C$. So if the third case holds, then
so does one of the first two cases. Furthermore, the first two cases are essentially identical.
So, we may assume that $z \in g(\tilde C) \cap \tilde V$.

Consider a horizontal path $\tilde q$ in $\tilde V$ starting at $z$, 
running in the direction of $\tilde \alpha$. (See Figure \ref{Fig:CCupV}.) Its image $g^{-1}(\tilde q)$ starts at $g^{-1}(z) \in \tilde C$ and runs along a
horizontal path in $\tilde Q$. Since we have applied a non-trivial covering transformation, this path $g^{-1}(\tilde q)$
has had a non-trivial element of $O(2)$ applied to it. By the time we have reached the end of $\tilde q$,
we reach a point that is in the same fibre as the start of $\tilde q$ and there are $m$ points of $\tilde E$ directly between the start and
end of $\tilde q$,
where $m$ is the integer from Proposition \ref{Prop:ProductRegion}. Hence, the same is true of $g^{-1}(\tilde q)$.
But $g^{-1}(z)$ lies in the grid $D(\tilde x, 81e)$. So, by Proposition \ref{Prop:CovariantConstant},
the first point of $D(g^{-1}(z), 81e)$ that lies in the same fibre as $g^{-1}(z)$ in the specified transverse
direction differs from $g^{-1}(z)$ by translation in the direction of $\tilde \alpha$. So, $g^{-1}(\tilde q)$
must run in the same direction as $\tilde \alpha$. Hence, $g$ must preserve the $\tilde \alpha$ direction
and the specified transverse orientation on $N(\tilde B_{E+})$. But this implies that $g$ represents the
trivial element of $O(2)$, and hence $g$ was the identity covering transformation. This is a contradiction.
\end{proof}

\begin{figure}[h]
\includegraphics[width=3in]{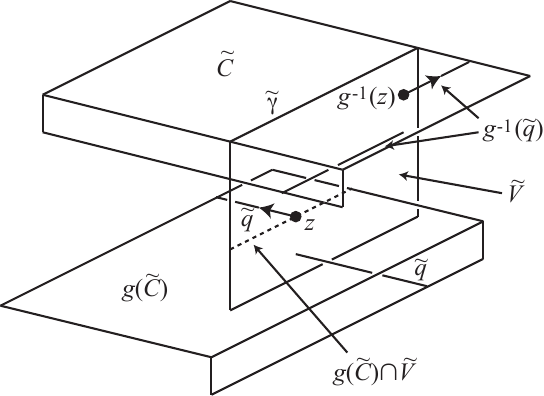}
\caption{An arrangement from the proof of Lemma \ref{Lem:CCupVEmbedded}}
\label{Fig:CCupV}
\end{figure}

\subsection{An embedded torus in $N(B_+)$}
\label{Sec:EmbeddedTorus}

Let $\gamma$ be the image of $\tilde \gamma$ in $Q$. Let $\gamma'$ be its translate
that goes through $y'$. Let $\delta$ be the subset of $\beta$ that lies between $\beta \cap \gamma$
and $\beta \cap \gamma'$. Define $\delta' \subset \beta'$ similarly.

Then $\gamma \cup \delta \cup \gamma' \cup \delta'$ forms an embedded parallelogram $Q'$
in $N(B_{E+})$. The sides $\delta$ and $\delta'$ follow the same itinerary in $N(B_{E+})$,
and have a portion of the vertical rectangle $R$ between them. Thus,
these close up to form an embedded annulus.

The opposite sides of this annulus have the image of $\tilde V$ between them.
Attaching $\tilde V$ to these two boundary components gives an embedded torus $T$ in
$N(B_{E+})$. (See Figure \ref{Fig:TorusConstruction}.)

Let $\sigma$ be a curve in $T$ that is the union of a parallel copy of $\gamma$ and a fibre in $R$.

\begin{figure}[h]
\includegraphics[width=4in]{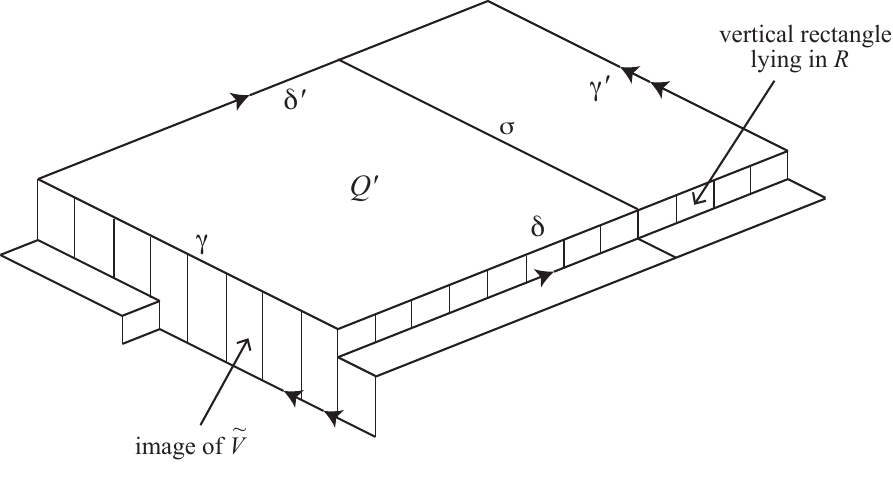}
\caption{The torus $T$}
\label{Fig:TorusConstruction}
\end{figure}

This torus $T$ is either disjoint from its image under the involution of $N(B_{E+})$ or invariant under the involution.
In the case where $T$ is disjoint from its image, let $T'$ be the union of $T$ and this image. When $T$ is
invariant under the involution, let $T' = T$.

\subsection{Removing a multiple of the summand $T'$}

This torus or tori $T'$ is a summand of $E_+$, because of the following lemma, which is Lemma 8.2 in \cite{Lackenby:PolyUnknot}.

\begin{lemma}
\label{Lem:WeightAndSummation}
Let $S_1$ and $S_2$ be normal surfaces carried by a branched surface. Suppose
that in each patch of the branched surface, the weight of $S_2$ is at least the weight of $S_1$. Then $S_1$ is a
summand of $S_2$.
\end{lemma}

In the case where $\tilde x$ lies in the grid $D(\tilde x, 256e)$,
let $m$ be the number of sheets of $\tilde E_+$ strictly between $\tilde x$ and $r_{D(\tilde x, 81e)}(\tilde x)$. 
(This is the constant from Proposition \ref{Prop:ProductRegion}.)
In the case where $\tilde \alpha$ is a closed geodesic, let $m$ be zero.

\begin{lemma}
$(m+1)T'$ is a summand of $E_+$.
\end{lemma}

\begin{proof}
When $m = 0$, this is a consequence of Lemma \ref{Lem:WeightAndSummation}. So let us assume that $\tilde x$ lies in the grid $D(\tilde x, 256e)$.
By Lemma \ref{Lem:WeightAndSummation}, it suffices to show that $(m+1)T'$ is carried by the branched surface $N(B_{E+})$ and
that in each patch, the weight of $E_+$ is at least that the weight of $(m+1)T'$. At each point $z$ in $\tilde Q$, $r_{D(z,81e)}(z)$ and $z$
lie in the same patch of $N(\tilde B_{E+})$. If any part of $\tilde E_+$ lies between them, this also lies in the same patch. Just
as $\tilde Q$ was used to create the torus $T'$, so these parts of $E_+$ therefore may be used to create the tori
$(m+1)T'$, which is then a summand of $E_+$. (See Figure \ref{Fig:ParallelSummands}.)
\end{proof}

\begin{figure}[h]
\includegraphics[width=2in]{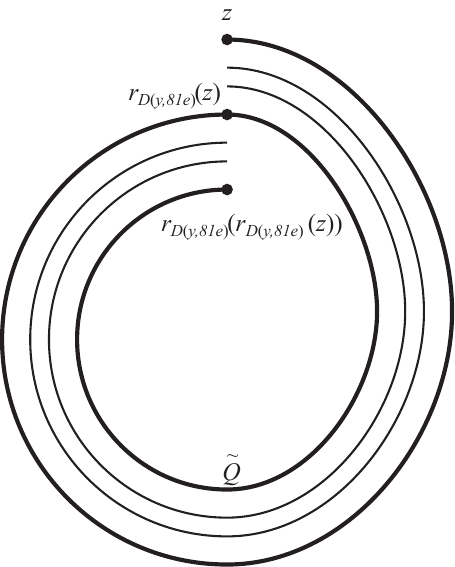}
\caption{Other parts of $E_+$ may lie between $z$ and $r_{D(z,81e)}(z)$ but these remain parallel to $Q$}
\label{Fig:ParallelSummands}
\end{figure}

Let $\overline E_+$ be the surface carried by $B_{E+}$ satisfying $E_+ = \overline E_+ + (m+1)T'$.

\subsection{Dehn twisting along $T'$}

\begin{proposition}
\label{Prop:Twisting}
The surface $\overline E_+$ is obtained from $E_+$ by Dehn twisting some number of times
around $T'$.
\end{proposition}

We will need the following.

\begin{lemma}
\label{Lem:DesumDehnTwist}
 Let $\overline{S}$ be a surface carried by a branched surface $B$, and let $T$ be a torus carried
by $B$. Suppose that $\overline{S} \cap T$ consists of $k$ curves that are essential on $T$ and so that the
summation directions are all coherently oriented around $T$. Then for any multiple $n$ of $k$,
$\overline{S} + nT$ is obtained from $\overline{S}$ by Dehn twisting some number of times around $T$.
\end{lemma}

\begin{proof}
Consider a simple closed curve on $T$ that intersects each curve of $\overline{S} \cap T$ once.
Then in a regular neighbourhood of this curve, the surfaces $\overline{S}$ and $T$ sit as shown in Figure
\ref{Fig:TorusSummation}. It is now clear
that, when $n$ is a multiple of $k$, then $\overline{S} + nT$ is obtained from $\overline{S}$ by Dehn twisting about $T$ $n/k$ times.
\end{proof}

\begin{figure}[h]
\includegraphics[width=2in]{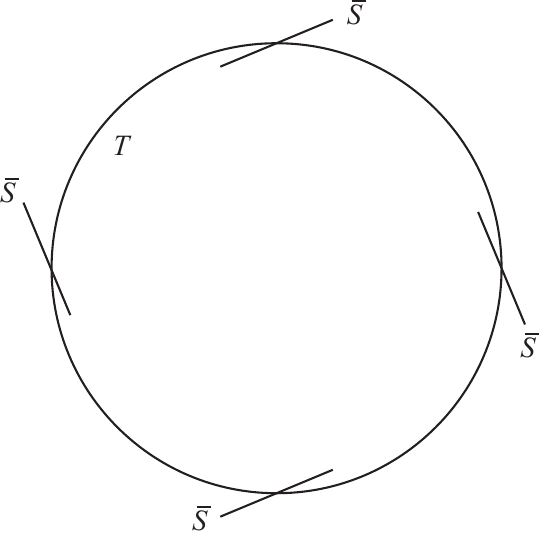}
\caption{The case $k = 4$}
\label{Fig:TorusSummation}
\end{figure}

Thus, to prove Proposition \ref{Prop:Twisting}, we just need to show that $\overline E_+ \cap T'$ consists of essential curves in $T'$ and that, for each component $T$ of $T'$, $|\overline E_+ \cap T|$ is a multiple of $m+1$. 
The following lemma will be used to deal with the number of these curves.

\begin{lemma}
\label{Lem:NumberCurvesMultiples}
Let $\zeta'$ be a collection of parallel essential simple closed curves on a torus $T$.
Let $\zeta$ be a simple closed curve of $T$ with slope other than that of $\zeta'$ and that intersects each component of $\zeta'$
in points of the same sign. Then $|\zeta'|$ divides $|\zeta \cap \zeta'|$.
\end{lemma}

\begin{proof}
The assumption that intersects each component of $\zeta'$
in points of the same sign implies that the number of such points is just the modulus of the intersection between
the slopes. Hence, $|\zeta \cap \zeta'|$ is just $|\zeta'|$ times this intersection number.
\end{proof}

\begin{proof}[Proof of Proposition \ref{Prop:Twisting}]
Let $\sigma$ be the curve on the torus $T$ defined in Section \ref{Sec:EmbeddedTorus}. 
We see that all the intersection points between $\sigma$ and $\overline{E}_+ \cap T$ have the same sign and respect the direction of the normal summation. 
Hence, $\zeta = \sigma$ and $\zeta' = \overline{E}_+ \cap T$
satisfy the hypotheses of Lemma \ref{Lem:NumberCurvesMultiples}. So for each component $T$ of $T'$,
$|E_+ \cap T|$ divides $|\sigma \cap E_+|$, which is $m+1$. Hence, by Lemma \ref{Lem:DesumDehnTwist},
$\overline  E_+$ is obtained from $E_+$ by Dehn twisting some number of times about $T'$.
\end{proof}

\subsection{The case where $T$ is disjoint from its image}

Suppose now that $T$ is disjoint from its image under the involution of $N(B_{E+})$. Hence, it restricts to
a torus, which we will also call $T$, in $N(B_E)$. This is disjoint from the boundary of $M$.
The Dehn twist defined in Proposition \ref{Prop:Twisting} restricts to a Dehn twist about $T$ in $M$.
Let $\overline{E}$ be the result of removing $(m+1)$ summands of $T$ from $E$, viewed as a 
surface carried by $B_E$. Then $\overline{E}$ is obtained from $E$ by applying some power $h$ of this Dehn twist
about $T$. Let $S_2'$ be the surface $(S' \cut N(B_E)) \cup \overline E$. Then $S'_2$ is obtained
from $S'$ by applying $h$. The surface $S \cut S'$ may intersect $N(B)$, but near $T$, it runs
parallel to $S'$. Hence, we may also obtain a surface $S_2$ from $S$ by applying $h$.
Similarly, we can obtain the 3-manifold $h(N')$ and handle structure $h(\mathcal{H}')$.
So,  $(N, X, h(Y), P \cap N, \mathcal{H}, h(S))$ and $(h(N'), X, h(Y'), P \cap N', h(\mathcal{H}'), h(S'))$ are nested admissible envelopes.
(See Figure \ref{Fig:Torus}.)
Note that the binding weight $w_\beta(h(S'))$ is strictly less than the binding weight $w_\beta(S')$,
since we have removed $(m+1) w_\beta(T)$, which is positive. This contradicts one of the
assumptions of Theorem \ref{Thm:DistanceToBoundaryEuclidean}. So, the theorem is proved in
this case.

\begin{figure}[h]
\includegraphics[width=4in]{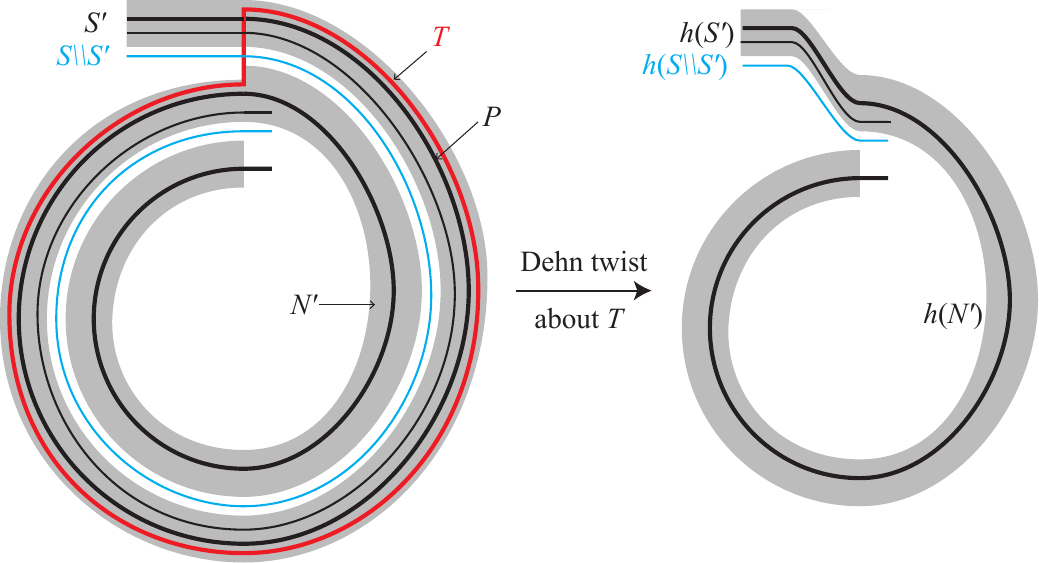}
\caption{Dehn twisting along the torus $T$ in $M$}
\label{Fig:Torus}
\end{figure}

\subsection{The case where $T$ is invariant under the involution}
In this case, $T$ restricts to an annulus $A$ in $N(B_E)$. The Dehn twist about $T$ restricts to 
a Dehn twist about $A$. Since we doubled $N(B_E)$ along $\partial M' \cap N(B)$, this annulus
$A$ has boundary lying in $\partial M'$.

When $T$ was disjoint from its image under the involution, there may have been parts of $S \cut S'$
lying between the sheets of $E$ near $T$. However, in the current case, this does not arise, for
the following reason. If there was any part of $S \cut S'$ between the sheets of $E$ near $T$,
then they would run parallel to $E$, and so would have to intersect $\partial M'$. Consider such a
point of intersection. Above and below this point would be tiles of $S'$ that are of the same
approximate type. However, these tiles intersect $\partial M'$, and so must actually be parallel in $N'$,
and so there cannot in fact be $S \cut S'$ between them.

If any component of $\partial A$ does not lie in $\partial M$,
then we may extend $A$ to an annulus $\tilde A$ properly embedded in $M$ by attaching
a collar in the wedge region $N \cut N'$. This annulus
$\tilde A$ lies in $N$ and is carried by the branched surface of which $N$ is thickening.
The Dehn twist about $A$ extends to a Dehn twist $h$ about $\tilde A$.
Again, $(N, X, h(Y), P \cap N, \mathcal{H}, h(S))$ and $(h(N'), X, h(Y'), P \cap N', h(\mathcal{H}'), h(S'))$ are nested admissible envelopes.
Again, the binding weight $w_\beta(h(S'))$ is strictly less than the binding weight $w_\beta(S')$,
since we have removed $(m+1) w_\beta(A)$, which is positive. Again this proves 
Theorem \ref{Thm:DistanceToBoundaryEuclidean} in this case.

\section{The thin part of the surface}
\label{Sec:Thin}

In Section \ref{Sec:CutVertices}, the cut vertices of $S'$ were defined. These are the non-manifold points of $S'$. Modifications
to $S'$ were given, but it was important that these took place away from the cut vertices. As a result, the 
surface $S'_{\mathrm{cut}}$ was defined in Section \ref{Sec:CutVertices}, 
which is obtained from $S'$ by cutting along its cut vertices. The remnants of
the cut vertices in $S'_{\mathrm{cut}}$ are not forgotten. We can view the modifications in Section \ref{Sec:Parallelism} as being
applied to $S'_{\mathrm{cut}}$. We will see later that we can apply these modifications until the binding weight of
the non-bigon components of $S'_{\mathrm{cut}}$ is small, in the sense that it is bounded above by a polynomial function 
of the initial arc index of the link. However, this does not imply that the binding weight of $S'$ is small,
since there may be parts of $S'$ that consist of long lines of bigons joined up in a row. We call these
the `thin' parts of $S'$. In this section, we explain how to modify these thin parts so that they too have small
binding weight.

\subsection{The thin part}
\label{Subsec:ThinDefn}

The \emph{thin part} of the surface $S'$ consist of the union of the bigon components of $S'_{\mathrm{cut}}$, together
with the cut vertices of $S'$ that are incident to two bigons and no other tiles. See Figure \ref{Fig:ThinPart}.
We denote the thin part by $S'_{\textrm{thin}}$.

\begin{figure}[h]
\includegraphics[width=3in]{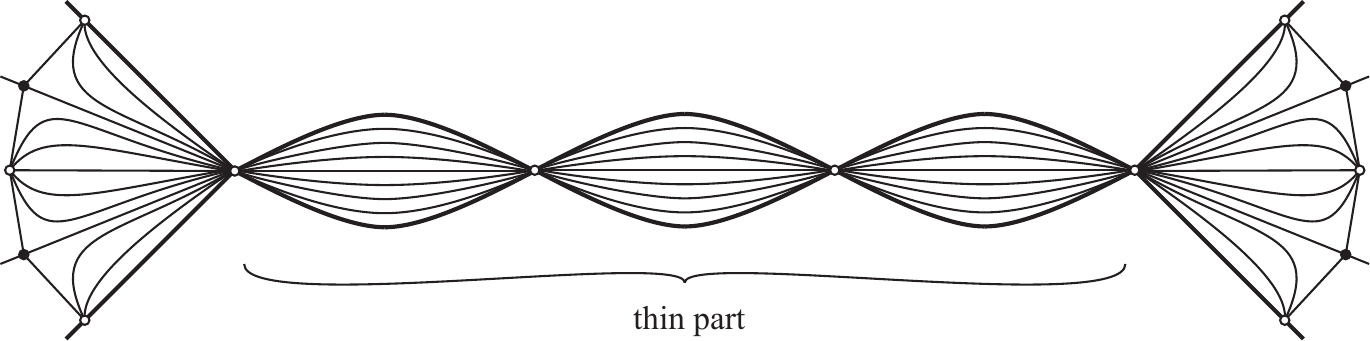}
\caption{The thin part of $S'$}
\label{Fig:ThinPart}
\end{figure}

Recall that $S'$ is properly embedded in the manifold $N'$. Let $M'$ be obtained from $N'$ by attaching $M \cut N$.
Thus, $M'$ is obtained from $M$ by inserting wedges, but we suppose that when a wedge is inserted in a way that
creates a cut vertex, then we do not quite insert the full wedge, in order to maintain $M'$ as a 3-manifold.

The following is the main result of this section.

\begin{proposition}
\label{Prop:IsotopeThin}
Suppose that the component of $M'$ containing $S'$ is not an $I$-bundle over a compact surface
with $P$ lying in the vertical boundary as a collection of essential simple closed curves.
Suppose also that if the component of $M'$ containing $S'$ has any JSJ annuli,
then $S'$ is the JSJ annuli.
Let $e_{\textrm{bigon}}$ be an upper bound for the number of $\calH'$-parallelism types of bigons in $S'$.
There is a pattern-isotopy of $S'$ taking it to a nearly admissible surface $S'''$ with binding weight
at most $$(2 e_{\textrm{bigon}} + w_\beta(\partial M') + 1) \, w_\beta(S' \cut S'_{\textrm{thin}}).$$
\end{proposition}

\begin{remark}\label{Rem:BigonTypes}
This proposition refers to the number of $\calH'$-parallelism types of bigons in $S'$. This is at most
the number of approximate star types of tiles of $S'$. This is because when two bigons have the same approximate star type,
then they are actually parallel in $\calH'$, as their vertices lie on $\partial M'$.
\end{remark}

\subsection{The thin bundle}
\label{Subsec:ThinBundle}

The \emph{thin bundle} $\calB$ is a 3-dimensional submanifold of $M'$, constructed as follows. Take a regular 
neighbourhood of the thin part of $S'$. Between any two bigon tiles of $S'_{\textrm{thin}}$ that are $\calH'$-parallel, insert
the space between them. Between any two vertices of the thin part of $S'$ that are $\calH'$-parallel, insert
the space between them. This is an $I$-bundle over a surface, for the following reason. Each bigon tile 
of $S'_{\textrm{thin}}$ is a product $I \times I$, where the second factor is given by the $\theta$ value. A regular
neighbourhood of each vertex in the thin part of $S'$ is again a product $I \times I$, when this
vertex has been suitably enlarged so that the surface is properly embedded in $M'$. When we 
form regular neighbourhoods of these bigon tiles and these vertices, the result is still an $I$-bundle,
and it remains so when we attach the spaces between them.

\subsection{Annular simplifications}
\label{Subsec:AnnularSimps}
We now describe a modification to $M'$.
Suppose that there is a clean annulus $A$ properly embedded in
$M'$. Suppose also that there is a clean annulus $A'$ in $\partial M'$ such that $\partial A = \partial A'$. Assume
that $A \cup A'$ bounds a 3-manifold $G$ such that
\begin{enumerate}
\item either $G$ is product region between $A$ and $A'$, or $G$ lies in a 3-ball in $M'$;
\item the intersection between $G$ and the thin bundle $\calB$ is a union of components of $\calB$;
\item $A$ is a vertical boundary component of $\calB$.
\end{enumerate}
Removing the interiors of $G$ and $A'$ from $M'$ is called an \emph{annular simplification}.

Note that an annular simplification can be achieved by an isotopy of an earlier surface in the partial hierarchy, as follows.
Since $A'$ is disjoint from the boundary pattern, it lies within some surface $S_i$ of the partial hierarchy.
If we replace $S_i$ by the surface $(S_i \cut A') \cup A$, then this new surface is isotopic to $S_i$.
All the vertices of $A$ are vertices of $A'$, and so this does
not increase the binding weight of the hierarchy. This isotopy of the surface has the effect of removing
$G$ from $M'$. However, $G$ is added to the part of the 3-manifold on the other side of $S_i$. Note
that when constructing the partial hierarchy, we always ensured that it was
adequately separating, and we therefore guaranteed that on the two sides of
$A'$, there were different components of $M'$. Thus, whenever we perform an annular
simplification to the component of $M'$ containing $S'$, we reduce the number of
components of $\calB$ in that component of $M'$. 

So let us suppose that no further annular simplifications can occur to the component of $M'$ containing $S'$.

\subsection{The generalised thin bundle}
\label{Subsec:GeneralisedThin}

By using arguments analogous to ones introduced in \cite{Lackenby:Composite}, we will now 
enlarge $\calB$ to a bigger $I$-bundle $\calB_+$, which we call the \emph{generalised thin bundle}.
Throughout this enlargement process, $\calB_+$ will have the following properties:
\begin{enumerate}
\item $\calB_+$ is an $I$-bundle over a compact surface;
\item the horizontal boundary of $\calB_+$ is $\calB_+ \cap \partial M'$, and is clean;
\item the vertical boundary of $\calB_+$ is a union of vertical boundary components of the thin bundle.
\end{enumerate}

\begin{proposition}
\label{Prop:MaximalBundle}
Suppose that $\calB_+$ is maximal with the above properties. Suppose also that the component of
$M'$ containing $S'$ admits no annular simplifications. Then each component of $\calB_+$ 
either is an $I$-bundle over a disc or has incompressible vertical boundary.
\end{proposition}

\begin{proof}
Let $\calB_+'$ be the union of the components of $\calB_+$ that are not $I$-bundles over discs. We will
show that the vertical boundary of $\calB_+'$ is incompressible.

Suppose that, on the contrary, it has a compression disc $D$. Let $A$ be the vertical boundary component containing 
$\partial D$. By the definition of $\calB_+'$, $D$ does not lie entirely in $\calB_+'$. Its interior is disjoint from $\calB_+'$ (by the definition of a compression disc), but it may intersect $\calB_+ - \calB_+'$. 

Now, $A$ compresses along $D$ to give two discs $D'_1$ and $D'_2$ properly embedded in $M'$. Their boundaries are clean. Since $(M',P)$ is irreducible and
boundary-irreducible, $D'_1$ and $D'_2$ are parallel to discs $D_1$ and $D_2$ in $\partial M' - P$, via 3-balls $B_1$ and $B_2$. There are 
two cases to consider: where $B_1$ and $B_2$ are disjoint and where they are nested.

Let us suppose first that they are disjoint. Then, $A \cup D_1 \cup D_2$ bounds a 3-ball $B$, by the irreducibility of $M'$. Since
the interior of $D$ is disjoint from $\calB_+'$, this ball $B$ does not lie in $\calB_+'$. So, we may extend the $I$-bundle structure of  $\calB_+'$ over $B$, contradicting the maximality of $\calB_+$. 

Let us now suppose that $B_1$ and $B_2$ are nested; say that $B_2$ lies in $B_1$. Let $A'$ be $D_1 \cut D_2$. Then, $A'$ is a clean annulus in $\partial M'$ such that $\partial A = \partial A'$. Let $G$ be the 3-manifold bounded by $A \cup A'$. This lies in the 3-ball $B_1$. So, $M'$ admits an 
annular simplification, which is a contradiction.

Thus, the vertical boundary of $\calB_+'$ is incompressible. 
\end{proof}

\begin{proof}[Proof of Proposition \ref{Prop:IsotopeThin}]
Let $\partial_v \calB_+$ be the vertical boundary of the generalised thin bundle. It is a union of components of the vertical boundary
of $\calB$. 

We claim that the binding weight of $\partial_v \calB_+$ is at most $2e_{\textrm{bigon}}$, where $e_{\textrm{bigon}}$ is an upper bound 
for the number of $\calH'$-parallelism  types of bigons in $S'$. It is composed of a union of bigon tiles of $S'$, and each such tile is
outermost in its $\calH'$-parallelism class. The claim follows immediately.

We will show that, after a pattern-isotopy, $S'$ intersects $\calB_+$ in discs that are vertical with respect to the $I$-bundle structure on $\calB_+$.
Furthermore, this does increase the binding weight of the surface outside of $\calB_+$ and it does not increase the number
of components of $S' \cap \partial_v \calB_+$.
First note that, because the vertical boundary of $\calB_+$ is equal to a union of vertical boundary components of
$\calB$, and $S'$ intersects $\mathcal{B}$ in bigons, we may perform a small isotopy of $S'$ so that afterwards it intersects $\partial_v \calB_+$ in vertical arcs. 
Suppose now that $S' \cap \calB_+$ has a boundary compression disc $D$ that intersects $\partial \calB_+$ in $\partial_h \calB_+$.
Since $S'$ is boundary-incompressible (with respect to the pattern), the arc $S' \cap D$ separates off a clean 
disc in $S'$. This and $D$ separate off a ball disjoint from the pattern, and we may isotope
$S'$ across this ball. This has the effect of reducing the number of components of $S' \cap \partial_v \calB_+$.
It also does not increase the binding weight of $S'$ outside of $\calB_+$.
Note also that $S' \cap \calB_+$ is incompressible. For suppose that it had a compression
disc $D$. Since $S'$ is incompressible, $\partial D$ bounds a disc $D'$ in $S'$.
This misses the boundary of $M'$, and hence does not intersect $\partial_v \calB_+$. But it must then
lie in $\calB_+$, which is a contradiction, proving the claim. 
Thus $S' \cap \calB_+$ is now incompressible and has no boundary compression disc
that intersects $\partial \calB_+$ in $\partial_h \calB_+$. The only such surfaces are, up to isotopy
preserving $\partial_h \calB_+$ and $\partial_v \calB_+$, vertical or horizontal in $\calB_+$. However, since $S' \cap \partial_v \calB_+$
is vertical, we deduce that $S' \cap \calB_+$ must be vertical, as required.

Let $\calB^a_+$ be those components of $\calB_+$ that are $I$-bundles over annuli, and where
both vertical boundary components are inessential in $M'$.
We now isotope $S'$ off $\calB_+ - \calB^a_+$. We do this first for the components of $\calB_+- \calB^a_+$ that are $I$-bundles over discs.
In any such $D^2 \times I$ component, its intersection with $S'$ is isotopic to $\alpha \times I$, where $\alpha$ is a collection of
arcs properly embedded in $D^2$. Pick an outermost such arc, and replace it by a disc parallel to 
a subset of $\partial_v \calB_+$. Repeating this process, we eventually isotope $S'$ off the $D^2 \times I$
components of $\calB_+$.

Now consider the components of $\calB_+ - \calB^a_+$ that are not $I$-bundles over discs. By Proposition \ref{Prop:MaximalBundle}, their vertical
boundary is a union of incompressible annuli. A component $B$ of $\calB_+ - \calB^a_+$ cannot have all its vertical boundary
components being inessential in $M'$, for the following reason. If any is inessential, then it is parallel to an annulus in $\partial M'$ that intersects
the pattern in a (possibly empty) union of core curves. If all the product regions between these annuli
and the vertical boundary are disjoint from the interior of $B$, then this component of $M'$
is an $I$-bundle, and the boundary pattern intersects each vertical boundary component in a (possibly empty) collection
of parallel core curves, which is contrary to assumption. On the other hand, if one of these product regions
contains $B$, then $\partial_vB$ is an essential subsurface of an annulus, and hence is a collection of annuli.
This implies that $B$ must be an $I$-bundle over an annulus, which is a contradiction. Hence,
 $\calB_+ - \calB^a_+$ is non-empty, then we deduce
that $M'$ contains a clean essential annulus. In this situation, we are assuming that $S'$ consists
of the JSJ annuli. Therefore, $S'$ can be pattern-isotoped off the incompressible vertical boundary of $\calB_+$. So,
there is a product region between $S'$ and this vertical boundary. We can pattern-isotope $S'$ across this product region,
so that part of it runs parallel to a subset of the vertical boundary. It is possible that this product region does
not lie in $\calB_+$, in which case, the new part of $S'$ lies in $\calB_+$. However, it remains the case that
$S'$ intersects $\calB_+$ in vertical discs, or possibly annuli parallel to some vertical boundary components of $\calB_+$.
It also remains the case that $S'$ is disjoint from the components of $\calB_+$ that are $I$-bundles over discs.
Hence, we may repeat this process until $S'$ is disjoint from the vertical boundary of $\calB_+$. If any
components of $S'$ are parallel to a vertical boundary component of $\calB_+$, they may be pattern-isotoped
off $\calB_+$.

At the end of this process, the resulting surface $S'''$ is disjoint from $\calB_+ - \calB_+^a$. In each component
of $B^a$ of $\calB_+^a$, the surface is vertical. If any component $S'''$ is inessential in $B_a$, we may isotope
it out of $B_a$, so that it runs parallel to a subset of the vertical boundary. The remaining essential components
of $S''' \cap B_a$ may be arranged so that they intersect each vertex in $\partial_h B_a$ at most once.
The resulting surface differs by an isotopy and some power of a Dehn twist about one of the boundary components
of $\partial_v B_a$. This Dehn twist is pattern-isotopic to the identity on $M'$.

Each component of  $S''' \cut S'$ in $\calB_+ - \calB_+^a$ runs 
parallel to a vertical subset of $\partial_v \calB_+$. It therefore has binding weight at most $2e_{\textrm{bigon}}$. 
Each such component starts and ends on a vertex of $S' \cut S'_{\textrm{thin}}$
that is not in its thin part. Hence, the total binding weight of $S''' \cut S'$ in $\calB_+ - \calB_+^a$ is at most $2e_{\textrm{bigon}} \, w_\beta(S' \cut S'_{\textrm{thin}})$.
In $\calB_+^a$, each component of $S''' \cut S'$ also starts and ends on a vertex of $S' \cut S'_{\textrm{thin}}$
that is not in its thin part. So the binding weight of those components is at most $w_\beta(\partial M') \, w_\beta(S' \cut S'_{\textrm{thin}})$.
The remainder of $S'''$ is $S' \cut S'_{\textrm{thin}}$, and so we obtained our required upper bound on the binding
weight of $S'''$.
\end{proof}

\section{The Euler characteristic argument}
\label{Sec:EulerCharacteristic}

\begin{theorem}
\label{Thm:NumberLowValVertices}
Let $(M,P)$ be a compact orientable irreducible boundary-irreducible 3-manifold with boundary pattern.
Suppose that $M$ is a subset of $S^3$ with $\partial M$ in alternative admissible form and $P$ in an arc presentation.
Let $S$ be a properly embedded surface in $M$ in generalised admissible form. Let $(N,X, Y, P \cap N, \mathcal{H}, S)$
and $(N',X, Y', P \cap N', \mathcal{H}', S')$ be nested admissible envelopes. Suppose that there is no homeomorphism
$h$ as in Theorem \ref{Thm:DistanceToBoundaryEuclidean}. Let $e$ be the number of approximate star types of tiles of $S'$. 
Let $S' \setminus S'_{\mathrm{thin}}$ be obtained from $S'$ by cutting along its cut vertices and then discarding any bigon tiles.
Let $S''$ be obtained from $S'_{\mathrm{thin}}$ by discarding any toral or annular components that lie in $E$ and have binding
weight at most $23552e^2$.
Let $S_+$ be the double of $S''$, which
inherits a singular foliation. Then the number
of non-cut vertices of $S_+$ with valence at most 3 is at least 
$$\frac{w_\beta(S'')}{31\pi (2392e)^2} + \frac{72}{31} \chi(S) - \frac{5}{31}|S \cap P|.$$
\end{theorem}

\begin{proof}
Recall that the cut-adjusted Euler characteristic $\chi_c$ was defined in Section \ref{Sec:CutVertices}
and that $\chi_c(S'') = \chi_c(S' \setminus S'_{\mathrm{thin}}) = \chi_c(S'_\mathrm{cut}) = \chi(S')$. Also define the \emph{cut-adjusted Euler characteristic} $\chi_c(S_+)$
of $S_+$ to be $\chi(S_+)$ minus its number of cut vertices. Then $\chi_c(S_+) = 2\chi_c(S'') = 2\chi(S')$.
Note that $S'$ is homeomorphic to $S$ plus possibly some clean discs, and so $\chi(S') \geq \chi(S)$.
Form a cell complex on $S_+$ by declaring that each vertex of $S_+$ is a 0-cell and each
generalised saddle is a 0-cell. Each separatrix is a 1-cell. Each tile is a 2-cell.
Let $E$ and $F$ be the edges and faces of this cell complex. Let $C$ be the cut vertices
of $S_+$. Since the vertices $V$ and generalised saddles $X$
form the 0-cells of this complex, then obviously
$$|V| + |X| - |E| + |F| - |C|= \chi_c(S_+) \geq 2 \chi(S).$$
Since each tile of $S_+$ has four separatrices in its boundary, $2|E| = 4|F|$, and so $|E| = 2 |F|$.
Each edge runs from a vertex to a generalised saddle, and so contributes 1 to $d(v)/2 + d(x)/2$ 
for some vertex $v$ and generalised saddle $x$, where $d$ is their valence.
So
$$|E| = \frac{1}{2} \sum_{v \in V} d(v) + \frac{1}{2} \sum_{x \in X} d(x).$$
Therefore, 
\begin{align*}
2 \chi(S) & \leq |V| + |X| - |E| + |F| - |C| \\
&= |V| + |X| - |E|/2 - |C| \\
& = \sum_{v \in V} \left (1 - \frac{d(v)}{4} \right ) + 
\sum_{x \in X} \left (1 - \frac{d(x)}{4}\right) - |C|.
\end{align*}
We now introduce the \emph{cut-adjusted valence} of a vertex $v$ in $S_+$ to be
$$d_c(v) = 
\begin{cases}
d(v) + 4 & \text{if } v \text{ is a cut vertex} \\
d(v) & \text{otherwise}.
\end{cases}
$$
So,
$$2 \chi(S) \leq \sum_{v \in V} \left (1 - \frac{d_c(v)}{4} \right ) + 
\sum_{x \in X} \left (1 - \frac{d(x)}{4}\right).$$
For each integer $k \geq 2$, let $v_k$ be the number of vertices in $S_+$ with cut-adjusted valence $k$. Then
$$3v_1 + 2 v_2 + v_3 =
\sum_{\scriptstyle{v \in V \atop d_c(v) < 4}} (4 - d_c(v)) \geq
8\chi(S) + \sum_{\scriptstyle{v \in V \atop d_c(v) > 4}} (d_c(v) - 4) + 
\sum_{x \in X} \left (d(x) - 4 \right).$$
Note that
$$\sum_{\scriptstyle{v \in V \atop  d_c(v) > 4}} (d_c(v) - 4) = 
\sum_{k > 4} \sum_{\scriptstyle{v \in V \atop d_c(v) = k}} (d_c(v) - 4)
= \sum_{k > 4} v_k (k-4) \geq \frac{1}{5} \sum_{k > 4} k v_k.$$
Similarly, because each generalised saddle has valence that is an even integer at least 4,
$$\sum_{x \in X} (d(x) - 4) \geq \frac{1}{3} \sum_{\scriptstyle  x \in X \atop \scriptstyle d(x) \not= 4} d(x).$$
So,
$$
3v_1 + 2 v_2 + v_3
\geq \left ( \frac{1} {5}  \sum_{k > 4} k v_k + \frac{1}{3} \sum_{\scriptstyle{x \in X \atop d(x) \not=4}} d(x) \right) 
+ 8 \chi(S). 
$$
In particular,
\begin{equation}
3v_1 + 2 v_2 + v_3 \geq \frac{1}{5} \sum_{k > 4} k v_k + 8 \chi(S) \geq \sum_{k > 4} v_k + 8 \chi(S),\tag{1}
\end{equation}
and
$$3 v_1 + 2 v_2 +  v_3 \geq \frac{1}{3} \sum_{\scriptstyle{x \in X \atop d(x) \not=4}} d(x) + 8 \chi(S).$$

We now wish to bound $\textrm{deg}(S_+ \setminus E_+)$ from above in terms of
the valence of the non-Euclidean saddles and vertices. Let $v_4^{E_+}$ be the number of 4-valent non-cut vertices in the interior of $E_+$. 
Each of these vertices in the interior of $E_+$ contributes $1$ to the area of $E_+$
and so by Lemma \ref{Lem:AreaEuclidean} and Theorem \ref{Thm:DistanceToBoundaryEuclidean},
$$v_4^{E_+} \leq \pi (2392e)^2 \mathrm{deg}(S_+ \setminus E_+).$$
The remaining 4-valent non-cut vertices lie in the copy of $S \cap P$ in $S_+$, and therefore
$$v_4 \leq v_4^{E_+} + |S \cap P|.$$
Each point in $S_+ \setminus E_+$ is a generalised saddle that is not a saddle, a vertex with cut-adjusted valence other than 4 or a point in
$S \cap P$ with cut-adjusted valence 4.
So, we deduce that
\begin{align*}
&\qquad \mathrm{deg}(S_+ \setminus E_+) \\
&\leq 4 |S \cap P| + \sum_{k \not=4} k v_k + 
\sum_{\scriptstyle{x \in X \atop d(x) \not=4}} d(x) \\
&\leq 4 |S \cap P| + 3v_1 + 2v_2 +  v_3 + \sum_{k >4} k v_k + 
\sum_{\scriptstyle{x \in X \atop d(x) \not=4}} d(x) \\
&\leq 4 |S \cap P| + 3v_1 + 2v_2 +  v_3 + 15v_1 + 10v_2 + 5v_3 - 40 \chi(S) + 9v_1 + 6v_2 +  3v_3 - 24\chi(S).
\end{align*}
Therefore, 
\begin{equation}
\begin{split}
27v_1 + 18 v_2 + 9 v_3 
&\geq \mathrm{deg}(S_+ \setminus E_+) + 64\chi(S) - 4 |S \cap P| \\
&\geq \frac{v_4^{E_+}}{\pi (2392e)^2} + 64\chi(S) - 4 |S \cap P| \\
&\geq \frac{v_4}{\pi (2392e)^2} + 64\chi(S) - 5|S \cap P|.
\end{split}
\tag{2}
\end{equation}
Adding (1) to (2), we deduce that
$$
31v_1 + 21 v_2 + 11 v_3 \geq 
\frac{v_4}{\pi (2392e)^2} + \sum_{k \not= 4} v_k +72\chi(S) -5  |S \cap P| .$$
Therefore,
$$31v_1 + 21 v_2 + 11 v_3 \geq \frac{|V|}{\pi (2392e)^2} + 72 \chi(S) - 5|S \cap P|.$$
Finally note that each vertex in $S''$ gives rise to one or two
vertices in $V$, and each vertex in $V$ arises in this way. Therefore,
$$w_\beta(S'') \leq |V| \leq 2 w_\beta(S'').$$
Hence, the number of non-cut vertices of $S_+$ with valence at most 3 is at least
$$\frac{w_\beta(S'')}{31\pi (2392e)^2} +  \frac{72}{31}\chi(S) - \frac{5}{31}|S \cap P|.$$
\end{proof}

\section{Modifications in the case of a broken section}
\label{Sec:BrokenSection}

In the hierarchy described in Section \ref{Sec:Hierarchies}, the final surface might be a broken section.
Recall that this is where the exterior of the partial hierarchy is a collection of 3-balls and solid tori
with clean longitudes. The solid tori patch together to form a circle bundle. The intersection between any
two such solid tori is a collection of annuli that are unions of circle fibres. The intersection of any three such solid
tori is a union of finitely many fibres. We let $M$ denote the circle bundle, we let $A$ be these annuli and we
let $C$ be the union of the fibres meeting three solid tori and the fibres meeting $\partial M$ and two solid tori.
The broken section is obtained from a section of $M$ by perturbing it a little
in each component of $M \cut A$. Thus in each component of $M \cut A$, it is a meridian disc and the discs
are all disjoint from each other. Let $P$ be the boundary pattern for $M$. 

In this section, we present the minor modifications to our previous arguments that
are needed to deal with broken sections.

Some components of $M$ are dealt with in a different way from others. We let $M_{\text{st}}$ be the
union of the solid toral components of $M$, and let $M_{\text{nst}}$ be the union of the 
remaining components.

\subsection{The handle structure $\mathcal{K}$}

We are considering the situation where the penultimate manifold of the hierarchy is $M \cut A$ plus some 3-balls.
As previously, we create the polyhedral decomposition $\Delta$ for the exterior of $K$ that
contains the partial hierarchy $H$ has a subcomplex. This also collapses to a triangulation
$\mathcal{T}$. This restricts to a triangulation
$\mathcal{T}_M$ for $M$. The annuli $A$ are simplicial in $\mathcal{T}$ and $\mathcal{T}_M$. 
As in Section \ref{Subsec:HierarchyExteriorHS}, we enlarge $\mathcal{T}_M$ to $\hat \Delta_M$, and we let $\mathcal{K}$
be the dual handle structure for $M$. Since $A$ was simplicial in $\mathcal{T}_M$, it naturally
intersects $\mathcal{K}$ in the 1-handles, 2-handles and 3-handles. In particular, for
each 1-handle, $A$ is either disjoint from it or $A$ intersects the 1-handle in a co-core.

\subsection{The components that are not solid tori} For the moment, we focus on $M_{\textrm{nst}}$.

\begin{lemma}
Each component of $\partial M_{\textrm{nst}} \cut P$ is a disc. 
\end{lemma}

\begin{proof}
Each component of
$M \cut A$ is a solid torus with a clean longitudinal annulus in its boundary. Hence, $\partial M_{\textrm{nst}} \cut P$
consists of discs and annuli, where the annuli are a union of fibres. So if any component of $\partial M_{\textrm{nst}} \cut P$
is an annulus, it lies in $\partial N(K)$, since otherwise it would be incident to another solid toral component and hence would lie in $A$
rather in $\partial M_{\textrm{nst}}$. So there is an incompressible clean annulus $A'$ properly embedded in $M_{\textrm{nst}}$
that is not parallel to a clean annulus in $\partial M_{\textrm{nst}}$. A core curve of $A'$ is a fibre of the circle bundle, and hence is
$\pi_1$-injective in the exterior of $K$. So, $A'$ is incompressible in the exterior of $K$. It is disjoint from
the JSJ tori of the exterior of $K$, since we cut along those in an earlier step in the partial hierarchy,
and it does not lie in the Seifert pieces, since these have been decomposed
into solid tori by the partial hierarchy. Hence, $A'$ is boundary parallel in the exterior of $K$. Consider the product
region between $A'$ and the relevant annulus in $\partial N(K)$. The first surface in the partial hierarchy to
intersect this region would have to be inessential, contrary to the construction of the hierarchy. This contradiction
proves the lemma.
\end{proof}

\subsection{The branched surface $B$ and its natural handle structure}

According to Proposition \ref{Prop:SectionCircleBundle}, some section $S$ of $M_{\text{nst}}$ 
can be pattern-isotoped, via an isotopy preserving $A$, to a weakly fundamental surface
in $\mathcal{T}_M$. Moreover, this surface intersects $P$ as few times as possible,
and among such surfaces, it has least extended weight. As in Section \ref{Subsec:BranchedSurfaceCarryingAdmissible}, the branched
surface $B$ is constructed, which has a patch for each tile type of $S$. The annuli
$A$ intersect $N(B)$ in a union of fibres. Let $\mathcal{H}$ denote the natural handle
structure on $N(B)$.

\subsection{The arc index of $\partial S \cap M_{\textrm{nst}}$}

\begin{lemma}
\label{Lem:ArcIndexBoundary}
The arc index of $\partial S \cap M_{\textrm{nst}}$ is at most $|S \cap P| 96 w (w - \sum_{i=1}^n \chi(S_i))$, where $w$ is binding
weight of the partial hierarchy before $S$.
\end{lemma}

\begin{proof}
The curves $\partial S \cap M_{\textrm{nst}}$ are cut into arcs by $P$. Each such arc lies in a disc component of $\partial M_{\textrm{nst}} \cut P$.
Since $S$ has minimal extended weight up to pattern-isotopy, each arc
of $\partial S \cap M_{\textrm{nst}}$ must intersect therefore each edge of $\mathcal{T}_M$ at most once.
The number of edges of $\mathcal{T}_M$ is at most $96 w (w - \sum_{i=1}^n \chi(S_i))$
by Lemma \ref{Lem:NumberTetrahedra}. 
\end{proof}

\subsection{New boundary pattern}

Define $P'$ to be the boundaries of the tiles of the partial hierarchy in $\partial M_{\textrm{nst}}$. Then $|S \cap P'|$ is at most the arc
index of $\partial S \cap M_{\textrm{nst}}$, which is bounded above by Lemma \ref{Lem:ArcIndexBoundary}.

\subsection{Modifications arising from low-valence interior vertices}

In Section \ref{Sec:ModificationsAdmissible}, various modifications were made to the next surface in the partial hierarchy.
In this case, we make these modifications to the section $S$. Note that we never insert a wedge,
since this is only performed around low-valence vertices in $\partial S$ that are disjoint from the pattern.
However, every vertex in $\partial S$ is a vertex of $\partial M$, and is therefore in the pattern $P'$.

The modifications that we make will move the section $S$, but they will not move $A$. Hence,
the relative position of $S$ and $A$ may change. In particular, $S \cut A$ need not remain
meridian discs for the solid tori $M \cut A$. We will adjust $A$ suitably in Section \ref{Sec:DiscSwaps}.

As in Section \ref{Sec:Parallelism}, we keep track of an admissible envelope
$(N, X, Y, P \cap N, \mathcal{H}, S)$. In order to maintain consistency with the early argument,
we will also consider another surface $S'$ and another admissible envelope 
$(N', X, Y', P \cap N', \mathcal{H}', S')$. However, since we will not be performing any
wedge insertions, $S'$, $N'$, $Y'$ and $\mathcal{H}'$ will be equal to $S$, $N$, $Y$ and
$\mathcal{H}$.

The annuli $A$ initially intersect $N(B)$ in a union of fibres, and we will maintain
this property for the intersection with $N$.

When we consider interior 2-valent vertices as in Section \ref{Subsec:InteriorTwoValRevisited}, we need
to be careful also to consider the annuli $A$.
The only situation where care is required here is when $A$ intersects the star of each vertex $s^i$
in an arc that runs from $s^i_1$ to $s^i_2$ via $s^i$. But after the modification in Section \ref{Subsec:InteriorTwoValRevisited},
there remains a copy of $s^i_1$ and $s^i_2$, and we simply ensure that $A$ includes the fibre between them.

The modification involving 3-valent interior vertices is also easy to handle. There a tile that is incident
to the vertex is collapsed, thereby combining two generalised saddles. If $A$ intersects that tile, it does so in a fibre,
and we can ensure that after the modification, $A$ runs through the new combined generalised saddle.

\subsection{The Euclidean subsurface}

As before, we must consider the Euclidean subsurface of $S' = S$. This is defined in
the same way as previously, except that $E$ also does not include the intersection $S \cap C$.
The branched surface $B_E$ is disjoint from $\partial M \cup C$, as is $E \setminus N_2(S' \cut E)$.
It therefore lies in the manifold $M_C = M \setminus N(C)$. This is a circle bundle over
the circle, and we can view the annuli $A$ as properly embedded in $M_C$.

We need to ensure that the conclusion of Theorem \ref{Thm:DistanceToBoundaryEuclidean} still holds. That theorem had two hypotheses,
which we vary to the following. The first hypothesis is that there is no homeomorphism $h \colon N \rightarrow N$
such that all the following hold:
\begin{enumerate}
\item $h$ is equal to the identity on $P \cap N$ and $X$ (and so extends to $M_C$ and $M$);
\item $h$ is a Dehn twist along a torus properly embedded in $M$;
\item $(N, X, Y, P \cap N, \mathcal{H}, h(S))$ is an admissible envelope;
\item $w_\beta(h(S)) < w_\beta(S)$.
\end{enumerate}
The second hypothesis is there is no pattern-isotopy taking $S$ to a surface $S_2$, and a
nested admissible envelope
$(N_2, X_2, Y_2, P \cap N_2, \mathcal{H}_2, S_2)$ satisfying
$$(w_\beta(S_2), \text{approx-type}(S_2)) < (w_\beta(S), \text{approx-type}(S)),$$
$$c_+(\mathcal{H}_2) \leq c_+(\mathcal{H}), \qquad c(\mathcal{H}_2) \leq c(\mathcal{H}).$$

We now explain why the conclusion of Theorem \ref{Thm:DistanceToBoundaryEuclidean} remains
true with these modified hypotheses.

The second hypothesis is used in the proof of Proposition \ref{Prop:GridFirstReturnTildeE}.
There, a collection of disjoint grids in $\tilde E_+$ are considered. In the proof, a
disc $D_-$ in one of these grids is analysed. Its boundary is a boundary component
of an annulus that is vertical in $N$. The other boundary component of this annulus
lies in $S$ and bounds a disc $D_+$ in $S$. If the discs have different binding
weights, then we may replace one by the other to reduce the binding weight of $S$.
On the other hand, if they have the same binding weight, then we can
replace as many parallel copies of $D_+$ as possible by parallel copies of $D_-$.
This reduces the number of approximate star types of tiles of $S$.

The proof proceeds by constructing a torus $T$ carried by $N(B_+)$. In our situation,
$T$ is actually carried by $N$. Then a power of a Dehn twist $h$ about $T$ reduces the binding
weight of $S$. This contradicts our first hypothesis.

\subsection{The thin part} In our previous argument, the insertion of wedges may introduce cut vertices to $S'$.
As a result, the surface $S'_\mathrm{cut}$ was defined. However, in our situation, no
wedges are inserted, and therefore no cut vertices are introduced. Hence, $S'_\mathrm{cut} =
S' = S$, and the thin part of $S'$ is empty.

\subsection{The Euler characteristic argument} This is almost unchanged. The only variation
is that we have modified the boundary pattern from $P$ to $P'$, and we have also declared
that the Euclidean subsurface is disjoint from $C$. Hence, the revised statement of
Theorem \ref{Thm:NumberLowValVertices} is that the number of interior vertices of $S$ with valence
at most $3$ is at least
$$\frac{w_\beta(S)}{31\pi (2392e)^2} +  \frac{72}{31}\chi(S) - \frac{5}{31}|S \cap P'| - \frac{5}{31}|S \cap C|.$$

\subsection{Controlling the annuli $A$}
\label{Sec:DiscSwaps}

As we will see in Section \ref{Sec:PolynomiallyBounded}, we will modify $S$ until its
weight is bounded. These modifications all take place away from $C$ and so are contained
in the manifold $M_C = M \cut N(C)$. Let $F = S \cap M_C$ and let $F'$ be the resulting surface
in $M_C$. We have not moved $A$,
and so $F\cup A$ might not be isotopic to $F' \cup A$. We now explain how to remedy
this. We will consider the following modifications.

Let $S_1$ and $S_2$ be two surfaces properly embedded in a 3-manifold $M_C$ and
intersecting transversely. Suppose there is a simple closed curve of $S_1 \cap S_2$
that bounds a disc in one of the surfaces. Suppose that the disc is $D_1 \subset S_1$.
Suppose also that the interior of $D_1$ is disjoint from $S_2$. Assume also that $\partial D_1$
bounds a disc $D_2$ in $S_2$, which might have interior intersecting $S_1$. Then 
removing $D_2$ from $S_2$ and replacing it by a parallel copy of $D_1$, thereby removing
$\partial D_1$ from $S_1 \cap S_2$, is known as \emph{disc swap}.

\begin{lemma}
\label{Lem:SectionAnnuliHomeo}
Let $M_C$ be a compact orientable 3-manifold that is a circle bundle over the circle.
Let $A$ be a collection of properly embedded vertical annuli so that $M_C \cut A$ is a
collection of solid tori. Let $F$ be a section of $M_C$. Let $F'$ be obtained from $F$
by a homeomorphism of $M_C$ supported in the interior of $M_C$, and suppose also that $F'$ intersects
$A$ transversely. Suppose that $F''$ and $A''$ are obtained from $F'$ and $A$ by a sequence of disc
swaps until no more such disc swaps are
possible. Then there is a homeomorphism of $M_C$, supported in the interior of $M_C$ and taking $F'' \cup A''$ 
to $F \cup A$.
\end{lemma}

\begin{proof}
Since $F$ is a section, it intersects each component of $\partial A$ exactly once.
As $F''$ is obtained from $F$ by a homeomorphism of $M_C$ supported in the interior of $M_C$,
$F''$ also intersects each component of $\partial A$ exactly once.
Hence, $F''$ intersects each component of $A$ in an essential arc plus possibly 
some inessential simple closed curves. However, because no further
disc swaps are possible, there are no simple closed curves of $F'' \cap A$.
By a further isotopy preserving $A$, we may assume that 
the arcs $F'' \cap A$ are transverse to the circle fibres.  

Let $N(A \cup \partial M_C)$ be a thin regular neighbourhood of $A \cup \partial M_C$. 
We may assume that $F'' \cap N(A \cup \partial M_C)$ is homeomorphic to $F \cap N(A \cup \partial M_C)$.
Now, $F$ is obtained from $F \cap N(A \cup \partial M_C)$ by attaching discs. Since $F$ and $F''$
have the same Euler characteristic, $F''$ must also be obtained from $F'' \cap N(A \cup \partial M_C)$
by attaching discs. These must be one such disc in each component of $M_C \cut A$, forming
a meridian disc for that component. Thus, there is an isotopy preserving $A$ and supported in 
the interior of $M_C$ taking $F''$ to a section. Hence there is a homeomorphism of $M_C$, 
supported in the interior of $M_C$ and taking $F'' \cup A''$  to $F \cup A$.
\end{proof}

Thus, the procedure that we use is as follows. We have applied a homeomorphism of $M_C$
taking $F$ to $F'$, until it has
controlled weight. We have not moved $A$, and so its weight is unchanged. Suppose that
$F'$ and $A$ admit a disc swap. Then it is possible to perform a disc swap without
increasing the binding weight of the surfaces, as follows. Consider a simple
closed curve of $F' \cap A$ that bounds a disc $D$ in one of the surfaces, and
suppose that among all such curves, the binding weight of $D$ is minimal. Then
we may assume that the interior of $D$ is disjoint from $F' \cap A$. By the incompressibility
of the surfaces, $\partial D$ bounds a disc $D'$ in the other surface. By our
minimality assumption, the binding weight of $D'$ is at least the binding weight
of $D$. Hence, the disc swap that removes $D'$ and replaces it by $D$ does not
increase the binding weight of the surfaces.

So, we may perform disc swaps to $F'$ and $A$ without increasing the binding weight, until
no more disc swaps are possible. Then by Lemma \ref{Lem:SectionAnnuliHomeo},
the resulting surfaces are strongly equivalent to $F \cup A$.

\subsection{The solid toral components}
The above argument dealt with $M_{\text{nst}}$. The solid toral components must be
handled in a different way, since they do not necessarily contain a clean incompressible annulus that is
not parallel to a clean annulus in $\partial M$.
However, unlike other circle bundles, a solid torus has the advantage that cutting
along a section decomposes it into a ball. Hence, we simply remove $A$ and $C$
from $M_{\text{st}}$, let $S_{\text{st}}$ be a section in $M_{\text{st}}$, and then treat $S_{\text{st}}$ in the same way as
the remaining surfaces in the hierarchy.

\section{Realising the hierarchy with polynomially bounded binding weight}
\label{Sec:PolynomiallyBounded}

In this section, we prove the following result, which arranges an exponentially controlled hierarchy into
admissible form with polynomially bounded binding weight. This key result is central to this
paper, and relies on all the machinery developed so far, including normal surface theory
and the structure of arc presentations.

\begin{theorem}
\label{Thm:PolyBoundExchange}
Let $K$ be a non-split link in $S^3$, other than the unknot. Assign boundary pattern to $\partial N(K)$
as described in Theorem \ref{Thm:WeaklyFundamentalHierarchy}.
Let $H = \{ S_1, \dots, S_\ell \}$ 
be an adequately separating, $(\lambda,2^{15})$-exponentially controlled hierarchy for the exterior of $K$, where
$$\lambda = \sum_{i=1}^\ell -3\chi(S_i) + 6 |S_i| + 6|S_i \cap P_i|.$$
Here, $P_i$ is the boundary pattern
in the exterior of the first $i-1$ surfaces. Let $H'$ be obtained from $S_1 \cup \dots \cup S_\ell$
by the procedure in Section \ref{Sec:BrokenSection}, which at the final stage removes
the intersection between $H$ and the interior of the solid toral components of
the relevant circle bundle and then inserts weakly fundamental meridian discs. Let
$$q(x) = (10^{80} \lambda^8)^{\sum_{i=0}^\ell {22}^i} x^{22^\ell}.$$
Let $D$ be an arc presentation for $K$ with arc index $n$.
Then there is a sequence of at most $n^2 q(n)$ exchange moves and cyclic permutations, and 
a homeomorphism of the exterior of $K$ that equals the identity on $\partial N(K)$ and that
takes $\partial N(K) \cup H'$ to an admissible hierarchy with binding weight at most $q(n)$.
\end{theorem}

\begin{proof}
 Let $k$ be a positive integer that is at most the length of the hierarchy $H$.
For $k < \ell$, let $H_k$ be the partial hierarchy consisting of the first $k$ surfaces of $H$, including $S_0 = \partial N(K)$. 
Let $H_\ell = H'$.
Let 
$$q_k(x) = (10^{80} \lambda^8)^{\sum_{i=0}^k {22}^i} x^{22^k}.$$
We will show by induction on $k$ that 
there is a sequence of at most $n^2 q_k(n)$ exchange moves and cyclic permutations and 
an isotopy of the exterior of $K$ that takes $\partial N(K) \cup H_k$ to an admissible partial hierarchy with weight at most $q_k(n)$.
Setting $k$ to be the length $\ell$ of the hierarchy will clearly then prove the theorem.

The induction starts by performing at most $n+1$ stabilisations to $D$ so that, for each component of $\partial N(K)$
with non-empty boundary pattern, the slope of this pattern equals the writhe of the relevant component of $K$.
We then place $\partial N(K)$ into admissible form. We may do this
simply by taking a thin regular neighbourhood of $K$. This requires no exchange moves
or cyclic permutations, and the resulting binding weight of $\partial N(K)$ is at most $4n+2$.
Note that $q_0(n) \geq 4n+2$.

We now establish the inductive step. Let $\{S_0, \dots, S_k \}$ be the surfaces of the partial hierarchy $H_k$, where $k < \ell$,
and let $S = S_{k+1}$ be the next surface. Let $P$ be the boundary pattern on the exterior of
$H_k$. Suppose that $H_k$ is admissible, and has binding weight at most $w$. 
We will show that there is a sequence of
at most $10^{40} \lambda^4 w^{11}$
steps that takes $\partial N(K) \cup H_{k+1}$
to an admissible hierarchy with binding weight at most 
$10^{80} \lambda^8 w^{22}$. 
Thus, the binding weight is at most
$$10^{80} \lambda^8 (q_k(n))^{22} = 10^{80} \lambda^8 (10^{80} \lambda^8)^{\sum_{i=1}^{k+1} {22}^i} n^{22^{k+1}}
= (10^{80} \lambda^8)^{\sum_{i=0}^{k+1} 22^i} n^{22^{k+1}}= q_{k+1}(n).$$
Each step will use at most $n^2$ exchange moves and cyclic permutations.
So, the number of exchange moves and cyclic permutations so far is at most
$$n^2 q_k(n) + n^2 10^{40} \lambda^4 (q_k(n))^{11} < n^2 10^{80} \lambda^8 (q_k(n))^{22} = n^2 q_{k+1}(n).$$
This will establish the inductive step.

So, we now suppose that $H_k$ is in admissible form with binding weight at most $w$. In particular, it has no
annular tiles.

\medskip
\emph{Step 1.} Modify $H_k$ so that it is well-spaced.

We use the procedure described in Section \ref{Subsec:WellSpaced} to make the hierarchy well-spaced.
This increases its binding weight to at most $7w$.

\medskip
\emph{Step 2.} Constructing a triangulation for the exterior of $H_k$.

A triangulation for the exterior $M$ of $H_k$ is described in Section \ref{Subsec:TriangulatingExterior}. Lemma \ref{Lem:NumberTetrahedra} states
that it has at most $t = 112 w (7w - \sum_{i=1}^k \chi(S_i) ) \leq 112 w (7w + \lambda)$ tetrahedra.

\medskip
\emph{Step 3.} Realising $S$ as an exponentially controlled normal surface.

By assumption, $S$ can be realised as a $(\lambda,2^{15})$-exponentially controlled surface,
and so the weight of $S$ is at most $\lambda 2^{15t}$.

\medskip
\emph{Step 4.} Homotope $S$ to a nearly embedded surface $\hat S$ with boundary
in an arc presentation.

As in Section \ref{Subsec:HomotopeAdmissible}, we now homotope $S$ to a nearly embedded alternative admissible surface $\hat S$
with binding weight at most $200(7w)^3 (7w + \lambda)^2 \lambda 2^{15t}$. It is normal with respect to the
handle structure $\mathcal{K}$.

\medskip
\emph{Step 5.} Forming the branched surface $B$ that carries $S$.

In Section \ref{Subsec:BranchedSurfaceCarryingAdmissible}, a branched surface $B$ was described that carries $S$.
Let $N = N(B)$ be its thickening and let $\mathcal{H}$ be its natural handle structure.
We set $X$ to be $\partial_h N$, and set $Y$ to be empty. In Section \ref{Sec:AdmissibleEnvelope},
we defined the admissible envelope $(N, X, Y, P \cap N, \mathcal{H}, S)$.
We also initially set $(N', X, Y', P \cap N', \mathcal{H}', S')$ to be equal to 
$(N, X, Y, P \cap N, \mathcal{H}, S)$ and note that these form nested admissible envelopes.
In Section \ref{Subsec:ComplexityHS}, we defined complexities for these handle structures.
By Lemma \ref{Lem:ComplexityNaturalHS}, these complexities satisfy
$$c(\mathcal{H}) \leq c_+(\mathcal{H}') \leq 11648 (7w)^2(7w+\lambda).$$
We will modify the envelopes as described in Section \ref{Sec:Parallelism},
but their complexities will not increase. The surface $S'$ may have some cut vertices.
We let $S' \setminus S'_{\mathrm{thin}}$ be the result of cutting $S'$ along these cut vertices and discarding
any bigon tiles. Let $E$ be the Euclidean subsurface of $S' \setminus S'_{\mathrm{thin}}$.
Let $S''$ be obtained from $S' \setminus S'_{\mathrm{thin}}$ by removing any toral or clean annular
components of $S'$ lying in $E$ with at most $23552e^2$ vertices. 
Let $e$ be the number of approximate star types of tiles of $S''$. 
By Propositions \ref{Prop:StarTypeNonCut} and \ref{Prop:NumberWeakTileTypes},
$$e \leq 16 c_+(\mathcal{H}) + 32c(\mathcal{H}')\leq 559104 (7w)^2(7w+\lambda).$$

\medskip 
\emph{Step 6.} Finding many vertices with small valence.

Let $S_+$ be the double of $S''$.  By Theorem \ref{Thm:NumberLowValVertices}, when $k < \ell$,
the number of non-cut vertices of $S_+$ with valence at most $3$ is 
at least
$$\frac{w_\beta(S'')}{31\pi (2392e)^2} + \frac{72}{31} \chi(S) - \frac{5}{31}|S \cap P| \geq \frac{w_\beta(S'')}{31\pi (2392e)^2} - \lambda .$$
At the final step, we might be dealing with a broken section, in which case,
the number of non-cut vertices of $S_+$ with valence at most $3$ is (in the terminology of Section \ref{Sec:BrokenSection}) at least
\begin{align*}
& \frac{w_\beta(S'')}{31\pi (2392e)^2} + \frac{72}{31} \chi(S) - \frac{5}{31}|S \cap P'| - \frac{5}{31}|S \cap C| \\
&\qquad \geq \frac{w_\beta(S'')}{31\pi (2392e)^2} - \lambda - \frac{5}{31}|S \cap P| 96(7w)(7w+\lambda) \\
&\qquad \geq \frac{w_\beta(S'')}{31\pi (2392e)^2} - 388w^2 \lambda.
\end{align*}
Their stars come in at most $e$ approximate types. Thus, in $S''$, we may find a collection of at least 
$$\frac{1}{e}
\left(\frac{w_\beta(S'')}{62\pi (2392e)^2} - 194w^2\lambda  \right)$$
low-valence vertices all with parallel stars.

\medskip
\emph{Step 7.} Modifying $S$ and $S'$ and their nested admissible envelopes.

In Section \ref{Sec:Parallelism}, procedures were described which reduced the binding weight.
Each was applied to a collection of parallel stars with low valence.
Thus, the binding weight of $S''$ reduces by at least 
$$\frac{1}{e}
\left(\frac{w_\beta(S'')}{62\pi (2392e)^2} - 194 w^2 \lambda\right) \geq
\frac{1}{e}
\left(\frac{w_\beta(S'')}{10^{10} e^2} - 194 w^2 \lambda \right)
$$
So $w_\beta(S'') - 194 \cdot 10^{10} \lambda w^2 e^2$ decreases by at least a factor $(1 - 1/(10^{10} e^3))$, for the following reason:
\begin{align*}
& w_\beta(S'') - 194 \cdot 10^{10} \lambda w^2e^2 - w_\beta(S'')/(10^{10} e^3) + 194\lambda w^2/e \\
& \leq
(1 - 1/(10^{10} e^3)) (w_\beta(S'') - 194 \cdot 10^{10} \lambda w^2 e^2).
\end{align*}
We perform these operations until $w_\beta(S'') - 194\cdot 10^{10} \lambda w^2 e^2 < 1$, in other words until
$w_\beta(S'') \leq 194 \cdot 10^{10} \lambda w^2 e^2$.

\medskip 
\emph{Step 8.} Reintroducing the toral and annular components.

When forming $S''$, toral and annular components of $S'_{\mathrm{thin}}$ were removed. Each had binding
weight at most $23552e^2$. The number of such components is at most $2e$, since this is an upper bound
for the maximal number of disjoint surfaces carried by $B_E$. The above modifications have not increased their binding weight.
So, we may add them back to the surface, giving a surface with binding weight at most
$$194 \cdot 10^{10} \lambda w^2 e^3 + 47104 e^3 \leq 2 \cdot 10^{12} \lambda w^2 e^3.$$

\medskip
\emph{Step 9.} Counting the number of iterations.

Let $x$ be the number of times that we applied Steps 6 and 7.
Each time such a step was repeated, $w_\beta(S'') - 194 \cdot 10^{10} \lambda w^2e^2$ was multiplied by 
a factor of at most $(1- 1/(10^{10} e^3))$. The initial binding weight was at most 
$$200(7w)^3(7w+\lambda)^2 \lambda 2^{15t}  \leq 200(7w)^3(7w+\lambda)^2 \lambda  2^{1792 w (7w + \lambda )} \leq 2^{11760 w(w+\lambda )}.$$
Since, after $x-1$ steps, the binding weight is still more than 1, we deduce that
$$2^{11760 w(w + \lambda )} (1-1/(10^{10}e^3))^{x-1} \geq 1$$
and hence that
$$(\log 2) 11760 w(w+\lambda ) + (x-1) \log  (1-1/(10^{10}e^3)) \geq 0.$$
Rearranging and using the inequality $\log(1-y) \leq - y$ for $0<y < 1$ gives
$$(x-1) / (10^{10} e^3) \leq (\log 2) 11760 w(w+ \lambda ).$$
Hence, 
$$x \leq 1+ 10^{10} (\log 2) 11760 (559104)^3 7^9 w^7 (w+\lambda )^4 \leq 10^{40} \lambda^4 w^{11}.$$
The final inequality uses that $x+y \leq xy$ for integers $x,y \geq 2$.
As explained above, the total number of exchange moves and cyclic permutations is therefore at most $n^2 q_{k+1}(n)$.

\medskip
\emph{Step 10.} Bounding the binding weight of the partial hierarchy.

Each step increases the binding weight of the hierarchy $\{S_0, \dots, S_k \}$ by at most 6.
Hence, after these steps, it has binding weight at most  $w + 6 \cdot 10^{40} \lambda^4 w^{11} \leq 7 \cdot 10^{40} \lambda^4 w^{11}$.

\medskip
\emph{Step 11.} Controlling the thin part of $S'$.

By Proposition \ref{Prop:IsotopeThin} and Remark \ref{Rem:BigonTypes}, there is a pattern-isotopy of $S'' \cup S'_{\mathrm{thin}}$ taking it to a surface 
$S_{k+1}$
with binding weight
at most 
$$(2e + 1 +  7 \cdot 10^{40} \lambda^4 w^{11}) \, w_\beta(S'') \leq (2e + 1 +  7 \cdot 10^{40} \lambda^4 w^{11}) \cdot 2 \cdot 10^{12} \lambda w^2 e^3$$
$$\leq 10^{54} \lambda^5 w^{13} e^3
\leq 10^{79} \lambda^8 w^{22}.$$ 
In the case of where $S_{k+1}$ is a broken section, we need to do disc swaps with $S_{k+1}$ and the
annuli $A$ so that $S_{k+1} \cut A$ is meridian discs for the solid tori forming the circle bundle.
As explained in Section \ref{Sec:DiscSwaps}, this can be achieved without increasing the binding weight
of $S_{k+1}$ and $A$. Furthermore, since this isotopy takes place in the exterior of $K$, no
exchange moves or cyclic permutations are required. Finally, the section is perturbed to a broken
section, which might increase its weight by a factor of at most $3$, but it remains at most 
$10^{80} \lambda^8 w^{22}$.
\end{proof}

\section{Isotoping a link through a handle structure}
\label{Sec:IsotopeThroughHS}

Let $\mathcal{H}$ be a handle structure of the 3-sphere. We say that a link $K$ \emph{respects} the handle structure $\mathcal{H}$ if the following all hold:
\begin{enumerate}
\item $K$ lies in the union of the 0-handles and 1-handles of $\mathcal{H}$; 
\item it intersects each 0-handle in a properly embedded 1-manifold; 
\item it intersects each 1-handle $D^1 \times D^2$ in arcs,
each of which is of the form $D^1 \times \{ \ast \}$ for some point $\ast \in {\rm int}(D^2)$.
\end{enumerate}

Our aim in this section is to start with a link $K$ that respects $\calH$ and to perform an ambient isotopy of $K$ taking it into a single 0-handle.
This can clearly be done. However,
we wish to perform this isotopy in a controlled way. We therefore introduce the
following moves. Each starts and ends with $K$ respecting the handle structure.

\medskip
\noindent \emph{Move 1}: Sliding across a 2-handle

Let $H_2$ be a 2-handle of $\mathcal{H}$, and let $N(H_2)$
be a small regular neighbourhood of it. So, for each 0-handle and 1-handle to which $H_2$
is attached, $N(H_2)$ intersects that handle in a collection of balls. Note that $N(H_2)$
has a natural product structure as $D^2 \times D^1$. Suppose that $K \cap N(H_2)$
consists of a single arc of the form $\alpha \times \{ \ast \}$, where $\alpha$ is an
arc in $\partial D^2$ and $\{ \ast \}$ is a point in ${\rm int}(D^1)$. Suppose also that
this arc $K \cap N(H_2)$ lies in a single 0-handle. Then we isotope $\alpha \times \{ \ast \}$
across the disc $D^2 \times \{ \ast \}$, replacing it with the arc 
$(\partial D^2 \cut \alpha) \times \{ \ast \}$. (See Figure \ref{Fig:Sliding2Handle}.)

\begin{figure}[h]
\includegraphics[width=4in]{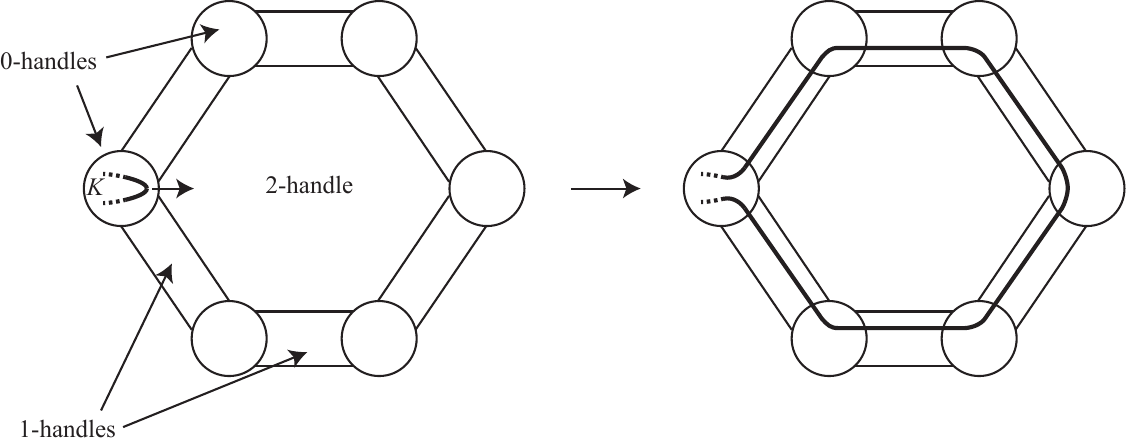}
\caption{Sliding across a 2-handle}
\label{Fig:Sliding2Handle}
\end{figure}

\medskip
\noindent \emph{Move 2}: Sliding along a 1-handle

Let $H_1$ be a 1-handle of $\mathcal{H}$, and let $N(H_1)$ be
a regular neighbourhood of $H_1$. This has a product structure $D^1 \times D^2$.
Suppose that all but one arc of $K \cap N(H_1)$ is of the form $D^1 \times \{ \ast \}$
for some point $\ast \in {\rm int}(D^2)$. Suppose that the remaining arc $\gamma$ is a 
concatenation of three arcs $\alpha \times \{ p_1 \}$, $\alpha \times \{ p_2 \}$
and $\{ q \} \times \beta$, where $\alpha$ is an arc in $D^1$ such that $\alpha \cap \partial D^1$
is a single point, where $q$ is the remaining point $\partial \alpha \backslash \partial D^1$,
and where $\beta$ is an arc in the interior of $D^2$ joining $p_1$
to $p_2$. Suppose also that $\gamma$ intersects $H_1$ in two arcs. Then, we replace
$\gamma$ by an arc $\gamma'$ that is disjoint from $H_1$. Once again,
$\gamma'$ is the concatenation of three arcs $\alpha' \times \{ p_1 \}$, $\alpha' \times \{ p_2 \}$
and $\{ q' \} \times \beta$, where $\alpha'$ is a subarc of $\alpha$. We also permit
the reverse of this move. (See Figure \ref{Fig:Sliding1Handle}.)

\begin{figure}[h]
\includegraphics[width=4in]{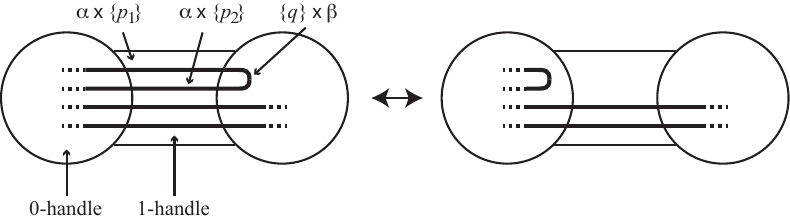}
\caption{Sliding across a 1-handle}
\label{Fig:Sliding1Handle}
\end{figure}

\medskip
\noindent \emph{Move 3}: Handle-preserving isotopy

The final move consists of an ambient isotopy of the 3-sphere,
that preserves each handle, and such that $K$ respects $\mathcal{H}$ throughout.

\begin{proposition}
\label{Prop:IsotopyMoves}
Let $K$ be a link that respects a handle structure
$\mathcal{H}$ on the 3-sphere. Then there is a sequence of the above moves, after which
$K$ lies in the interior of a single 0-handle.
\end{proposition}

\begin{proof}
We start by defining a system of arcs and discs embedded in $\mathcal{H}$, which we term
the \emph{specified} arcs and discs. The specified arcs are easily defined: they are just the co-cores
of the 2-handles. In other words, for a 2-handle $D^2 \times D^1$, the specified arc is $\{ 0 \} \times D^1$,
where $\{ 0 \}$ is the central point in $D^2$.

We now need to define the specified discs. Each disc will include the co-core of a 1-handle, which is
$\{ \ast \} \times D^2$ for some point $\{ \ast \}$ in ${\rm int}(D^1)$. The specified discs will also
include parts lying in the 2-handles, as follows. For a 2-handle $D^2 \times D^1$, the intersection
with the co-cores of the 1-handles is of the form $P \times D^1$, where $P$ is a finite collection
of points in $\partial D^2$. For each point $p$ in $P$, there is a radial arc $\alpha$ in $D^2$
running in a straight line from $p$ to $0$. We term $\alpha \times D^1 \subset D^2 \times D^1$
a \emph{strip}. (See Figure \ref{Fig:SpecifiedDiscs}.) For each co-core of a 1-handle, we attach all the strips to which it is incident.
The result is a \emph{specified disc}. (See Figure \ref{Fig:SpecifiedDiscStrips}.)

\begin{figure}[h]
\includegraphics[width=2in]{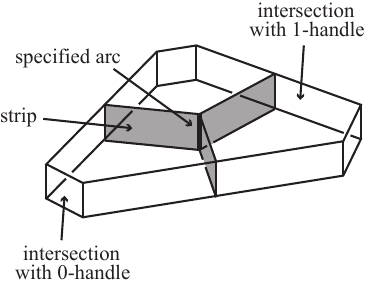}
\caption{The specified arc and strips within a 2-handle}
\label{Fig:SpecifiedDiscs}
\end{figure}

\begin{figure}[h]
\includegraphics[width=3in]{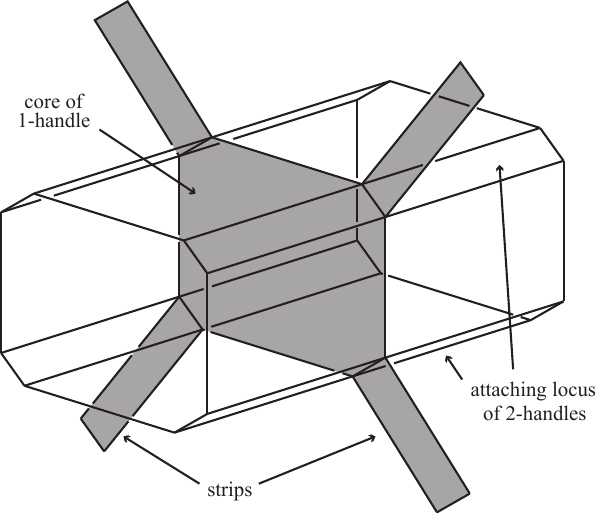}
\caption{A specified disc associated with a 1-handle}
\label{Fig:SpecifiedDiscStrips}
\end{figure}

Using the specified arcs and discs, one can form a handle structure $\mathcal{H}'$ for the 3-sphere as follows. The 3-handles of $\mathcal{H}'$
and $\mathcal{H}$ are the same. A thin
regular neighbourhood of each specified arc forms a 2-handle of $\mathcal{H}'$. A thickening of each
specified disc forms a 1-handle. The remainder of the manifold is the 0-handles of $\mathcal{H}'$.
Now $\mathcal{H}$ and $\mathcal{H}'$ are, in some sense, the same handle structure. More precisely,
there is a homeomorphism $h \colon S^3 \rightarrow S^3$ taking each handle of $\mathcal{H}'$ to a handle of
$\mathcal{H}$. This homeomorphism expands each 2-handle of $\mathcal{H}'$ until it fills the 2-handle of $\mathcal{H}$
that contains it. It also lengthens each 1-handle of $\mathcal{H}'$.

We now pick an isotopy of $S^3$ that takes $h^{-1}(K)$ into a 0-handle of $\mathcal{H}'$. By general position, we may
assume that the link misses the specified arcs, except at finitely many points in time, when it crosses through some
specified arc transversely. Hence, we may also arrange that, at this point in time, it moves across a 2-handle of
$\mathcal{H}'$ in a single slide. Similarly, the link intersects each specified disc transversely in finitely many points,
except at finitely many instances, where one of two possible things happens. One possibility is that 
the link approaches the specified disc, then just touches it,
and then goes through it, creating two new points of intersection. Hence,
the link moves across a 1-handle of $\mathcal{H}'$ in a slide.
The other possibility is the reverse of this. This isotopy
of $S^3$ gives a 1-parameter family of homeomorphisms of $S^3$ that starts at the identity. Applying
$h$ to each of these gives an isotopy, that moves $K$ into the interior of a 0-handle of $\mathcal{H}$,
and that is of the required form.
\end{proof}

\section{From the isotopy to Reidemeister moves}
\label{Sec:IsotopyToReidemeister}

In this section, we complete the proof of our main theorem. We are given a diagram $D$ of a link, with link type $K$, where
$D$ has $c$ crossings. We pick a fixed diagram $D'$ for $K$. We need to prove that there is a polynomial $p_K$ such that
$D$ and $D'$ differ by a sequence of at most $p_K(c)$ Reidemeister moves.

Let $\mu$ be the constant in Theorem \ref{Thm:ProducesWeaklyFundamentalHierarchy}.
We first pick a fixed $(\mu,2^{15})$-exponentially controlled hierarchy $H$ for the exterior of $K$, which exists by Theorem \ref{Thm:ProducesWeaklyFundamentalHierarchy}. We also ensure that it is adequately separating.

We isotope $D$ to form a rectangular diagram. By Lemma \ref{Lem:IsotopeToRectangular}, we can arrange that this rectangular diagram
has arc index $n$ where $n \leq (81/20)c$. Since only an isotopy of the plane was used, no Reidemeister moves are needed at this
stage. This rectangular diagram determines an arc presentation also with arc index $n$. Theorem \ref{Thm:PolyBoundExchange}
provides a polynomial $q$, which depends only on $K$ and $H$, such that, after at most $n^2 q(n)$ cyclic permutations and exchange moves, 
the hierarchy has been placed in admissible form with binding weight at most $q(n)$.

\subsection{The handle structure arising from the hierarchy}
\label{Subsec:HSFromHierarchy}

The hierarchy $H$ determines a handle structure for the exterior of $K$, as follows. The union of $\partial N(K)$ and $H$ is a
2-complex. Thicken this to a handle structure.
The remainder of the exterior of $K$ is a collection of 3-balls, which we take to be 3-handles of the handle structure.

We now extend this to a handle structure $\calH$ for $S^3$, as follows. Pick a small regular neighbourhood of $K$,
denoted $N_{\textrm{small}}(K)$, lying in the interior of $N(K)$. So $N(K) \cut N_{\textrm{small}}(K)$ is a disjoint union of copies of $T^2 \times I$,
one for each component of $K$.
We give each component of $N_{\textrm{small}}(K)$ a handle structure as a 0-handle and a 1-handle. Pick a meridional simple closed curve
on each component of $\partial N(K)$, as well as a simple closed curve on $\partial N(K)$ that has slope equal to a longitude plus some number of
meridians. Let $m$ and $l$ be these meridian and integral curves. On each component of $\partial N(K)$, pick $m$ and $l$ so that they respect the
handle structure on this torus, and so that they intersect in a single point $p$ in some 0-handle. Between $p$ and a point on the 0-handle of $N_{\textrm{small}}(K)$, 
attach a 1-handle that is vertical in the $T^2 \times I$ region. Now attach a 2-handle that runs along $m - p$, then along this 1-handle, then around the 0-handle intersecting $K$, and then back along the 1-handle. Similarly attach a 2-handle than runs along $l - p$, then along the vertical 1-handle,
then along the 1-handle intersecting $K$, then back along the vertical 1-handle. These two 2-handles and the vertical 1-handle have decomposed the $T^2 \times I$
region into a ball, which we take to be a 3-handle.

We term $\calH$ as the handle structure \emph{arising from} $H$. It depends on $H$ as well as the choice of curves $m$ and $l$.
Note that $K$ respects $\mathcal{H}$, in the sense of Section \ref{Sec:IsotopeThroughHS}.

We now pick another fixed copy $K'$ of $K$ lying within a single 0-handle of $\mathcal{H}$. By Proposition \ref{Prop:IsotopyMoves}, there
is a sequence of moves, as described in Section \ref{Sec:IsotopeThroughHS}, which take $K$ to $K'$. Since $H$ has polynomially bounded weight,
it is not very surprising that one should be able to pass from
the rectangular diagram for $K$ to a diagram for $K'$ using a polynomially bounded number of Reidemeister moves.
So, if $D'$ and $K'$ are chosen so that a suitable projection of $K'$ is $D'$, then this will complete the proof of Theorem \ref{Thm:Main}.
In the remainder of this section, we give some further details of this argument.

\subsection{From an arc presentation to a rectangular diagram}

We have $K$ sitting inside the 3-sphere in a specific way, arising from its arc presentation. We now recall
the construction from \cite{Cromwell} that arranges $K$ so that its projection to a plane is the
associated rectangular diagram.

The 3-sphere is the join $S^1_\phi \ast S^1_\theta$, where $S^1_\phi$ is the binding circle.
Since $K$ is an arc presentation, it intersects the binding circle at finitely points, the vertices of $K$, and it intersects each
page $\mathcal{D}_t$ either in the empty set or in a single open arc joining distinct vertices of $K$.
We can arrange that each such open arc is a concatenation of two arcs, each of which runs straight
from a vertex of $K$ to the point $\mathcal{D}_t \cap S^1_\theta$.
We may also assume that the vertices of $K$ avoid the value $\phi = 0$, and that the
pages $\mathcal{D}_t$ intersecting $K$ avoid $\theta = 0$.

We now perturb $K$ off $S^1_\phi$ and $S^1_\theta$ as follows. We view each page $\mathcal{D}_t$ as 
an open unit disc. We replace an arc of intersection $K \cap \mathcal{D}_t$, by removing its intersection
with some small disc around the origin in $\mathcal{D}_t$ and replacing it by an arc going
around the circumference of this disc. We choose the arc so that it avoids taking the value $\phi = 0$.
We do something similar near the binding circle. Instead of going straight to a vertex on the binding
circle and then away again, we replace this by an arc transverse to the pages with constant $\phi$ value.
We choose this arc so that it avoid the value $\theta = 0$.

Thus, $K$ is now disjoint from $S^1_\theta \cup S^1_\phi$.
The complement of $S^1_\theta \cup S^1_\phi$ is a copy of $S^1_\theta \times S^1_\phi \times (0,1)$,
where the first two factors are parametrised by the $\theta$ and $\phi$ coordinates respectively.
Projection onto the first two factors will give the diagrammatic projection map for $K$.
In fact, because $K$ avoids $\theta = 0$ and $\phi = 0$, its image lies in $(0,2 \pi) \times (0,2 \pi)$.

We have arranged that $K$ is a concatenation of three types of arc. One type has constant $\theta$
and $\phi$ values and so projects to a point in $(0,2\pi) \times (0,2\pi)$. One type, lying in a page
and staying near $S^1_\theta$, has constant $\theta$ value. So it projects to a vertical arc
in $(0,2\pi) \times (0,2\pi)$. The third type of arc lies near $S^1_\phi$ and has constant $\phi$ value,
and so projects to a horizontal arc. Thus, we have a projection of $K$ to $(0,2\pi) \times (0,2\pi)$
consisting of a concatenation of horizontal and vertical arcs. Moreover, when two such arcs cross,
the vertical arc is the over-arc. This is the required rectangular diagram for $K$.

When $K$ is in this form, it inherits a natural framing from the diagram, and hence each component
has a writhe, which is the integral slope of its framing. We perform some initial stabilisations to the
rectangular diagram, so that the writhe of each component is equal to the integral slope of the
relevant component of $l$.

\subsection{A projection of the 1-skeleton}
\label{Subsec:Proj1Skel}

The hierarchy $H$ determines a cell structure for the exterior of $K$. We will now create a projection
for its 1-skeleton.

Each 0-cell lies on the binding circle $S^1_\phi$. It therefore has a specific $\phi$-value $s$. So it
corresponds to an arc $(0,2\pi) \times \{ s \}$ in $(0,2\pi) \times (0,2\pi)$. A regular neighbourhood 
of this 0-cell is a 0-handle of $\calH$. We may initially assume that the projection of this 0-handle to
the plane is a small regular neighbourhood of this arc. We may also identify the 0-handle with 
$D^2 \times I$, so that its projection onto the first factor is the diagrammatic projection map.

Each 1-cell of the cell structure is an arc in the boundary of some surface of the hierarchy. It
is therefore in arc presentation. Hence, it projects to a concatenation of horizontal and vertical
arcs in the plane. We thicken each 1-cell so that it projects to a regular neighbourhood of
the union of these arcs.

This specifies the projection of the 0-cells and 1-cells in the exterior of $K$. In addition, there are
the 0-handles and 1-handles that form $N_{\textrm{small}}(K)$ and the 1-handles that are vertical in $N(K) \cut N_{\textrm{small}}(K)$.
We may assume that each 0-handle is a regular neighbourhood of a vertex of $K$. Each 1-handle
in $N_{\textrm{small}}(K)$ runs along the remainder of a component of $K$, which is in an arc presentation.
So this specifies its projection. Each vertical 1-handle can be chosen so that it is a regular neighbourhood
of an arc with constant $\phi$-value. So it projects to a regular neighbourhood of a horizontal arc
in the plane.

\subsection{Location of the 2-handles}

Aside from the 2-handles lying in $N(K)$, each other 2-handle is a thickened 2-cell, and this 2-cell is 
made up of a union of tiles. So it suffices to explain the location of each tile.
A tile has a product foliation. Each leaf has a constant $\theta$-value and runs between two vertices.
Thus, as long as $\theta \not= 0$, the leaf specifies a vertical arc in the diagram. Thus, when the leaf
where $\theta = 0$ is removed, the tile is a union of such arcs.

The two 2-handles in each component of $N(K)$ interpolate between an arc in $\partial N(K)$ of slope $m$ or $l$
and a corresponding arc in $\partial N_{\textrm{small}}(K)$. Initially, before the hierarchy is constructed,
$K$ is in an arc presentation with arc index $n$ and $N(K)$ is a thin regular neighbourhood of it.
If this had remained the case, it would be easy to realise the 2-cells in $N(K)$ as a union of tiles
with binding weight at most $2n$. However, as we build the hierarchy $H$, we modify $N(K)$ in two ways.
One modification is to perform exchange moves and cyclic permutations to $K$, and these have a corresponding
effect on $N(K)$ and $N_{\textrm{small}}(K)$. This modification has no effect on the binding weight of
the 2-cells in $N(K)$. The other modification is to perform wedge insertions along $\partial N(K)$.
These expand $N(K)$ but leave $N_{\textrm{small}}(K)$ unchanged. So at each wedge insertion,
the 2-handles in $N(K)$ might expand. This expansion increases the binding weight of their constituent
2-cells by at most $2$. So the total binding weight after these modifications is at most $2q(n)$, because
$q(n)$ is an upper bound for the number of wedge insertions.

\subsection{Shrinking the 0-handles}

As described above, it is natural to take each 0-handle so that it projects to the regular neighbourhood of
a horizontal arc in the plane. It is convenient to shrink this 0-handle, so that it actually projects to a little disc,
with the property that the interior of the disc is disjoint from the projection of the interior of the 1-handles.
We need to extend the 1-handles attached to this 0-handle, but this can be achieved by adding a regular
neighbourhood of a horizontal arc to each endpoint of each 1-handle.

\subsection{From slides to Reidemeister moves}
\label{Subsec:SlidesToRMs}

We have picked a sequence of slides across handles and handle-preserving isotopies that takes the initial copy of $K$ in $\mathcal{H}$
to the copy $K'$ that lies in a 0-handle. We can realise each of these moves in 3-space. Projecting, we
obtain a sequence of diagrams that starts with $D$ and ends with $D'$. We therefore
need to bound the number of Reidemeister moves that relate successive diagrams in this
sequence. We will arrange that these diagrams are actually rectangular diagrams.

It will useful to have an upper bound $d$ on the arc index of in each rectangular diagram of this sequence.
Each diagram arises from a fixed realisation of the link in the handle structure $\mathcal{H}$.
Each 0-handle of $\mathcal{H}$ is a copy of $D^2 \times I$ and is embedded in $(0,2\pi) \times (0,2\pi) \times (0,1)$ 
in a way that respects their product structures. So, the intersection of the link $K$ with
this 0-handle has a specific projection to $(0,2\pi) \times (0,2\pi)$. In particular, this projection consists of
a bounded number of straight arcs, and we may assume that each of these is vertical or horizontal. This bound depends only on the arrangement of $K$
in $\mathcal{H}$, and does not depend on the way that $\mathcal{H}$ is embedded in 3-space.
We now consider the intersection of the link with a 1-handle of $\mathcal{H}$. As observed in Section \ref{Subsec:Proj1Skel},
this 1-handle is a regular neighbourhood of an arc $\alpha$ that is in an arc presentation,
with arc index bounded above by the binding weight of the hierarchy. Now, $K$ respects
the product structure on this 1-handle, and therefore each arc of intersection between $K$
and the 1-handle runs in a thin regular neighbourhood of $\alpha$. Its diagrammatic
projection is therefore a concatenation of at most $w$ horizontal and at most $w$ vertical arcs, where $w$ is
the binding weight. We may also ensure that the projections of the 0-handles
are disjoint from each other. We can also arrange that the projection of each 1-handle
avoids the projection of each 0-handle, except at the points where the 1-handle is
attached to a 0-handle. However, distinct 1-handles may have projections that overlap.
We therefore deduce that each rectangular diagram in this sequence has arc index at most $d = c_1 w + c_2$, 
where $c_1$ and $c_2$ are constants depending only on $K$, and $w$ is
the binding weight of the hierarchy. 

We now find an upper bound on the number of Reidemeister moves that result from each of
the moves described in Section \ref{Sec:IsotopeThroughHS}. We start with the process of sliding along a 1-handle,
because this is the simplest. Here, the knot remains in the union of the 0-handles and 1-handles. Moreover, its
intersection with the 0-handles barely changes. In particular, the projections of the part of
the link that lies in the 0-handles does not change, apart from near a small regular neighbourhood
of the 1-handle that is involved in the move. As $K$ is
slid through the 1-handle, a series of type 2 Reidemeister moves are performed.
These occur when the core curve $\alpha$ of the 1-handle and some other point of $K$
have the same projection. Therefore, the number of type 2 Reidemeister moves is at most
the number of points of intersection between the projections of $\alpha$ and $K$, which is at most
$2wd$.

We next consider the case of handle-preserving isotopy. This changes the position of the
link within the 0-handles and the 1-handles. However, we can perform these isotopies
separately, first of all changing the position with the 0-handles and then changing the
position within a small regular neighbourhood of the 1-handles. The isotopies supported
within the 0-handles have the effect of isotoping the link, and in so doing, result in
Reidemeister moves. The way that the diagram changes within the image of a 0-handle
depends only on the choice of isotopy within $\mathcal{H}$, and not on the way that it is
embedded. Thus, we deduce that the number of Reidemeister moves for this part of
the process is a constant, independent of $w$. The next part of the process isotopes
the link in a regular neighbourhood of the 1-handles, respecting the product structure
on the 1-handles throughout. This may be achieved using a Reidemeister 2 move,
followed by several Reidemeister 3 moves. The
number of type 3 moves is equal to the number of points of projection between the
core $\alpha$ of the 1-handle and the projection of the link in the remaining handles.
Hence, the number of moves is at most a quadratic function of $w$.

Finally, we consider the situation where the link is slid across a 2-handle. It is convenient
to consider the reverse of this move. This will require the same number of Reidemeister moves.
The 2-handle is a thickening of a 2-cell $C$ that is nearly admissible. Initially, $K$ runs
around $\partial C$. We will isotope $K$ across $C$, one tile at a time. Let $C_-$ be the
union of the tiles that have not been moved across. Then we will arrange that $C_-$ is
a disc (except after the final step, when it is empty). We will also arrange that $K$ runs
around $\partial C_-$.

We may pick a tile $T$ in $C_-$, such that after removing this tile from $C_-$, the result
is still a disc. This is for the following reason. Consider any tile that shares at least one edge with $\partial C_-$.
If we were to remove this tile from $C_-$ and create something other than a disc, then at least one edge
of the tile would start and end on $\partial C_-$ and with its interior lying in the interior of $C_-$.
Consider the union of all such arcs, over all tiles
sharing at least one edge with $\partial C_-$. Each is properly embedded in $C_-$ and the interiors
of these arcs are disjoint. So we may find such an arc $\beta$ that is outermost in $C_-$. It separates off
a sub-disc of $C_-$ with interior disjoint from the arcs. Pick any tile in that disc that shares at least
one edge with $\partial C_-$. (Avoid the tile containing $\beta$, unless it is the only tile in that sub-disc.)
Then we can let $T$ be this tile.

Suppose first that $T$ has exactly two edges of $\partial C_-$ in its boundary, and that these intersect
at a saddle of $C$. Then we may slide $K$ across this tile, maintaining it is an arc in a page throughout.
In the diagram, each such arc is vertical. Hence, we obtain a 1-parameter family of vertical arcs.
As the arc passes over some other vertical arc, this is achieved by an exchange move.
By Lemma \ref{Lem:ExchangeRM}, this can be performed using at most $d$ Reidemeister moves. It might be
the case that one of the arcs in the tile has $\theta = 0$. At this point, the arc in the diagram 
switches from a vertical arc with $\theta$ just above $0$ to one just below $2\pi$, or vice versa. This is
a cyclic permutation, using at most $d^2$ Reidemeister moves. Hence, in total, the number of
Reidemeister moves is at most $2d^2$.

Now suppose that $T$ has exactly one edge in $\partial C_-$. This runs between a saddle and a vertex.
Consider the other edge of the tile emanating from this saddle. Then we may introduce a finger to $K$
so that it runs along this edge and back again. This is achieved by sliding across a 1-handle.
In doing so, we need at most $2wd$ type 2 Reidemeister moves.
We can now slide $K$ across the tile, as above.

The situation where there are three or four edges of $\partial C_-$ is similar. We can slide $K$ across the
tile, and then perform the reverse of a finger move. When the tile has two edges of $\partial C_-$ in its
boundary, but they intersect in a vertex, we introduce a finger running along one the remaining edges,
then we slide across the tile, and then we remove a finger. This requires at most $2d^2 + 4wd$ type 2 Reidemeister moves.

Each 2-handle is a thickening of a 2-cell $C$, and this consists of at most $4w$ tiles, for the following
reason. The number of tiles in $C$ is at most the number of half-tiles plus twice the number of 
full square tiles. This is equal to $4x^i(C)$ and $2x^b(C)$, where $x^i(C)$ and $x^b(C)$ denote the number of
interior saddles and boundary saddles in the 2-cell. By Lemma \ref{Lem:EulerInequality},
this is equal to $4v^i(C) + 2v^b(C)  - 2$ where $v^i(C)$ and $v^b(C)$ denote the number
of vertices in the interior and boundary of $C$. This is at most $4w$.

Sliding across each tile uses at most $2d^2 + 4wd$ Reidemeister
moves, we have used at most $8d^2 w+ 16 w^2 d$ Reidemeister moves in total, which is at most a cubic function of $w$.

Thus, we have proved the following.

\begin{proposition} 
\label{Prop:ReidemeisterMovesToFixed}
Let $K$ be a non-split link, and let $H$ be a hierarchy
for the exterior of $K$. Let $D'$ be a fixed diagram for $K$.
Then there is a constant $k$, depending only on $K$, $H$ and $D'$
with the following property. If $K$ is an arc presentation, and $H$ is in admissible
form with binding weight $w$, then there is a sequence of at most
$kw^3$ Reidemeister moves taking the rectangular diagram for $K$ that
arises from the arc presentation to $D'$.
\end{proposition}

\subsection{Split links}
We are nearly in a position to prove the main theorem of this paper, Theorem \ref{Thm:Main}. But first we explain how to deal with
links that are split.

\begin{theorem}
\label{Thm:Split}
Let $L$ be a split link that is a distant union $L_1 \sqcup \dots \sqcup L_s$ of non-split links. Let $D$ be a rectangular diagram for $L$
with arc index $n$. Then there is a sequence of at most $(11n)^{11}$ exchange moves, cyclic permutations and destabilisations taking $D$
to a disconnected rectangular diagram, where the components are diagrams of $L_1, \dots, L_s$.
\end{theorem}

\begin{proof}
Theorem 1.4 of \cite{Lackenby:PolyUnknot} provides a sequence of at most $(14n)^{10}$ exchange moves, cyclic permutations and destabilisations taking $D$
to a disconnected rectangular diagram. Note that these moves do not increase the arc index of the rectangular diagram, and so this is still at most $n$. 
If any of the components of this diagram represents a split link, then at most $(14n)^{10}$ further such moves makes this disconnected.
Repeating in this way, we eventually reach the required rectangular diagram where the components are diagrams of $L_1, \dots, L_s$.
Since each component has arc index at least $2$, the number of times we repeat is at most $n/2$. So the total number of moves
is at most $(n/2)(14n)^{10} \leq (11n)^{11}$.
\end{proof}

\subsection{Proof of the main theorem}

\begin{proof}[Proof of Theorem \ref{Thm:Main}]
Let $D$ be some diagram for $K$ with crossing number $c$. In the case of the theorem we will consider $D = D_1$ or $D = D_2$.
Let $D'$ be some fixed diagram for $K$. We will provide an upper bound on the number of Reidemeister
moves needed to relate $D$ to $D'$. Hence, this will give an upper bound on the number of Reidemeister
moves needed to relate $D_1$ to $D_2$.

If $K$ is split, then write it as a distant union $L_1 \sqcup \dots \sqcup L_s$ of non-split links.
We may assume that $D'$ is a distant union of diagrams for $L_1, \dots, L_s$.
The given diagram $D$ can be isotoped to a rectangular diagram with arc index $n \leq (81/20)c$,
by Lemma \ref{Lem:IsotopeToRectangular}. By Theorem \ref{Thm:Split}, this may be
converted to a disconnected rectangular diagram, where the components are diagrams of $L_1, \dots, L_s$,
using at most $(11n)^{11}$ exchange moves, cyclic permutations and destabilisations.
It therefore suffices to consider these diagrams individually.

Thus, we may assume that $K$ is non-split. We may also assume that $K$ is not the unknot,
since the main result of \cite{Lackenby:PolyUnknot} establishes that $p_K(c) = (236c)^{11}$ works in this case.

Let $H$ be a fixed adequately separating, $(\mu,2^{15})$-exponentially controlled hierarchy for the exterior of $K$, 
which exists by Theorem \ref{Thm:WeaklyFundamentalHierarchy} and Remark \ref{Rem:AdequatelySeparating}.
Here, $\mu$ is the constant from Theorem \ref{Thm:ProducesWeaklyFundamentalHierarchy}. This determines a handle structure for
the exterior of $K$. Also let $m$ and $l$ be fixed curves on each component of $\partial N(K)$ that intersect once,
that have meridional and integral slope, and that respect the handle structure. These then
determine an associated handle structure $\calH$ for the 3-sphere. Let $f$ be the sum of the
absolute values of the integer slopes of $l$.

As above, the diagram $D$ can be isotoped to a rectangular diagram with arc index at most $(81/20)c$,
by Lemma \ref{Lem:IsotopeToRectangular}. We then perform stabilisations so that each component
of $K$ has writhe equal to the framing of the relevant component of $l$. This requires at most
$f+c$ stabilisations. Let $n \leq f + (101/20)c$ be the resulting arc index.
Theorem \ref{Thm:PolyBoundExchange} provides a polynomial $q$, which depends only on $K$, $H$ and $D'$, 
such that, after at most $n^2 q(n)$ cyclic permutations and exchange moves,
the hierarchy has been placed in admissible form with binding weight at most $q(n)$.
Each of these modifications to the diagram requires at most $n^2$ Reidemeister moves.

By Proposition \ref{Prop:ReidemeisterMovesToFixed}, this diagram can be converted to $D'$ using
at most $k (q(n))^3$ Reidemeister moves, where $k$ depends only on $K$, $H$ and $D'$.

So we have converted $D$ to $D'$ using at most the following number of Reidemeister moves:
\begin{align*}
& f + c + (f + (101/20)c)^4 q(f + (101/20)c) + k (q(f + (101/20)c))^3 \\
&\qquad \leq 2k (q(f + (101/20)c))^3.
\end{align*}
So we may set $p_K(c) = 2k (q(f + (101/20)c))^3$.
\end{proof}

\begin{remark}
The constant $k$ in the above proposition is algorithmically computable.
Using the hierarchy $H$ and the choice of meridians $m$ and integral curves $l$, the handle structure $\mathcal{H}$ for the 3-sphere may be constructed.
The initial position of the link $K$ within $\mathcal{H}$ is determined. There is certainly
some isotopy of the 3-sphere taking $K$ to a link $K'$ that lies within a 0-handle of $\mathcal{H}$.
Such an isotopy can always be found eventually, since the isotopy is a level-preserving simplicial
map $S^3 \times [0,1] \rightarrow S^3 \times [0,1]$ and all possible such simplicial maps can be enumerated.
Eventually one will be found that forms an isotopy taking $K$ into a 0-handle. We can then
let $K'$ be the resulting link in the 0-handle. The 0-handle has a product structure as $D^2 \times [0,1]$,
and the projection onto $D^2$ gives a diagram $D''$ for $K$, possibly after a small perturbation
of $K'$. The procedure that is used above for 
creating the moves from the isotopy is deterministic, and so it specifies a sequence of
Reidemeister moves taking $D$ to $D''$. Finally, since $D''$ and $D'$ are both diagrams of the
same link $K$, and so some sequence of Reidemeister moves relating them may always be found.

The final step is algorithmic, but it does not provide a good \emph{a priori} bound for the
number of Reidemeister moves relating $D''$ and $D'$. However, if instead we had initially
chosen $D'$ to be equal to $D''$, then this step would not have been required.
\end{remark}

\section{Explicit polynomials}
\label{Sec:Explicit}

We will want to determine an actual value for the constant $k$ for the figure-eight knot and for torus knots.
In order to do so, we will need a more quantified argument. 

In the previous section, we defined the handle structure of the 3-sphere arising from a hierarchy.
It will be helpful to introduce the following slightly more general version. As previously,
we start with the handle structure on the exterior of $K$ arising from the hierarchy and then attach handles in $N(K)$. Also as before,
we pick a union $l$ of integral curves, one on each component of $\partial N(K)$, that respects the handle structure on the boundary.
But now we pick multiple disjoint meridians $m$ on $\partial N(K)$, where we 
require each component of $\partial N(K)$ to contain at least one of these meridians, but maybe more than one. Each
of the meridians respects the handle structure and we require that each intersects $l$ exactly once.

We now give $N_{\textrm{small}}(K)$ a handle structure consisting of $|m|$ 0-handles and $|m|$ 1-handles,
where $|m|$ is the number of meridians. In $N(K) \cut N_{\textrm{small}}(K)$, we inserts $|m|$ vertical
1-handles, each running from one of the 0-handles of $N_{\textrm{small}}(K)$ to a point of intersection
between $m$ and $l$. Now insert $|m|$ 2-handles, each running around some component of $m$,
down a vertical 1-handle, around a 0-handle in $N_{\textrm{small}}(K)$ and then back up the vertical 1-handle.
Also insert $|m|$ 2-handles, each running along a component of $l \cut m$,
then along a vertical 1-handle, then along a 1-handle in $N_{\textrm{small}}(K)$ and then back up another
vertical 1-handle. Finally, we insert $|m|$ 3-handles into $N(K) \cut N_{\textrm{small}}(K)$.
We also say that this is a handle structure on the 3-sphere \emph{arising from the hierarchy}.
As before, it also depends on extra data, specifically the choice of integral curves $l$ and
meridians $m$.

We will also make an extra
hypothesis on the handle structure $\calH$ on the 3-sphere, which we now describe.

Let $\calH$ be a handle structure of 3-manifold satisfying Convention \ref{Convention:HS}.
A handle structure $\calH'$ is obtained from $\calH$ by a \emph{collapse} if:
\begin{enumerate}
\item there are handles
$H_-$ and $H_+$ of $\calH$, where the index of $H_+$ is one more than the index of $H_-$,
and where the core of $H_+$ intersects the co-core of $H_-$ once; $H_- \cup H_+$ is therefore a 3-ball; 
let $D_-$ be the union of points in $\partial (H_- \cup H_+)$ where $H_-$ and $H_+$
are attached to handles of lower index, and let $D_+$ be $\partial (H_- \cup H_+) \cut D_-$;
\item the handles $H_-$ and $H_+$ are removed;
\item any handles of higher index that are incident to $D_+$ are instead
attached along $D_-$ using the natural identification between $D_-$ and $D_+$.
\end{enumerate}
Thus, we permit handle pairs to be collapsed even when handles of higher index are
attached to them. We say that a collapse is \emph{elementary} if, in the above definition,
$D_+$ lies in the boundary of the 3-manifold, or equivalently, $H_-$ and $H_+$ are
disjoint from any other handles of higher index.
We say that $\calH$ is \emph{completely collapsible} if there is a sequence of collapses that takes
it to a single 0-handle.

\begin{proposition} 
\label{Prop:ElementaryCollapses}
If a handle structure $\calH$ of a 3-manifold is collapsible, then there is a sequence of elementary collapses that
takes it to a single 0-handle.
\end{proposition}

\begin{proof}
This is by induction on the number of handles. The induction starts trivially when $\calH$ has a single handle,
which must be a 0-handle, and so no elementary collapses are required.

So consider a collapsible handle structure $\calH$, and suppose that the lemma is true for all collapsible 
handle structures with fewer handles than $\calH$.

By assumption, $\calH$ has a sequence of collapses taking it to a 0-handle. We claim that there is such
a sequence where all the 3/2 collapses happen first, then all the 2/1 collapses and then all the 1/0 collapses.
This just follows from the easy observation that we may interchange the order of successive collapses
if a 1/0 collapse happens just before a 3/2 or 2/1 collapse or a 2/1 collapse happens before a 3/2 collapse.
The handle structure that results after performing such a pair of collapses does not depend on the
order in which they happen. So, we now assume that the collapses to $\calH$ are ordered in this way.

We may similarly interchange the order of the 3/2 collapses, and we will feel free to do so, for the following reason. We form
a graph, using only the 2-handles and 3-handles that are removed in the 3/2 collapses.
The vertices of the graph correspond to the 3-handles and to the components of intersection between the
boundary of the manifold and the 2-handles that are involved in the 3/2 collapses. The edges of the graph
correspond to the 2-handles that are removed in these collapses. Each 2-handle that is removed in 
the 3/2 collapses must be incident to at least one 3-handle. When a 3/2 collapse is performed,
the homotopy type of this graph is unchanged, except when a 3/2 collapse removes an entire
component of the graph. Hence, the graph is forest. In particular, it has more vertices than edges.
However, there is the same number of 2-handles as 3-handles involved in the 3/2 collapses.
Therefore, there is a vertex corresponding to a component of intersection between the manifold boundary
and a 2-handle involved in a 3/2 collapse.
Hence, this 2-handle and the adjacent 3-handle admit an elementary
collapse. In this way, there is a sequence of elementary collapses that removes all the 3-handles.

We now consider the first 2/1 collapse. Say this involves handles $H_2$ and $H_1$, and let
$\calH'$ be the resulting handle structure. If this was the final 2/1 collapse, then it was in fact
an elementary collapse. Then, inductively, $\calH'$ admits a sequence of elementary collapses
taking it to a 0-handle. So, we may assume that the next collapse is also a 2/1 collapse.
Inductively, $\calH'$ admits a sequence of elementary collapses taking it to a 0-handle.
So, we may assume that the next collapse to $\calH'$ is an elementary collapse, involving
handles $H'_2$ and $H'_1$. Note that $H'_1$ and $H'_2$ are also handles in $\calH$.

Suppose first $H'_2 \cup H'_1$ is disjoint from $H_2 \cup H_1$ in $\calH$. Then we may interchange the order
of the collapses. Let $\calH''$ be the handle structure obtained from $\calH$ by collapsing
$H_2'$ and $H_1'$. Note that $\calH''$ is obtained from $\calH$ by an elementary collapse.
Since $\calH''$ collapses to $\calH'$, it is also totally collapsible. Hence, by induction,
there is a sequence of elementary collapses taking $\calH''$ to a 0-handle, as required.

So suppose now that $H_2' \cup H'_1$ intersects $H_2 \cup H_1$. Since handles of the same
index are disjoint, there are two possibilities: $H_2'$ runs over $H_1$, or $H_2$ runs over $H'_1$.

Suppose that $H_2'$ runs over $H_1$, but that $H_2$ does not run over $H_1'$. The collapse of $H_2$
and $H_1$ does not change the components of intersection between $H'_1$ and the 2-handles.
So, in $\calH$, $H'_1$ only has $H'_2$ running over it once and is not incident to any other 2-handles.
Hence, $H'_1$ and $H'_2$ admit an elementary collapse in $\calH$. The resulting handle structure is
completely collapsible, since we may collapse $H_2$ and $H_1$ and then follow the remaining collapses that
were applied to $\calH$. Hence, by induction, it completely collapses
using only elementary collapses. So the same is true of $\calH$.

Now suppose that $H_2$ runs over $H_1'$. The collapse of $H_1$ and $H_2$ then introduces
new components of intersection between $H_1'$ and the 2-handles. But in $\calH'$, $H_1'$ intersects
the 2-handles just once. We therefore deduce that in $\calH$, $H_1$ either is disjoint from 
the 2-handles other than $H_2$ or has just a single component of intersection with these handles.
In the former case, the collapse of $H_1$ and $H_2$ is an elementary collapse, and the
proof is complete then. On the other hand, when $H_1$ intersects the other 2-handles exactly once,
this intersection must be with $H_2'$. We deduce that $H_2$ runs over $H_1$ once and $H_1'$ once,
and that $H_2'$ runs over $H_1$ once. Furthermore, $H'_1$ intersects the 2-handles just once,
in $H_2$. So we perform an elementary collapse to $H_2$ and $H'_1$. The resulting handle structure
is completely collapsible, since we may then collapse $H_1$ and $H'_2$, and then perform the
remaining collapses that were applied to $\calH' \cut (H'_1 \cup H'_2)$. Hence, again inductively,
$\calH$ admits a sequence of elementary collapses taking it to a 0-handle.

Thus, we are reduced to the final case where only 1/0 collapses are performed. Thus,
$\calH$ is a thickened tree. In particular, it may collapsed to a single 0-handle by collapsing
the leaves of the tree one at time. This is a sequence of elementary collapses.
\end{proof}

\begin{proposition}
\label{Prop:CollapsibleHierarchy}
Let $K$ be a non-split link, and let $H$ be a hierarchy
for the exterior of $K$. Let $l$ be a collection of integral curves on $\partial N(K)$,
one on each component of $\partial N(K)$. Let $m$ be a collection of disjoint meridians,
with at least one on each component of $\partial N(K)$. We require that these curves
respect the handle structure and that each component
of $m$ intersects $l$ in a single point.
Let $\calH$ be the associated handle structure of the 3-sphere.
Suppose that, when some 3-handle is removed from $\calH$, the result
is completely collapsible. This determines an isotopy of the 3-sphere
taking the complement of the 3-handle to a 0-handle of $\mathcal{H}$.
Let $K'$ be the image of $K$ after this isotopy, and suppose that the projection of
the 0-handle $D^2 \times I$ onto the first factor gives a diagram $D'$
for $K'$. Let $a$ be the number of 1-cells in this cell structure, and
let $b$ be the maximal number of 1-cells that any 2-cell runs over, counted with multiplicity.
Then the constant $k$ in Proposition \ref{Prop:ReidemeisterMovesToFixed} may be taken to be $216a^4b^{3a+3}$.
\end{proposition}

\begin{proof}
By Proposition \ref{Prop:ElementaryCollapses}, we may remove the 3-handle from $\calH$ and
then perform a sequence of elementary collapses to reach a single 0-handle.

When a 2-handle and a 3-handle are collapsed, this has no effect on the link, since the link
respects the handle structure. This collapse also does not affect the number of 1-cells
that the remaining 2-cells run over. So we ignore these collapses. Each of the
remaining collapses removes a 1-cell, and so there are $a$ of these. 

After the $i$th such collapse, each 0-handle intersects $K$ in a tangle, which will project to a union
of vertical and horizontal arcs. Let $t_i$ be the total number of these arcs. 
Initially, exactly $|m|$ of the tangles contain three vertical and horizontal arcs,
and the rest are empty. So initially, $t_0 = 3|m| \leq 3a$.

Let $m_i$ be an upper bound for the number of times that the link runs
over each 1-handle after the $i$th collapse. This only increases when a 1-handle and a 2-handle are collapsed.
Before such a collapse,
the link needs to be slid off the 1-handle. Each such slide may increase the number
times it runs over the other 1-handles by at most $b$. Initially, $m_0 = 1$. Then $m_{i+1} \leq b m_i$. Hence, after $i$ collapses,
we have $m_i \leq b^i$.

After all the 3-handles have been removed, we may then perform all the collapses involving
a 1-handle and a 2-handle. As discussed above, before the collapse, we first slide at most $m_i$ arcs of $K$ across
the 2-handle. Each slide introduces at most $b$ components to the tangles, and each such component
is a concatenation of three vertical and horizontal arcs.
Hence, $t_{i+1} \leq t_i + 3bm_i \leq t_i + 3b^{i+1}$. So inductively, $t_i$ is at most $3a + 3b^{i+1}$, which is at most $3ab^{i+1}$.
Therefore, the arc index is at most $t_i + 2m_iw \leq 3ab^{i+1}w$. Set $d_i$ to be this upper bound $3ab^{i+1}w$.

Consider the case where a 1-handle and a 2-handle are collapsed. Before the collapse, we first slide at most $m_i$ arcs of $K$ across
the 2-handle. As explained in the proof of Proposition \ref{Prop:ReidemeisterMovesToFixed}, the number of Reidemeister moves performed
in each such slide is at most $8d_i^2w + 16 w^2  d_i\leq 24 d_i^2 w \leq 216a^2b^{2i+2}w^3$.

Consider now when a 0-handle and a 1-handle are collapsed. We retract the 1-handle, taking the 0-handle with it.
In doing so, we pass at most $t_i$ arcs in the projection of the 0-handle past at most $d_i$ other arcs,
at most $w$ times. This is achieved using at most $t_id_iw$ exchange moves, each of which requires at most $d_i$
Reidemeister moves. Hence, the number of Reidemeister moves is at most $t_i d_i^2 w \leq 27a^3b^{3i+3} w^3$.
After this, the two 0-handles are merged into a single 0-handle, and the diagrams in these 0-handles
are merged into a single diagram. This does not increase $t_i$ or $m_i$.

We deduce that the total number of Reidemeister moves is at most $216a^4b^{3a+3}w^3$.
\end{proof}

\begin{remark}
Each 1-cell of the cell structure, other than the 1-cells with interior in $\mathrm{int}(N(K))$, arises where three surfaces of the hierarchy meet
or where a surface in the hierarchy meets $\partial N(K)$.
At the endpoints of such a 1-cell are points where some $S_i$ meets the pattern $P_i$. Also taking into
account the 1-cells with interior in $\mathrm{int}(N(K))$, we have
$$a = 2|m| + 2 \sum_i |S_i \cap P_i|.$$
Moreover, each 2-cell lying in the hierarchy or $\partial N(K)$ can run over each 1-cell at most once. There are
other 2-cells lying in $N(K)$, and these run over the 1-handles at most $\textrm{length}(l) + 3$ and
$\textrm{length}(m) + 2$ times. 
Hence, 
$$b \leq \max \{ a, \textrm{length}(l) + 3, \textrm{length}(m) + 2 \}.$$
\end{remark}

Combining these estimates with Theorem \ref{Thm:PolyBoundExchange}, Proposition \ref{Prop:CollapsibleHierarchy}
and the proof of Theorem \ref{Thm:Main}, we obtain the following.

\begin{theorem}
\label{Thm:CollapsiblePolynomial}
Let $K$ be a non-split link other than the unknot.
Let $S_1, \dots, S_\ell$ be a hierarchy for the exterior of $K$ as in Proposition \ref{Prop:CollapsibleHierarchy}
and that is adequately separating and $(\lambda, 2^{15})$-exponentially controlled, where
$$\lambda = \sum_{i=1}^\ell -3\chi(S_i) + 6 |S_i| + 6|S_i \cap P_i|.$$
Let $m$ and $l$ be curves on $\partial N(K)$ as in Proposition \ref{Prop:CollapsibleHierarchy},
and let $f$ be the sum of the absolute values of the integral slopes of $l$. Suppose that the resulting
handle structure of the 3-sphere is totally collapsible once some 3-handle is removed.
Let
\begin{align*}
a &= 2|m| + \sum_{i=1}^\ell 2|S_i \cap P_i|, \\
b &= \max \{ a, \textrm{length}(l) + 3, \textrm{length}(m) + 2 \}
\end{align*}
Then we may set 
$$p_K(c) = 512a^4b^{3a+3} (10^{80} \lambda^8)^{3 \sum_{i=0}^\ell 22^i} (f + (101/20)c)^{3 \cdot 22^\ell}.$$
\end{theorem}

\section{The figure-eight knot}
\label{Sec:Fig8}

\begin{theorem}
For $K$ the figure-eight knot, we may set $p_K(c) = (10^{108} c)^{15460896}$.
\end{theorem}

\begin{proof}
The figure-eight knot is shown in Figure \ref{Fig:HierarchyFig8} (i). We now construct an adequately separating, 
$(\lambda, 2^{15})$-exponentially controlled hierarchy for
its exterior $M$, using the procedure given in Section \ref{Sec:Hierarchies}. The length $\ell$ of this hierarchy will be 5.
The constants $\lambda$ and $a$ defined in Theorem \ref{Thm:CollapsiblePolynomial} will be $\lambda=492$ and $a = b = 128$. 
The framing $f$ is $4$.
The resulting handle structure on $S^3$ will be totally collapsible,
once a suitable 3-handle has been removed.
Hence, by Theorem \ref{Thm:CollapsiblePolynomial},
$$512a^4b^{3a+3} (10^{80} \lambda^8)^{3 \sum_{i=0}^\ell 22^i} (f + (101c/20))^{3(22)^\ell}$$
would work as $p_K(c)$. This is bounded above by $(10^{108} c)^{15460896}$, which we may set as $p_K(c)$.

\begin{figure}[h]
\includegraphics[width=4in]{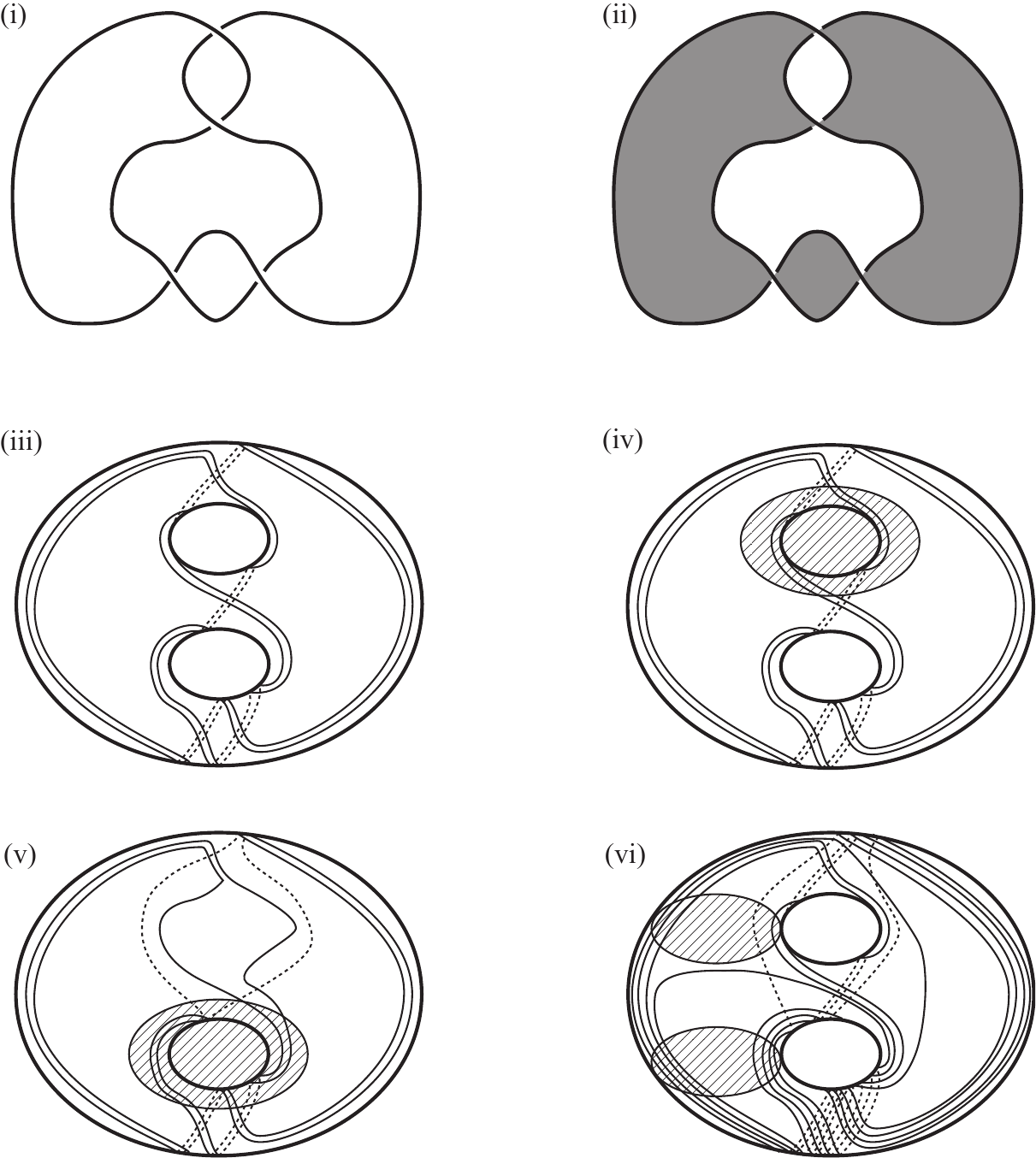}
\caption{Some of the surfaces used to build a weakly fundamental hierarchy for the figure-eight knot exterior}
\label{Fig:HierarchyFig8}
\end{figure}

Since $M$ is hyperbolic, we may jump straight to Step (3) in Section \ref{Sec:Hierarchies}. Let $S_1$ be the orientable double cover of the surface shown in Figure \ref{Fig:HierarchyFig8} (ii). Note that $-\chi(S_1) + 2|S_1| + 2|S_1 \cap P_1| = 4$. It was shown by Hatcher and Thurston \cite{HatcherThurston} that $S_1$ is incompressible and boundary-incompressible. Note that it is not a fibre in a fibration of $M$ over the circle, since it is separating and connected. Thus, we are permitted to cut along this surface in Step (3). This surface intersects $\partial N(K)$ into two curves, each with integral slope $4$.

This creates a manifold $M_2$, which has two components. One component,  which we denote by $M_2'$, consists of the exterior of the surface shown in Figure \ref{Fig:HierarchyFig8} (ii). The other component, which we denote by $M_2''$ is a regular neighbourhood of this checkerboard surface. Now $M_2''$ is an $I$-bundle chamber, and so it cannot be decomposed at this point. The component $M_2'$ is shown in Figure \ref{Fig:HierarchyFig8} (iii). It has a non-trivial JSJ decomposition, as follows. It contains a properly embedded disc that intersects the pattern four times, as shown in Figure \ref{Fig:HierarchyFig8} (iv). A regular neighbourhood of this disc is an $I$-bundle over the disc. Attach onto this $I$-bundle a regular neighbourhood of $\partial M_2' \cap \partial N(K)$. The result is an $I$-bundle over a surface, which we denote by $M_3'$. One can compute that the base surface is a once-punctured M\"obius band. Let $M_3''$ be $M_2' \cut M_3'$. This is shown in Figure \ref{Fig:HierarchyFig8} (v). Then $M_3'$ and $M_3''$ intersect in an annulus $S_2$, which we claim is the JSJ annulus for $M_2'$. Note that $M_3''$ is a solid torus, and its boundary pattern consists of two curves, each of which winds three times around the solid torus. Thus, $M_3''$ is Seifert fibred, whereas $M'_3$ is an $I$-bundle, and so these do indeed form the pieces of the JSJ decomposition of $M_2'$. As required by Step 3 (iii), we let this annulus $S_2$ be the next surface in the hierarchy.

Since $M_3''$ is a Seifert fibred solid torus with boundary pattern equal to a union of fibres, it is simple. Hence, Step 4(v) applies to it. We therefore let $S_3$ be two parallel copies of a meridian disc for this solid torus, each of which can be chosen to intersect the pattern 6 times.

The $I$-bundle $M_2''$ has now inherited some boundary pattern in its horizontal boundary. One can check that this contains no clean essential annulus, and hence is simple. Thus, Step 4(v) can be applied to it. We decompose along $S_4$, two parallel copies of the discs shown in Figure \ref{Fig:HierarchyFig8} (vi). Two of these discs intersects the pattern 8 times and two intersects the pattern 6 times. 

The only remaining component that is not a 3-ball is $M_3'$, which is an $I$-bundle over a once-punctured M\"obius band. Again one can check that this contains no clean annuli. We can decompose along $S_5$, which is two parallel copies of two vertical discs. Each of these discs each intersects the pattern 4 times. We choose the location of these discs carefully as follows. As observed above, $\partial N(K)$ intersected $S_1$ in two curves, each with slope $4$. These divide $\partial N(K)$ into two annuli. There is a product structure on $N(K)$ as the product of an annulus $A$ and an interval $I$ for which $A \times \partial I$ is a copy of each of these annuli. When further decompositions were made, one of these annuli inherited further boundary pattern. Specifically, when we decompose along $S_4$, one of the annuli inherits 8 new arcs of pattern, which come in four parallel pairs. We choose the eight discs $S_5$ so that the eight arcs $S_5 \cap (A \times \partial I)$ are almost parallel to the existing eight arcs of pattern, via the product structure $A \times I$.

After these decompositions, the resulting manifolds are 3-balls and so the hierarchy terminates.

We may set $l$ to be a curve of slope $4$ running along one of the components of $\partial S_1$. Its length is $16$. The framing $f$ is therefore $4$. We pick $m$ to consist of $8$ meridians, each running along an arc in $\partial S_5$, an arc in $\partial S_4$ and two short arcs in $\partial S_1$. We chose the location of $S_5$ carefully so that we could set $m$ to be these 8 meridians.

We now briefly explain how to verify the collapsibility hypothesis in Proposition \ref{Prop:CollapsibleHierarchy}. We must first pick a 3-handle to be removed. We pick one of 3-handles lying in $M_3''$, the one containing the point at infinity in Figure \ref{Fig:HierarchyFig8}.

We will shortly give the sequence of collapses that takes the resulting handle structure to a single 0-handle. But first we present a simplified example.
Consider the handle structure of the 3-ball shown in Figure \ref{Fig:HierarchyFig8Collapse}. This is a thickening of the following cell structure. Start with a standardly embedded closed orientable genus two surface. Then attach compression discs for its complementary handlebodies. We use two compression discs in each handlebody, as shown in Figure \ref{Fig:HierarchyFig8Collapse}. These discs decompose the `inside' handlebody into a 3-ball, which we take to be a 3-cell. The boundaries of these discs decompose the genus two surface into an annulus. We pick an essential arc in this annulus,
as shown in Figure \ref{Fig:HierarchyFig8Collapse}, and declare that this is a 1-cell. Thus, the surface has now been decomposed to a disc, which we take to be a 2-cell of the cell structure. It is not hard to collapse this cell structure, as follows. First, collapse the 2-cell which lies in the genus two surface with the inside 3-cell. Then collapse each compression disc along with a 1-cell in its boundary. Then finally retract the resulting tree to a point.

\begin{figure}[h]
\includegraphics[width=2in]{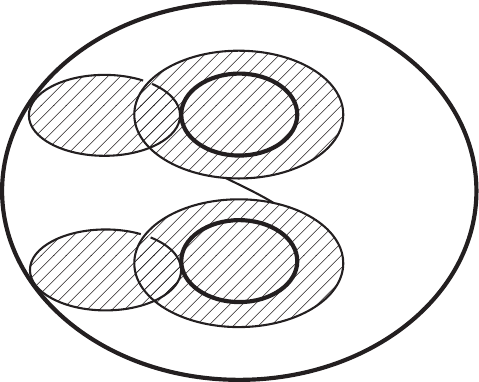}
\caption{A collapsible handle structure of a 3-ball}
\label{Fig:HierarchyFig8Collapse}
\end{figure}

Our aim is to retract the given handle structure of the 3-ball to the one shown above. We start with the handle structure on $N(K)$. We collapse the eight vertical 1-handles in $N(K) \cut N_{\textrm{small}}(K)$ with the eight 0-handles intersecting $K$, and collapse the 1-handles
through which $K$ runs with the 2-handles incident to $l$. We also collapse the sixteen 1-handles that contain the arcs of intersection between $m$ and $\partial S_1$. After these collapses, the handle structure on $N(K)$ is quite simple. It consists of sixteen 0-handles and sixteen 1-handles, which combine to give two thickened circles running along $\partial S_1$. It consists of 16 further 1-handles, eight lying in each component of $A \times \partial I$. It also consists of eight 2-handles, each forming a meridian, as well as eight further 2-handles in each component of $A \times \partial I$. Finally, it contains eight 3-handles. We may collapse this handle structure first of all onto one component of $A \times \partial I$, and then onto a single component of $\partial S_1$. After this, it just consists of eight 0-handles and eight 1-handles glued in a circular fashion.

We now consider $M'_3$, which was an $I$-bundle that has been decomposed along eight vertical discs. After this decomposition, it has become seven balls, each of which is an $I$-bundle over a disc. One of these balls contains the disc shown in Figure \ref{Fig:HierarchyFig8} (iv), and in the boundary of that ball, there are  two copies of that disc. We collapse each of the remaining six balls onto one of their components of intersection with $S_1 \cap S_2$. We then collapse the remaining 3-ball onto one of the copies of the disc shown in Figure \ref{Fig:HierarchyFig8} (iv).

The two discs in $S_3$ are parallel, with a 3-handle between them. We collapse one of these discs with the 3-handle. Similarly, the discs $S_4$ come in two pairs, with each pair cobounding a 3-handle. We collapse these two 3-handles with two of the discs. 

The handle structure that we have reached is very close to the simple one described above. We have the genus two surface, to which four compression discs as shown in Figure \ref{Fig:HierarchyFig8Collapse} are attached. The genus two surface is subdivided into 2-cells by its intersection with the four discs, but also by other remaining cells, for example the ones running along $N(K)$. But we can collapse these extraneous cells, until we have exactly the handle structure in Figure \ref{Fig:HierarchyFig8Collapse}. As described above, this is completely collapsible.
\end{proof}

\section{Torus knots}
\label{Sec:TorusKnots}

\begin{theorem}
When $K$ is a torus knot, we may set $p_K(c) = (10^{80}c)^{21962}$.
\end{theorem}

\begin{proof}
Let $K$ be an $(r,s)$-torus knot with $0 < r < s$.
We present an exponentially controlled hierarchy for the exterior of $K$ with length 3. 
The exterior of $K$ is Seifert fibred and so we give $\partial N(K)$ boundary pattern that is a single longitude.
The torus knot $K$ lies on a standardly embedded torus. The intersection between this torus and the
exterior of $K$ is an annulus, which we take to $S_1$. It intersects the longitude $2rs$ times. By Proposition \ref{Prop:AnnuliWeaklyFundamental}, this is weakly fundamental and hence, by Corollary \ref{Cor:WeaklyFundExpControlled}, it is $(2 + 4rs,2^{15})$-exponentially controlled. This decomposes the exterior of $K$ into two solid tori. Pick one of them, and let $S_2$ be two copies of a meridian disc for it that intersects the pattern as few times as possible. This number is $2r$, say. By Theorem \ref{Thm:MakeWeaklyFundamental}, it can be made weakly fundamental and hence $(4 + 8r,2^{15})$-exponentially controlled. The other solid torus now inherits a boundary pattern, where every complementary region is a disc. A meridian disc intersects this pattern $2s$ times. 
Let $S_3$ be two copies of this surface. This is weakly fundamental by Theorem \ref{Thm:MakeWeaklyFundamental} and hence is $(2 + 8s,2^{15})$-exponentially controlled.
For this hierarchy, $\lambda = 6 +12rs + 24r + 24s$. Thus, the polynomial $q$ provided by Theorem \ref{Thm:PolyBoundExchange} is
\begin{align*} q(x) &= (10^{80} (42 + 12rs + 24r + 24s)^8 )^{\sum_{i=0}^3 22^i } x^{22^i} \\
&= (10^{80} (6 +12rs + 24r + 24s)^8)^{11155} x^{10648}.
\end{align*}

Thickening this hierarchy gives a handle structure on the exterior of $K$. We extend this to a handle structure $\calH$ on $S^3$, using the method described in Section \ref{Subsec:HSFromHierarchy}. This requires a choice of simple closed curves $m$ and $l$ on $\partial N(K)$. For $l$, we pick a curve lying in $\partial N(S_1)$, which has
slope equal to a longitude plus $rs$ meridians. For $m$, we take it to be some meridional curve that respects the handle structure and that intersects $l$ once.

We will not use the estimate in Proposition \ref{Prop:CollapsibleHierarchy}, since this will not be strong enough. 
Instead, we will construct an isotopy of $K$ through $\calH$ explicitly. We first slide $K$ across the 2-handle 
that is vertical in $N(K) \cut N_{\textrm{small}}(K)$ and that runs along $l$. The 2-handle is a thickening of a 2-cell with binding weight at most 2$q(n)$, and hence comprised of at most $8q(n)$ tiles. As we slide over the 2-handle, the arc index of $K$ remains at most $d \leq n + 8q(n) \leq 9q(n)$. Each slide across a tile uses at most $2d^2+4wd$ Reidemeister moves, where $w$ is the binding weight of the hierarchy. So, the number of Reidemeister moves is at most $1296 (q(n))^3 + 288 w (q(n))^2$.

We then slide $K$ across the 1-handle that is vertical in $N(K) \cut N_{\textrm{small}}(K)$. This requires at most $2wd$ Reidemeister moves. After this, the knot runs along the 0-handles and 1-handles containing $l$. It runs along each such handle once.

Now consider one of the discs in $S_2$ and one of the discs in $S_3$.  Extend them a little so that their boundaries are on the standard torus containing $K$.
Their boundaries intersect once and hence the union of these boundaries is a wedge $W$ of two circles. We will isotope the knot so that afterwards it lies in a small regular neighbourhood of $W$. We can think of the standard torus as a square with side identifications, with the sides of the square form $W$. Currently, the knot lies on this torus as a curve of slope $(r,s)$. The two curves $\partial S_1$, together with $\partial S_2$ and $\partial S_3$ divide the torus into discs. We can slide the knots across these discs until it lies in the boundary $W$ of the square. There are $r+s+1$ of these discs. The knot is divided up into $r+s$ arcs. Each arc of $K$ is slid across each 2-cell at most once. So we need to perform at most $(r+s+1)(r+s)$ slides of $L$ across 2-cells to get into a regular neighbourhood of the wedge of circles. It then goes $r$ times round one circle in the wedge and $s$ times round the other. We now slide $K$ across the discs bounded by these circles. We need to perform $r+s$ slides. As described in Section \ref{Subsec:SlidesToRMs}, each slide across a 2-handle is achieved by at most $24w^3$ Reidemeister moves. Here, $w$ is the binding weight of the hierarchy. After this, the knot lies in a single 0-handle as required. Its projection is in fact the standard diagram $D'$ for the torus knot with crossing number $(r-1)s$.

Thus, as argued in the proof of Theorem \ref{Thm:Main},
the number of Reidemeister moves needed to convert an arbitrary diagram $D$ to a standard diagram $D'$ is at most
$$rs + c + (rs + (101/20)c)^4 q(rs + (101/20)c) + k (q(rs + (101/20)c))^3,$$
where $k = 1602 + 24(r+s)(r+s+2)$. We can now use the fact that the crossing number of the torus knot is $(r-1)s$ and so $c \geq (r-1)s$. 
So $c \geq rs/2$ and $c \geq r+s+2$. So, $k \leq 24c^2$. Thus,
the number of Reidemeister moves is at most
\begin{align*}
& rs + c + (rs + (101/20)c)^4 q(rs + (101/20)c) + k (q(rs + (101/20)c))^3 \\ 
&\leq 3c + (141/20)^4 c^4 (10^{80} (6 + 12rs + 24r + 24s)^8)^{11155}  (141c/20)^{10648} \\
& \quad+  (1602 + 24c^2)(10^{80} (6 + 12rs + 24r + 24s)^8)^{33465} (141c/20)^{31944} \\
&\leq (10^{11} c)^{299666}.
\end{align*}
So we may set $p_K(c) = (10^{11}c)^{299666}$, as claimed.
\end{proof}

\bibliography{prm-biblio}

\begin{thebibliography}{10}

\bibitem{BaldwinSivek}
John~A. Baldwin and Steven Sivek.
\newblock On the complexity of torus knot recognition.
\newblock {\em Trans. Amer. Math. Soc.}, 371(6):3831--3855, 2019.

\bibitem{Bennequin}
Daniel Bennequin.
\newblock Entrelacements et \'equations de {P}faff.
\newblock In {\em Third {S}chnepfenried geometry conference, {V}ol. 1
  ({S}chnepfenried, 1982)}, volume 107-108 of {\em Ast\'erisque}, pages
  87--161. Soc. Math. France, Paris, 1983.

\bibitem{BirmanMenasco}
Joan~S. Birman and William~W. Menasco.
\newblock Studying links via closed braids. {IV}. {C}omposite links and split
  links.
\newblock {\em Invent. Math.}, 102(1):115--139, 1990.

\bibitem{Budney}
Ryan Budney.
\newblock J{SJ}-decompositions of knot and link complements in {$S^3$}.
\newblock {\em Enseign. Math. (2)}, 52(3-4):319--359, 2006.

\bibitem{CowardLackenby}
Alexander Coward and Marc Lackenby.
\newblock An upper bound on {R}eidemeister moves.
\newblock {\em Amer. J. Math.}, 136(4):1023--1066, 2014.

\bibitem{Cromwell}
Peter~R. Cromwell.
\newblock Embedding knots and links in an open book. {I}. {B}asic properties.
\newblock {\em Topology Appl.}, 64(1):37--58, 1995.

\bibitem{CullerShalen}
M.~Culler and P.~B. Shalen.
\newblock Bounded, separating, incompressible surfaces in knot manifolds.
\newblock {\em Invent. Math.}, 75(3):537--545, 1984.

\bibitem{Dynnikov}
I.~A. Dynnikov.
\newblock Arc-presentations of links: monotonic simplification.
\newblock {\em Fund. Math.}, 190:29--76, 2006.

\bibitem{Haken}
Wolfgang Haken.
\newblock Some results on surfaces in {$3$}-manifolds.
\newblock In {\em Studies in {M}odern {T}opology}, volume Vol. 5 of {\em
  Studies in Mathematics}, pages 39--98. Math. Assoc. America,, 1968.

\bibitem{HassLagarias}
Joel Hass and Jeffrey~C. Lagarias.
\newblock The number of {R}eidemeister moves needed for unknotting.
\newblock {\em J. Amer. Math. Soc.}, 14(2):399--428, 2001.

\bibitem{HatcherThurston}
A.~Hatcher and W.~Thurston.
\newblock Incompressible surfaces in {$2$}-bridge knot complements.
\newblock {\em Invent. Math.}, 79(2):225--246, 1985.

\bibitem{Hemion}
Geoffrey Hemion.
\newblock On the classification of homeomorphisms of {$2$}-manifolds and the
  classification of {$3$}-manifolds.
\newblock {\em Acta Math.}, 142(1-2):123--155, 1979.

\bibitem{Jaco-Rubinstein:PL-equivariant-surgery}
William Jaco and J.~Hyam Rubinstein.
\newblock P{L} equivariant surgery and invariant decompositions of
  {$3$}-manifolds.
\newblock {\em Adv. in Math.}, 73(2):149--191, 1989.

\bibitem{Johannson}
Klaus Johannson.
\newblock {\em Homotopy equivalences of {$3$}-manifolds with boundaries},
  volume 761 of {\em Lecture Notes in Mathematics}.
\newblock Springer, Berlin, 1979.

\bibitem{Kuperberg}
Greg Kuperberg.
\newblock Algorithmic homeomorphism of 3-manifolds as a corollary of
  geometrization.
\newblock {\em Pacific J. Math.}, 301(1):189--241, 2019.

\bibitem{Lackenby:Composite}
Marc Lackenby.
\newblock The crossing number of composite knots.
\newblock {\em J. Topol.}, 2(4):747--768, 2009.

\bibitem{Lackenby:PolyUnknot}
Marc Lackenby.
\newblock A polynomial upper bound on {R}eidemeister moves.
\newblock {\em Ann. of Math. (2)}, 182(2):491--564, 2015.

\bibitem{Matveev}
Sergei Matveev.
\newblock {\em Algorithmic topology and classification of 3-manifolds},
  volume~9 of {\em Algorithms and Computation in Mathematics}.
\newblock Springer, Berlin, second edition, 2007.

\bibitem{Murasugi:Braid}
Kunio Murasugi.
\newblock On the braid index of alternating links.
\newblock {\em Trans. Amer. Math. Soc.}, 326(1):237--260, 1991.

\bibitem{Waldhausen:Klasse}
Friedhelm Waldhausen.
\newblock Eine {K}lasse von {$3$}-dimensionalen {M}annigfaltigkeiten. {I},
  {II}.
\newblock {\em Invent. Math.}, 3:308--333; ibid. 4 (1967), 87--117, 1967.

\bibitem{Waldhausen:Gruppen}
Friedhelm Waldhausen.
\newblock Gruppen mit {Z}entrum und {$3$}-dimensionale {M}annigfaltigkeiten.
\newblock {\em Topology}, 6:505--517, 1967.

\bibitem{Waldhausen}
Friedhelm Waldhausen.
\newblock On irreducible {$3$}-manifolds which are sufficiently large.
\newblock {\em Ann. of Math. (2)}, 87:56--88, 1968.

\end{thebibliography}
\bibliographystyle{plain}

\end{document}